\documentclass[11pt,reqno]{amsart}
\usepackage{graphicx,amsmath,amsthm,amsfonts,amssymb,mathrsfs}
\usepackage[usenames,dvipsnames]{xcolor}
\usepackage{tikz}
\usepackage{hyperref}
\hypersetup{%
	colorlinks=true, linkcolor=blue,
	citecolor=ForestGreen
}
\usepackage[paper=letterpaper,margin=.9in]{geometry}
\usepackage[inline]{enumitem} % change enum. items
\usepackage{caption,subcaption}
\theoremstyle{plain}
\newtheorem{thm}{Theorem}[section]
\newtheorem{innercustomthm}{Theorem}
\newenvironment{customthm}[1]
{\renewcommand\theinnercustomthm{#1*}\innercustomthm}
{\endinnercustomthm}
\newtheorem{cor}[thm]{Corollary}
\newtheorem{prop}[thm]{Proposition}
\newtheorem{lem}[thm]{Lemma}
\theoremstyle{definition}
\newtheorem{defn}[thm]{Definition}
\newtheorem{rmk}[thm]{Remark}
\numberwithin{equation}{section}
\numberwithin{table}{section}
\allowdisplaybreaks

\makeatletter
\newcommand{\myitem}[1]{%
\item[#1]\protected@edef\@currentlabel{#1}%
}
\makeatother

\usepackage{acronym}
\acrodef{SHE}{Stochastic Heat Equation}
\acrodef{SBE}{Stochastic Burgers Equation}
\acrodef{SPDE}{Stochastic Partial Differential Equation}
\acrodef{KPZ}{Kardar--Parisi--Zhang}
\acrodef{6V}{Six Vertex}
\acrodef{ASEP}{Asymmetric Simple Exclusion Process}

\newcommand{\DefinText}[1]{\textbf{#1}}	%definition in text

\newcommand{\gen}{\mathsf{L}}		%generator
\newcommand{\sgn}{\text{sgn}}		%sign
\newcommand{\ind}{\mathbf{1}}		%indicator
\newcommand{\Ex}{\mathbb{E}}		%expectation
\renewcommand{\Pr}{\mathbb{P}}		%probability
\newcommand{\set}[1]{{\{#1\}}}		%set
\newcommand{\metric}{\text{d}_{C^{-1}(\R^2)}}
\newcommand{\TrP}{\Pr_{\overrightarrow{\text{S6V}}}}	%transition probability for S6V
\newcommand{\TrPr}{\Pr_{\overleftarrow{\text{S6V}}}}	%transition probability for reversed S6V
\newcommand{\TrPrc}{\mathbf{U}}							%transition probability for S6V, prescaled form, contour
\newcommand{\TrPrcfr}{\tilde{\mathfrak{F}}}				%the `fraction' in the formula
\newcommand{\TrPrct}{\tilde{\mathfrak{D}}}				%the time dependenece in the formula

\newcommand{\Gibbs}{\mathcal{SG}}			%Stochastic Gibbs states
\newcommand{\Zsv}{Z}						%Z_S6V
\newcommand{\Zasep}{Z_\textup{ASEP}}			%Z_asep
\newcommand{\Nasep}{N_\textup{ASEP}}			%N_asep
\newcommand{\filt}{\mathscr{F}}				%filtration
\newcommand{\mg}{M}							%the martingale increament: Z\xi		
\newcommand{\limZ}{\mathcal{Z}}				% limiting Z (that solves SHE)
\newcommand{\limH}{\mathcal{H}}				% limiting H (that solves KPZ eq)
\newcommand{\limU}{\mathcal{U}}				% limiting H (that solves SBE)
\newcommand{\limM}{\mathcal{M}}				% limiting M in linear MG prob
\newcommand{\limN}{\mathcal{N}}				% limiting N in linear MG prob

\newcommand{\Bdd}{\mathcal{B}_\e}					%expdecay stuff, random
\newcommand{\Decay}{\mathcal{G}_\e}					%expdecay stuff, random
\newcommand{\decay}{\gamma_\e}						%expdecay stuff, deterministic
\newcommand{\Zg}{Z_{\nabla}}						%gZ.Z
\newcommand{\tilZ}{\tilde{Z}}				% \til Z := \eta Z \eta Z
\newcommand{\Xbd}{\mathcal{X}_\text{bdd}}			%bounded stuff involves ZZ
\newcommand{\Yg}{\mathcal{Y}_{\nabla}}				%stuff involves gZZ+ZgZ, unconstraint
\newcommand{\Ygg}{\mathcal{Y}_{\nabla,\nabla}}		%stuff involves gZgZ
\newcommand{\Xg}{\mathcal{X}_{\nabla}}				%stuff involves gZZ+ZgZ, constraint
\newcommand{\Xgg}{\mathcal{X}_{\tilZ}}				%stuff involves gZgZ
\newcommand{\Xgmomt}{B_{\mathcal{X}_{\nabla}}}		% || \int \Xg(t,x) dt ||_2
\newcommand{\Xggmomt}{B_{\mathcal{X}_{\tilZ}}}		% || \int \Xgg(t,x) dt ||_2
\newcommand{\etacnt}{\eta_\text{c}}					%centered occupation variables

\newcommand{\move}{K}								% N(t,x)-N(t+1,x)
\newcommand{\moveCent}{\overline{K}}				% \move(t,x) - \Ex[\move(t,x)|\filt(t)]
\newcommand{\noise}{\xi}							%white noise
\newcommand{\den}{\rho}								%density
\newcommand{\RW}{R}									%random walk, centered, tilted
\newcommand{\Rw}{S'}								%random walk, tilted
\newcommand{\rw}{S}									%random walk, pre-processed
\newcommand{\hk}{\mathsf{p}}						%heat kernel for \RW
\newcommand{\hke}{\mathsf{p}_\e}					%heat kernel for \RW

\newcommand{\Circ}{\mathcal{C}}					%circiular contours
\newcommand{\zoneC}{\Gamma(t,\e)}				%the z1-contour
\newcommand{\zoneCLim}{\Gamma_{*}}				%the z1-contour, limiting
\newcommand{\Magic}{\mathcal{M}}				%the magic contour |z-1/2|=1/2
\newcommand{\MagicC}{\mathcal{M}'}				%modified \magic, slightly enlarged, and cut by |z|=1 near z=1
\newcommand{\magicC}{\mathcal{M}'}				%\magicC, with t-depdence
\newcommand{\Magicc}{\mathcal{M}''}				%modified \magic, slightly bumpbed near z=0
\newcommand{\magicc}{\mathcal{M}''}				%\magicc, with t-depdence
\newcommand{\zplv}{\mathcal{N}}					%level set of |z\pole(z)|
\newcommand{\Circc}{\tilde{\mathcal{C}}}		%circiular contours centered near 1/2
\newcommand{\pole}{\mathfrak{p}_\varepsilon}	%the `pole' function
\newcommand{\polee}{\mathfrak{q}_\varepsilon}	% z\pole(z)
\newcommand{\poleLim}{\mathfrak{p}_*}			%the `pole' function, limiting
\newcommand{\rad}{\mathfrak{u}}					%the `radius' function
\newcommand{\SG}{\mathbf{V}}					%semigroup
\newcommand{\SGe}{\mathbf{V}_\e}				%semigroup, scaled form
\newcommand{\SGFr}{\mathbf{V}^\text{fr}_\e}		%semigroup, free part
\newcommand{\SGIn}{\mathbf{V}^\text{in}_\e}		%semigroup, interacting part
\newcommand{\SGfr}{\mathfrak{F}}				%the `fraction' in the formula
\newcommand{\SGt}{\mathfrak{D}}					%the time dependenece in the formula
\newcommand{\SGfre}{\mathfrak{F}_\e}			%the `fraction' in the formula, scaled
\newcommand{\SGte}{\mathfrak{D}_\e}				%the time dependenece in the formula, scaled
\newcommand{\SGtee}{\mathfrak{D}_{\e}}			%the 1/t-th root of \SGte
\newcommand{\SGteLim}{\mathfrak{D}_*}			%\SGte, limit
\newcommand{\SGtt}{\mathfrak{H}_\e}				%product of \SGte
\newcommand{\SGttLim}{\mathfrak{H}_*}			%product of \SGte, limit
\newcommand{\SGbk}{\mathbf{V}_\text{blk}}		%the bulk part of SGIn
\newcommand{\SGres}{\mathbf{V}_\text{res}}		%the resdiue part of SGIn
\newcommand{\zp}{\mathfrak{J}} 					%a terms in SGres
\newcommand{\SGasep}{\mathbf{V}_{\textup{ASEP}}}				%semigroup
\newcommand{\SGeasep}{\mathbf{V}_{\e,\textup{ASEP}}}			%semigroup, scaled form
\newcommand{\SGasepfr}{\mathfrak{F}^\textup{ASEP}_\e}			%fraction
\newcommand{\SGasepE}{\mathfrak{E}^\textup{ASEP}_\e}			%energy

\newcommand{\mue}{\mu_\e}
\newcommand{\taue}{\tau_\e}
\newcommand{\lambdae}{\lambda_\e}

\newcommand{\Zmg}{Z_\text{mg}}
\newcommand{\Adr}{A_\text{dr}}
\newcommand{\Amg}{A_\text{mg}}
\newcommand{\Agdr}{A_{\nabla,\text{dr}}}
\newcommand{\Agmg}{A_{\nabla,\text{mg}}}
\newcommand{\hatAdr}{\hat{A}_\text{dr}}
\newcommand{\hatAmg}{\hat{A}_\text{mg}}
\newcommand{\Atdr}{A_{\hk-I}}
\newcommand{\hatAgdr}{\hat{A}_{\nabla,\text{dr}}}
\newcommand{\hatAgmg}{\hat{A}_{\nabla,\text{mg}}}

\newcommand{\img}{\mathbf{i}}

\newcommand{\Z}{\mathbb{Z}}
\newcommand{\C}{\mathbb{C}}
\newcommand{\R}{\mathbb{R}}
\newcommand{\e}{\varepsilon}
\newcommand{\ic}{\text{ic}}
\newcommand{\mut}{\lfloor\mu t\rfloor}
\newcommand{\muet}{\lfloor\mue t\rfloor}
\newcommand{\mues}{\lfloor\mue s\rfloor}

\renewcommand{\bar}[1]{\overline{#1}}

\renewcommand{\hat}[1]{\widehat{#1}}
\renewcommand{\tilde}[1]{\widetilde{#1}}

\usepackage{graphicx}
\newcommand*{\Cdot}{{\raisebox{-0.5ex}{\scalebox{1.8}{$\cdot$}}}} % largedot
\graphicspath{{./figs/}}

\newcommand{\uu}{\mathbf{U}}

\newcommand{\Yfin}[1]{\mathbb{Y}^{#1}}
\newcommand{\Xlf}[1]{\mathbb{X}_{\geq #1}}
\newcommand{\Xinf}{\mathbb{X}}
\begin{document}
\title[SPDE Limit of the Six Vertex Model]
{Stochastic PDE Limit of the Six Vertex Model}
\author[I.\ Corwin]{Ivan Corwin}
\address{I.\ Corwin,
	Departments of Mathematics, Columbia University,
	\newline\hphantom{\hspace{15pt}I.\ Corwin}
	2990 Broadway, New York, NY 10027}
\email{corwin@math.columbia.edu}
\author[P.\ Ghosal]{Promit Ghosal}
\address{P.\ Ghosal,
	Departments of Statistics, Columbia University,
	\newline\hphantom{\hspace{15pt}P.\ Ghosal}
	1255 Amsterdam Avenue, New York, NY 10027}
\email{pg2475@columbia.edu}
\author[H.\ Shen]{Hao Shen}
\address{H.\ Shen,
	Departments of Mathematics, University of Wisconsin - Madison,
	\newline\hphantom{\hspace{15pt}H.\ Shen}
480 Lincoln Drive, Madison, WI  53706}
\email{pkushenhao@gmail.com}
\author[L-C.\ Tsai]{Li-Cheng Tsai}
\address{L-C.\ Tsai,
	Departments of Mathematics, Columbia University,
	\newline\hphantom{\hspace{15pt}L-C.\ Tsai}
	2990 Broadway, New York, NY 10027}
\email{lctsai.math@gmail.com}

%\subjclass[2010]{
%}

\begin{abstract}
\noindent
We study the \emph{stochastic} six vertex model and prove that under weak asymmetry scaling (i.e., when the parameter $\Delta\to 1^+$ so as to zoom into the ferroelectric/disordered phase critical point) its height function fluctuations converge to the solution to the \ac*{KPZ} equation. We also prove that the one-dimensional family of stochastic Gibbs states for the \emph{symmetric} six vertex model converge under the same scaling to the stationary solution to the stochastic Burgers equation.

Our proofs rely upon the \emph{Markov (self) duality} of our model. The starting point is an exact microscopic Hopf--Cole transform for the stochastic six vertex model which follows from the model's known one-particle Markov self-duality. Given this transform, the crucial step is to establish \emph{self-averaging} for specific quadratic function of the transformed height function. We use the model's two-particle self-duality to produce explicit expressions (as Bethe ansatz contour integrals) for conditional expectations from which we extract time-decorrelation and hence self-averaging in time. The crux of our Markov duality method is that the entire convergence result reduces to precise estimates on the one-particle and two-particle transition probabilities. Previous to our work, Markov dualities had only been used to prove convergence of particle systems to linear Gaussian SPDEs (e.g. the stochastic heat equation with additive noise).
%
%Previous work on ASEP, a limit of this model, relied on an involved trick introduced by Bertini and Giacomin to control the convergence of the martingale part of the resulting microscopic stochastic heat equation to space-time white noise. That trick does not seem extendable to our setting. In its place we introduce a new method to control this convergence using the two-particle Markov self-duality. Thus, the entire convergence result is reduced to readily verifiable estimates on the one-particle and two-particle heat kernels for the stochastic six vertex model.
\end{abstract}
\maketitle
\tableofcontents

\section{Introduction}\label{introduction}

The \ac{6V} model and the \ac{KPZ} equation are widely studied models in equilibrium and non-equilibrium statistical mechanics. In this paper we demonstrate how a certain scaling limit of the former model converges to the later equation. This limit comes from scaling into the critical point dividing the ferroelectric and disordered phases of the model. Our results apply for both the \emph{stochastic} and \emph{symmetric} \ac{6V} models (Theorems \ref{thm:S6V} and \ref{thm:6V} respectively). The technical core of this paper is the \emph{Markov duality method}: One-particle duality allows us to perform a microscopic Hopf--Cole transform of the model's height function process into a discrete stochastic heat equation, and prove tightness of that resulting equation; and two-particle duality controls the quadratic variation of the martingale part and proves precise self-averaging in time.

The structure of this introduction is as follows:
Section \ref{sec.s6v} introduces the stochastic \ac{6V} model and records our first main result, it convergence to the \ac{KPZ} equation (Theorem \ref{thm:S6V}).
Section \ref{sec.6v} introduces the symmetric \ac{6V} model and records our second main result, the convergence of the one-parameter family of stochastic Gibbs states  to the stationary solution to the \emph{stochastic Burgers equation} (Theorem \ref{thm:6V}). This section also describes the model with
external fields and how the stochastic Gibbs states arise in the (conjectural) phase diagram for the model's Gibbs states.
Section \ref{sec:ASEPKPZ} recalls how the \ac{KPZ} equation arises as a scaling limit for ASEP (a well studied continuous time limit of the stochastic \ac{6V} model). The purpose of this is to highlight (in the simplest case possible) the key technical challenge in proving such results---self average of the quadratic variation.
Section \ref{sect:introDual} briefly introduces our Markov duality method in the context of ASEP and provides some historical context for it. This approach is developed fully for the stochastic \ac{6V} model in the main body of the paper.
Section \ref{sec:liter} provides a brief review of related literature studying the symmetric and stochastic \ac{6V} models, \ac{KPZ} equation, and Markov dualities.

%(Maybe too early to get so technical. We do emphasize these points later one.)(Li-Cheng)
%Existing methods, namely those of \cite{Bertini1997},
%do not seem to apply in this discrete time setting,
%\note{actually, \ac{6V} being discrete-time, probability \emph{all} existing approaches
%don't apply}
%and, in order to prove these results we introduce a new \emph{Markov duality method}.
%This reduces the key steps of proving convergence
%to estimating the one and two-particle transition probability for the \ac{6V} model.

\subsection{KPZ equation as a limit of the stochastic six vertex model}\label{sec.s6v}
The stochastic \ac{6V} model is a discrete time interacting particle system introduced in 1992 by Gwa and Spohn \cite{GS92}. The model depends on two parameters $b_1,b_2\in (0,1)$ which are used to define (positive) weights on six type of vertices---see the top row of Figure \ref{tbl:S6VWeights}. Treating the solid lines entering a vertex from below or the left as \emph{inputs} and those exits above or to the right as \emph{outputs}, these vertex are conservative (i.e., the number of input lines equals the number of output lines)  and stochastic (i.e., for fixed inputs, the sum of weights over outputs is always 1, and the individual weights are non-negative). Given a down-right path in $\Z^2$ and a specification of boundary condition inputs along the path, the stochastic \ac{6V} model is a measure on the vertices to the up and right of the path, or equivalently a measure on the collection of solid lines which leave the boundary inputs and continue in the up and right directions. The measure is defined recursively: starting with vertices with inputs given, the outputs are randomly chosen amongst all possible outputs with probabilities given by the associated vertex weights. The left-side of Figure \ref{fig:sv} illustrates when the boundary condition inputs are specified on the coordinate axes for the first quadrant. See Section \ref{sec.s6vs} for a more precise definition of the model (including a bi-infinite version) and Section \ref{sec.s6vlit} for a brief review of related literature.

\begin{figure}[ht]
        \centering
        \begin{tabular}{|c|c|c|c|c|c|c|}
           \hline
Non-crossing paths
&
\includegraphics[width=35pt]{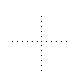}
% \begin{tikzpicture}[scale=0.5]
% \draw[thick,white] (0,-1.4) -- (0,1.4);
% \draw[thick,white] (-1.4,0) -- (1.4,0);
%            \draw[dotted] (-1,0) -- (0,0);
%            \draw[dotted] (0,0) -- (1,0);
%            \draw[dotted] (0,-1) -- (0,0);
%            \draw[dotted] (0,0) -- (0,1);
%            \end{tikzpicture}
&
\includegraphics[width=35pt]{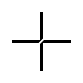}
%            \begin{tikzpicture}[scale=0.5]
% \draw[thick,white] (0,-1.4) -- (0,1.4);
% \draw[thick,white] (-1.4,0) -- (1.4,0);
%            \draw[ultra thick] (-1,0) -- (0,0);
%            \draw[ultra thick] (0,0) -- (1,0);
%            \draw[ultra thick] (0,-1) -- (0,0);
%            \draw[ultra thick] (0,0) -- (0,1);
%            \draw[thick][white] (-.1,-.1) -- (.1,.1);
%            \end{tikzpicture}
&
\includegraphics[width=35pt]{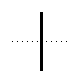}
%            \begin{tikzpicture}[scale=0.5]
%  \draw[thick,white] (0,-1.4) -- (0,1.4);
% \draw[thick,white] (-1.4,0) -- (1.4,0);
%            \draw[dotted] (-1,0) -- (0,0);
%            \draw[dotted] (0,0) -- (1,0);
%            \draw[ultra thick] (0,-1) -- (0,0);
%            \draw[ultra thick] (0,0) -- (0,1);
%            \end{tikzpicture}
&
\includegraphics[width=35pt]{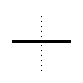}
%            \begin{tikzpicture}[scale=0.5]
% \draw[thick,white] (0,-1.4) -- (0,1.4);
% \draw[thick,white] (-1.4,0) -- (1.4,0);
%            \draw[ultra thick] (-1,0) -- (0,0);
%            \draw[ultra thick] (0,0) -- (1,0);
%            \draw[dotted] (0,-1) -- (0,0);
%            \draw[dotted] (0,0) -- (0,1);
%            \end{tikzpicture}
&
\includegraphics[width=35pt]{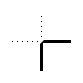}
%            \begin{tikzpicture}[scale=0.5]
% \draw[thick,white] (0,-1.4) -- (0,1.4);
% \draw[thick,white] (-1.4,0) -- (1.4,0);
%            \draw[dotted] (-1,0) -- (0,0);
%            \draw[dotted] (0,0) -- (0,1);
%            \draw[ultra thick] (0,-1) -- (0,0);
%            \draw[ultra thick] (0,0) -- (1,0);
%            \draw[thick][white] (-.1,-.1) -- (.1,.1);
%            \end{tikzpicture}
&
\includegraphics[width=35pt]{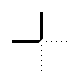}
%            \begin{tikzpicture}[scale=0.5]
% \draw[thick,white] (0,-1.4) -- (0,1.4);
% \draw[thick,white] (-1.4,0) -- (1.4,0);
%            \draw[ultra thick] (-1,0) -- (0,0);
%            \draw[ultra thick] (0,0) -- (0,1);
%            \draw[dotted] (0,-1) -- (0,0);
%            \draw[dotted] (0,0) -- (1,0);
%            \draw[thick][white] (-.1,-.1) -- (.1,.1);
%            \end{tikzpicture}
\\
           \hline
Stochastic weights &        $1$ & $1$ & $b_1$ & $b_2$ & $1-b_1$ & $1-b_2$ \\
            \hline
Symmetric weights&
        $a$ & $a$ & $b$ & $b$ & $c$ & $c$ \\
            \hline
Asymmetric weights&
        $e^{-H-V} a$ & $e^{H+V}a$ & $e^{-H+V}b$ & $e^{H-V}b$ & $c$ & $c$ \\
            \hline
        \end{tabular}
        \caption{Six vertices with their stochastic, symmetric and asymmetric weights.}
        \label{tbl:S6VWeights}
\end{figure}

If the boundary condition inputs are specified entirely on the horizontal axis, it is natural to think of vertical solid lines as particles evolving in time (as measured by the $y$-coordinate) via the following Markovian update. Start with left-most particle\footnote{If there is no left-most particle, the dynamics can be still be defined with some care---see Section~\ref{sec.s6vs}.}. With probability $b_1$ it stays put, and with $1-b_1$ it moves one to the right. The particle continues to move right with probability $b_2$ per step until it either stops, or it hits the next particle. When no collision happens, repeat these rules for the next particle to the right. If a collision occurs, the moving particle stops at that site and the next particle starts moving to the right with probability $1$, and continues to move with probability $b_2$ (as usual). See Section \ref{sec.s6vlit} for a discussion of some limit of the stochastic \ac{6V} model.%, including ASEP.

\begin{figure}[ht]
   % gives better spacing than \begin{center}...\end{center}
  \includegraphics[width=5in]{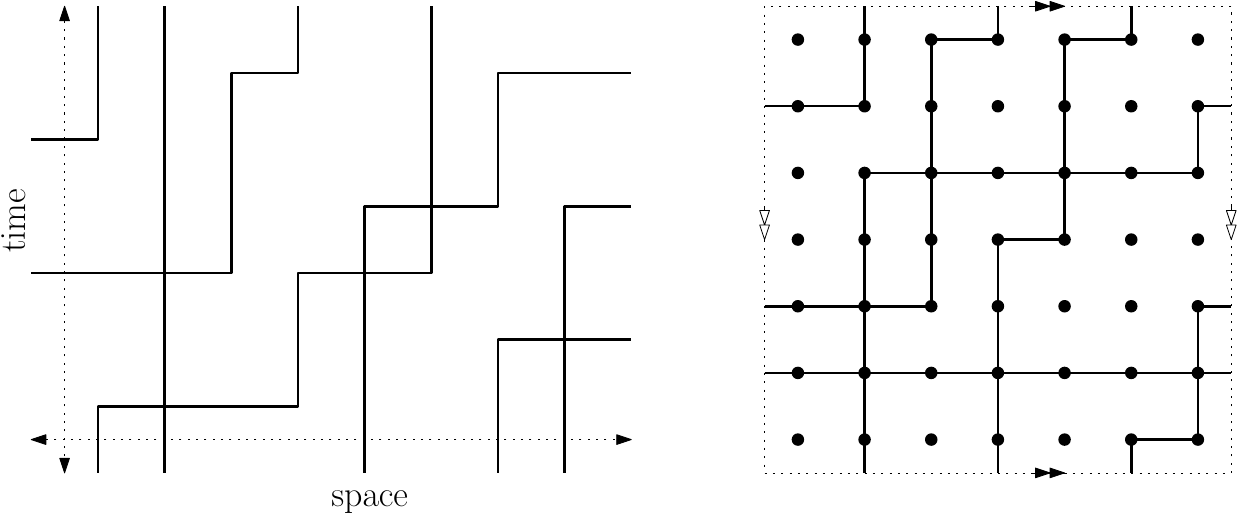}
  \caption{Left: Particle trajectories for the stochastic \ac{6V} model with boundary condition inputs along the coordinate axes. Right: Periodic boundary conditions.}
  \label{fig:sv}
\end{figure}

Define the height function $ N(t,x) $ for the stochastic \ac{6V} model to be equal to the net number of particles which have moved across the time-space line between $(0,0)$ and $(t,x)$ (i.e., summing $1$ for each left-to-right move and $-1$ for each right-to-left move---see Figure \ref{fig.s6vpart}). For a precise definition as well as a construction of $ N(t,x) $ for bi-infinite configurations, see Section~\ref{sec.s6vs}.
Given such $ N(t,x) $, we first linearly interpolate in $ x\in\Z $
and then linearly interpolate in $ t \in \Z_{\geq 0} $ to make $ N(t,x) \in C([0,\infty),C(\R)) $.
Hereafter, we endow the space $ C(\R) $ and $ C([0,\infty),C(\R)) $ the topology of uniform convergence over compact subsets,
and write $ \Rightarrow $ for the weak convergence of probability laws.

\begin{figure}[ht]
\centering
\includegraphics[width=.45\textwidth]{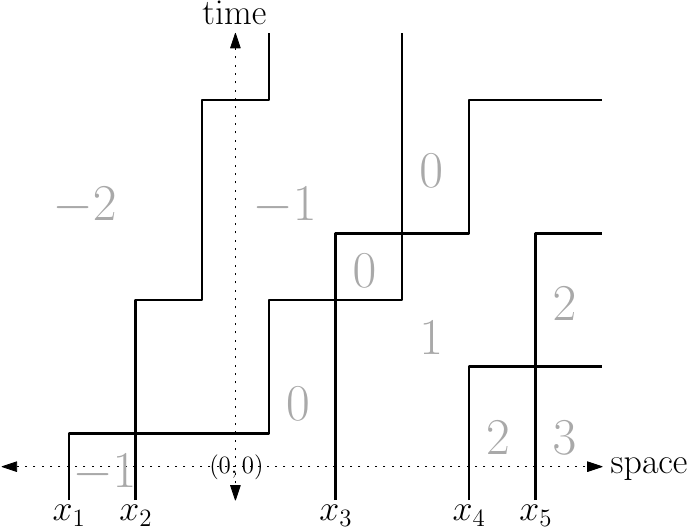}
\caption{The stochastic \ac{6V} particle trajectories and associated height functions. Here we assume a left-most particle and label the first five particles $x_1,\ldots, x_5$. The lines represent their temporal trajectories. The dark grey numbers represent the height function $N(t,y)$ for different regions. The height function changes value when crossing particle trajectories (increasing as one crosses from left to right).}\label{fig.s6vpart}
\end{figure}

Our main result for the stochastic \ac{6V} model states that,
under weak asymmetry scaling where $ b_1\in (0,1) $ is fixed
and $ \tau=b_2/b_1=e^{-\sqrt\e} \to 1 $,
an analog of \eqref{eq.bgres} holds for $ N(t,x) $.
To setup notations, we \emph{fix} any density $ \den\in(0,1) $ hereafter, and let
\begin{align}
	\label{eq:lambda}
	\lambda
	&=
	\frac{1-b_2\tau^{-\den}}{b_1-(b_1+b_2-1)\tau^{-\den}}
	=
	\frac{1-b_1\tau^{1-\den}}{b_1-(b_1+b_1\tau-1)\tau^{-\den}},
\\
	\label{eq:mu}
	\mu
	&
	= \frac{\tau^{-\den}(1-b_1)(1-b_2)}{(b_1-(b_1+b_2-1)\tau^{-\den})(1-b_2\tau^{-\den})}
	= \frac{\tau^{-\den}(1-b_1)(1-b_1\tau)}{(b_1-(b_1+b_1\tau-1)\tau^{-\den})(1-b_1\tau^{1-\den})}.
\end{align}
The reason of choosing these values of the parameters $\lambda,\mu$ will be clear in Section~\ref{sect:mHC}.
Specifically, under the weak asymmetry scaling where $ b_1\in(0,1) $ fixed and $ \tau =\taue = b_2/b_1 := e^{-\sqrt\e} $,
we have $ \lambda=\lambdae $ and $ \mu=\mue $, which,
up to first order in $ \sqrt\e $, read
\begin{align}
	\label{eq:lambdae}
	\lambdae
	&=
	1 - \den \sqrt\e + \mathcal{O}(\e),
	%%\tfrac{\den^2-3b_1\den^2+2b_1\den}{1-b_1} \e + \mathcal{O}(\e^{\frac32}),
\\
	\label{eq:mue}
	\mue
	&=
	1 + \tfrac{b_1-2b_1\den}{b_1-1} \sqrt{\e} + \mathcal{O}(\e).
%	+ \tfrac{b_1+b_1^2-2b_1\den-6b_1^2\den+6b_1^2\den}{2(1-b_1)^2} \e
\end{align}
%\hao{We need the next order (see the reason below):
%\[
%\lambdae =
%	1 - \den \sqrt\e + \frac{\den (b_1 (2-3\den) + \den)}{2-2b_1}\e   + \mathcal{O}(\e^{\frac32}),
%\]
%}
We adopt standard notation $ \mathcal{O}(a) $ to denote a generic quantity
such that $ \sup_{0<a<1}|\mathcal{O}(a)|a^{-1}<\infty $.
Recall the \ac{KPZ} equation (see Section \ref{sect:SHE} for its definition; Sections \ref{sec:ASEPKPZ} and \ref{sec.KPZlit} a literature review).

\begin{align}
	\label{eq:KPZ}
	\partial_t \limH(t,x) = \frac{\nu_*}{2} \partial_x^2 \limH(t,x)
	- \frac{\kappa_*}{2}\big( \partial_x \limH(t,x)\big)^2 + \sqrt{D_*}\noise(t,x),
\end{align}
with coefficients
%\footnote{\hao{The limiting coefficient $\nu_*$ will be calculated in \eqref{e:VarR-nustar}....}}
\begin{align}
	\label{eq:ceffints}
	\nu_* := \frac{2b_1}{1-b_1},
	\quad
	\kappa_* := \frac{2b_1}{1-b_1},
	\quad
	D_* := \frac{2b_1\den(1-\den)}{1-b_1}.
\end{align}
\begin{thm}
\label{thm:S6V}
Consider the stochastic \ac{6V} model, with parameter $ b_1> b_2 \in (0,1) $.
\begin{enumerate}[leftmargin=5ex, label=(\alph*)] %
\item \textbf{(Near stationary initial conditions)} Fix a density $\rho\in (0,1)$.
With $ \e\downarrow 0 $ denoting a scaling parameter,
we start the stochastic \ac{6V} model from a sequence of initial conditions $ \{N_\e(0,x)\}_{\e>0} $,
and let $ N_\e(t,x) $ denote the resulting height function.
Assume that $ \{N_\e(0,x)\}_{\e>0} $ is near stationary with density $\rho$ (Definition~\ref{def:nearStat}),
and that for some $ C(\R) $-valued process $ \limH^\ic(x) $,
\begin{align}
	\sqrt\e \big( N_\e(0,\e^{-1}x) -  \rho \e^{-1}x \big)
	\Longrightarrow
	\limH^\ic(x),
	\quad
	\text{ in } C(\R).
\end{align}
Then, under the weak asymmetry scaling where $ b_1\in(0,1) $ is fixed, $ \tau =\taue = b_2/b_1 := e^{-\sqrt\e} $, and $\lambda$ and $\mu$ depend on $\e$ as in \eqref{eq:lambdae} and \eqref{eq:mue}, we have
\begin{align}
	\label{eq:S6VtoKPZ}
	\sqrt\e \Big( N_\e\big(\e^{-2}t,\e^{-1}x+\mue \e^{-2} t\big)
	 - \den(\e^{-1}x+\mue \e^{-2} t) \Big) - \e^{-2}t\log\lambdae
	\Longrightarrow
	&\limH(t,x),
\\
	\notag
	&\text{in } C([0,\infty), C(\R)),
\end{align}
where $ \limH(t,x) $ is the Hopf--Cole solution (defined in Section~\ref{sect:SHE}) of
the \ac{KPZ} equation~\eqref{eq:KPZ} with initial condition~$ \limH^\ic(x) $.
\item \textbf{(Step initial condition)}
Start the stochastic \ac{6V} model from the step initial condition $ N(0,x) = (x)_+:=\max(0,x) $,
and let $ N_\e(t,x) $ denote the resulting height function. Fix $\rho\in (0,1)$.
Under the weak asymmetry scaling where $ b_1\in(0,1) $ is fixed, $ \tau =\taue = b_2/b_1 := e^{-\sqrt\e} $, and $\lambda$ and $\mu$ depend on $\e$ as in \eqref{eq:lambdae} and \eqref{eq:mue}, we have
\begin{align*}
	\sqrt\e \Big(
		N_\e\big(\e^{-2}t,\e^{-1}x+\mue \e^{-2} t\big)
		- \den(\e^{-1}x+\mue \e^{-2} t) \Big)
		- \e^{-2}t\log\lambdae
		&- \log \tfrac{\den(1-\den)}{\sqrt\e}
	\Longrightarrow
	\limH(t,x),
\\
	&\text{in } C((0,\infty), C(\R)),
\end{align*}
where $ \limH(t,x) $ is the Hopf--Cole solution of
the \ac{KPZ} equation~\eqref{eq:KPZ} with narrow wedge initial condition (see Section~\ref{sect:SHE}).
\end{enumerate}
\end{thm}
\begin{rmk}
It is worth remarking on the freedom to choose arbitrary $\rho\in (0,1)$ in the theorem. For the near stationary initial conditions, $\rho$ controls the density of particles (or vertical lines) as well as the \emph{characteristic} velocity around which we focus. For step initial data, $\rho$ determines a velocity within the rarefaction fan (and gives the density around that velocity). Previous KPZ equation limit results for ASEP where limited to $\rho=1/2$ since the arguments become more complicated in a moving frame or with a disproportionate number of particles to holes.
\end{rmk}

\begin{rmk}
\cite{CT15} proves KPZ equation convergence for a portion of the class of higher spin stochastic vertex models \cite{Corwin2016}. Those models fall into two types -- those with spin $I,J \in \Z_{\geq 1}$ in which the number of particles or arrows per edge is bounded by $I$ or $J$ (depending on the edge's orientation) and those with non-integer spin in which there may be an infinite number of particles or arrows per edge. \cite{CT15} analyzed this second class, specifically under scaling in which the expected number of particles per site diverges with $\e$. This simplifies analysis quite dramatically since \cite{CT15} is able to Taylor expand the quadratic martingale in the density yielding an analysis which completely avoids the key complexity which we encounter here. The stochastic \ac{6V} model comes from taking $I=J=1$ and hence the number of particles per site is either 0 or 1. Though we do not address the general $I,J\in \Z_{\geq 1}$ class herein, we expect our Markov duality method is applicable there.
\end{rmk}

\begin{rmk}
Plugging the expansions \eqref{eq:lambdae} and \eqref{eq:mue} for $\lambdae$ and $\mue$
into \eqref{eq:S6VtoKPZ},
one can see that the two terms of vertical shifting of height function,
namely $-\sqrt\e \den ( \mue \e^{-2} t )$ and $- \e^{-2}t\log\lambdae$,
are both  of order $\mathcal{O}(\e^{-\frac32})$;
but their order $\mathcal{O}(\e^{-\frac32})$ parts
 cancel out.
Therefore \eqref{eq:S6VtoKPZ} states that the rescaled and tilted height function
subtracting $\mathcal{O}(\e^{-1}) t$ converges to the solution to KPZ equation.
\end{rmk}

\noindent \emph{Proof sketch.}
Proposition \ref{prop:mSHE} provides an exact \emph{microscopic Hopf--Cole transform} through which the stochastic \ac{6V} model height process is relates to a microscopic \ac{SHE}. This transformation is readily seen as a consequence of the (one-particle) Markov self-duality given in Corollary \ref{cor:CP16}.
Theorem \ref{thm:S6V:} proves convergence of this microscopic \ac{SHE} to the continuum \ac{SHE}. When translated back into the stochastic \ac{6V} model height function, this implies Theorem \ref{thm:S6V}.

The proof of Theorem \ref{thm:S6V:} boils down to showing tightness and identifying the limiting linear and quadratic martingale problem. The first two items follow in a standard manner from moment bounds provided by Proposition \ref{prop:momt}. Controlling the quadratic variation is the hard part. Proposition \ref{prop:qv} does this by proving a form of self-averaging for the quadratic variation (which itself is a quadratic in the solution to the microscopic \ac{SHE}). The proof of the self-average relies upon the two-particle duality through Proposition \ref{prop:Zduality}. That duality reduces the calculation of conditional expectations to computations involving the transition probability for a two-particle version of the stochastic \ac{6V} model. Such transition formulas can be written explicitly using Bethe ansatz---see Proposition \ref{prop:TransitionProbability} or the  formula in \eqref{eq:SG}.  Proposition \ref{prop:SG} contains very precise estimates on the two-point transition probabilities which are proved via involved steepest descent analysis on the double-contour integral formulas encoding these probabilities.

In Sections \ref{sec:ASEPKPZ} and \ref{sect:introDual} (and Appendix \ref{sect:ASEP}) we explain how these ideas work in the simpler the context of ASEP. For ASEP, there are other methods which can be used to prove self-averaging. Presently, our  Markov duality method is the only approach which works for the \ac{6V} model.

\subsection{Stochastic Burgers equation as a limit of symmetric six vertex model}\label{sec.6v}

The symmetric \ac{6V} model is a foundational model in 2D equilibrium statistical mechanics.
It is defined with respect to a pre-imposed a choice of boundary condition on a compact domain in $\Z^2$, e.g. periodic boundary condition on a rectangular domain as in Figure \ref{fig:sv}. Then, one chooses an assignment of vertices inside the domain which fit together (i.e. output lines match input lines from vertices to the right or above) with probability proportional to the product of vertex weights. These weights are specified by $a,b,c>0$ (in fact, by scaling only two of these matter) as in Figure \ref{tbl:S6VWeights} and the model is called \emph{symmetric} since reflecting the vertices over the diagonal does not change their weight. To go from such a product of weights to a probability requires dividing by a normalizing constant (also called a partition function) which is the sum over all configurations of the product of weights. The need to normalize was not present in the case of  stochastic weights as it equals 1 there.

\subsubsection{Conjectural phase diagram for symmetric six vertex model Gibbs states} \label{sec:Conjectural phase diagram}

How does the symmetric \ac{6V} model behave as the mesh size goes to zero? Is there a limit shape? How does the height function fluctuate around it? How much do boundary conditions or external fields effect these limits? These questions are intertwined with understanding the {\it extremal, translation invariant, ergodic infinite volume Gibbs states} (or simply Gibbs states for short) and their associated {\it free energies}. These can be thought of as distributions on configurations of vertices on $\Z^2$ which satisfy the symmetric \ac{6V} Gibbs property---the marginal distribution restricted to any compact subdomain, given the state of the boundary vertices, is given by the above symmetric \ac{6V} model probability prescription (i.e., product over weights of vertices normalized to be a probability distribution).

While much has been conjectured about the symmetric \ac{6V} Gibbs states (e.g. their phase diagram, free energy, uniqueness, and fluctuations) very little has been proved---see Section \ref{sec.6vlit} for some further discussion.
The description we provide here (i.e. in this Section~\ref{sec:Conjectural phase diagram})
can be found, for instance, in \cite{Nolden92, Bukman1995, Res2010} and is essentially conjectural. We include it here to motivate the importance of studying the ``\emph{stochastic Gibbs states}''  in Section \ref{sec:stochasticgibbsstates}.
The discussion in this Section~\ref{sec:Conjectural phase diagram} will  not be used in any proofs.

The Gibbs states for the symmetric \ac{6V} model (with a given choice of $a,b,c$) are believed to arise as infinite volume limits of the periodic boundary condition asymmetric \ac{6V} model in which there are horizontal and vertical external fields of strength $H,V\in \R$ (see Figure \ref{tbl:S6VWeights}). These fields reward  the occurrence of horizontal or vertical lines by factors of $e^{H/2}$ and $e^{V/2}$ and penalize the absence of lines by $e^{-H/2}$ and $e^{-V/2}$. Consider any rectangle enclosed in the interior of the fundamental domain of the periodic model. Then, regardless of the choices of external fields, conditioned on the vertices on the boundary of the rectangle, the configuration inside is given by the symmetric, zero-field \ac{6V} model weights. This is because all possible vertex configurations inside the rectangle have the same number of vertical and horizontal lines. This is analogous to the fact that for a simple random walk with drift, the marginal distribution of the walk given a fixed starting and ending level is drift-independent.

\begin{figure}[ht]
   % gives better spacing than \begin{center}...\end{center}
\hspace{5mm}
  \begin{subfigure}{0.4\textwidth}
  \includegraphics[width=2.5in]{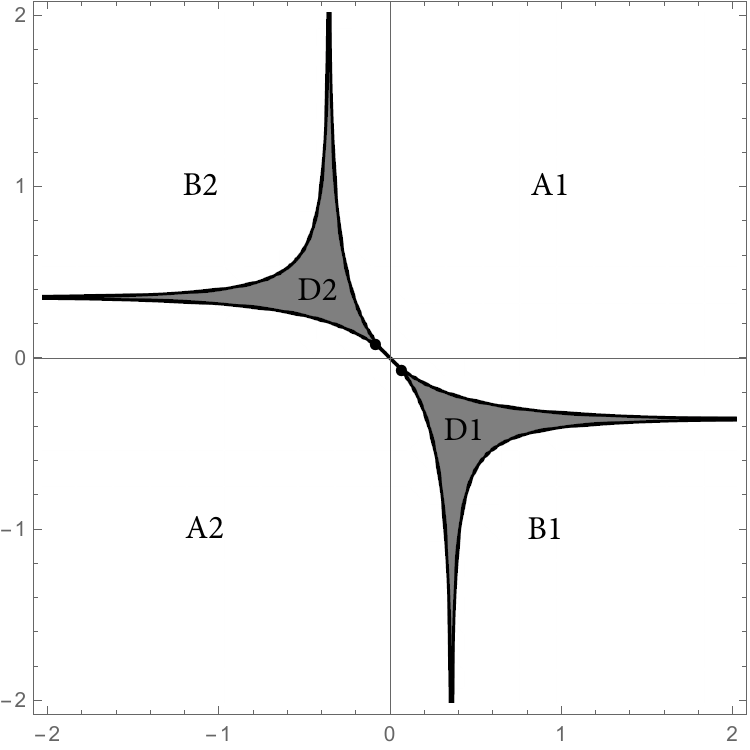}
  \put(-87,-7){$H$}
  \put(-187,90){$V$}
  \caption{}
  \label{fig1}
  \end{subfigure}
  \begin{subfigure}{0.4\textwidth}
  \includegraphics[width=2.5in]{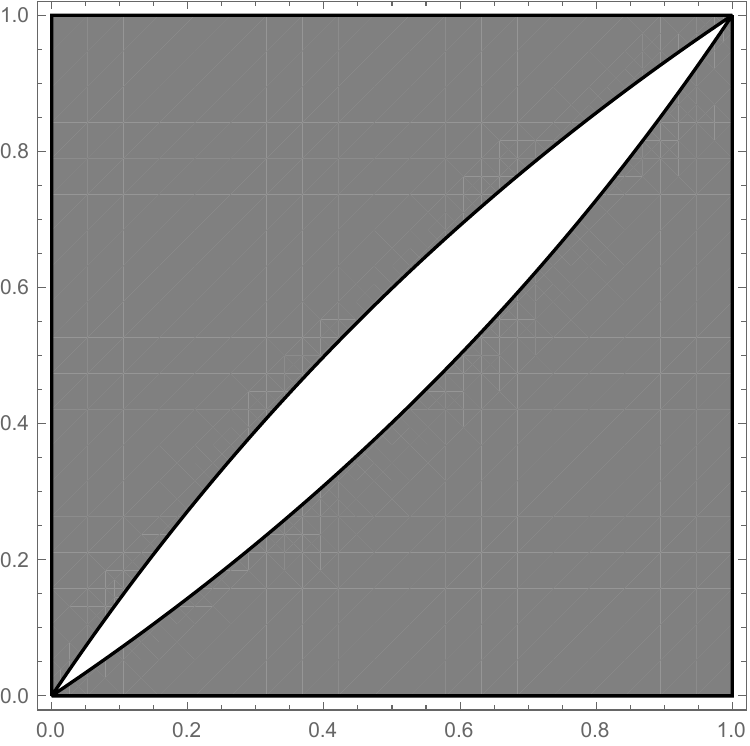}
  \put(-85,-7){$h$}
  \put(-185,90){$v$}
  \caption{}
  \label{fig2}
  \end{subfigure}
  \caption{\ac{6V} model with parameters  $(a,b,c)\approx (.201,.1,.1)$ (or $u,\eta=.1$) and $\Delta \approx 1.005$. (a): Phase diagram mapping $(H,V)$ onto Gibbs states. (b): Average density of horizontal and vertical   lines $(h,v)$ accessible as $(H,V)$ varies. The A1 phase maps to $(h,v)=(1,1)$, A2 to $(0,0)$, B2 to $(0,1)$, B1 to $(1,0)$. The disordered phase D2 maps to the grey area above the diagonal in the $(h,v)$ plot, and D1 to the reflected area. The disordered phase extends asymptotically vertically and horizontally so as to separate the A and B phases. The two {\it conical points} are where D2/D1, A1 and A2 touch. Each conical point maps to the entire boundary of the white lens around the $(h,v)$ diagonal. Inside the lens there are no (extremal) Gibbs states with those specified densities.}
  \label{fig:phases}
\end{figure}

Gibbs states are believed to be uniquely indexed by their average density $(h,v)\in[0,1]^{2}$ of horizontal and vertical lines (respectively). It is not necessary that every $(h,v)$ will have a corresponding Gibbs state which realizes those densities. \cite{Res2010} describes the conjectural mapping (derived based on Bethe ansatz calculations) between $(H,V)$ and $(h,v)$. The nature of this mapping depends on the parameter
\begin{align}
	\label{eq:Delta}
	\Delta = \frac{a^2+b^2-c^2}{2ab}.
\end{align}
We will focus on the case when $\Delta>1$ and $a>b+c$ (the other possible case when $\Delta>1$ is $b>a+c$ and that can be recovered by a simple transformation of vertices) in which the conjectural phase diagram is given in Figure \ref{fig:phases}\footnote{When $|\Delta|<1$ the conical points in the phase diagram disappear and the two disordered phases merge. When $\Delta<-1$ a new antiferroelectric phase emerges for $H,V$ near zero. The associated Gibbs state is composed of diagonal bands of zig-zags made up only the $c$-type vertices.} -- see the caption beneath the figure regarding how different phases in $(H,V)$ picture are mapped into regions in $(h,v)$ picture. There are four frozen phases A1,A2,B1,B2 which arise when $H$ and $V$ are sufficiently positive or negative. Between them are {\it disordered} phases D1,D2 which map onto values of $(h,v)$ in the grey region. \cite{Nien84} (see more recently \cite{KMSW17}) conjectured that the fluctuations in the disordered phase are log-correlated and related to the Gaussian free field (or central charge 1 CFT). Such a result has only been proved at the free-fermion ($\Delta=0$) point \cite{Ken00,Ken01,Ken09}.

In Figure \ref{fig:phases}(A) the disordered regions D1 and D2 terminate near the origin at {\it conical points} connected by a line between the A1 and A2 phases. In Figure \ref{fig:phases}(B) these conical points are mapped to the entire boundary between the grey disordered phase and the white excluded phase (i.e. the lens around the diagonal which do not have corresponding extremal  Gibbs states). Different  Gibbs states  arise at a conical point, depending on the angle in the $(H,V)$-plane along which one approaches the conical point; these Gibbs states have  different line densities $(h,v)$ as parametrized by the boundary of the lens in Figure \ref{fig:phases}(B).
 \cite{Bukman1995} argued that the one-parameter family of Gibbs states  arising at the conical points coincides with the one-parameter family of so-called ``{\it stochastic} Gibbs states'', which we now discuss in Section~\ref{sec:stochasticgibbsstates}.

\subsubsection{Stochastic Gibbs states and their scaling limits}\label{sec:stochasticgibbsstates}

Whereas even the existence of disordered Gibbs states is only conjectural for $\Delta\neq 0$, the one-parameter family of ``stochastic Gibbs states'' (which enjoy the symmetric $(a,b,c)$ Gibbs property) is  readily constructed owing to their connection with the stochastic \ac{6V} model. Fix $(a,b,c)$ and consider the stochastic \ac{6V} model with parameters\footnote{This relation can be reversed to give $\Delta = \frac{b_1+b_2}{2\sqrt{b_1 b_2}}$.}
\begin{equation}\label{eq.bs}
b_1 = \tfrac{b}{a}\big(\Delta + \sqrt{\Delta^2-1}\big), \quad \textrm{and}\quad  b_2 = \tfrac{b}{a}\big(\Delta - \sqrt{\Delta^2-1}\big).
\end{equation}
Choose $(h,v)\in [0,1]^2$ such that
\begin{equation}\label{vhbeq}
%\frac{v}{1-v}(1-b_2) = \frac{h}{1-h}(1-b_1).
\frac{v}{1-v}(1-b_1) = \frac{h}{1-h}(1-b_2).
\end{equation}

There is a one-parameter family of solutions $(h,v)$ to this relation.
Consider boundary condition inputs for the stochastic \ac{6V} model on the first quadrant where with probability $h$ there are horizontal lines coming in from the $y$-axis, and with probability $v$ there are vertical lines coming in from the $x$-axis (all these occur independently). \cite{A16aa}  proves that this boundary condition is {\it stationary} so that if one shifts the coordinates of the origin into the third quadrant, the marginal distribution restricted to the first quadrant remain unchanged. Shifting the origin back to $(-\infty,-\infty)$ defines a Gibbs state referred to as  {\bf stochastic Gibbs state} with line densities $(h,v) $,
which we denote by $ \Gibbs(b_1,b_2;h,v) $
(see Lemma \ref{lem:stat} below for precise statement of this construction.)
Figure \ref{fig:char} illustrates such a stochastic Gibbs state.

\begin{figure}[ht]
  \includegraphics[width=4in]{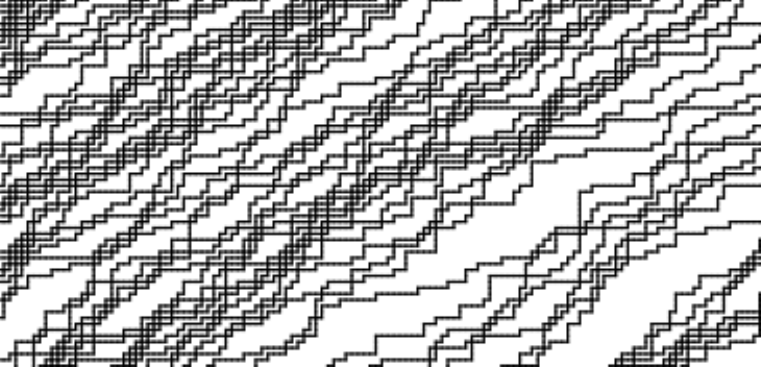}
  \caption{A sample of a stochastic Gibbs state on a finite box.}
  \label{fig:char}
\end{figure}

The densities $(h,v)$ in this one-parameter family of stochastic Gibbs states
$ \Gibbs(b_1,b_2;h,v) $ coincide with the densities which are conjectured to arise from the conical point (i.e. the boundary of the white lens in Figure \ref{fig:phases})\footnote{In fact, \eqref{vhbeq} only gives  upper %bottom
boundary of the lens. The other boundary comes from applying the diagonal symmetry of the symmetric model.}.
Let us briefly make this matching to the formula for that lens boundary given in \cite{ResSri16b}.
When $\Delta>1$ and $a>b+c$,
Baxter introduced a convenient  (projective) parameterization of $ (a,b,c) $:
\begin{align}
	\label{eq:projPara}
	a= \sinh(u+\eta),\qquad b= \sinh(u),\qquad c=\sinh(\eta),
\end{align}
%with $u,\eta>0$;
%and when $b>a+c$,
%$$
%a= \sinh(u-\eta),\qquad b= \sinh(u),\qquad c=\sinh(\eta)
%$$
%with $0<\eta<u$.
with $u,\eta>0$. Note in particular $ \Delta = \cosh(\eta) $.
In terms of this parameterization,
the conjectural (see, for example, \cite{ResSri16b}) one-parameter family of  Gibbs states arising from the conical points have horizontal and vertical line densities given
by the relation\footnote{In \cite{ResSri16b}, $t=2h-1$ and $s=2v-1$. There was a transcription error in  \cite[Eq.~(34)]{ResSri16b} (which related a result from \cite{Bukman1995}). What was written there as $\tanh(u+\eta)$ should be $\tanh(u)$ (as stated here) \cite{PrivateNicolaiAmol}.}
\begin{align}
	\label{eq:slope}
	h= \frac{v\big(1\pm \tanh(u)\big)}{1\pm \tanh (u)(2v-1)}.
\end{align}
and the conical points arises from choosing $(H,V) =  (\pm \eta/2,\mp \eta/2)$.
\cite{A16aa} proves that $ \Gibbs(b_1,b_2;h,v) $ does
prescribe a translation-invariant Gibbs state for the symmetric six-vertex model with weights $(a,b,c)$, see Proposition~\ref{GibbsMeasureProposition} below.
%
%This curve defines the boundary of the white (inaccessible) and grey (disordered) region in the phase diagram in Figure \ref{fig:phases}. In terms of this parameterization, the conical point in the $(H,V)$ plane occurs at $(\pm \eta/2,\mp \eta/2)$.
%It is straightforward to match the above one-parameter family of $(s,t)$ with those which arise from the corresponding stochastic Gibbs states in Proposition \ref{GibbsMeasureProposition}. It would be nice to see a proof that these are only Gibbs state which arise from the conical point.

Our main theorem on symmetric \ac{6V} model describes the large scale behaviors of the stochastic Gibbs state when $ \Delta \downarrow 1 $---that is, when we zoom into the \emph{ferroelectric-disorder interface}.

A natural quantity describing large scale behavior of Gibbs states
is the empirical distributions of vertical or horizontal lines.
We will focus on vertical lines, and the analogous result on horizontal lines is obtained through exchanging $ x $- and $ y $-axes.
Given a tiling on $ \Z^2 $ by the six vertices from Figure~\ref{tbl:S6VWeights},
for each point $ (x,y)\in\Z^2 $, we let $ u(x,y) $ denote the indicator function
for having an \emph{incoming} (i.e., from below) vertical line.
More explicitly,
\begin{align}
	\label{eq:vl:ind}
	u(x,y)
	=
	\ind\set{
		(x,y) \text{ is tiled with }
		\raisebox{-4pt}{\includegraphics[width=20pt]{vertex1111}},
		\raisebox{-4pt}{\includegraphics[width=20pt]{vertex1010}},
		\text{ or }
		\raisebox{-4pt}{\includegraphics[width=20pt]{vertex1001}}
	}.
\end{align}
%Similarly to our discussion in stochastic \ac{6V} model,
%the empirical distribution involves certain centering and scaling.
%We will be operating in Baxter's projective parametrization~\eqref{eq:projPara}.
%To setup notation, fix vertex weights $ a>b>0 $ and $ \Delta\in(1,\frac{(b/a)^2+1}{2b/a}) $ (which then specifies the weight $ c>0 $ through~\eqref{eq:Delta}).
%Fix further a pair of slopes $ (s,t)\in(0,1) $ such that
%\begin{align}
%\label{eq:slopes}
%	\tfrac{s}{1-s} \big(1-\tfrac{b}{a}(\Delta+\sqrt{\Delta^2-1})\big)
%	=
%	\tfrac{t}{1-t} \big(1-\tfrac{b}{a}(\Delta-\sqrt{\Delta^2-1})\big).
%\end{align}
%\begin{align}
%	\label{eq:slopes}
%	t(\den,\tfrac{a}{b},\Delta)
%	:=
%	\frac{(1-\frac{b}{a}(\Delta-\sqrt{\Delta^2-1})\den}{(1-\frac{b}{a}(\Delta-\sqrt{\Delta^2-1})\den+(1-\frac{b}{a}(\Delta+\sqrt{\Delta^2-1})(1-\den)}.
%\end{align}
For a fixed $ v\in(0,1) $ average density of vertical lines,
we define the scaled empirical distribution $  U_\e $, acting on $ f\in C^\infty_c(\R^2) $ ($ C^\infty $ with compact support) as
\begin{align}
	\label{eq:empirical}
	\langle U_\e, f \rangle
	:=
	\e^{\frac52} \sum_{x,y\in\Z} \big(u(x,y)-v\big) f(\e^{-1}x-\mue \e^{-2}y,\e^{-2}y).
\end{align}
Here $ \mue $ is the proper centering of the reference frame in order to observe \ac{KPZ}-type fluctuations.
In terms of Baxter's projective parametrization~\eqref{eq:projPara},
%for fixed $ u>0 $, $ \mue $ is obtained by matching $ b_1,b_2 $ into $ u,\eta $ and via~\eqref{eq.bs}--\eqref{eq:projPara} in~\eqref{eq:mu},
%and set $ \eta=\eta_{\e}=\frac12\sqrt\e$ and $ \tau=\taue=e^{-\sqrt\e} $.
$ \mue $ is obtained by matching $ b_1,b_2 $ into $ u,\eta $  via~\eqref{eq.bs}\eqref{eq:projPara} in~\eqref{eq:mu},
and setting $u=u_{\e}=\zeta\sqrt\e$ for some fixed $\zeta\in(0,\infty)$, $\eta=\eta_{\e}=\frac12\sqrt\e$ and $\rho=v$, $ \tau=\taue=e^{-\sqrt\e} $.

Informally speaking, the $ \e\downarrow 0 $ limit of the empirical distribution $ U_\e $
is described by the stationary solution of the \ac{SBE}:
\begin{align}
	\label{eq:SBE}
	\partial_t \limU = \frac{\nu_*}{2} \partial_{x}^2\limU
	-
	\frac{\kappa_*}{2} \partial_x\big( \limU^2\big) + \sqrt{D_*}\partial_x\noise.
\end{align}

To formulate our result precisely,
first note that the solution $ \limU $ of the \ac{SBE}~\eqref{eq:SBE} is a distribution (i.e., generalized function) valued process.
In the following we will work with the space $ C^{-1}(\R^2) $ of distributions.
For $ f\in C^\infty_c(\R^2) $, write $ f_\delta(x,y):= f(\delta^{-1} x, y) $ for the corresponding scaled function.
This scaling probes only the regularity in $ x $.
For linear functionals $ U,U' $ on $ C^\infty_c(\R^2) $, define
\begin{align}
	\notag
	\Vert U \Vert_{C^{-1}(\R^2),[-\ell,\ell]^2}
	&:=
	\sup\big\{ |\langle U, f_\delta\rangle| \delta: \delta \in (0,1), f\in C^\infty_c(\R^2),
\\
	\label{eq:C-1:}
	&
	\hphantom{:=\sup\big\{ |\langle U, f_\delta\rangle| \delta^{2}:}
	\text{supp}(f)\subset[-\ell,\ell]^2, \ \Vert f \Vert_\infty + \Vert \partial_x f \Vert_\infty \leq 1 \big\},
\\
	\label{eq:C-1}
	\metric(U,U') &:= \sum_{\ell=1}^\infty \big( 2^{-\ell} \wedge \Vert U-U' \Vert_{C^{-1}(\R^2),[-\ell,\ell]^2} \big).
\end{align}
The space $ C^{-1}(\R^2) $ consists of linear functional $ U:C^\infty_c(\R^2)\to \R $ satisfying $ \metric(U,0)<\infty $,
endowed with the metric $ \metric(\Cdot,\Cdot) $.

To define the stationary solution of the \ac{SBE}~\eqref{eq:SBE},
consider the Hopf--Cole solution $ \limH_\text{stat}(t,x) \in C([0,\infty),C(\R)) $ of the \ac{KPZ} equation~\eqref{eq:KPZ},
with initial condition
\begin{align}
	\label{eq:Brownian}
	\limH_\text{stat}(0,x) = \sqrt{\den(1-\den)}B(x),
	\quad
	\den = v,
\end{align}
where $ B(x) $ denote a two-sided standard Brownian motions.
It is known~\cite{Bertini1997,FQ15} that
the Brownian motion~\eqref{eq:Brownian} is quasi-stationary for the \ac{KPZ} equation~\eqref{eq:KPZ}.
That is, $ \limH_\text{stat}(t_0,\Cdot)-\limH_\text{stat}(t_0,0) \stackrel{\text{law}}{=} \sqrt{\den(1-\den)}B(\Cdot) $, for any $ t_0\in[0,\infty) $.
This and the uniqueness of Hopf--Cole solutions implies that
\begin{align}
	\label{eq:stat:SBE}
	\limH_\text{stat}(t+t_0,x) - \limH_\text{stat}(t_0,0)	
	\stackrel{\text{law}}{=}
	\limH_\text{stat}(t,x),
	\quad
	\text{as }
	C([0,\infty),C(\R))
	\text{-valued processes}
\end{align}
for any $ t_0>0 $.
Utilizing~\eqref{eq:stat:SBE}, we show in Section~\ref{sect:pfthm6V} that
the centered height process $ (\limH_\text{stat}(t,x) - \limH_\text{stat}(t,0)) $ can in fact be extended to all values of $ t>-\infty $.
\begin{prop}
\label{prop:constrK}
There exits a $ C(\R,C(\R)) $-valued process $ \mathcal{K}(t,x) $ such that, for any fixed $ t_0\in\R $,
\begin{align}
	\label{eq:KeqH}
	\mathcal{K}(t-t_0,x) \stackrel{\text{law}}{=} \limH_\text{stat}(t,x) - \limH_\text{stat}(t,0),
	\quad
	\text{as }
	C([0,\infty),C(\R))
	\text{-valued processes in }(t,x).
\end{align}
\end{prop}
\noindent
Note that in the above proposition $ \mathcal{K}(t,x) $ is a process with $t\in \R$.
Given this,
the solution of $ \limU $ of the \ac{SBE} is defined as
\begin{align}
	\label{eq:SBE:}
	\limU: C^\infty_c(\R^2) \to \R,
	\quad
	\langle \limU, f \rangle
	:=
	- \int_{\R^2} \partial_x f(x,y) \mathcal{K}(y,x) dxdy.
\end{align}
Given that $ \limH_\text{stat}\in C(\R_+\times\R) $, it is straightforward to check $ \limU\in C^{-1}(\R^2) $.

The following is our main result on symmetric \ac{6V} model:
\begin{thm}
\label{thm:6V}
Consider the symmetric \ac{6V} model with vertex weights $(a,b,c)$ given by Baxter's projective parameters $(u,\eta)$ as in~\eqref{eq:projPara}.
Let  $ \eta=\eta_\e = \frac12\sqrt\e $ such that $ \Delta = \cosh(\eta_\e) \downarrow 1$.
There exist parameters $u=u_\e = \frac12\zeta\sqrt\e +o(\sqrt\e)$
for some constant $\zeta\in(0,\infty)$ such that,
for any densities $ (h,v)\in(0,1)^2 $ with $v$  fixed and $ h=h_\e $  given by~\eqref{eq:slope} (with the $\pm$ symbol fixed to be $-$), we have
\begin{align*}
	 U_\e \Longrightarrow \limU
	\quad
	\text{ in } C^{-1}(\R^2)
	\qquad \mbox{as }\e\to 0 .
\end{align*}
Here $U_\e$ is the
 the empirical distribution as in \eqref{eq:empirical} of the stochastic Gibbs state $ \Gibbs(b_1,b_2;h_\e,v) $,
with $ b_1,b_2 $  given in terms of %$ u,\eta_\e $
$ u_\e,\eta_\e $
through~\eqref{eq.bs}\eqref{eq:projPara};
and $ \limU $ is   the solution to \ac{SBE} given as in~\eqref{eq:SBE:},
with coefficients
\begin{align} \label{e:SBEcoeff}
	\nu_* = 2\zeta,
	\qquad
	\kappa_* = 2\zeta,
	\qquad
	D_* := 2\zeta v(1-v).
\end{align}
%Here, as $\zeta$ varies in $(0,\infty)$, one has $b_1\in (0,1)$
%leading to a family of coefficients as in \eqref{eq:ceffints}.
\end{thm}

\begin{rmk}
In order to see the \ac{SBE} limit here it is necessary to look along the characteristic (in the sense of Burger's equation) direction. In \eqref{eq:empirical}, this is reflected in the slope $\mue$ which is  given by the derivative of $h(v)$ evaluated at $v$ (as can be checked using \eqref{eq:slope}).
\end{rmk}

\noindent \emph{Proof sketch.}
This result is proved in Section \ref{sect:pfthm6V}. Since the stochastic Gibbs states come from a suitably chosen stochastic \ac{6V} model, we can apply Theorem \ref{thm:S6V} to prove convergence. The convergence is for positive times, but using the stationarity, we can extend it easily to all time.

\subsection{KPZ equation as a limit of ASEP}\label{sec:ASEPKPZ}

\acp{SPDE} describe the evolution of systems in the presence of random noise. The construction and approximation theory for non-linear SPDEs has attracted significant attention and enjoyed major breakthroughs in recent years (see, for instance, \cite{Bertini1997,Hairer13,Hairer14, GP2015a,GJ14a,GP15b}). Such equations are believed to describe the fluctuations of microscopic systems around their hydrodynamic limits.

The \ac{KPZ} equation is a model for random growth processes, interacting particle systems, and directed polymers \cite{Corwin12,QS15}. Writing $\limH(t,x)$ for the height at time $t\geq 0$ above $x\in\R$, the equation reads:
\begin{equation}
	\label{eq.kpz}
	\partial_t \limH(t,x) = \tfrac{\nu}{2} \partial_x^2 \limH(t,x) - \tfrac{\kappa}{2}\big( \partial_x \limH(t,x)\big)^2 + \sqrt{D}\noise(t,x),
\end{equation}
where $ \noise(t,x) $ denotes the Gaussian space-time white noise,
and $ \kappa\neq 0 \in\R $ and $ \nu,D >0 $ are constants measuring the strength of each term in~\eqref{eq.kpz}.

Making sense of \eqref{eq.kpz} is confounded by the non-linearity---solutions are rough enough that this does not make classical sense. The simplest, though indirect, approach is through the Hopf--Cole transform---one simply \emph{defines}
$\limH(t,x) =-\frac{\nu}{\kappa} \log \limZ(t,x)$ where $\limZ$ solves the \ac{SHE} (with multiplicative noise)\footnote{The positivity and well-posedness of \eqref{eq.SHEs} follows classical methods, see \cite{Corwin12,Q11} for further details.}:
\begin{equation}\label{eq.SHEs}
	\partial_t \limZ(t,x) = \tfrac{\nu}{2} \partial_x^2 \limZ(t,x) + \tfrac{\kappa\sqrt{D}}{\nu}\noise(t,x) \limZ(t,x).
\end{equation}
There are two other definitions which have been introduced recently and yield equivalent solutions: energy solutions \cite{GJ14a,GP2015a} and the regularity structures \cite{Hairer14}/paracontrolled distributions \cite{GP15b} (these last notions are for periodic $x\in [0,1]$). See also renormalization group techniques in \cite{Antti}.

How does the \ac{KPZ} equation arise from microscopic systems? Fixing $(b,z)\in \R^2$ and letting (for the moment) $\limH_{\e}(t,x) :=\e^{b} \limH(\e^{-z}t,\e^{-1}x)$ one sees that $\limH_{\e}$ satisfies a version of \eqref{eq.kpz} with scaled coefficients (see, for instance, \cite{Q11}). There are no choices for $(b,z)$ besides $(0,0)$ which leave the equation invariant. One may, however, simultaneously scale coefficients in \eqref{eq.kpz} to compensate for the effects of the $(b,z)$-scaling. This is a proxy for understanding how discrete models may converge to \eqref{eq.kpz} when one performs $(b,z)$-scaling while also scaling model parameters to effectively tune coefficients. This is called \DefinText{weak scaling}, and significant efforts have sought to show \DefinText{weak \ac{KPZ} universality}, meaning that general classes of processes converge to \eqref{eq.kpz} under such weak scaling.

Even though the focus of this work is on the \ac{6V} model, we focus for the moment on ASEP since it is a simpler process and allows us to cleanly identify the key challenge in proving the \ac{KPZ} equation limit for the stochastic \ac{6V} model.
The \ac{ASEP} is a continuous-time particle system in which particles inhabit sites indexed by $\Z$ and jump left and right according to continuous time exponential clocks with rates $\ell\geq 0$ and $r\geq 0$
(fix $\ell\geq r$ and $\ell+r=1$) subject to exclusion (jumps to occupied sites are suppressed). The \ac{ASEP} height function $ \Nasep(t,x) $ is defined just as for the stochastic \ac{6V} model and has $1$/$0$ slopes entering occupied/vacant sites (see Figure \ref{fig.asep}). \ac{ASEP} arises as a continuous time limit of the stochastic \ac{6V} model when $b_1=\e \ell$, $b_2=\e r$, time is scale to be $\e^{-1}t$ and particles are viewed in a moving frame with velocity $\e^{-1}$ (see \cite{BCG2016,A16b}).

\begin{figure}[ht]
\centering
\includegraphics[width=.4\textwidth]{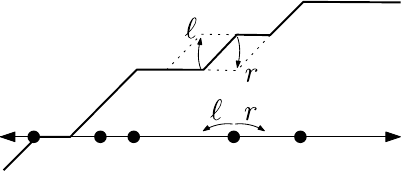}
\caption{ASEP particle configuration with the associated height function above it. Left jumps correspond to adding a rhombus and right jumps do the opposite.}\label{fig.asep}
\end{figure}

The \ac{ASEP} was the first discrete space system proved to converge to the \ac{KPZ} equation: \cite{Bertini1997} proved that for \emph{nearly stationary} initial condition with density $ \den=\frac12 $ (Definition~\ref{def:nearStat}),
under weak asymmetry scaling where $\ell-r = \sqrt\e$,
\begin{equation}\label{eq.bgres}
	\sqrt\e \Big( \Nasep(\e^{-2}t, \e^{-1} x) - \tfrac12 \e^{-1} x - \tfrac14 \e^{-\frac32} t  \Big)
	\stackrel{\e\to 0}{\Longrightarrow}
	\limH(t,x),
\end{equation}
as a space-time process. The starting point for this result was an observation in \cite{Gart88} that ASEP admits a \DefinText{microscopic Hopf--Cole transform}:
\begin{equation}\label{eq.mch}
\textrm{Setting } \tau = r/\ell \textrm{, and } Q(t,x) = \tau^{\Nasep(t,x)},\qquad dQ(t,x) = L^1_{r,\ell} Q(t,x) +Q(t,x) dM(t,x).
\end{equation}
Here $L^1_{r,\ell}$ is the generator of a simple continuous time random walk
with left and right jump rates given by $r$ and $\ell$ (note the exchange in left and right rates),
and $dM(t,x)$ is a martingale with explicit quadratic variation (see Appendix~\ref{sect:ASEP}).

The convergence in \eqref{eq.bgres} is shown not at the level of the height function, but rather its exponential,
by showing that the above microscopic \ac{SHE} \eqref{eq.mch} converges under the scalings in \eqref{eq.bgres} to its continuum version \eqref{eq.SHEs}.
Given tightness of the exponential process (which follows from detailed estimates on the random walk transition probability),
the convergence to~\eqref{eq.SHEs} is achieved via martingale problems (see Section \ref{sect:pfthmS6V}).
That is, the \ac{SHE} is uniquely characterized by a linear and quadratic martingale problem which, respectively, identify the drift and the noise.

Convergence of the linear problem follows easily by approximating $L^{1}_{r,\ell}$ with the Laplacian. The convergence of the quadratic problem is rather involved and ultimately boils down to showing that
\begin{equation}\label{eq.nab}
	\nabla Q(t,x+1) \nabla Q(t,x)
	\text{ self-averages in } t.
\end{equation}
Such expressions arise from the quadratic variation of the $dM(t,x)$. Here $(\nabla f)(x) = f(x+1)- f(x)$.
In \eqref{eq.nab}, ``self-averaging'' refers to a phenomena
where the moments of the integral of the expression over a long time interval of length $\mathcal O(\e^{-2})$
will vanish as $\e \to 0$, see \eqref{e:zero after time-averaging}.
For ASEP, this phenomena is explained more in Appendix~\ref{sect:ASEP}, in particular, see \eqref{e:heuristicASEP}.
In the case of stochastic six vertex model, the precise statement of ``self-averaging''
is given in Proposition~\ref{prop:qv}. See Remark~\ref{rem:self-av-mean}.

The statement~\eqref{eq.nab} is natural from the perspective of hydrodynamic limit theory.
Indeed, \cite{Q11} demonstrated how the replacement lemma (i.e., local equilibrium) can be used to prove~\eqref{eq.nab}.
The proof in \cite{Bertini1997} proceeded through a different, iterative scheme.
Roughly speaking, it seeks to close a sequences of inequalities starting from~\eqref{eq.mch}.
Crucial to the closing of inequalities (and hence to this scheme as a whole) is a non-trivial summation identity for the random walk transition probability.
%As explained in \cite{Q11}, the proof of this key bound in \cite[Lemma~4.8]{Bertini1997} is ``very hard to follow''. Hence, \cite{Q11} provides another proof of a sufficient bound using known tools (e.g. one and two-block estimates) from the hydrodynamic theory for ASEP.

\subsection{Markov duality method}
\label{sect:introDual}

The \emph{Markov duality method} that we employ in this article
provides a new way to obtain optimal control over the conditional expectation of the expression in \eqref{eq.nab} (and related terms).
More importantly, the method also applies to the general class of discrete time stochastic vertex models introduced in \cite{Corwin2016}---in particular, to the stochastic \ac{6V} model.
Presently, none of the other methods used for \ac{KPZ} equation convergence results seem to be applicable to the stochastic \ac{6V} model. The quadratic variation for the stochastic \ac{6V} model takes a more complicated form (as in~\eqref{eq:Theta1}--\eqref{eq:Theta2}) than that of \ac{ASEP}.
This being the case, the approach of \cite{Bertini1997} for closing inequalities does not appear to generalize.

Hydrodynamic theory methods like energy solutions \cite{GJ14a,GP2015a} or the approach to self-averaging given in \cite{Q11} relies heavily upon continuous time Markov process methods. In fact, hydrodynamic theory for discrete-time processes is not particularly well-developed as many of the basic tools that work in continuous time fail to generalize. The model considered here is updated sequentially in discrete time (see Section~\ref{sec.s6vs}), so, from the perspective of \emph{Markov chains}, the update of each particle depends on configurations of \emph{infinitely} many other particles.
This intricate feature further impedes generalizing methods of continuous time Markov process and hydrodynamic limit theory.
%As of yet, the methods from hydrodynamic theory for the \ac*{S6V} model have not been developed\footnote{In fact, we do not know of any discrete time processes for which such results have been proved.}.
%

Other methods like regularity structures \cite{Hairer14}, paracontrolled distributions \cite{GP15b} and renormalization group methods \cite{Antti} have not yet been sufficiently developed to deal with processes that are driven by a process-dependent noise (see, however, the recent work of \cite{Matetski} for progress on this in the context of regularity structures).
More precisely, this refers to the fact that the martingale in~\eqref{eq.mch} have a $ Q $-dependent quadratic variation.
%Regularity structures \cite{Hairer14} and paracontrolled distributions \cite{GP15b} have not yet been sufficiently developed to deal with discrete space systems, let alone discrete time.
Additionally, those methods are presently restricted to periodic boundary conditions. The Markov duality method works for discrete time processes with general initial condition on the full line. Its shortcoming is that it requires the existence of (at least $k=1,2$) Markov dualities like below. See Section \ref{sec.KPZlit} for further discussion on literature related to \ac{KPZ} equation convergence results.

\medskip

The microscopic Hopf--Cole transform \cite{Gart88} is the $k=1$ case  of ASEP Markov duality \cite{BCS14}:
\begin{equation}\label{eq.ASEPMD}
\textrm{For }k\geq 1\textrm{ and }\vec{x} = (x_1<\cdots <x_k)\in \Z^k, \quad \frac{d}{dt} \Ex\Big[\prod_{i=1}^{k} Q(t,x_i)\Big] = L^{k}_{r,\ell} \Ex\Big[\prod_{i=1}^{k} Q(t,x_i)\Big].
\end{equation}
Here $\Ex$ is the expectation of the ASEP height process, and $L^{k}_{r,\ell}$ acts on $\vec{x}$ as the space-reversed generator of $k$-particle ASEP with locations $\vec{x}$. For $k=1$, removing expectations yields \eqref{eq.mch}. Replacing $Q(t,x)$ by its discrete derivative $\tilde{Q}(t,x) := Q(t,x)-Q(t,x-1) $ yields a similar duality  due to \cite{MSchu97}.

The Markov duality method uses the $Q$ and $\tilde{Q}$ duality for $k=2$ to prove convergence of the discrete quadratic martingale problem to that of the SHE. For example, the key term in \eqref{eq.nab} can be rewritten as $\tilde{Q}(t,x+1)\tilde{Q}(t,x)$ and duality shows that for $x_1<x_2$ and $t>s$,
$$
\Ex\big[\tilde{Q}(t,x_1)\tilde{Q}(t,x_2)\big\vert \filt(s)\big] = \sum_{y_1< y_2} p_{t-s}(\vec{x}\to \vec{y}) \tilde{Q}(s,y_1)\tilde{Q}(s,y_2)
$$
where $p_{t-s}(\vec{x}\to \vec{y})$ is the two-particle space-reversed ASEP transition probability from $\vec{x}=(x_1,x_2)$ to $\vec{y}=(y_1,y_2)$ in time $t-s$. Once in this form, the discrete differentiation can be transferred to the transition probabilities and the proof of self-averaging reduces to fine estimates on such derivatives of the two-particle heat kernel. In essence, duality turns a hydrodynamic problem (involving the local equilibration in the collective behavior of many particles) into a diffusive problem (involving the fluctuations of a handful of particles).

The Bethe ansatz (for ASEP, see \cite{TracyWidom08,TracyWidom08e} or Appendix \ref{sect:ASEP}) provides a means to extract very precise estimates for finite particle system transition probabilities. We also remark that whereas previous results on the KPZ equation limit for ASEP have assumed density near $1/2$, the duality method works just as well for any density and for any moving frame in the rarefaction fan.
%Also, as explained \note{do we explain this?} in  Appendix \ref{sect:ASEP}, this method provides much more precise self-averaging estimates for ASEP as compared to the previous approaches in \cite{Bertini1997,Q11}.

The major downside of our Markov duality method is that such dualities like \eqref{eq.ASEPMD} do not hold for generic systems and their occurrence is often due to algebraic structures which are not very flexible to perturbations (see Section \ref{sec.duallit} for further discussion). However, it was shown in \cite{Corwin2016,JK17} that the stochastic \ac{6V} model (in fact its higher spin generalizations too) enjoy the same sort of duality as in \eqref{eq.ASEPMD} (see Section \ref{sect:duality}). We see the main technical accomplishment of this paper to be the use of this duality method to control the quadratic martingale.

Let us attempt to put the Markov duality method into historical context. The first instance where Markov duality was used to prove an SPDE limit was in the work of \cite{DPS} which focused on the {\it very} weakly asymmetric simple exclusion process (with weaker asymmetry than in \cite{Bertini1997}). Since the asymmetry in that work was sufficiently weak, the limiting SPDE was a linear (Gaussian) SPDE -- the additive SHE. The approach of \cite{DPS} relied on estimates for occupation variable correlation functions. For the symmetric (SSEP) model, these functions satisfy closed equations due to a Markov self-duality for SSEP. In the presence of asymmetry, \cite{DPS} derived an infinite hierarchy of relations for correlation functions which, for very weak asymmetry, they could control in a perturbative manner using the SSEP duality (see \cite{DeMasiPresutti, Ravishankar} for further discussion of this approach).

For stronger asymmetry (as considered in \cite{Bertini1997} and herein), the \cite{DPS} perturbation method breaks down. Instead, we use the ASEP self-dualities (which are non-local and generalize the SSEP correlation functions in certain cases) which yield a closed hierarchy. Moreover, we only need to use the one and two particle duality, as opposed to the full hierarchy (i.e., arbitrarily many dual particles).

\subsection{Further literature}\label{sec:liter}

\subsubsection{Symmetric six vertex model}\label{sec.6vlit}

Introduced in 1935 by Pauling \cite{LPauling35} as a model for 2D ice and then in its general form in 1941 by Slater \cite{Slater41} to model potassium dihydrogren phosphate, the symmetric \ac{6V} model has found many applications across physics and mathematics as well as prompted the discovery of new algebraic structures such as quantum groups and new symmetric functions. The \ac{6V} model was exactly solved in Lieb's breakthrough work \cite{Lieb67} which was the first time the ideas of Bethe ansatz were applied to a statistical mechanics model. This work immediately (e.g. \cite{SuthL67,YangYang66}) opened up the field to many important developments including coordinate/algebraic Bethe ansatz, quantum groups, domain-wall boundary conditions, connections to symmetric functions---see the reviews/books \cite{Baxter89, Nolden92, Fadeev, Korepin, jimbo, Res2010,Bleher,gaudin,KKK15, BP16b}). The results of this paper probe the behavior of the \ac{6V} model as $\Delta\searrow 1$. There are many other interesting phase transitions in the \ac{6V} model---for instance when $a=b$ (i.e., the Fierz, or F model---studied first in \cite{Rys}), as $c\to 2a$ (or equivalently $\Delta\to -1$) there is a remarkable infinite order phase transition in the free energy (see \cite{LiebWu} for further information).

\subsubsection{Stochastic six vertex model}\label{sec.s6vlit}

Study of this special case of the asymmetric \ac{6V} model was initiated in 1992 by Gwa and Spohn \cite{GS92}. However, the relation between the conical points and the stochastic \ac{6V} model was conjectured in 1995 by Bukman and Shore \cite{Bukman1995}, though there was earlier discussion about the existence of these conical points in  \cite{JS84}.

The study of the stochastic \ac{6V} model was recently reinitiated in \cite{BCG2016} wherein they proved the prediction from \cite{GS92} that the stochastic \ac{6V} model was in the \ac{KPZ} universality class. This was demonstrated at the level of convergence of the one-point distribution (to the GUE Tracy-Widom distribution) for a special boundary condition on the first quadrant with no lines coming from the $y$-axis and no anti-lines coming from the $x$-axis (i.e., step initial condition). This result did not involve any special weak scaling, hence convergence to the GUE Tracy-Widom distribution and not the one-point distribution for the \ac{KPZ} equation. \cite{AB16,A16aa} then extended the one-point convergence to other initial condition, including the stationary case (i.e., the stochastic Gibbs state).

In that case, \cite{A16aa} computed an exact one-point formula and proved convergence to the stationary \ac{KPZ} distribution (the Baik-Rains distribution) in the characteristic direction. In principle one could take the weakly asymmetric scaling limit of that formula and match it with the formula for the stationary \ac{KPZ} equation proved in \cite{BCFV} (though that would only prove a one-point convergence result, as opposed to the process level result herein). In a similar spirit, \cite{BO16} showed that under weakly asymmetric scaling, one point distribution of the stochastic \ac{6V} model converges to that of the \ac{KPZ} equation (see also \cite{BG16}). The scaling considered in \cite{BO16} is different than here---essentially they also tune $b_1,b_2\to 1$ (herein they converge to a value strictly less than 1). It is quite likely that our approach could apply under the scaling used in \cite{BO16}, though we do not pursue that here.

\cite{BBCW} recently studied a half-space version of the stochastic \ac{6V} model and demonstrated that its one-point asymptotics match the prediction from other models in the \ac{KPZ} universality class. It may be possible to adapt methods from \cite{corwin2016open} (see also \cite{Shalin}) to connect the half-space stochastic \ac{6V} model to the \ac{KPZ} equation under weakly asymmetric scaling, though we do not pursue that here.

The stochastic \ac{6V} model admits a higher spin analog wherein more than one line can move along each edge in $\Z^2$ (i.e., multiple particles can occupy the same site, or move together). These models have recently been studied in \cite{Corwin2016,BP16a} and admit some similar asymptotics as the stochastic \ac{6V} model. The Markov duality method should also apply to these models (as they all enjoy the same duality as the stochastic \ac{6V} model).

There are other limits of the stochastic \ac{6V} model besides the \ac{KPZ} equation and ASEP, e.g. Hall-Littlewood PushTASEP \cite{BufPet, BCG2016, BBW16, G17} and Brownian motions with oblique reflection \cite{SS15}.
%Another limit considered in forthcoming work \cite{BorGorforth} leads (conjecturally) to the stochastic telegraph equation---a linear SPDE driven by additive space-time white noise.
Another limit considered in parallel to the present paper is in the work of \cite{BorGorforth}. They consider a different type of limit in which $b_1$ and $b_2$ both tend to $1$. \cite{BorGorforth} proves a law of large numbers and some Gaussian fluctuation results under this scaling. Moreover, they conjecture  (and prove in a certain low density regime) convergence to the stochastic telegraph equation---a linear hyperbolic SPDE driven by additive space-time white noise. That conjecture has now been proved in \cite{ShenTsai}. It would be natural to try to fill-out the scaling limits which sit between our results and those of \cite{BorGorforth,ShenTsai}.

%We hope to apply our duality method to prove that convergence, as well as probe some additional limits which may interpolate between that and the \ac{KPZ} equation. We also plan to apply our duality method to the \emph{dynamic} stochastic \ac{6V} model introduced in \cite{Borodinellitpic} which also enjoys a Markov duality, as shown in \cite{BorodinCorwinMarkov}.

\subsubsection{Kardar-Parisi-Zhang equation}\label{sec.KPZlit}

The \ac{KPZ} equation \eqref{eq.kpz} was introduced in 1986 by Kardar, Parisi and Zhang \cite{KPZ86}. In 1995 Bertini and Cancrini \cite{Bertini1995} provided the first justification for the Hopf--Cole solution to the KPZ equation. Bertini and Giacomin \cite{Bertini1997} soon after proved the first discrete convergence result (for ASEP) to the \ac{KPZ} equation.
This converge result has recently been  extended in works of \cite{Amir11, Q11}. \cite{Dembo2016} extended the convergence result to certain non-nearest-neighbor exclusion processes which do not satisfy an exact microscopic Hopf--Cole transform.

The first convergence result to the \ac{KPZ} equation for a discrete time particle system was recently proved in \cite{CT15}. The systems considered therein were infinite spin versions of the higher spin vertex models studied in \cite{Corwin2016}. The scaling there was different than the weakly asymmetric scaling here. In particular, the number of particles per site diverges under their scaling. This simplified the study of the quadratic martingale problem considerably. In particular, due to the divergence of the number of particles per site, the key bound which plays a central role in this work and in that of \cite{Bertini1997} becomes straightforward and does not require any sort of trick to control. Other recent \ac{KPZ} equation convergence works, following the style of \cite{Bertini1997}, have included the ASEP-$(q,j)$ \cite{CSL16}, Hall-Littlewood PushTASEP \cite{G17}, and open ASEP \cite{corwin2016open,Shalin,Labbe16b}.

The energy solution method for \ac{KPZ} equation convergence was initiated in the work of the Jara and Gon\c calves \cite{GJ10} (cf. \cite{Assing}). Initially this approach only provided tightness and it was not known whether energy solutions were unique. Uniqueness (and hence the identification with the Hopf--Cole solution) was proved in \cite{GP2015a}. This approach has been applied to prove that a wide variety of particle systems converge to the \ac{KPZ} equation, see \cite{GJ14a, GJS15,TGS16,GJ13,GJ16,GonPerMar}. Those results require stationary initial condition and the method of proof relies heavily upon having well-developed hydrodynamic theory estimates available. Quite recently, \cite{Yang} has extended this method to include more general initial data such as flat.

Regularity structures and paracontrolled distributions provide another route to prove convergence results to the \ac{KPZ} equation. These notions of solutions were introduced by Hairer \cite{Hairer13,Hairer14} and Gubinelli and Perkowski \cite{GP15b} (cf. \cite{GIP}), and have since been used to prove convergence for some space-time regularized versions of the equation \cite{HaiShen, HaiQua, DieGubPer}. \cite{HaiMat,CanMat,erhard2017discretisation} has recently developed a discrete space-time version of regularity structures, which may prove useful in demonstrating convergence of various discrete processes to the \ac{KPZ} equation.
It is worth noting that presently due to technical challenges involved with going to the full line,
these works on the KPZ equation using regularity structures or paracontrolled distributions are restricted to periodic spatial coordinates. Finally, there is also a renormalization group method which has been applied to the KPZ equation in \cite{Antti}, though it also is also restricted to a periodic setting.

\subsubsection{Markov duality}\label{sec.duallit}

Markov dualities are extremely useful notions within probability. An early example of a self-duality was for the simple symmetric exclusion processes (SSEP) \cite{Liggett85} where it played a key role in proving that only extremal, translation invariant, ergodic invariant distributions of SSEP on $\Z^d$ are the Bernoulli product distributions. Whereas that duality applied, in fact to SSEP on any graphs, asymmetric particle system dualities seem to be much more rigid and dependant upon algebraic structures only present for one spatial dimension. The first such example was found in \cite{MSchu97} where the $\tilde{Q}$ version of the duality in \eqref{eq.ASEPMD} was first discovered based on the affine quantum group $U_q[\mathfrak{sl}_2]$ symmetry of ASEP (see also \cite{SS94}). The self duality of ASEP  has played an important role in demonstrating that ASEP belongs to the \ac{KPZ} universality class (see, for instance, \cite{BCS14, IC14Springer} and the reference therein).

Recently, a generalized version of ASEP (called ASEP-$(q,j)$) which enjoys a generalization of the ASEP self-duality was introduced in \cite{CGRS16} based on higher spin representations of $U_q[\mathfrak{sl}_2]$. Self duality has been also proved \cite{BS15, JK16} in certain multi-species versions of ASEP using higher rank quantum group symmetries in the spirit of \cite{CGRS16} (see also \cite{Wheeleretal} which relates duality to the Knizhnik-Zamolodchikov equation).

The stochastic \ac{6V} model (as well as higher spin vertex models) duality was discovered and proved in \cite{Corwin2016} (see \cite{JK17} for an algebraic proof of some of the dualities from \cite{Corwin2016} based on properties of the $R$ matrix and quantum group co-product, and see \cite{Yier} for a discussion of an fix to a mistake present in \cite{Corwin2016}). It is this duality for the stochastic \ac{6V} model that plays a pivotal role in this paper and is discussed in more detail in Section \ref{sect:duality}.

%
%For symmetric models, duality has played an important role in hydrodynamics and spde limits; see
%
%A. De Masi, E. Presutti, Mathematical methods for hydrodynamic limits, Lecture Notes
%in Mathematics, Springer-Verlag, Berlin (1991).
%
%
%Exact formulas for two interacting particles and applications in particle systems with duality
%Gioia Carinci, Cristian Giardina, Frank Redig
%arXiv:1711.11283

\addtocontents{toc}{\protect\setcounter{tocdepth}{0}}
\subsection*{Outline}
In Section~\ref{S6VMIntro} we give a brief discussion the stochastic and symmetric \ac{6V} models.
This includes the definition of the stochastic model with \emph{bi-}infinite configurations,
the construction of stochastic Gibbs states,
and how they fit into the stochastic and symmetric models.
Then, to setup the premise of our analysis,
in Section~\ref{sect:duality} we recall the self-duality of the stochastic model,
and in Section~\ref{sect:HC}, we introduce the microscopic Hopf--Cole transform.
Specifically, once the transform is introduced,
Theorem~\ref{thm:S6V} on the convergence of the stochastic model to \ac{KPZ} naturally translates
into the corresponding, equivalent statement in terms of convergence toward the \ac{SHE}, Theorem~\ref{thm:S6V:}.
In Section~\ref{sect:pfthm}, we settle the main results Theorems~\ref{thm:S6V:} and \ref{thm:6V}
while assuming Proposition~\ref{prop:qv}.
The latter is a statement on self-averaging of the relevant quadratic variation.
Proving Proposition~\ref{prop:qv} makes up the core of our analysis.
In Section~\ref{sect:SG}, we perform steepest-decent-like analysis on the given
contour integral formula for the semigroup.
The analysis produces estimates on the semigroup and its gradients, jointly over all relevant points in spacetime.
Then, in Section~\ref{sect:qv},
we incooperate these estimates into the stochastic model via duality and prove Proposition~\ref{prop:qv}.

To make connection with \ac{ASEP}, in Appendix~\ref{sect:ASEP},
we briefly recall its Hopf--Cole transform and the structure of the relevant martingale.
Given this setup, we explain how, for \ac{ASEP}, our duality approach could serve as an alternative
to the approach of \cite{Bertini1997} for controlling the quadratic variation.

\subsection*{Acknowledgements}
The authors wish to thank J. Dub\'edat, P. di Francesco, and T. Spencer for discussions about the disordered phase of the \ac{6V} model, A. Aggarwal for discussions about his work on this subject,  Yier Lin for careful reading of part of the manuscript and Krishnamurthi Ravishankar for helpful historical references about the role of dualities in SPDE limits. I. Corwin was partially supported by the NSF through DMS:1811143 and DMS:1664650, and the Packard Foundation through a Packard Fellowship for Science and Engineering. H. Shen is partially supported by the NSF through DMS:1712684.  L-C. Tsai is partially supported by the NSF through DMS:1712575 and the Simons Foundation through a Junior Fellowship. Both I. Corwin and P. Ghosal collaborated on this project during the 2017 Park City Mathematics Institute, partially funded by the NSF through DMS:1441467.
\addtocontents{toc}{\protect\setcounter{tocdepth}{2}}

\section{Stochastic and symmetric six vertex models}\label{S6VMIntro}
We now provide more detailed definitions of the stochastic and symmetric \ac{6V} models.

\subsection{Stochastic six vertex model as a particle system and its height function}\label{sec.s6vs}

\subsubsection{Defining the left-finite process}
In \cite[Section 2]{BCG2016}, the stochastic \ac{6V} model is defined on the first quadrant $\Z_{>0}^2$ by first specifying the configuration of lines coming from the bottom and left boundary and then inductively filling in the quadrant. Specifically, once it is determined whether  lines are entering a given vertex from below and from the left, the stochastic weights in  Figure \ref{tbl:S6VWeights} specify the probability according to which one chooses (independently over vertices) the outgoing line configuration. Proceeding recursively in this manner defines the stochastic \ac{6V} model distribution on the entire quadrant (for the given boundary condition).

If we restrict ourselves to boundary conditions where there are no lines coming from the left boundary, then the lines from the bottom can be seen as the trajectories of a discrete time sequential update exclusion-type particle system. Under this interpretation, time is measured by the $y$-axis, and the particles are identified with vertical lines and their moves are identified with the horizontal lines. We define below this particle system and allow particles to start anywhere on $\Z$ as long as there is always a left-most particle. After doing that, we explain how to extend our definition to two-sided infinite particle configurations (as will be necessary to state our main results).

\begin{defn}%[S6V model as a particle system with left-finite configurations]
\label{def:S6V:lf}
For $w\in \Z$ define the space of left-finite ordered particle configurations with left-most label $w$ to be
\begin{equation}\label{eq:RightFiniteParticleConfiguartion}
\Xlf{w}
:= \Big\{\vec{x}=(-\infty=x_{w-1}<x_w<x_{w+1}<\ldots):
	x_i\in \Z\cup\{\pm\infty\}, \text{ for } i\in\Z_{\geq w} \Big\}.
\end{equation}
Here $x_i$ represents the location of the particle labeled $i$. Notice that we have placed a virtual particle $x_{w-1}$ at $-\infty$. We allow $\Xlf{w}$ to contain configurations with infinitely many particles as well as finitely many particles. In the later case, there will be some $w'$ such that $x_{i}=+\infty$ for all $i>w'$.

Having defined our state space $\Xlf{w}$ we proceed to describe the discrete time Markov chain $\big(\vec{x}(t)\big)_{t\in\Z_{\geq 0}}$ where $\vec{x}(t)\in \Xlf{w}$ for each $t$. Fix $b_1,b_2\in(0,1)$ and let
$$\tau=b_2/b_1\in(0,1)$$
denote their ratio. We will assume that $b_2<b_1$ so that $\tau\in(0,1)$ throughout. The algebraic results do not generally depend on this, but when we perform asymptotics we will use this asymmetry. Given $\vec{x}(t)$, we choose $\vec{x}(t+1)$ according to the following sequential (left to right) procedure. For each $i \geq w$ (starting with $i=w$ and progressing sequentially to $i=w+1$, $i=w+2$, etc),
choose $x_i(t+1)$ so that (recall that $x_j(t+1)$ for $j<i$ have already been updated)\\
$(a)$ if $x_{i-1}(t+1)<x_i(t)$, then
\begin{align*}
\Pr\big(x_i(t+1)=x_i(t)+j\big)&=\left\{\begin{array}{l@{,}l}
b_1 & \text{ if }j=0;\\
(1-b_1)(1-b_2)b^{j-1}_2
	& \text{ if }1\leq j \leq x_{i+1}(t)-x_{i}(t)-1;\\
(1-b_1) b^{j-1}_2 & \text{ if }j= x_{i+1}(t)-x_{i}(t) ;\\
0 &\text{ otherwise};
\end{array}\right.
\end{align*}
$(b)$ if $x_{i-1}(t+1)=x_{i}(t)$, then
\begin{align*}
\Pr\big(x_{i}(t+1)=x_{i}(t)+j\big)
&=
\left\{\begin{array}{l@{,}l}
(1-b_2)b^{j-1}_2 & \text{ if }1\leq j <x_{i+1}(t)-x_{i}(t);\\
 b^{j-1}_2 & \text{ if }j=x_{i+1}(t)-x_{i}(t) ;\\
 0 & \text{ otherwise}.
\end{array}\right.
\end{align*}
Since we have assumed the convention $x_{w-1}(t)=-\infty$, the particle $x_w$ is always updated by rule $(a)$.

In words, sequentially (starting with particle $x_w$) each particle $x_i$ wakes up and moves one to the right with probability $1-b_1$. Once awake, the particle continues moving right with probability $b_2$ for each step. If $x_i$ eventually moves into the location occupied already by $x_{i+1}$, then $x_i$ stops moving and stays put, while $x_{i+1}$ is forced to wake up and move one to the right (after which it continues with the probability $b_2$ rule as above). Once the particle $x_i$ stops, that is its new position $x_i(t+1)$.

To each state $\vec{x}(t)\in \Xlf{w}$ we may associate occupation variables and a height function as follows: Define the $\{0,1\}$-valued \DefinText{occupation variables}
\begin{align}
	\eta(t,y) :=\ind_{\set{x_n(t)=y \text{ for some }n\in \Z_{\geq w}}}
\end{align}
where the indicator function is $1$ if the site $y$ is occupied by a particle at time $t$, and $0$ otherwise.
Likewise, define the \DefinText{height function}
%\footnote{This height function $N(t,y)$ slightly different than the height function discussed in the introduction. That one changed valued by $\pm$ as $y$ increases, whereas this one changes values by $0$ or $1$. They can be related by setting $h(t,y)=2N(y,t)-y$. We find it easier to work with the height function $N(t,y)$ in the main text.}
% (or integrated current) $h:\Z_{\geq}\times \Z\to \Z$
%\begin{align}
%	\label{eq:HeightFunc}
%	h(t,y) := N(t,y)-N(0,0)
%\end{align}
%where $N:\Z_{\geq}\times \Z\to \Z$  is defined by
\begin{align*}
N(t,y) := N_y\big(\vec{x}(t)\big)-N_0\big(\vec{x}(0)\big).
\end{align*}
(We have centered $N$ so that $N(0,0)=0$.) In the above definition, we have used the following notation. For $y\in \Z$, $N_y:\Xlf{w}\to \Z_{\geq w-1}$ and (for later use) $\eta_y:\Xlf{w}\to \{0,1\}$) are defined by\footnote{Note that $\eta(t,y)=\eta_{y}(\vec{x}(t))$.
We distinguish the notation $\eta(t,y)$ as a process and the notation $\eta_{y}$ as a function on particle configurations $\vec{x}$ merely for convenience.}
\begin{align} \label{eq:def-Nyx}
N_y(\vec{x}) := \max\big\{n:x_n\leq y\big\} \qquad \text{and} \quad\eta_y(\vec{x}) := N_y(\vec{x})-N_{y-1}(\vec{x}).
\end{align}
%In other words, $h(t,y)$ tracks the number of lines at or to the left of $(t,y)$ with centering so $h(0,0)=0$.
In particular, one has $N_{x_n}(\vec{x})=n$, and $N_y(\vec{x}) = w-1$ if $y$ is to the left of all particles in $\vec{x}$.
It follows that  % $h(t,y)-h(t,y-1) = N(t,y)-N(t,y-1) = \eta(t,y)$,
$ N(t,y)-N(t,y-1) = \eta(t,y)$, so that the space-time  level-lines of %$h(t,y)$ and
$N(t,y)$ correspond with the trajectories of $\vec{x}(t)$. See Figure \ref{fig.s6vpart} for an illustration.
\end{defn}

Under the dynamics described above in Definition~\ref{def:S6V:lf}, the height function $N(\Cdot,t)$ evolves in $ t $ as a Markov chain.
%\footnote{The notation $t\mapsto N(\Cdot, t)$ means the function which takes $t$ to height functions (defined on the full line). By convention, such height functions will have height $w-1$ asymptotically as $y$ goes to $-\infty$. The Markov chain is defined on that state-space.}.
We may describe its transitions explicitly.
\begin{defn}\label{def:Bs}
Let $X\sim \text{Ber}(\den)$ mean that $X$ is a Bernoulli random variable taking values in $\{0,1\}$ with $\Pr(X=1)=\den$. Let $\big\{B(t,y;\eta), B'(t,y;\eta): t\in\Z_{\geq 0}, y\in\Z, \eta\in\{0,1\}\big\}$
denote a countable collection of independent Bernoulli variables,
with $B(t,y;\eta)\sim \text{Ber}\big(1-b_1^\eta\big)$ and $B'(t,y;\eta)\sim \text{Ber}\big(b_2^{1-\eta}\big)$.
\end{defn}

Using the Bernoulli random variables from the above definition we see that
\begin{align}\label{eq:UpdateINS6VOccupationProcess}
	N(t+1,y)
	&\stackrel{\text{law}}{=}
	\left\{\begin{array}{l@{,}l}
   		N(t,y) -  B'(t,y;\eta(t,y)) & \text{  if } N(t+1,y-1)=N(t,y-1)-1,
   \\
    	N(t,y) -  B(t,y;\eta(t,y)) & \text{  if } N(t+1,y-1)=N(t,y-1).
    \end{array}
    \right.
\end{align}

\subsubsection{Defining the bi-infinite process}
Since the stochastic \ac{6V} model is sequentially updated,  it is not a priori clear how to define it when there are infinitely many particles to the left and right of the origin. \cite{CT15} showed that it is possible to restate the stochastic \ac{6V} model in terms of a parallel update rule which readily admits a bi-infinite extension. We restate this result below as well as include a convergence result showing how to approximate the bi-infinite process with left-finite ones.
\begin{defn}
Denote the space of bi-infinite order particle configurations by
$$
\Xinf= \big\{
	\cdots<x_{-1}<x_0<x_1<\cdots : x_i\in \Z\cup \{-\infty,+\infty\}\big\}.
$$
Notice that we have included left and right finite configurations in $\Xinf$
by having imaginary particles at $-\infty$ or $\infty$.
%(with the convention that $\pm \infty<\pm\infty$). \hao{what does $\pm \infty<\pm\infty$ mean?}
\end{defn}

\begin{lem}\label{lem:binfinite}
Consider a bi-infinite configuration $\vec{x}\in \Xinf$ and let $\vec{x}_{\geq w}= \big(x_{i}:i\geq w\big)\in \Xlf{w}$ for any $w\in \Z$. Let $N(0,y) = N_y\big(\vec{x}\big)-N_0\big(\vec{x}\big)$ and $N^{w}(t,y) = N_y\big(\vec{x}_{\geq w}(t)\big)-N_0\big(\vec{x}_{\geq w}(0)\big)$ where $\vec{x}_{\geq w}(t)$ is the stochastic \ac{6V} Markov chain at time $t$ with initial condition $\vec{x}_{\geq w}$. Likewise, let $\eta(0,y) = N(0,y)-N(0,y-1)$ and $\eta^w(t,y) = N^w(t,y)-N^w(t,y-1)$. Let $B(t,y,\eta)$ and $B'(t,y,\eta)$ be as in Definition \ref{def:Bs}. Then for any $t\in \Z_{\geq 0}$ and $w,y\in \Z$, we have that
$$
N^w(t,y)- N^w(t+1,y) =
\sum_{y'=x_w}^y \prod_{z=y'+1}^y \Big(B'\big(t,z; \eta^w(t,z)\big)-B\big(t,z;\eta^w(t,z)\big)\Big) B\big(t,y';\eta^w(t,y')\big).
$$
Furthermore for any $y\in \Z$, as $w\to -\infty$, $N^w(t,y)\to N(t,y)$ in $L^p$ for all $p\geq 1$ and in probability. The limit $N(t,y)$ is specified inductively in $t$ (with $t=0$ as the base case) by the (convergent) relation %\note{check with the sums above and below should go to $y$ or $y-1$. There was some confusion in the earlier draft.}
\begin{align}
	\label{eq:moveBer}
	N(t,y)- N(t+1,y)
	= \sum_{y'\leq y} \prod_{z=y'+1}^y \Big(B'\big(t,z;\eta(t,z)\big)-B\big(t,z;\eta(t,z)\big)\Big) B\big(t,y';\eta(t,y')\big)
\end{align}
and hence satisfies \eqref{eq:UpdateINS6VOccupationProcess}. From $N(t,y)$ we define $\eta(t,y) = N(t,y)-N(t,y-1)$, and we may uniquely define $\vec{x}(t)$ so that the particles of $\vec{x}(t)$ track the level lines of $N(t,y)$.
\end{lem}

%{\color{blue}
%I think that \eqref{eq:moveBer} is not the right formula.
%Assuming $N(t+1,y-1)=N(t,y-1)-1$, we can write \eqref{eq:moveBer}  as
%\begin{align*}
%{} & N(t,y)  - N(t+1,y)
%\\
%&= \sum_{y'\leq y-1}
%\Big(B\big(t,y,\eta(t,y)\big)-B'\big(t,y,\eta(t,y)\big)\Big)
%\prod_{z=y'+1}^{y-1}
%\Big(B\big(t,z,\eta(t,z)\big)-B'\big(t,z,\eta(t,z)\big)\Big)
%B'\big(t,y',\eta(t,y')\big)
%\\
%&\qquad +B'\big(t,y,\eta(t,y)\big)
%\\
%&=
%\Big(B\big(t,y,\eta(t,y)\big)-B'\big(t,y,\eta(t,y)\big)\Big)
%+B'\big(t,y,\eta(t,y)\big)
%=B\big(t,y,\eta(t,y)\big)
%\end{align*}
%where we used the inductive assumption
%$N(t,y-1)-N(t+1,y-1)=1$ to go from the 2nd+3rd line to the 4th line.
%This is not consistent with \eqref{eq:UpdateINS6VOccupationProcess}.
%A similar calculation shows that
%assuming $N(t+1,y-1)=N(t,y-1)$
%one has $N(t,y)  - N(t+1,y) =B' \big(t,y,\eta(t,y)\big)$,
%again contradicts with \eqref{eq:UpdateINS6VOccupationProcess}.
%
%The right formula should be
%\begin{align*}
%	N(t,y)- N(t+1,y)
%	= \sum_{y'\leq y} \prod_{z=y'+1}^y \Big(B'\big(t,z,\eta(t,z)\big)-B\big(t,z,\eta(t,z)\big)\Big) B\big(t,y',\eta(t,y')\big)
%\end{align*}
%and this will also be consistent with \cite[Lemma 2.3]{CT15}.
%If we all agree this is indeed the case, I will correct the relevant formulae in the tightness section.
%}

\begin{proof}
The result is a special case of the statement and proof of \cite[Lemma 2.3 and Remark 2.5]{CT15}. In \cite{CT15} the authors consider a more general higher-spin version of the stochastic \ac{6V} model \cite{Corwin2016} with arbitrary horizontal spin $J$ as well as parameters $\alpha,q,\nu$. Our stochastic \ac{6V} model corresponds with taking $J=1$ (spin-$\tfrac{1}{2}$), $ \nu=1/q=\tau $,
and matching $ b_1=\frac{1+ q\alpha}{1+\alpha} $ and $ b_2=\frac{\alpha+q^{-1}}{1+\alpha} $.
\end{proof}

Unless specified otherwise, the stochastic \ac{6V} model now means the bi-infinite  version of Lemma \ref{lem:binfinite}.

\subsubsection{Stationary initial condition}

A key aspect of studying an interacting particle system is to identify its stationary distributions, in particular those which are translation invariant and ergodic. These distributions are the first step towards identifying the hydrodynamic equations and non-universal constants which arise in the \ac{KPZ} scaling theory (see, for instance, \cite{Spohn12} and references therein). For ASEP these are characterized by one parameter $\den\in [0,1]$ and given by product distribution $\text{Ber}(\den)$ on occupation variables. The same distributions turn out to be stationary of the stochastic \ac{6V} model. In fact, as shown in \cite{A16aa}, the stationary stochastic \ac{6V} model enjoys a sort of stationarity along down-right paths very much akin to that of certain exactly solvable directed polymer and last passage percolation models (see, for instance, \cite{Sep12}).

\begin{defn}
\label{def:station}
Consider the stochastic \ac{6V} model with parameters $b_1, b_2$.
Choose $(h,v)\in [0,1]^2$ such that \eqref{vhbeq} holds, namely
%$ \frac{v}{1-v}(1-b_2) = \frac{h}{1-h}(1-b_1) $.
$ \frac{v}{1-v}(1-b_1) = \frac{h}{1-h}(1-b_2) $.
The stationary stochastic \ac{6V} model on the first quadrant is defined relative to $ (h,v) $
by specifying that on the $y$-axis ($x$-axis) horizontal (vertical)  lines enter from the boundary independently with probabilities $h$ ($v$).
\end{defn}

%The following lemma shows in what sense this boundary condition is stationary for the stochastic \ac{6V} model and allows us to define the stationary stochastic \ac{6V} model on all of $\Z^2$.

\begin{lem}\label{lem:stat}
Consider the stationary stochastic \ac{6V} model on the first quadrant from Definition~\ref{def:station}.
Then, along any fixed down-right lattice path in the first quadrant (i.e., a collection of vertices in $\Z_{\geq 0}^2$
so that each vertex follows the previous one by adding $(1,0)$ or $(0,-1)$ to its coordinates)
the sequence of incoming line occupancy variables
(i.e., whether a horizontal or vertical line enter vertices along the path)
are independent and incoming horizontal lines are present with probability $ h $
while incoming vertical lines are present with probability $ v $.
Consequently, we can define the stationary stochastic \ac{6V} model on all of $\Z^2$
by taking the distributional limit as $n\to \infty$
of the model on the first quadrant with the origin shifted to $(-n,-n)$.
We refer to this distribution (of vertex configurations on $ \Z^2 $)
as the \DefinText{stochastic Gibbs states} with densities $ (h,v) $,
and denote it by $ \Gibbs(b_1,b_2;h,v) $.
\end{lem}

\begin{proof}
This is the content of \cite[Lemma~A.2]{A16aa}.
\end{proof}

The distribution $\Gibbs(b_1,b_2;h,v)$ does not treat the $x$-axis and $y$-axis directions differently. In terms of the particle process interpretation for the stochastic \ac{6V} model, this stationary distribution corresponds to starting with particles independently at each site of $\Z$ with probability $v$. The parameter $h=h(v)$ then corresponds to the probability that a particle crosses a given vertical column at a given time, and the stationarity says that these events are all independent. The function $h(v)$ is called the \emph{flux}.

\subsection{Stochastic Gibbs states for the symmetric six vertex model}

As discussed in the introduction, the stochastic Gibbs states constructed in Lemma \ref{lem:stat} are Gibbs states for a symmetric \ac{6V} model in the ferroelectric phase with parameters matched accordingly.

\begin{prop}\label{GibbsMeasureProposition}
Consider positive $ (a,b,c) $ such that $ a>b+c $ and such that $\Delta>1$ (recall $ \Delta $ from~\eqref{eq:Delta}). Let $ b_1,b_2 $ be given as in~\eqref{eq.bs},
and $ (h,v)\in[0,1]^2 $ satisfy \eqref{vhbeq}, namely,  $ \frac{v}{1-v}(1-b_1) = \frac{h}{1-h}(1-b_2) $.
Then, the stationary stochastic \ac{6V} distribution
$ \Gibbs(b_1,b_2;h,v) $ from Lemma \ref{lem:stat} is a extremal, translation invariant, ergodic infinite volume Gibbs state for the symmetric
\ac{6V} model on $\Z^2$ with weights $ (a,b,c) $, and $(h,v)$ gives the density of horizontal and vertical lines under this Gibbs state.
%Likewise, $\mathcal{P}(b_2,b_1;\rho_2,\rho_1)$ is similarly a Gibbs state.
\end{prop}
\begin{proof}
A version of this result seems to have been first observed in \cite{Bukman1995}. More recently, it appeared in \cite{ResSri16b}; \cite[Proposition~A.3]{A16aa} provides a proof.
\end{proof}

\section{Self duality for stochastic six vertex model}\label{sect:duality}

%\hao{We should write $\mu,\tau,\lambda$ instead of $\mue,\taue,\lambdae$ throughout the paper to lighten the notation. I've deleted these $\e$ whenever I see.}

The Markov duality method we introduce in this paper for showing convergence of the stochastic \ac{6V} model to the \ac{KPZ} equation relies upon the model's self-duality (in particular the one and two-particle duality), which we present in this section. This result was first proved for the stochastic \ac{6V} model with left-finite initial condition in \cite{Corwin2016}. We recall that result first, and then extend it by approximation to the bi-infinite stochastic \ac{6V} model defined in Lemma \ref{lem:binfinite}.

Let us first recall the general definition of Markov duality.

\begin{defn}\label{def:dual}
Given two Markov chains (in discrete time) or processes (in continuous time) $x(t)\in X$ and $y(t)\in Y$,
we say $x(t)$ and $y(t)$ are dual with respect to a duality function $H:X\times Y\to \R$ if for all $x\in X$, $y\in Y$ and $t\geq 0$
\begin{align*}
%	\label{eq:DualityCondition}
	\Ex^x\Big[H\big(x(t),y\big)\Big] = \Ex^y\Big[H\big(x,y(t)\big)\Big].
\end{align*}
Here, $\Ex^x$ denotes the expectation when the process $x(t)$ has been started with the initial condition $x(0)=x$, and $\Ex^y$ likewise for the $y$ variables.
\end{defn}
Our stochastic \ac{6V} self duality theorem is actually a duality between the stochastic \ac{6V} model and its $k$-particle space reversal ($k\geq 1$ is arbitrary), which we define now.
\begin{defn}\label{def:kpart}
Let $\Yfin{k}= \{(y_1<\cdots <y_k)\in \Z^k\}$ denote the state space of ordered $k$-particle configurations (sometimes called a discrete Weyl chamber).
The reversed stochastic \ac{6V} (or $\overleftarrow{\text{S6V}}$) model with $k$-particles is the Markov chain
$ \vec{y}(t)=(y_1(t)<\ldots<y_k(t))\in \Yfin{k} $ defined such that $ -\vec{y}(t) := (-y_k(t)<\cdots<-y_1(t))\in \Yfin{k} $
evolves according to the stochastic \ac{6V} dynamics given in Definition~\ref{def:S6V:lf}.
%We use the overhead left arrow to denote this reversed process and its states.
For $\vec{x},\vec{y}\in \Yfin{k}$, let
$\TrPr(\vec{x}\to \vec{y};t)$ denote the transition probability
that the reversed stochastic \ac{6V} started from $\vec{y}(0)=\vec{x}$ has $\vec{y}(t) = \vec{y}$.
Likewise, we let $\TrP(\vec{x}\to \vec{y};t)$ denote the transition probability
that the (usual) stochastic \ac{6V} started from $\vec{y}(0)=\vec{x}$ has $\vec{y}(t) = \vec{y}$.
%The reversed stochastic \ac{6V} (or $\overleftarrow{\text{S6V}}$) model with $k$-particles is the Markov chain $\cev{y}(t)\in \Yfin{k}$
%defined such that $(-y_k(t)<\cdots<-y_1(t))$
%evolves according to the stochastic \ac{6V} dynamics given in Definition~\ref{def:S6V:lf}. We use the overhead left arrow to denote this reversed process and its states.
%For $\cev{y}^1,\cev{y}^2\in \Yfin{k}$, let
%$\TrPr\big(\cev{y}^1\to \cev{y}^2;t\big)$ denote the transition probability that the reversed stochastic \ac{6V} started from $\cev{y}(0)=\cev{y}^1$ has $\cev{y}(t) = \cev{y}^2$.
\end{defn}

\begin{prop}\label{prop:CP16}
Fix $k\in \Z_{\geq 1}$, $w\in \Z$ and parameters $b_1,b_2\in (0,1)$ with $b_2<b_1$ (and recall that $\tau = b_2/b_1$). Let $\vec{x}(t)\in \Xlf{w}$ denote the stochastic \ac{6V} model with left-finite configurations (recall Definition \ref{def:S6V:lf}, as well as the notation $N_y(\vec{x})$ and $\eta_y(\vec{x})$ defined therein) and let $\vec{y}(t)\in \Yfin{k}$ denote the (reversed) $\overleftarrow{\text{S6V}}$ model with $k$-particles (Definition \ref{def:kpart}). Then $\vec{x}(t)$ and $\vec{y}(t)$ are dual with respect to the following two duality functions (recall Definition \ref{def:dual})
\begin{align*}
	H(\vec{x},\vec{y}) := \prod_{i=1}^{k} \tau^{N_{y_i}(\vec{x})},
	\qquad \textrm{and}\qquad
	\tilde{H}(\vec{x},\vec{y}) := \prod_{i=1}^{k} \eta_{y_i+1}(\vec{x})\, \tau^{N_{y_i}(\vec{x})}.
\end{align*}
\end{prop}
\begin{proof}
This is a special case of the dualities proved for the higher spin stochastic vertex models in \cite[Theorem~2.23]{Corwin2016}
(see also Section~5.5 therein). In fact, \cite{Yier} found a mistake in the proof of the $ \hat{\mathsf{G}_n}(\vec{g},\vec{n}) $ duality (our $\tilde{H}$ duality herein) and provided a correct proof for that case. 
Note that the corresponding duality function $ \tilde{H} $ (called $ \hat{\mathsf{G}_n}(\vec{g},\vec{n}) $ therein)
takes a slight different form here.
Under current notation, the duality function in \cite[Theorem 2.23]{Corwin2016} corresponds to
$
	\tilde{H}'(\vec{x},\vec{y}) := \prod_{i=1}^{k} \eta_{y_i}(\vec{x})\, \tau^{N_{y_i}(\vec{x})}.
$
One readily sees that
\begin{align*}
	\tilde{H}\big(\vec{x},(y_1,\ldots,y_k)\big)
	=
	\tau^{k}
	\tilde{H}'\big(\vec{x},(y_1+1,\ldots,y_k+1)\big),
\end{align*}
so the duality of the latter readily implies that of the former.
\end{proof}

For our applications, we need to extend this duality to the bi-infinite stochastic \ac{6V} model.
This is accomplished by appealing to the approximation result given in Lemma \ref{lem:binfinite}.
Let $ (\filt(t))_{t\in\Z_{\geq 0}} $ denote the canonical filtration of the stochastic \ac{6V} model.
%We will also now work with the height function $h(t,y) = N(t,y)-N(0,0)$ since our subsequent results will be stated in terms of that (since it is conveniently centered so $h(0,0)=0$).

\begin{cor}
\label{cor:CP16}
Fix $k\in \Z_{\geq 1}$, $b_1,b_2\in (0,1)$ with $b_2<b_1$, and let $\tau = b_2/b_1$.
The result of Proposition~\ref{prop:CP16} also hold for the bi-infinite stochastic \ac{6V} model $\vec{x}(t)\in \Xinf$.
In particular, letting $N(t,y)$ denote the height function associated in Lemma \ref{lem:binfinite} to $\vec{x}(t)$, and recall the reversed stochastic \ac{6V} model transition probability $\TrPr$ from Definition \ref{def:kpart}, this implies that
\begin{align*}
%	\label{eq:SSVMDuality1}
	\Ex\Big[\prod_{i=1}^{k} \tau^{N(t+s,y_i)} \Big\vert \filt(s) \Big]\,
	&= \,\,\sum_{\vec{y}'\in \Yfin{k}} \TrPr\big(\vec{y}\to\vec{y}\,';t\big) \prod_{i=1}^{k} \tau^{N(s,y'_i)}
\\
	&= \,\,\sum_{\vec{y}'\in \Yfin{k}} \TrP\big(\vec{y}\,'\to\vec{y};t\big) \prod_{i=1}^{k} \tau^{N(s,y'_i)},	
\\
%	\label{eq:SSVMDuality2}
	\Ex\Big[\prod_{i=1}^{k} \eta(t+s,y_i+1) \tau^{N(t+s,y_i)} \Big\vert \filt(s) \Big]\,
	&= \,\,\sum_{\vec{y}'\in \Yfin{k}} \TrPr\big(\vec{y}\to\vec{y}\,';t\big) \prod_{i=1}^{k} \eta(s,y'_i+1)\tau^{N(s,y'_i)}
\\
	&= \,\,\sum_{\vec{y}'\in \Yfin{k}} \TrP\big(\vec{y}\,'\to\vec{y};t\big) \prod_{i=1}^{k} \eta(s,y'_i+1)\tau^{N(s,y'_i)}.
\end{align*}
Above, the expectation is over the height function $N(t+s,\Cdot)$ conditioned on its values $N(s,\Cdot)$ at time $s$, and $\eta$ is coupled to $N$ so that $\eta(t,y) = N(t,y)-N(t,y-1)$.
\end{cor}
\begin{proof}
We will give the proof for the $ H $ duality as the $ \tilde{H} $ duality follows identically.
Without loss of generality we assume that $s=0$.
It suffices also to show that the duality holds for just $t=1$ since general $t$ follows inductively.
%Finally, note that once the duality is extended to bi-infinite configurations, the rest of the corollary follows immediately.

Recall the notation $\vec{x}_{\geq w}$ and $\vec{x}_{\geq w}(t)$ from Lemma \ref{lem:binfinite} for the bi-infinite stochastic \ac{6V} model cutoff to be left-finite with first particle $x_w$. Applying the duality in Proposition~\ref{prop:CP16} implies that
\begin{align*}
	\Ex\Big[\prod_{i=1}^{k} \tau^{N_{y_i}(\vec{x}_{\geq w}(1))}\Big]
	=
	\sum_{\vec{y}'\in \Yfin{k}}\TrPr\big(\vec{y}\to\vec{y},';1\big) \prod_{i=1}^{k} \tau^{N_{y'_i}(\vec{x}_{\geq w})},
\end{align*}
where the expectation is over $\vec{x}_{\geq w}(t)$ at $t=1$ with initial condition $\vec{x}_{\geq w}$ at $t=0$.
In order to prove the corollary, we must show that taking $w\to -\infty$, both sides of the above equation converge to their bi-infinite version.
The left-hand side converges as $w\to -\infty$ to $\Ex[\prod_{i=1}^{k} \tau^{N_{y_i}(\vec{x}(1))}]$.
This is because, by Lemma~\ref{lem:binfinite} $N_y(\vec{x}_{\geq w}(t))$ converges in probability to $N_y(\vec{x}(t))$
and in a single time step $N_{y}(\vec{x}_{\geq w}(t))$ may change by at most one, hence the argument of the expectation is a bounded function.
To show the right-hand side convergence, we bound (for some constant~$C<\infty$)
\begin{align*}
	\TrPr\big(\vec{y}\to \vec{y}\,';1\big) \leq C \prod_{i=1}^{k} b_2^{y_i-y'_i} \ind_\set{y_i\geq y'_i}
\end{align*}
and then use the fact that $\tau^{-1}b_2=b_1<1$ to apply dominated convergence.
The above  bound follows since for the reversed stochastic \ac{6V} model $\vec{y}(t)$,
the increments (up to a minus sign) $ -(y_i(0)-y_i(1)) $ can be stochastically upper bounded by $1+\text{geo}(b_2)$,
where $ \text{geo}(b_2) $ is  a geometric random variable with values in $\Z_{\geq 0}$ with parameter $b_2$.
This proves the first identity for the $ H $ duality.

For the second identity, by definition of space-reversed stochastic \ac{6V} model, we have
\begin{align*}
	\TrPr(\vec{y}\to\vec{y}\,';t) = \TrP(-\vec{y}\to-\vec{y}\,';t),
\end{align*}
where, for $ \vec{y}=(y_1<\ldots<y_k)\in\Yfin{k} $, $ -\vec{y}:=(-y_k<\ldots<-y_1)\in\Yfin{k} $ denotes the space-reversed configuration.
Further, the stochastic \ac{6V} model enjoys a space-time reversal symmetry:
\begin{align*}
	\TrP(-\vec{y}\to-\vec{y}\,';t) = \TrP(\vec{y}\,'\to\vec{y};t).
\end{align*}
To see this, notice that $ (-\vec{y},-\vec{y}\,') \mapsto (\vec{y}\,',\vec{y}) $ amounts to a vertical and horizontal flip in the vertex model configration.
Under such flips, the weights for
($ \includegraphics[width=10pt]{vertex1111} $,
$ \includegraphics[width=10pt]{vertex0000} $,
$ \includegraphics[width=10pt]{vertex0101} $,
$ \includegraphics[width=10pt]{vertex1010} $)
remain unchanged,
while the weights for
($ \includegraphics[width=10pt]{vertex1001} $,
$ \includegraphics[width=10pt]{vertex0110} $)
swap.
Given fixed initial and terminal conditions $ (\vec{y},\vec{y}\,') $,
it is readily checked that \ac{6V} measures are invariant under the prescribed swap.
From these consideration we conclude
\begin{align*}
	\TrPr(\vec{y}\to\vec{y}\,';t) = \TrP(-\vec{y}\to-\vec{y}\,';t) = \TrP(\vec{y}\,'\to\vec{y};t).
\end{align*}
This proves the second claimed identity.
\end{proof}

Owing to its Bethe ansatz solvability, the $k$-particle (reversed) stochastic \ac{6V} model admits explicit integral formulas for transition probabilities. We will make use of the $k=1,2$ cases of these formulas, but since the general $k$ result is not any more complicated, we record it below. Note that in the below formula (and subsequent calculations involving it) we will use $\vec{x}$ and $\vec{y}$ to denote $k$-particle configurations (as opposed to $\vec{y}$ and $\vec{y}'$ as in our discussion on duality).
\begin{prop}\label{prop:TransitionProbability}
Fix $k\in \Z_{\geq 1}$ and parameters $b_1,b_2\in (0,1)$ with $b_2<b_1$. Then for any $\vec{x},\vec{y}\in \Yfin{k}$ (where $\Yfin{k}$ is the discrete Weyl Chamber defined in Definition~\ref{def:kpart}) and $t\in \Z_{\geq 0}$,
\begin{align}
	\TrP\big(\vec{y}\to \vec{x};t\big)  = \TrPrc(\vec{y},\vec{x};t)
\end{align}
where $\uu(\vec{y},\vec{x};t)$ is defined for all $\vec{y},\vec{x}\in \Z^k $ by
\begin{equation}\label{eq:SGc}
	\TrPrc(\vec{y},\vec{x};t)
	=
	\oint_{\Circ_r}\cdots \oint_{\Circ_r} \sum_{\sigma\in \mathfrak{S}_k}(-1)^{\sigma}\!\prod_{1\leq i<j\leq k}\TrPrcfr(z_i,z_j,\sigma)
	\prod_{i=1}^k z^{x_{\sigma(i)}-y_i-1}_{i} \TrPrct(z_i)^t \frac{dz_i}{2\pi \img}.
\end{equation}
Here $\Circ_r$ is a circular contour (counter-clockwise oriented) centered at the origin with a large enough radius $r$ so as to include all poles of the integrand, $\mathfrak{S}_k$ is the set of all permutations on the set $\{1,\ldots, k\}$, $(-1)^{\sigma}\in \{-1,1\}$ is the sign of the permutation, and
\begin{align*}
	\TrPrcfr(z_i,z_j,\sigma) := \frac{1-(1+\tau^{-1})z_{\sigma(i)}+\tau^{-1}z_{\sigma(i)}z_{\sigma(j)}}{1-(1+\tau^{-1})z_i+\tau^{-1}z_iz_j},
	\qquad
	\TrPrct(z) := \left(\frac{b_1+(1-b_1-b_2)z^{-1}}{1-b_2z^{-1}}\right).
\end{align*}
\end{prop}
\begin{proof}
This is a special case of \cite[Theorem~3.6, Eq.~(26)]{BCG2016} with $c_1=1-b_1$, $c_2=1-b_2$  and $a_1=a_2=1$.
\end{proof}

\section{Hopf--Cole transform: reformulation of Theorem~\ref{thm:S6V}}
\label{sect:HC}

One-particle $H$ duality (Proposition~\ref{prop:CP16}) implies that $\Ex[\tau^{N(t,x)}] $ solves the evolution equation for a one-particle stochastic \ac{6V} model.
As is true for general finite variance homogeneous random walks on $ \Z $, this evolution equation is a discrete heat equation and after proper centering and scaling, it will go to the continuous heat equation on $ \R $.
In this section we describe (see Proposition \ref{prop:mSHE}) the martingale part that is left when ones does not take expectations,
as well as the proper centering of the process $ \tau^{N(t,x)} $ that gives $ \Zsv(t,x) $,
the microscopic Hopf--Cole transform of $ N(t,x) $.

Given such a transform,
we reformulate the convergence to \ac{KPZ} equation (i.e., Theorem~\ref{thm:S6V})
as an \emph{equivalent} statement of convergence to \ac{SHE} (see Theorem~\ref{thm:S6V:}).

\subsection{Microscopic Hopf--Cole transform}
\label{sect:mHC}
%Introduced by G\"artner in \cite{GJ1985}, the G\"artner transform of any exclusion is a non-linear transformation which leads to the stochastic equation of the limiting dynamics. In what follows, we first introduce the parameters of interest and then, we turn to define the G\"artner transform of the S6V model.
%\begin{defn}\label{def:model}
%%\note{change $1- \nu \to \den$ below.}
%Fix $\nu\in (0,1)$, and (as throughout the paper) $b_1,b_2\in (0,1)$ with $\tau = b_2/b_1<1$. Let $\{R_i\}_{i=1}^\infty$ be i.i.d. random variables with probability mass function (abbreviated pmf in what follows)

Recall that $ \den\in(0,1) $ is a fixed parameter representing the average density.
Referring back to Theorem~\ref{thm:S6V},
we notice that the convergence results involve centering and tilting of the height function $ N(t,x) $.
Our first step here is hence to introduce the corresponding centering and tilting of $ \tau^{N(t,x)} $.
To setup notation, consider the stochastic \ac{6V} model with a \emph{single} particle starting from $ x=0 $.
This is simply a discrete-time random walk $ X(t) = \rw(1)+\ldots + \rw(t) $,
with i.i.d.\ increments $ \rw(1),\ldots,\rw(t) $ that have distribution $ \rw(i)\stackrel{\text{law}}{=}\rw $, where
\begin{align}
	\label{eq:rw}
	\Pr\left(\rw=n\right)
	=
	\left\{\begin{array}{c@{,}l}
		(1-b_1)(1-b_2)b^{n-1}_2 & \text{ when }n>0,
		\\
		b_1 & \text{ when } n=0,
		\\
		0 & \text{ otherwise.}
	\end{array}\right.
\end{align}
Now, with $ N(t,x) $ being tilted by $ -\den x $ in~\eqref{eq:S6VtoKPZ},
we consider the analogous tilt of $ \rw $:
\begin{align}
	\label{eq:Rw}
	\Pr\left(\Rw=n\right)
	:=
	\lambda \Ex[\tau^{-\den\rw}\ind_{\set{\rw=n}}]
	=\lambda \tau^{-\rho n } \Pr(\rw=n).
\end{align}
The parameter $ \lambda = (\Ex[\tau^{-\den\rw}])^{-1} $ is in place to ensure~\eqref{eq:Rw} defines a random variable,
and the variable $ \Rw $ has mean $ \mu= \Ex[\Rw]>0 $.
From~\eqref{eq:rw}, it is straightforward to check\footnote{The computation for $\lambda$ simply boils down to a geometric series.
The computation for $\mu$ boils down to a sum of the form
$\sum_{n\ge 0} (n+1) (b_2\tau^{-\rho})^n$; this multiplied by $(1-b_2 \tau^{-\rho})$ again gives a geometric series.}
that $ \lambda $ and $ \mu $ are given by~\eqref{eq:lambda}--\eqref{eq:mu}.
We further consider the corresponding centered variable $ \RW := \Rw-\mu $.
With $ \mu $ being the centering parameter (in Theorem~\ref{thm:S6V}) that sets the reference frame along the characteristic,
we let
\[
 \Xi(t) = \Z - t \mu
 \]
denote a shifted integer lattice to accommodate the centering by $ \mu $.
Given this notation, we define the \DefinText{(microscopic) Hopf--Cole (i.e., G\"artner) transform} of the stochastic \ac{6V} model as
\begin{align}
	\label{eq:ColeHopfTransform}
	\Zsv(t,x) := \lambda^{t}\tau^{N( t, x+\mu t )-\den( x+\mu t)},
	\quad
	x \in \Xi(t),
\end{align}
where $ \lambda $ and $ \mu $ are given in~\eqref{eq:lambda}--\eqref{eq:mu}.

It is straightforward to verify that the $ k=1 $ duality for $ \tau^{N_y(\vec x)} $
(Proposition~\ref{prop:CP16}) implies that
\begin{align}
	\label{eq:Zsv:mSHE:}
	\Ex\big[\Zsv(t+1,x-\mu) \big| \filt(t) \big]
	=
	\big( \hk\Zsv(t) \big)(x-\mu),
\end{align}
where $ \hk $ acts on functions $ f: \Xi(t)\to\R $ as
\begin{align*}
	(\hk f)(x) := \sum_{y\in \Xi(t)} \hk(x-y) f(y)
	= \sum_{y:\, x-y\in \Xi(1)} \!\!\!\! \hk(x-y) f(y) ,
	\quad
	x \in \Xi(t+1),
\end{align*}
with a kernel $ \hk(\Cdot) $ given by the probability mass function of $ \RW $, i.e.,
\begin{align}
	\label{eq:hk}
	\hk(x) := \Pr(\RW=x)
	=
	\left\{\begin{array}{c@{,}l}
		%\lambda(1-b_1)(1-b_2)b^{x+\mu-1}_1 \tau^{(x+\mu-1)(1-\den)-1}
		\lambda(1-b_1)(1-b_2)b^{x+\mu-1}_2 \tau^{-\den(x+\mu)}
		& \text{ when }x+\mu \in \Z_{>0},
		\\
		\lambda b_1 & \text{ when } x+\mu=0,
		\\
		0 & \text{ otherwise}
	\end{array}\right.
\end{align}
%We likewise define the generator
%\begin{align}
%	\label{eq:gen}
%	(\gen f)(x) := \sum_{y\in \Xi(t)} \big( \hk(x-y) - \ind_{\set{x+\mu=y}} \big) f(y),
%	\quad
%	x\in \Xi(t+1).
%\end{align}
While the kernel $\hk(\Cdot)$ is independent of $t$,
strictly speaking the domain and range of the {\it operator} $ \hk $ depends on $ t $ because it maps functions on $ \Xi(t) $ to functions on $ \Xi(t+1) $.
We however drop this dependence in our notation $ \hk $.
We will consider also the $ t $-the power of $ \hk $ (viewed as an operator), i.e., $ \hk(t) := \hk^t $
(so in particular $ \hk=\hk(1) $), namely
\[
\begin{tikzpicture}
\node at (0,0) {$\cdots\stackrel{\hk}{\longrightarrow}\R^{\Xi(t)} \stackrel{\hk}{\longrightarrow}\R^{\Xi(t+1)} \stackrel{\hk}{\longrightarrow}\R^{\Xi(t+2)}
\stackrel{\hk}{\longrightarrow}\R^{\Xi(t+3)}\stackrel{\hk}{\longrightarrow}\cdots$};
%\draw[->, bend right = 10] (-3,-.3) to (3,-.3);
\draw[->] (-3,-.3) --(-3,-.5) --(3,-.5) -- (3,-.3) ;
\node at (0,-.8) {$\hk^3$};
 \end{tikzpicture}
 \]
and $\hk^t$ has kernel
%{\color{red}
%No it is indeed $x_i\in \Xi(1)$. Consider $t=2$:
%\[
%(\hk^2 f)(x) = \sum_{y\in \Xi(t+1)} \hk(x-y) \sum_{z\in \Xi(t)} \hk(y-z)  f(z)
%= \sum_{y\in \Xi(t+1)} \sum_{z\in \Xi(t)} \hk(x-y)  \hk(y-z)  f(z)
%	\qquad
%	x \in \Xi(t+2),
%\]
%both  $x_1=x-y$ and $x_2=y-z$ are in $\Xi(1)$.
%}
\begin{align}
	\label{eq:hkt}
	\hk(t,x)
%	=
%	\big(\underbrace{\hk\ldots\hk}_{t} \big)(x)
	=
	\sum_{x_i\in \Xi(1), x_1+\ldots+x_t=x} \hk(x_1) \cdots \hk(x_t),
	\quad
	x\in \Xi(t).
\end{align}
Since $ \hk(x) $ is the probability mass function of $ \RW $,
the kernel $ \hk(t,x) $ is exactly the $ t $-step transition probability of a random walk with i.i.d.\ increment $ \RW $.
Given this interpretation and the aforementioned relation between $ \rw $ and $ \RW $, we have
\begin{align*}
	\hk(t,x) = \Pr\big[ \RW(1)+\ldots+\RW(t)=x \big]
	&= \lambda^t\Ex\big[ \ind_\set{\rw(1)+\ldots+\rw(t)=x+\mu t} e^{-\den(\rw(1)+\ldots+\rw(t))} \big]
\\
	&= \lambda^t e^{-\den(x+\mu t)}\TrP\big[ 0 \to x+\mu t ;t \big ].
\end{align*}
Combining this with Proposition~\ref{prop:TransitionProbability} for $ k=1 $ gives the following contour integral expression:
\begin{align}
	\label{eq:hk:contour}
	\hk(t,x)
	=
	\oint_{\Circ_r}	
		z^{x+(\mu t-\mut)} \frac{(\SGt(t,z))^t dz}{2\pi\img z},	
\end{align}
where $ \Circ_r $ denotes a counter-clockwise oriented, circular contour that is centered at origin, and
\begin{align}
	\label{eq:SGt}
	\SGt(t,z) &:= z^{\mut} \Big( \lambda \Big(\frac{b_1+(1-b_1-b_2)/(\tau^\den z)}{1-b_2/(\tau^\den z)}\Big) \Big)^{t}.
\end{align}

Equation~\eqref{eq:Zsv:mSHE:} states that
$ \Zsv(t+1,x-\mu)-(\hk\Zsv(t))(x) $ is an $ \filt $-martingale increment.
We now provide a precise description of this martingale increment.
Recall that the height function $ N(t,x) $ either decreases by one or remains constant within each update $ t\mapsto t+1 $.
This being the case,
\begin{align}
	\label{eq:move}
	\move(t,x) := N(t,x) - N(t+1,x)
\end{align}
defines a $ \{0,1\} $-valued (i.e., Bernoulli) random variable.
Consider further the centered variables
\begin{align}
	\label{eq:moveCent}
	\moveCent(t,x) := \move(t,x) - \Ex[ \move(t,x) | \filt(t) ].
\end{align}
\begin{prop}\label{prop:mSHE}
For any $t \in\Z_{\geq 0} $ and $ x\in \Xi(t) $, we have
\begin{align}
	\label{eq:mSHE}
	\Zsv(t+1,x-\mu) = \big( \hk\Zsv(t) \big)(x-\mu) + \mg(t,x),
\end{align}
where
\begin{align}
	\label{eq:mg}
	\mg(t,x) := \lambda(\tau^{-1}-1)\Zsv(t,x) \moveCent(t,x+\mu t)
\end{align}
is an $ \filt $-martingale increment,
i.e., $ \Ex[\mg(t,x)|\filt(t)]=0 $, $ t\in\Z_{\geq 0} $,
with
\begin{align}
	\label{eq:Quadvar}
	\Ex\left[\mg(t,x_1)\mg(t,x_2)|\filt(t)\right]
	&=
	(b_1\tau^{1-\den})^{|x_1-x_2|} \Theta_1(t,x_1\wedge x_2)\Theta_2(t,x_1\wedge x_2),
\\	
	\label{eq:Theta1}
	\Theta_1(t,x) &:= \lambda \tau^{-1} \Zsv(t,x)
		- \big(\hk\Zsv(t)\big)(x-\mu),
\\
	\label{eq:Theta2}
 	\Theta_2(t,x) &:= -\lambda \Zsv(t,x)+\big(\hk\Zsv(t)\big)(x-\mu).
\end{align}
\end{prop}
\begin{proof}
The result is a special case of the statement
and proof of \cite[Proposition~2.6]{CT15}.
In \cite{CT15} the authors consider a more general higher-spin version of the stochastic \ac{6V} model \cite{Corwin2016} with arbitrary non-negative integer valued horizontal spin $J$ as well as parameters $\alpha,q,\nu$. Our stochastic \ac{6V} corresponds with taking $ J=1 $ and $ \nu=1/q=\tau $ therein,
and matching $ b_1 \mapsto \frac{1+ q\alpha}{1+\alpha} $ and $ \tau^{-\den} \mapsto \den $.
\end{proof}

More generally, for $ k \geq 2 $, $ \Zsv(t,x) $ inherits a duality from $ \tau^{N(t,x)} $,
analogous to Corollary~\ref{cor:CP16} and Proposition~\ref{prop:TransitionProbability}.
The analogous semigroup integral formulas are obtained by a centering and tilting of $ \TrPrc $ (as in Proposition~\ref{prop:TransitionProbability}).
We state the duality and integral formula result for $ \Zsv $ only for $ k=2 $ (as we will only need that case).
For $ y_1<y_2\in\Xi(s) $ and for $ x_1<x_2\in\Xi(s+t) $, we define
\begin{align}
	\notag
	\SG\big((y_1,y_2),&(x_1,x_2);t\big)
	:=
	\oint_{\Circ_r}\oint_{\Circ_r}
	\Big(
		z_1^{x_1-y_1+(\mu t-\mut)}z_2^{x_2-y_2+(\mu t-\mut)}
\\
	\label{eq:SG}
		&-\SGfr(z_1,z_2) z_1^{x_2-y_1+(\mu t-\mut)}z_2^{x_1-y_2+(\mu t-\mut)}
	\Big)
	\prod_{i=1}^2
	\frac{\SGt(t,z_i)dz_i}{2\pi\img z_i}.
\end{align}
Here $ \Circ_r $ is a counter-clockwise oriented, circular contour
that is centered at origin, with a large enough radius $ r $
so as to include all poles of the integrand, $ \SGt(t,z) $ is defined in~\eqref{eq:SGt}, and
\begin{align}
	\label{eq:SGfr}
	\SGfr(z_1,z_2)
	&:=
	\frac{1+\tau^{-1+2\den}z_1z_2-(1+\tau^{-1})\tau^\den z_2}{1+\tau^{-1+2\den}z_1z_2-(1+\tau^{-1})\tau^\den z_1}.
\end{align}
\begin{rmk}
One could rewrite the formula~\eqref{eq:SG} in a seemingly simpler form:
\begin{align*}
	\SG\big((y_1,y_2),(x_1,x_2);t\big)
	=
	\oint_{\Circ_r}\oint_{\Circ_r}
	\Big(
		z_1^{x_1-y_1}z_2^{x_2-y_2}
		-\SGfr(z_1,z_2) z_1^{x_2-y_1}z_2^{x_1-y_2}
	\Big)
	\prod_{i=1}^2
	\frac{(\TrPrct(z_i))^t dz_i}{2\pi\img z_i},
\end{align*}
where $ \TrPrct(z) := z^{\mu}\lambda \frac{b_1+(1-b_1-b_2)/(\tau^\den z)}{1-b_2/(\tau^\den z)} $.
The expression, however, involves non-integer powers of $ z_i $, because $ x_i-y_j\not\in\Z $
and $ \mu\not\in\Z $ in general,
and having non-integer powers is undesirable for our analysis in the sequel.
With $ x_i- y_i\in \Xi(t) $,
we have that $ x_i-y_j+(\mu t-\mut) \in \Z $, so the formula~\eqref{eq:SG} involves only integer powers of $ z_i $.
\end{rmk}
\noindent
We adopt the following shorthand notation for centered occupation variables:
\begin{align}
	\label{eq:etaCentered}
	\etacnt(t,x)
	:=
	\eta(t,x+\mu t),
	\quad
	\etacnt^{+}(t,x) := \etacnt(t,x+1),
	\quad
	x \in \Xi(t).
\end{align}

\begin{prop}\label{prop:Zduality}
With $ \Zsv $ being the Hopf--Cole transform of the stochastic \ac{6V} model with parameters $ b_1>b_2\in(0,1) $,
for all $ x_1<x_2\in\Xi(t+s) $ and $ t,s\in \Z_{\geq 0} $, we have
\begin{align}
	\label{eq:duality:ZZ}
	&\Ex\Big[ \Zsv(t+s,x_1) \Zsv(t+s,x_2) \Big\vert \filt(s) \Big] = \sum_{y_1<y_2\in\Xi(s)} \!\!\!
	\SG\big((y_1,y_2),(x_1,x_2);t\big) \Zsv(s,y_1) \Zsv(s,y_2),
\\
	\notag
	&\Ex\Big[ (\etacnt^{+}\Zsv)(t+s,x_1) (\etacnt^{+}\Zsv)(t+s,x_2) \Big\vert \filt(s) \Big]
\\
	\label{eq:duality:eeZZ}
	&\hphantom{\Ex\Big[ (\eta\Zsv)(t+s,x_1)}
	=\sum_{y_1<y_2\in\Xi(s)} \!\!\!
		\SG\big((y_1,y_2),(x_1,x_2);t\big)
		\big(\etacnt^{+}\Zsv\big)(s,y_1)\big(\etacnt^{+}\Zsv\big)(s,y_2).
\end{align}
\end{prop}
\begin{proof}
Recall from~\eqref{eq:ColeHopfTransform} that $\Zsv(t,x) $ is obtained from $\tau^{N(t,x+\mu t)}$
 through centering and tilting.
Translating the $ k=2 $ duality (from Corollary~\ref{cor:CP16} and Proposition~\ref{prop:TransitionProbability})
in terms of the centered and tilted process $\Zsv(t,x) $, we see that~\eqref{eq:duality:ZZ}--\eqref{eq:duality:eeZZ} holds where
\begin{align}
	\notag
	&\SG\big((y_1,y_2),(x_1,x_2);t\big)	
\\
	\label{eq:SG:}
	&=
	\lambda^{2t} \tau^{-\den(x_1+x_2-y_1-y_2+2\mu t)}
	\TrPrc\big( (y_1+\mu s, y_2+\mu s),(x_1+\mu(t+s), x_2+\mu(t+s)) ;t \big).
\end{align}
Our goal now is to show that $ \SG $ given in~\eqref{eq:SG:}
can, indeed, be written as the contour integral in~\eqref{eq:SG}.
Referring back to the formula~\eqref{eq:SGc} for $ \TrPrc $, and combining it with~\eqref{eq:SG:}, we find that
\begin{align*}
	&\SG((y_1,y_2),(x_1,x_2);t)
	=
	\oint_{\Circ_r} \oint_{\Circ_r}
	\Big(
		(\tau^{-\den}z_1)^{x_1-y_1+(\mu t -\mut)} (\tau^{-\den}z_2)^{x_2-y_2+(\mu t -\mut)}
\\
	&-
		\frac{1-(1+\tau^{-1})z_2+\tau^{-1}z_1z_2}{1-(1+\tau^{-1})z_1+\tau^{-1}z_1z_2}
		(\tau^{-\den}z_1)^{x_2-y_1+(\mu t -\mut)} (\tau^{-\den}z_2)^{x_1-y_2+(\mu t -\mut)}
	\Big)
	\prod_{i=1}^2 \frac{\hat{\SGt}(z_i,t)dz_i}{2\pi\img z_i},
\end{align*}
where %, with $ \TrPrc(z) $ defined in Proposition~\ref{prop:TransitionProbability},
$ \hat{\SGt}(t,z) := (\tau^{-\den}z)^{\mut} \lambda^t \TrPrct(z)^t$.
Given this, the claimed result now follows by the change of variable $ \tau^{-\den}z_i:=\tilde{z}_i $.
\end{proof}
\subsection{The SHE}
\label{sect:SHE}
Proposition~\ref{prop:mSHE} states that $ \Zsv $ solves a discrete-time, discrete space \ac{SPDE}.
Examining this equation suggests that, under appropriate scaling,
$ \Zsv $ should converge to the solution of the \ac{SHE}:
\begin{align}
	\label{eq:SHE}
	\partial_t \limZ(t,x) = \frac{\nu_*}{2} \partial_x^2 \limZ(t,x) + \frac{\kappa_*\sqrt{D_*}}{\nu_*}\noise(t,x) \limZ(t,x).
\end{align}
The coefficients $ \nu_* $, $ \kappa_* $ and $ D_* $ are given in~\eqref{eq:ceffints}.
(Although $ \nu_* = \kappa_* $, we prefer to write the equation as above to better track the limiting coefficients.)

To formulate the convergence to \ac{SHE} precisely,
recall that a $ C([0,\infty),C(\R)) $-valued process $ \limZ $ is a \DefinText{mild solution} of~\eqref{eq:SHE}
with initial condition $ \limZ^\ic(x) $ if
\begin{align}
	\label{eq:SHEmild}
	\limZ(t,x) = \int_{\R} p(\nu_*t,x-y)\limZ^\ic(y) dy + \int_0^t \int_{\R} p(\nu_*(t-s),x-y) \limZ(s,y) \frac{\kappa_*\sqrt{D_*}}{\nu_*}\noise(s,y) dsdy,
\end{align}
for all $ t\in[0,\infty) $ and $ x\in\R $.
Given non-negative $ \limZ^\ic \in C(\R) $ that is not identically zero,
the \ac{SHE} permits a unique mild solution that stays positive for all $ t>0 $.
See, for example, \cite[Proposition 2.5]{Corwin12} and the references therein.
With the \ac{SHE} being an informal exponentiation of the \ac{KPZ} equation,
we say $ \limH $ is a \DefinText{Hopf--Cole solution} of the \ac{KPZ} equation~\eqref{eq:KPZ} if
\begin{align} \label{e:HCsolution}
 	e^{-\frac{\kappa_*}{\nu_*}\limH(t,x)} =e^{-\limH(t,x)}
\end{align}
is a mild solution of~\eqref{eq:SHE}.
So far our discussion has been for a $ C(\R) $-valued $ \limZ^\ic $,
which is the proper setup for near stationary initial conditions (defined in the following).
To accommodate the step initial condition,
$ \eta(0,x) = \ind_{\set{x \geq 0}} $, we need to also consider the \ac{SHE} starting from delta function $ \delta(x) $.
The mild solution is defined analogously:
\begin{align*}
	\limZ(t,x) = p(\nu_*t,x-y) + \int_0^t \int_{\R} p(\nu_*(t-s),x-y) \limZ(s,y) \frac{\kappa_*\sqrt{D_*}}{\nu_*}\noise(s,y) dsdy,
\end{align*}
for $ t>0 $ and $ x\in\R $.
For delta initial condition,
there exists a unique $ C((0,\infty),C(\R)) $-valued solution $ \limZ $, which is positive.
\footnote{For reference, see  \cite[Proposition~4.3]{Shalin} where
existence, uniqueness and positivity in
a more complicated case (i.e. with boundaries) are proved.}
For such $ \limZ $, we then define $ \limH(t,x):=\log(\limZ(t,x)) $
as the solution of the \ac{KPZ} equation~\eqref{eq:KPZ} with \DefinText{narrow wedge initial condition}.

As discussed above Theorem~\ref{thm:S6V},
we will prove convergence to the Hopf--Cole solution to the
KPZ equation under \DefinText{weak asymmetry scaling},
where
\begin{align*}
\rho\in(0,1), b_1\in(0,1)  \mbox{ are fixed,}
\qquad
\tau =\taue = b_2/b_1 =b_2^\e/b_1:= e^{-\sqrt\e}
\end{align*}
and $(\lambda,\mu)=(\lambdae,\mue)$ are defined in \eqref{eq:lambda}--\eqref{eq:mu} which behave asymptotically as \eqref{eq:lambdae}--\eqref{eq:mue}.
Under this scaling, the microscopic Hopf--Cole transform \eqref{eq:ColeHopfTransform} reads
\begin{align} \label{e:Gartner-eps}
	\Zsv(t,x)=\Zsv_\e(t,x) := e^{t\log \lambdae - \sqrt\e (N_\e( t, x+\mue t )-\den( x+\mue t))},
	\quad
	x \in \Xi(t).
\end{align}

Hereafter,
we adopt the standard notation $ \Vert X \Vert_n := (\Ex[|X|^n])^{\frac1n} $,
and say \DefinText{for all $ \e>0 $ small enough} if the referred statement holds for all $ \e\in (0,\e_0) $,
for some generic but fixed threshold $ \e_0>0 $ that may change from line to line.
Following \cite{Bertini1997}, we define near stationary initial conditions for the stochastic \ac{6V} model:
\begin{defn}
\label{def:nearStat}Fix any density parameter $\rho\in (0,1)$.
With $ \e\downarrow 0 $ being the scaling parameter,
consider a sequence of possibility random initial conditions $ \{N_\e(0,x)\}_{\e>0} $,
and let $ \Zsv_\e(0,x) $ denote the corresponding Hopf--Cole transformed initial data  defined through~\eqref{eq:ColeHopfTransform}.
We say the initial condition is \DefinText{near stationary with density $\rho$} if,
for any given $ n<\infty $ and $ \alpha\in(0,\frac12) $, there exist constants $ C=C(n,\alpha) $ and $ u=u(n,\alpha) $, such that
\begin{align}
	\label{eq:nearStat}
	\Vert \Zsv_\e(0,x)\Vert_{n} &\leq C\exp\left( u\e|x|\right),
\\
	\Vert \Zsv_\e(0,x)-\Zsv_\e(0,x') \Vert_{n}
	&\leq
	C \left(\e|x-x^\prime|\right)^{\alpha} \exp\left( u\e(|x|+|x'|)\right),
\end{align}
for all $ x,x' \in \Z $, and small enough $ \e>0 $.
\end{defn}

We now state our result on the convergence of $ \Zsv(t,x) $ to the \ac{SHE}.
Due to the round-about definition of the Hopf--Cole solution \eqref{e:HCsolution},
it is readily checked (see \eqref{e:Gartner-eps}) that,
Theorem~\ref{thm:S6V:} in the following is an \emph{equivalent} formulation of Theorem~\ref{thm:S6V}.
Given $ \Zsv(t,x) $, $ t\in\Z_{\geq 0} $, $ x\in\Xi(t) $,
we first linearly interpolate in $ x $ and then linearly interpolate in $ t $ to obtain\footnote{This is \emph{different} from exponentiating the interpolated height function. Nevertheless, under the weak asymmetry scaling $ \tau=\exp(-\sqrt\e) $, it is straightforward to verify that the difference between these two interpolation schemes is negligible as $ \e\to 0. $}
a $ C([0,\infty),\R) $-valued process.

\begin{customthm}{\ref*{thm:S6V}}
\label{thm:S6V:}
Consider the stochastic \ac{6V} model, with parameter $ b_1> b_2 \in (0,1) $.
\begin{enumerate}[leftmargin=5ex, label=(\alph*)] %
\item \label{thm:S6V:nearStat} \textbf{(Near stationary initial conditions)}  Fix a density $\rho\in (0,1)$.
Start the stochastic \ac{6V} model from a sequence (parameterized by $ \e $) of near stationary with density $\rho$ initial conditions,
and let $ Z_\e(t,x) $ denote the resulting Hopf--Cole transform.
If, for some $ C(\R) $-valued process $ \limZ^\ic $, we have
\begin{align}
	Z_\e(0,\e^{-1}x) \Longrightarrow \limZ^\ic(x),
	\quad
	\text{ in } C(\R),
\end{align}
then, under the weak asymmetry scaling %where $ b_1\in(0,1) $ is fixed, and $ \tau =\taue = b_2/b_1 := e^{-\sqrt\e} $,
we have
\begin{align*}
	{\Zsv}_\e(\e^{-2}t,\e^{-1}x) \Longrightarrow \limZ(t,x),
	\quad
	\text{in } C([0,\infty), C(\R)),
\end{align*}
where $ \limZ(t,x) $ is the mild solution of the \ac{SHE}~\eqref{eq:SHE} with initial condition~$ \limZ^\ic(x) $.
\item  \label{thm:S6V:step} \textbf{(Step initial condition)}
Start the stochastic \ac{6V} model from the step initial condition $ N(0,x) = (x)_+ $,
and let $ Z_\e(t,x) $ denote the resulting Hopf--Cole transform. Let $\rho\in (0,1)$ be fixed.
Under the weak asymmetry scaling %where $ b_1\in(0,1) $ is fixed, and $ \tau =\taue = b_2/b_1 := e^{-\sqrt\e} $,
we have
\begin{align*}
	\tfrac{\den(1-\den)}{\sqrt\e} {\Zsv}_\e(\e^{-2}t,\e^{-1}x)
	\Longrightarrow
	\limZ(t,x),
	\quad
	\text{in } C((0,\infty), C(\R)),
\end{align*}
where $ \limZ(t,x) $ is the mild solution of the \ac{SHE}~\eqref{eq:SHE} with delta initial condition~$ \delta(x) $.
\end{enumerate}
\end{customthm}

\section{Proof of Theorems~\ref{thm:S6V:} and~\ref{thm:6V}}
\label{sect:pfthm}

Hereafter, we will be assuming  the weak asymmetry scaling $ \tau=\taue=e^{-\sqrt\e} $ (for the stochastic model),
and the scaling $ \eta=\eta_\e = \frac12\sqrt\e $ (for the symmetric model under Baxter's projective parametrization~\eqref{eq:projPara}).
To highlight this dependence, for \emph{parameters} we write $ \lambda=\lambdae $, $ \mu=\mue $, etc.
On the other hand, to simplify notation, for \emph{processes} we often omit this dependence,
and write $ \Zsv_\e = \Zsv $, etc.
We also adopt the notation $ C(\alpha,\beta,\ldots)<\infty $ for a \emph{generic} deterministic finite constant
that may change from line to line, but depends only on the designated variables $ \alpha,\beta,\ldots $.
The dependence on $ (\den,b_1)\in(0,1)^2 $ will \emph{not} be indicated as they are \emph{fixed} throughout the article.

To prove Theorem~\ref{thm:S6V:}, in Section~\ref{sect:tight} we establish the tightness of $ \{ \Zsv(\e^{2}\Cdot,\e\Cdot) \}_\e $,
and then, in Section~\ref{sect:pfthmS6V}, we identify the limit point via martingale problems.
As noted earlier, the major technical step here is to establish self-averaging of the quadratic variation in the martingale problem.
We state this as Proposition~\ref{prop:qv} (postponing its proof to Section~\ref{sect:qv})
and give the rest of the proof of Theorem~\ref{thm:S6V:} in Section~\ref{sect:pfthmS6V}.

Given Theorem~\ref{thm:S6V} (or equivalently Theorem~\ref{thm:S6V:}),
Theorem~\ref{thm:6V} follows as a rather straightforward consequence.
In Section~\ref{sect:pfthm6V}, we establish Theorem~\ref{thm:6V}.

\subsection{Moment bounds and tightness}
\label{sect:tight}
In this subsection we prove the tightness of $ \{ \Zsv(\e^{2}\Cdot,\e\Cdot) \}_\e $
by establishing moment bounds on the process.
A useful tool in this context is the following bounds on the transition kernel $ \hk(t,x) $ (defined in~\eqref{eq:hk}--\eqref{eq:hkt}).

\begin{lem}\label{lem:kerbd}
For any $ u,T\in(0,\infty) $ and $ \alpha\in(0,1] $, there exist constants $C(u,T), C(u)>0$ such that
\begin{align}
	\label{eq:hk:sup}
	\hk(t,x) &\leq C\, (t+1)^{-\frac12},
\\
	\label{eq:hk:sum}
	\sum_{x\in\Xi(t)} \hk(t,x) e^{\e u |x|} &\leq C(u),
\\
	\label{eq:hk:sumx}
	\sum_{x\in\Xi(t)} |x|^{\alpha} \hk(t,x) e^{\e u |x|} &\leq C(\alpha,u) (t+1)^{\frac{\alpha}{2}},
\\
	\label{eq:hk:gd}
	\big|\hk(t,x)-\hk(t,x')\big|
	&\leq
	C(T) |x-x'|^\alpha t^{-\frac{\alpha+1}{2}},
\end{align}
for all $ x,x'\in\Xi(t) $ and $ t\in [0,\e^{-2}T]\cap\Z $.
\end{lem}
\begin{proof}
Given the contour integral expression~\eqref{eq:hk:contour} for $ \hk(t,x) $,
these bounds can be obtained by steepest-decent-like analysis.
This type of analysis is carried out in greater generality in Section~\ref{sect:SG} so we use a few results developed there in the following.
In particular, setting $ (x_i-y_i,\alpha) \mapsto (x,u+1) $ in~\eqref{eq:1contourbd} gives
\begin{align*}
	\hk(t,x)
	\leq
	\frac{C(u,T)}{\sqrt{t+1}} e^{ \frac{-(u+1)|x|}{\sqrt{t+1}+C(u)} },
\end{align*}
From this pointwise estimate the bounds~\eqref{eq:hk:sup}--\eqref{eq:hk:sumx} follow.
As for~\eqref{eq:hk:gd},
we set $ (x_i-y_i,\alpha) \mapsto (y,1) $ in~\eqref{eq:1contourbd:gd} (where $ \nabla f(x) := f(x+1)-f(x) $) to get
\begin{align}
	\label{eq:hkgdpf1}
	|\hk(t,y+1) - \hk(t,y)|
	\leq
	C(T) e^{ \frac{-|y|}{\sqrt{t+1}+C} } \frac{1}{t+1}.
\end{align}
Assume without lost of generality that $ x< x' $.
Summing~\eqref{eq:hkgdpf1} over $ y\in[x,x'-1] $ gives
\begin{align}
	\label{eq:hkgdpf2}
	|\hk(t,x') - \hk(t,x)|
	\leq
	\frac{C(T)}{t+1} \sum_{y\in[x,x'-1]} e^{ \frac{-|y|}{\sqrt{t+1}+C} }.
\end{align}
On the r.h.s.\ of \eqref{eq:hkgdpf2}, bounding the exponential factor $ \exp( \frac{-|y|}{\sqrt{t+1}+C} ) \leq 1 $ gives the bound $ \frac{C(T)|x'-x|}{t+1} $.
On the other hand, keeping the exponential factor but summing over $ y\in\Z $ instead gives the bound $ \frac{C(T)}{\sqrt{t+1}} $.
Taking the minimum of these two bounds we conclude
\begin{align*}
%	\label{eq:hkgdpf3}
	|\hk(t,x') - \hk(t,x)|
	\leq
	C(T)\,
	\Big( \frac{1}{\sqrt{t+1}} \wedge \frac{|x'-x|}{t+1} \Big)
	\leq
	\frac{C(T)}{\sqrt{t+1}} \Big( 1\wedge \frac{|x'-x|}{\sqrt{t+1}} \Big).
\end{align*}
Given that $ u\in(0,1] $, the last expression is bounded by
$ \frac{C(T)}{\sqrt{t+1}} ( \frac{|x'-x|}{\sqrt{t+1}} )^u $, which yields~\eqref{eq:hk:gd}.
\end{proof}

A major ingredient in proving moment bounds is a discrete analog of~\eqref{eq:SHEmild}, i.e., the mild form of the \ac{SHE}.
To derive it, fix $ t_1 \leq t_2 \in \Z_{\geq 0} $.
Since $ \hk(t):=\hk^t $, iterating~\eqref{eq:mSHE} $ (t_2-t_1) $-times starting from $ t=t_1 $ gives
\begin{align}
	\label{eq:mSHEmild}
	\Zsv(t_2,x) &= \big( \hk(t_2-t_1)\Zsv(t_1) \big)(x) + \Zmg(t_2,t_1,x),
\\
	\label{eq:Zmg}
\textrm{where} \quad	\Zmg(t_2,t_1,x) &:= \sum_{t=t_1}^{t_2-1} \big( \hk(t_2-t-1)\mg(t) \big)(x+\mu).
\end{align}
Recall the definitions of $ \move $ and $ \moveCent $ from~\eqref{eq:move}--\eqref{eq:moveCent},
and recall from~\eqref{eq:mg} that $ \mg $ is defined in terms of $ \moveCent $.
To pave the way for bounding moments of $ \Zmg $,
in the following lemma we construct a useful bound on conditional moments of $ \moveCent $.
Let $ \mathscr{P}_{2,3}(n) $ denote the set of partitions of $ \{1,\ldots,n\} $ into intervals of $ 2 $ or $ 3 $ elements.
Here intervals refers to set of the form $ U=[a,b] := [a,b]\cap\Z $, $ a\leq b \in \Z $.
For example,
\begin{align*}
	\mathscr{P}_{23}(6)
	=
	\big\{ \{[1,2],[3,4],[5,6]\}, \{[1,3],[4,6]\}, \{ [1,4], [5,6] \}, \{[1,2],[3,6]\} \big\}.
\end{align*}
Given an interval $ U=[a,b] $ and $ \vec{y}\in \Z^{n} $,
we write $ |\vec{y}|_U := |y_b-y_a| $.
%Recall that $ \Vert\,\Cdot\,\Vert_n := (\Ex[\,(\Cdot)^n\,])^{1/n} $.
%
\begin{lem}\label{lem:movebd}
Fix $ n\in\Z_{>0} $.
For all $ t\in\Z_{\geq 0} $ and $ y_1\leq \ldots\leq y_{n} \in \Z $,
we have
\begin{align*}
	\Big| \Ex\Big[ \prod_{i=1}^n \moveCent (t,y_i) \Big| \filt(t) \Big] \Big|
	\leq
	C(n) \sum_{\pi\in\mathscr{P}_{23}(n)} \prod_{U\in\pi} e^{-\frac{1}{C(n)}|\vec{y}|_U}.
\end{align*}
%, and $ f: \Xi(t)\to\R $ be a deterministic function.
%Set $ \bar{f} := \sup_{x\in\Xi(t)} |f(x)| $.
%\begin{align*}
%	\Big \Vert \sum_{x\in\Xi(t)} f(x) \mg(t,x) \Big \Vert_{2n}^2
%	\leq
%	C(n) \bar{f} \, \sum_{x\in\Xi(t)} |f(x)| \Vert Z(t,x) \Vert^2_{2n}.
%\end{align*}
\end{lem}
\begin{proof}
Fix $ n\in\Z_{>0} $, $ t\in \Z_{\geq 0} $, and $ y_1\leq \ldots \leq y_{n} \in \Z $.
Throughout this proof, we write $ C=C(n) $ and $ \Ex'[\,\Cdot\,] := \Ex[\,\Cdot\,|\filt(t)] $ to simplify notation. %and the sums over $ x, x_1, $ etc are always over $ \Xi(t) $.
We invoke the expression of $ \move(t,y) $ from~\eqref{eq:moveBer},
where $ B(t,\eta) $ and $ B'(t,\eta) $ are independent Bernoulli variables defined in Definition~\ref{def:Bs}.
To reduce notation, we set
\begin{align*}
	I(y',y) := \prod_{z=y'+1}^y \Big(B'\big(t,z;\eta(t,z)\big)-B\big(t,z;\eta(t,z)\big)\Big) B\big(t,y';\eta(t,y')\big)
\end{align*}
for the term within the sum in~\eqref{eq:moveBer},
and write $ \bar{I}(y',y) := I(y',y) - \Ex[I(y',y)|\filt(t)] $.
This gives $ \moveCent(t,y) = \sum_{y'\leq y} \bar{I}(y',y) $, and hence
\begin{align}
	\label{eq:movedecor1}
	\Ex'\Big[ \prod_{i=1}^{n}\moveCent(t,y_i) \Big]
	=
	\sum_{\vec{y}'\in Y} \Ex'\Big[ \prod_{i=1}^{n}\bar{I}(y'_i,y_i) \Big],
\end{align}
where
$ Y := \big\{ (y'_1,\ldots,y'_{n})\in\Z^{n} : y'_i \leq y_i, i=1,\ldots,n \big\} $.
The r.h.s.\ of~\eqref{eq:movedecor1} is summable.
To see this, note from Definition~\ref{def:Bs} (together with $ b_1,b^\e_2 $ being bounded away from $ 0 $ and $ 1 $ under our scale)
that we have $ \Ex'[I(y',y)^{\ell}] \leq \exp(-\frac{1}{C}|y'-y|) $, $ \ell\in\Z_{>0} $, which gives
\begin{align}
	\label{eq:moveI:bd}
	\Ex'\big[ |\bar{I}(y',y)|^{\ell} \big] \leq C e^{-\frac{1}{C}|y-y'|},
	\quad
	\ell \in \Z_{>0}.
\end{align}
From this we see that the r.h.s.\ of~\eqref{eq:movedecor1} is summable.

It is useful to arrange the r.h.s.\ of~\eqref{eq:movedecor1} according to how the $ \bar{I} $'s are dependent.
To this end, let $ \mathscr{P}=\mathscr{P}(n) $ denote the set of partitions of $ \{1,\ldots,n\} $ into intervals.
For $ (y'_1,\ldots,y'_{n}) \in Y $,
we say a pair of coordinates $ y_i,y_j $, $ i\neq j $, are \DefinText{connected} if $ [y'_i,y_i] \cap [y'_j,y_j] \neq \emptyset $.
Recall that the $ y $'s are ordered $ y_1\leq \ldots \leq y_{n} $, and recall that for $ \vec{y}'\in Y $ we have $ y'_i \leq y_i $.
This being the case, we see that if $ y_i $ and $ y_{j} $ are connected, for $ i<j $, then $ y_{i+1},\ldots,y_{j-1} $ must also be connected to %$ y_i $ and
$ y_j $.
Group indices (the $ i $'s) together
if the corresponding coordinates (the $ y_i $'s) are connected.
This grouping procedure maps each $ \vec{y}'\in Y $ into a partition $ p(\vec{y}')\in \mathscr{P}(n) $.
We then rewrite \eqref{eq:movedecor1} as
\begin{align}
	\label{eq:movedecor2}
	\Ex'\Big[ \prod_{i=1}^{n}\moveCent(t,y_i) \Big]
	=
	\sum_{\pi\in\mathscr{P}} \sum_{ \vec{y}'\in Y(\pi) } \Ex'\Big[ \prod_{i=1}^{n}\bar{I}(y'_i,y_i) \Big],
\end{align}
where $ Y(\pi) := \{ \vec{y}'\in Y: p(\vec{y}')=\pi \} $.
Since conditioning on $\filt(t)$ (so that $\eta(t)$ is fixed) the Bernoulli variables $ \{ B(t,y;\eta(t,y)),B'(t,y;\eta(t,y)):t\in\Z_{\geq 0},y\in \Z \}$ are independent,
the variables $ \bar{I}(y'_i,y_i) $ are independent among unconnected coordinates.
Consequently, the r.h.s.\ of~\eqref{eq:movedecor2} factorizes among unconnected coordinates
\begin{align}
	\label{eq:movedecor3}
	\Ex\Big[ \prod_{i=1}^{n}\moveCent(t,y_i) \Big| \filt(t) \Big]
	=
	\sum_{\pi\in\mathscr{P}} \sum_{ \vec{y}'\in Y(\pi) } \prod_{U\in\pi} \Ex\Big[ \prod_{i\in U}\bar{I}(y'_i,y_i) \Big| \filt(t) \Big].
\end{align}
For the special case of a singleton interval $ U =\{i_*\} $,
%using the independence of $ B(y,\eta),B'(y,\eta) $ gives
one has
 $ \Ex[ \bar{I}(y'_{i_*},y_{i_*}) | \filt(t) ]=0 $.
This implies the expectation on r.h.s.\ of~\eqref{eq:movedecor3} vanishes if any $ U\in\pi $ is a singleton.
We hence need only to sum over partitions consisting of non-singleton intervals, i.e.,
\begin{align}
	\label{eq:movedecor4}
	\Ex'\Big[ \prod_{i=1}^{n}\moveCent(t,y_i) \Big]
	=
	\sum_{\pi\in\mathscr{P}_{\geq 2}} \sum_{ \vec{y}'\in Y(\pi) } \prod_{U\in\pi} \Ex'\Big[ \prod_{i\in U}\bar{I}(y'_i,y_i) \Big],
\end{align}
where $ \mathscr{P}_{\geq 2}(n) := \{ \pi\in\mathscr{P}: \# U \geq 2, \forall U\in\pi \} $.

On the r.h.s.\ of~\eqref{eq:movedecor4},
using H\"{o}lder's inequality
$ |\Ex'[ \prod_{i\in U}\bar{I}(y'_i,y_i) ]| \leq \prod_{i\in U} (\Ex'[ \bar{I}(y'_i,y_i)^{\#U} ])^{\frac{1}{\#U}} $,
followed by using~\eqref{eq:moveI:bd} and $ \frac{1}{\#U} \geq \frac1n = C $, we find that
\begin{align}
	\label{eq:movedecor6}
	\Big| \Ex'\Big[ \prod_{i=1}^{n}\moveCent(t,y_i) \Big] \Big|
	\leq
	C \sum_{\pi\in\mathscr{P}_{\geq 2}} \sum_{ \vec{y}'\in Y(\pi) } \exp\Big( -\frac{1}{C} \sum_{i=1}^n |y_i-y'_i| \Big).
\end{align}
Fix a partition $ \pi=\{U_1,\ldots,U_{\#\pi}\} $.
%and write $ U_j = [a_{j},b_j] $ where $b_j:=a_{j+1}-1$.
We claim that the sum over $ \vec{y}'\in Y(\pi) $ in~\eqref{eq:movedecor6}
will lead us to the bound
\begin{align}
	\label{eq:moveClaim}
	\Big| \Ex'\Big[ \prod_{i=1}^{n}\moveCent(t,y_i) \Big] \Big|
	\leq
	C \sum_{\pi\in\mathscr{P}_{\geq 2}}\prod_{U\in\pi}
	 e^{ -\frac{1}{C}  |\vec{y}|_{U} }.
\end{align}
To prove this claim, letting $U=[a,b]\in \pi$,
we define a subset $\{i_0\le \cdots \le i_q\} \subset U$ inductively.
First, let $i_0:=a$.
Suppose that $i_0,\cdots,i_p$ have been defined.
If $i_p=b$ we stop the induction  with $q:=p$;
otherwise, let
\[
i_{p+1} := \max\{j\in (i_p,  b] \cap \Z \,:\, \exists i \le i_p \mbox{ s.t. $y_i$ and $y_j$ are connected} \}.
\]
The  set on the right hand side is non-empty by definition of a group.
In fact choosing $i_{p+1}$
to be any element in this set (not necessarily the max),
the following argument will still be valid,
and what is important is that
by construction $y'_{i_{p+1}}\le y_{i_{p}} \le y_{i_{p+1}}$.
Hence each sum over $ y'_{i_k} $ in \eqref{eq:movedecor6}, with $ i_k \in  \{i_0, \cdots,  i_q\} $, produces a factor of $ C\exp(-\frac{1}{C}|y_{i_{k}}-y_{i_{k-1}}|) $.
On the other hand, each sum over $ y'_{i} $, with $ i \in U\setminus  \{i_0, \cdots,  i_q\} $, produces a factor of $ C $.
Thus the sum over all $ y'_{i} $ with $ i \in U$ produces a factor
\[
C  \exp\Big(-\frac{1}{C}  \sum_{k=1}^q |y_{i_k} - y_{i_{k-1}}| \Big)
=C  e^{-\frac{1}{C}  |y_{b} - y_{a}| }.
\]
The claimed bound \eqref{eq:moveClaim}  immediately follows.

This is almost the desired result except that the sum is over $ \mathscr{P}_{\geq 2}(n) $ instead of $ \mathscr{P}_{23}(n) $.
To go from the former to the latter, we `chop' longer intervals into shorter intervals of length $ 2 $ or $ 3 $.
For example, if $ U=[1,5] $, we indeed have $ \exp( -\frac{1}{C}|\vec{y}|_{[1,5]} ) \leq \exp( -\frac{1}{C}|\vec{y}|_{[1,2]}) \exp( -\frac{1}{C}|\vec{y}|_{[3,5]} ) $.
More generally, for $ \# U \geq 4 $, we always have
$ \exp( -\frac{1}{C}|\vec{y}|_{U} ) \leq \prod_{V}\exp( -\frac{1}{C}|\vec{y}|_{V}) $,
where the $ V $'s partition $ U $ into intervals of length $ 2 $ or $ 3 $.
This completes the proof.
\end{proof}

We now proceed to derive moment bounds on $ \Zmg $ (defined in~\eqref{eq:Zmg})
which we view as a weighted sum of $ \mg(t,x) $.
In fact, we will consider a generic weighted sum with weight $ f(t,x) $.
Recall that $ \Vert\,\Cdot\,\Vert_n := (\Ex[\,(\Cdot)^n\,])^{1/n} $.
\begin{lem}\label{lem:Zmgbd}
Fix $ n\in\Z_{>0} $, $ t\in\Z_{\geq 0} $, $ t_1<t_2\in\Z_{\geq 0} $,
and let $ f(t,x) $ be a deterministic function defined on $ t\in[t_1,t_2]\cap \Z $ and $ x\in\Xi(t) $.
Write $ f_\infty(t) := \sup_{x\in\Xi(t)} |f(t,x)| $.
We have
\begin{align*}
	\Big \Vert \sum_{t=t_1}^{t_2-1}\sum_{x\in\Xi(t)} f(t,x) \mg(t,x) \Big \Vert_{2n}^2
	\leq
	\e C(n) \sum_{t=t_1}^{t_2-1} \sum_{x\in\Xi(t)}  |f_\infty(t) f(t,x)| \, \big\Vert Z(t,x) \big\Vert^2_{2n}.
\end{align*}
\end{lem}
\begin{proof}
Throughout this proof we write $ C=C(n) $.
Recall from Proposition~\ref{prop:mSHE} that $ \mg(t,x) $ is a martingale increment.
Hence the process
$
	\sum_{s=t_1}^{t} \sum_{x\in\Xi(t)} f(t,x) \mg(t,x),
$
$ t=t_1,\ldots,t_2-1 $,
is a martingale. Burkholder's inequality applied to this martingale gives
\begin{align}
	\label{eq:Burkholder}
	\Big \Vert \sum_{t=t_1}^{t_2-1}\sum_{x\in\Xi(t)} f(t,x) \mg(t,x) \Big \Vert_{2n}^2
	\leq
	C\,\sum_{t=t_1}^{t_2-1} \Big \Vert \sum_{x\in\Xi(t)} f(t,x) \mg(t,x) \Big \Vert_{2n}^2.
\end{align}
Recall that $ \mg(t,x) $ is given in terms of $ \Zsv(t,x) $ and $ \move(t,x+\mu t) $ as~\eqref{eq:mg}.
Under our scale $ \lambdae|1-\taue| \leq C \sqrt{\e} $.
Set $ G(t) := \sum_{x\in\Xi(t)} f(t,x) \moveCent(t,x+\mu t) $, we then have
\begin{align}
	\label{eq:Zmgbd1}
	\Big \Vert \sum_{t=t_1}^{t_2-1}\sum_{x\in\Xi(t)} f(t,x) \mg(t,x) \Big \Vert_{2n}^2
	\leq
	\e C\,\sum_{t=t_1}^{t_2-1} \Vert G(t) \Vert_{2n}^2.
\end{align}

To bound the last expression in~\eqref{eq:Zmgbd1}, we proceed to estimate
\begin{align*}
	\Ex[G(t)^{2n}] = \sum_{\vec{x}\in\Xi(t)^{2n}} \Ex\Big[ \prod_{i=1}^{2n} f(t,x_i) \Zsv(t,x_i) \moveCent(t,x_i+\mu t) \Big].
\end{align*}
Let us evaluate the r.h.s.\ by first conditioning on $ \filt(t) $.
Since $ \Zsv(t,x_i) $ is $ \filt(t) $-measurable, we may apply Lemma~\ref{lem:movebd} to bound the conditional expectation over $ \moveCent $. This yields,
for $ x_1 \leq \ldots \leq x_{2n} \in \Xi(t) $,
\begin{align}
	\label{eq:sasf}
	\Ex\Big[ \prod_{i=1}^{2n} f(t,x_i) \Zsv(t,x_i) \moveCent(t,x_i+\mu t) \Big| \filt(t) \Big]
	\leq
	\sum_{\pi\in\mathscr{P}_{23}} \prod_{U\in\pi} e^{-\frac{1}{C}|\vec{x}|_U} \prod_{i\in U} |f(t,x_i) \Zsv(t,x_i)|,
\end{align}
where we write $ \mathscr{P}_{23} := \mathscr{P}_{23}(2n) $ to simplify notation.
Sum both sides of~\eqref{eq:sasf} over the $ x_i $'s.
By paying a factor of $ n! = C $ we may and shall restrict the sum to \emph{ordered} tuples $ (x_1 \leq \ldots \leq x_{2n}) $.
Rewriting the resulting $ (2n) $-fold sum over $ (x_1,\ldots,x_{2n}) $ into iterated sums over $ (x_{i})_{i\in U} $, $ U\in\pi $, and rearranging the result accordingly, we then have
\begin{align*}
	\Ex[ G(t)^{2n} | \filt(t) ]
	\leq
	\sum_{\pi\in\mathscr{P}_{23}} \prod_{U\in\pi} \Big( \sum_{\vec{x}\in\Xi(t)^{\#U}_{\leq}}  e^{-\frac{1}{C}|x_{\#U}-x_1|} \prod_{i=1}^{\# U} |f(t,x_i) \Zsv(t,x_i)| \Big),
\end{align*}
where $ \Xi(t)^{j}_{\leq} := \{ (x_1 \leq \ldots \leq x_{j}) \in \Xi(t)^j \} $ denotes the set of ordered $ j $ tuples.
Within the last expression, apply Young’s inequality $ \prod_{U\in\pi}a_U \leq \sum_{U\in\pi} \frac{\# U}{2n} |a_U|^{\frac{2n}{\# U}} $.
Together with $ \#U =2,3 $, we have
\begin{align*}
	\Ex[ G(t)^{2n} | \filt(t) ]
	\leq
	C \sum_{j=2,3} \Big( \sum_{\vec{x}\in\Xi(t)^{j}_\leq}  e^{-\frac{1}{C}|x_{j}-x_1|} \prod_{i=1}^{j} |f(t,x_i) \Zsv(t,x_i)| \Big)^{\frac{2n}{j}}.
\end{align*}
Further bound $ \exp(-\frac{1}{C}|x_j-x_1| \leq \exp(-\frac{1}{Cj}\sum_{i=1}^j|x_i-x_1|) $ (because $ x_1 \leq \ldots \leq x_j $),
and then release the sum from ordered tuples $ \Xi(t)^j_{\leq} $ to unordered tuples $ \Xi(t)^j $.
Take $ (\Ex[\,\Cdot\,])^{1/n} $ on both sides of the result, and then apply $ ( \sum_{j=2,3} |a_j |)^{1/{n}} \leq 2 \sum_{j=2,3} |a_j|^{1/n} $. From this we obtain
\begin{align*}
	\Vert G(t) \Vert_{2n}^2
	&\leq
	C \sum_{j=2,3} \Big\Vert \sum_{\vec{x}\in\Xi(t)^{j}} \prod_{i=1}^{j} e^{-\frac{1}{C}|x_{i}-x_1|} |f(t,x_i) \Zsv(t,x_i)| \Big\Vert^{\frac2j}_{\frac{2n}{j}}.
\end{align*}
Pass $ \Vert\,\Cdot\Vert_{2n} $ into the sum by the triangle inequality,
and use H\"{o}lder's inequality to write
$ \Vert \prod_{i=1}^{2n} \Zsv(t,x_i) \Vert_{2n/j} \leq \prod_{i=1}^{2n} \Vert \Zsv(t,x_i) \Vert_{2n} $.
We then obtain
\begin{align}
	\label{eq:G1}
	\Vert G(t) \Vert_{2n}^2 &\leq
	C \sum_{j=2,3}  g_j(t)^{\frac2j},
	\quad
	g_j(t) :=
	\sum_{\vec{x}\in\Xi(t)^{j}} \prod_{i=1}^{j} e^{-\frac{1}{C}|x_{i}-x_1|} |f(t,x_i)| \, \big\Vert\Zsv(t,x_i)\big\Vert_{2n}.
\end{align}
Recall that $ f_\infty(t) := \sup_{x\in\Xi(t)} |f(t,x)| $.
Set $ \tilde{g}(t) := \sum_{x\in\Xi(t)} |f_\infty(t)f(t,x)| \, \Vert\Zsv(t,x)\Vert^2_{2n} $.
For the term $ g_3 $, using the Cauchy–Schwarz inequality over $ \sum_{x_3} $ gives
\begin{align*}
	g_3(t)
	&\leq
	\sum_{x_1,x_2\in\Xi(t)}
	\prod_{i=1}^{2} e^{-\frac{1}{C}|x_{i}-x_1|} |f(t,x_i) \Zsv(t,x_i)| \Big\Vert^{\frac2j}_{\frac{2n}{j}}.
	\,
	\Big( \sum_{x_3\in\Xi(t)} |f(t,x)|^2 \Vert\Zsv(t,x)\Vert^2_{2n} \Big)^\frac12 \, \Big( \sum_{x_3\in\Xi(t)} e^{-\frac{1}{C}|x_3-x_1|} \Big)^{\frac12}
\\
	&\leq
	C\, g_2(t) \tilde{g}(t)^\frac12.
\end{align*}
As for $ g_2(t) $,
since $  \taue = e^{-\sqrt\e} $ under current scaling,
referring to~\eqref{eq:ColeHopfTransform} we see that $ \Zsv(t,x_2) \leq \Zsv(t,x_1) e^{\sqrt{\e}|x_2-x_1|} $.
Using this to bound $ \Vert\Zsv(t,x_2)\big\Vert_{2n} $, and bounding $ |f(t,x_2)| $ by $ f_\infty(t) $, we have
\begin{align*}
	g_2(t)
	\sum_{x_1\in\Xi(t)}|f_\infty(t) f(t,x_1)| \, \big\Vert\Zsv(t,x_1)\big\Vert^2_{2n} \sum_{x_2\in\Xi(t)} e^{-(\frac{1}{C}-C\sqrt{\e})|x_2-x_1|}
	\leq
	C \tilde{g}(t).
\end{align*}
Combining the preceding bounds on $ g_2(t) $ and $ g_3(t) $ with~\eqref{eq:G1},
we arrive at
\begin{align*}
	\Vert G(t) \Vert_{2n}^2 \leq C \tilde{g}(t) := C \sum_{x\in\Xi(t)} |f_\infty(t)f(t,x)| \, \Vert\Zsv(t,x)\Vert^2_{2n}.
\end{align*}
Inserting this back into~\eqref{eq:Burkholder} gives the desired result.
\end{proof}

Given Lemma~\ref{lem:Zmgbd}, we are now ready to establish moment bounds on $ \Zsv $.
\begin{prop}\label{prop:momt}
\begin{enumerate}[label=(\alph*)]
\item[]
\item \label{prop:momt:nearStat}
Start the stochastic \ac{6V} model from near stationary initial conditions (as Definition~\ref{def:nearStat}, with $ \den\in(0,1) $ fixed as declared previously),
let $ u=u(n,\alpha) $ be the corresponding exponent,
and let $ Z(t,x) $ denote the resulting Hopf--Cole transform (with respect to a fixed density $ \rho\in(0,1) $).
For any $ \alpha\in (0,\frac12) $, $ n\in\Z_{>0} $, and $ T<\infty $,
there exist $ C=C(n,\alpha,T)<\infty $ such that
\begin{align}
	\label{eq:Zmomt}
	\Vert\Zsv(t,x)\Vert_{2n}
	&
	\leq e^{u\e|x|},
\\
	\label{eq:gZmomt}
	\Vert\Zsv(t,x)-\Zsv(t,x') \Vert_{2n}
	&\leq
	C  \left(\e|x-x'|\right)^{\alpha}
	e^{u\e(|x|+|x'|)},
\\
	\label{eq:tZmomt}
	\Vert\Zsv(t,x)-\Zsv(t',x)\Vert_{2n}
	&\leq
	C \left(\e^2|t-t'|\right)^{\frac{\alpha}{2}}
	e^{2u\e|x|}.
\end{align}
for any $ t,t'\in [0,\e^{-2}T] $ and $ x,x' \in \R $.
\item \label{prop:momt:step}
Start the stochastic \ac{6V} model from the step initial condition $ N(0,x)=(x)_+ $,
and let $ Z(t,x) $ denote the resulting Hopf--Cole transform (with respect to a fixed density $ \rho\in(0,1) $).
For each given $ n\in\Z_{>0} $ and $ \alpha\in(0,\frac14) $,
there exist $ C=C(n,\alpha)<\infty $ and $ \tau=\tau(n,\alpha)>0 $ such that
\begin{align}
	\label{eq:Zmomt:step}
	\big\Vert \tfrac{\den(1-\den)}{\sqrt\e}\Zsv(t,x) \big\Vert_{2n}
	&
	\leq (\e^{2}t)^{-\frac12},
\\
	\label{eq:gZmomt:step}
	\big\Vert \tfrac{\den(1-\den)}{\sqrt\e}(\Zsv(t,x)-\Zsv(t,x')) \big\Vert_{2n}
	&\leq
	C  \left(\e|x-x'|\right)^{\alpha} (\e^{2}t)^{-\frac{1+\alpha}{2}}.
\end{align}
for any $ t\in (0,\e^{-2}\tau] $ and $ x,x' \in \R $.
\end{enumerate}
\end{prop}
\begin{proof}
Fix $ n $, $ \alpha\in (0,\frac14) $, and $ u=u(\alpha,n) $.
Throughout this proof we write $ C=C(n,\alpha,T) $.
Recall that $ \Zsv(t,x) $ is defined on $ [0, \infty)\times\R $ by linear interpolations.
This being the case, it suffices to consider the lattice $ t,t'\in\Z_{\geq 0} $ and $ x,x'\in\Xi(t) $.
Generalization to continuum $ t,x $, etc., follows easily.
Hence throughout this proof we assume $ t,t'\in\Z_{\geq 0} $ and $ x,x'\in\Xi(t) $, etc.

\ref{prop:momt:nearStat}
We begin with~\eqref{eq:Zmomt}.
On the space of functions $ f:\Xi(t)\to \R $,
it is convenient to consider the norm $ [f]_{2u} := \sup_{x\in\Xi(t)} |f(x)|e^{-2u\e|x|} $.
Our goal is to bound
\begin{align*}
	D(t)
	:=
	\big| \Vert\Zsv(t,\Cdot)\Vert^2_{2n} \big|_{2u}
	=
	\sup_{x\in\Xi(t)} \Vert Z(t,x) \Vert_{2n}^2 e^{-2u\e|x|}.
\end{align*}
To this end, take $ \Vert\,\Cdot\,\Vert^2_{2n} $ on both sides of~\eqref{eq:mSHEmild} to obtain
\begin{align}
	\label{eq:Zmom1}
	\Vert \Zsv(t_2,x) \Vert_{2n}^2
	\leq
	\big( \big\Vert \big( \hk(t)Z(t) \big)(x)  + \Zmg(t_2,t_1,x) \big\Vert_{2n}  \big)^2
	\leq
	2(\Adr(x)^2 + \Amg(x)^2),
\end{align}
where
\begin{align}
	\label{eq:A1}
	\Adr(x) &:=\sum_{y\in\Xi(t_1)}\hk(t_2-t_1,x-y)\Vert Z(t_1,y) \Vert_{2n},
\\
	\label{eq:A2}
	\Amg(x) &:= \Vert \Zmg(t_2,t_1,x) \Vert_{2n}.
\end{align}
Applying $ [\,\Cdot\,]_{2u} $ to both sides of~\eqref{eq:Zmom1} yields
\begin{align}
	\label{eq:Zmom1.5}
	D(t)
	\leq
	2\big[ \Adr^2 \big]_{2u} + 2\big[ \Amg^2 \big]_{2u}.
\end{align}

We proceed to bound the r.h.s.\ of~\eqref{eq:Zmom1.5}. Write
\begin{align}
	\label{eq:Znormexp}
	\Vert Z(t_1,y) \Vert_{2n}
	\leq
	\Big( D(t_1) e^{2u\e|y|} \Big)^\frac12
	\leq
	 D(t_1)^\frac12 e^{u\e|y-x|} e^{u\e|x|}.
\end{align}
In~\eqref{eq:A1}, use the bound~\eqref{eq:Znormexp},
and then sum over $ y\in\Xi(t_2) $ with the aid of~\eqref{eq:hk:sum}.
We obtain $ \Adr(x)^2 \leq C\, D(t_1) e^{2u\e|x|} $, and hence $ [\Adr^2]_{2u} \leq C\,D(t) $.
Next, recall the definition of $ \Zmg(t_2,t_1,x) $ from~\eqref{eq:Zmg},
and write $ x_{-\mu} := x-\mu $ to simplify notation.
We apply Lemma~\ref{lem:Zmgbd} with $ f(t,y) = \hk(t_2-t_1-1,x_{-\mu}-y) $.
With the aid of~\eqref{eq:hk:sup} and~\eqref{eq:Znormexp}, we have
\begin{align*}
	\Amg(x)^2
	\leq
	\e C \sum_{t=t_1}^{t_2-1} \sum_{y\in \Xi(t)} \frac{1}{\sqrt{t_2-t+1}} \hk(t_2-t-1,x_{-\mu}-y) e^{2u\e|x-y|} e^{2u\e|x|}  D(t).
\end{align*}
Further using~\eqref{eq:hk:sum} to bound the sum over $ y\in\Xi(t) $, we obtain
\begin{align}
	\label{eq:A2bd}
	\Amg(x)^2 \leq e^{2u\e|x|}  C\,\e^2 \sum_{t=t_1}^{t_2-1} (\e^{2}(t_2-t_1))^{-\frac12} D(t).
\end{align}
This gives $ [\Amg^2]_{2u} \leq C\,\e^2 \sum_{t=t_1}^{t_2-1} (\e^{2}(t_2-t_1))^{-\frac12} D(t) $.
Inserting the preceding bounds on $ [\Adr]_{2u} $ and $ [\Amg]_{2u} $ into~\eqref{eq:Zmom1.5} gives
\begin{align}
	\label{eq:Zmom2}
	D(t_2)
	\leq
	C\, D(t_1)
	+
	C\,\e^2 \sum_{t=t_1}^{t_2-1} (\e^{2}(t_2-t_1))^{-\frac12} D(t).
\end{align}

Now, set $ E(t) := \sup_{s\in[0,t]\cap\Z} D(t) $.
From~\eqref{eq:Zmom2} we have
$ E(t_2) \leq C E(t_1) + E(t_2) C\,\e^2 \sum_{t=t_1}^{t_2-1} (\e^{2}(t_2-t_1))^{-\frac12} $.
Given that $ t_1 \leq t_2 \leq \e^{-2} T $, the last sum can be estimated by comparison to an integral, yielding
$ E(t_2) \leq C E(t_1) + C_* ((\e^{2}(t_2-t_1))^{\frac12} E(t_2) $, for some constant $ C_*=C_*(u,n,T) $.
Fixing $ \delta>0 $ small enough so that $ C_* \sqrt{\delta} < \frac12 $.
We then have $ E(t_2) \leq C E(t_1) $, for all $ t_1<t_2\in\Z_{\geq 0} $ with $ t_2-t_1 \leq \e^{-2}\delta $.
Iterate this inequality starting from $ t_1=0 $.
After $ \lceil T/\delta \rceil=C $ iterations we conclude that $ E(\lceil\e^{-2}T\rceil) \leq C^C E(0) = C E(0) $.
From the assumption~\eqref{eq:nearStat} of near stationary initial conditions,
we have $ E(0) \leq C $, so $ E(\lceil\e^{-2}T\rceil) \leq C $, which gives the desired result~\eqref{eq:Zmomt}.

Next we turn to~\eqref{eq:gZmomt}.
In~\eqref{eq:mSHEmild}, set $ (t_1,t_2)=(0,t) $, take the difference for $ x=x' $ and $ x=x $,
and then take $ \Vert\,\Cdot\,\Vert^2_{2n} $ on both sides of the result. We obtain
\begin{align}
	\label{eq:gZmom1}
	\Vert \Zsv(t,x')-\Zsv(t,x) \Vert_{2n}^2
	\leq
	2(\Agdr^2 + \Agmg^2 ),
\end{align}
where $ \Agdr := \sum_{y\in\Z} \hk(t,x-y)\Vert Z(0,y+x'-x)-Z(0,y) \Vert_{2n} $
and $ \Agmg:= \Vert \Zmg(t,0,x') - \Zmg(t,0,x) \Vert_{2n} $.
To bound $ \Agdr $, use~\eqref{eq:nearStat} in conjunction with~\eqref{eq:hk:sum} to get $ \Agdr \leq C\,|\e(x-x')|^{\alpha}(e^{\e u(|x|+|x'|)}) $.
As for $ \Agmg $, similar to the preceding procedure for bounding $ \Amg(x)^2 $,
here we apply Lemma~\ref{lem:Zmgbd} with $ f(t,y) = \hk(t-s-1,x'_{-\mu}-y)-\hk(t-s-1,x_{-\mu}-y) $.
With the aid of \eqref{eq:hk:gd} and \eqref{eq:Znormexp}, we obtain
\begin{align}
\label{eq:A'2}
\begin{split}
	\Agmg^2 &\leq \e C \sum_{s=0}^{t-1} \sum_{y\in \Xi(t)} \frac{|x-x'|^{2\alpha}}{(t-s+1)^{\frac12+\alpha}}
\\	
	&\big( \hk(t-s-1,x_{-\mu}-y) e^{2u\e|x-y|+2u\e|x|} + \hk(t-s-1,x'_{-\mu}-y) e^{2u\e|x'-y|+2u\e|x'|} \big)\, D(s).
\end{split}
\end{align}
Further using~\eqref{eq:hk:sum} to bound the sum over $ y\in\Xi(t) $,
together with $ D(s) \leq C $ (which we showed previously),
we obtain $ \Agmg^2 \leq C e^{2u\e|x'|+2u\e|x'|} \e^2 \sum_{s=0}^{t-1} (\e|x-x'|)^{2\alpha} (\e^2(t-s))^{-\frac{1}{2}-\alpha} $.
The last sum can be estimated by comparison to integrals, yielding $ \Agmg^2 \leq C e^{2u\e|x'|+2u\e|x'|} |\e(x-x')|^{2\alpha} $.
Inserting the preceding bounds on $ \Agdr $ and $ \Agmg^2 $ into~\eqref{eq:gZmom1} gives the desired result~\eqref{eq:gZmomt}.

Next, to show~\eqref{eq:tZmomt}, subtract $ \Zsv(t_1,x) $ from both sides of~\eqref{eq:mSHEmild},
and take $ \Vert\,\Cdot\,\Vert_{2n} $ of the result to get
\begin{align}
	\label{eq:tZmom1}
	\Vert Z(t_2,x)-Z(t_1,x) \Vert_{2n}
	\leq
	\Atdr(x) + \Amg(x),
\end{align}
where $ \Atdr(x) := \Vert (\hk(t_2-t_1)Z(t_1))(x)-Z(t_1,x)\Vert_{2n} $.
From~\eqref{eq:A2bd} and $ D(t) \leq C $ (which we showed previously) we have
$ \Amg(x) \leq C\,(\e^{2}(t_2-t_1))^\frac14 \leq C\,(\e^{2}(t_2-t_1))^\frac{\alpha}{2} $.
As for $ \Atdr(x) $, using $ \sum_{y\in\Xi(t_1)} \hk(t_2-t_1,x-y) =1 $ we write
\begin{align*}
	\Atdr(x)
	&=
	\Big\Vert \sum_{y\in\Xi(t_1)} \hk(t_2-t_1,x-y) \big( Z(t_1,y) - Z(t_1,x) \big) \Big\Vert_{2n} \\
	&\leq
	\sum_{y\in\Xi(t_1)}  \hk(t_2-t_1,x-y) \Big\Vert Z(t_1,y) - Z(t_1,x) \Big\Vert_{2n}.
\end{align*}
Within the last expression we apply the bound~\eqref{eq:gZmomt} with $ (x',x)=(y,x)\Xi(t_1)\times\Xi(t_2) $.
As noted previously,  the bound~\eqref{eq:gZmomt} extends to all $ x',x\in\R $ via linear interpolation.
Further using~\eqref{eq:hk:sumx} to bound the resulting sum over $ y\in\Xi(t_1) $. We then obtain
\begin{align*}
	\Atdr(x)
	\leq
	C \sum_{y\in\Xi(t_1)}  \hk(t_2-t_1,x-y) |\e(x-y)|^{\alpha} e^{u\e(|x|+|y|)}
	\leq
	C |\e^2(t_2-t_1)|^{\frac{\alpha}{2}} e^{2u\e|x|}.
\end{align*}
Inserting the preceding bounds on $ \Atdr(x) $ and $ \Amg(x) $ into~\eqref{eq:tZmom1} gives the desired result~\eqref{eq:tZmomt}.

\ref{prop:momt:step}
Set $ \hat{Z}(t,x) := \frac{\den(1-\den)}{\sqrt{\e}} Z(t,x) $ to simplify notation.
On the space of functions $ f: \Xi(t) \to \R $, it is convenient to consider the norm
\begin{align*}
	[f]_{*,t} := (\e^{2}t) \sup_{x\in\Xi(t)} |f(x)|  +  (\e^{2}t)^{\frac12} \e \!\! \sum_{x\in\Xi(t)} |f(x)|.
\end{align*}
We write $ \hat{D}(t) := [ \Vert Z(t) \Vert^2_{2n} ]_{*,t} = [ \Vert Z(t,\Cdot) \Vert^2_{2n} ]_{*,t} $, so in particular
\begin{align}
	\label{eq:D*sup}
	\e \sum_{x\in\Xi(s)} \Vert \hat Z(s,x) \Vert^2_{2n}	\leq (\e^2s)^{-\frac12} \hat{D}(s),
\\
	\label{eq:D*sum}
	\Vert \hat Z(s,x) \Vert^2_{2n}	\leq (\e^2s)^{-1} \hat{D}(s),
\end{align}
Multiplying both sides of~\eqref{eq:Zmom1} by $ \den(1-\den)\e^{-\frac12} $, here we have
\begin{align}
	\label{eq:stepmom1}
	\big\Vert \hat{Z}(t,x) \big\Vert^{2}_{2n} \leq 2 \hatAdr(t,x)^2+ 2 \hatAmg(t,x)^2,
\end{align}
where $ \hatAdr(t,x) := \sum_{x\in\Xi(t)} \hk(t,x-y) \hat{Z}(0,y) $ (note that here $ \hat{Z}(0,y) $ is deterministic),
and $ \hatAmg(t,x) := \frac{\den(1-\den)}{\sqrt{\e}} \Vert \Zmg(t,0,x)\Vert_{2n} $.
Apply $ [\,\Cdot\,]_{*,t} $ to both side of~\eqref{eq:stepmom1} yields
\begin{align}
	\label{eq:stepmom2}
	\hat{D}(t) \leq 2 \big[ \hatAdr(t)^2 \big]_{*,t} + 2 \big[ \hatAmg(t)^2 \big]_{*,t}.
\end{align}

We proceed to bound the r.h.s.\ of~\eqref{eq:stepmom2}.
Recall that $ N(0,x)=x_+ $ under the step initial condition, so $ \hat{Z}(0,x) = \e^{-\frac12}\den(1-\den) e^{-\e(\den(x)_+-x)} $.
From the last expression, it is straightforward to verify that
\begin{align}
	\label{eq:hatZ0:sumbd}
	\e \sum_{x\in\Z} \hat{Z}(0,y) \leq C.
\end{align}
Using this in conjunction with \eqref{eq:hk:sup} and~\eqref{eq:hk:sum} yields
\begin{align*}
	|\hatAdr(t,x)| \leq \frac{C}{\e\sqrt{t+1}} \leq C \, (\e^{2}t)^{-\frac12},
	\quad
	\e \sum_{x\in\Xi(t)} |\hatAdr(t,x)| \leq C.
\end{align*}
From these properties we deduce
\begin{align}
	\label{eq:hatA1bd}
	\big[\hatAdr(t)^2\big]_{*,t} \leq C.
\end{align}

Next, apply Lemma~\ref{lem:Zmgbd} with $ f(t,y) = \hk(t-s-1,x_-\mu-y) $ with the aid of~\eqref{eq:hk:sup} to get
\begin{align}
	\label{eq:hatA2bd1}
	\hatAmg(t,x)^2
	\leq
	\sum_{s=0}^{t-1} \frac{C\,\e^2}{ (\e^{2}(t-s))^{\frac12} } \sum_{y\in \Xi(s)} \hk(t-s-1,x_{-\mu}-y) \Vert \hat\Zsv(s,y) \Vert^2_{2n}.
\end{align}
We bound the sum over $ y\in\Xi(s) $ by using $ \sum_{y} |f_1(y)f_2(y)| \leq (\sup_{y} |f_1(y)|) \sum_{y} |f_2(y)| $
for two difference choices of $ (f_1,f_2) $.
For $ (f_1,f_2)=(\hk,\Vert\hat\Zsv\Vert^2_{2n}) $, we use~\eqref{eq:hk:sup} and~\eqref{eq:D*sum},
and for $ (f_1,f_2)=(\Vert\hat\Zsv\Vert^2_{2n},\hk) $, we use~\eqref{eq:D*sup} and $ \sum_{z} \hk(t-s-1,z)=1 $.
Taking the minimum of the results from two cases gives
\begin{align}
	\label{eq:hatA2bd2}
	\sum_{y\in \Xi(s)} \hk(t-s-1,x_{-\mu}-y) \Vert \hat\Zsv(s,y) \Vert^2_{2n}
	\leq
	C\,
	\Big( \frac{1}{(\e^{2}(t-s))(\e^2 s)^{\frac12}} \wedge \frac{1}{(\e^{2}(t-s))^{\frac12}(\e^2s)} \Big) \hat{D}(s).
\end{align}
Set $ \hat{E}(t) := \sup_{[0,t]\cap\Z} \hat{D}(s) $.
In~\eqref{eq:hatA2bd2}, bound $ \hat{D}(s) $ by $ \hat{E}(t) $, and bound the remaining integral by comparison to an integral.
Inserting the result in~\eqref{eq:hatA2bd2}, we obtain
\begin{align}
	\label{eq:hatA2bd3}
	\hatAmg(t,x)^2 \leq  C\, \hat{E}(t) (\e^{-2}t)^{-\frac12}.
\end{align}
On the other hand, sum~\eqref{eq:hatA2bd1} over $ x\in\Xi(t) $, using $ \sum_{z}\hk(t-s-1,z)=1 $ and~\eqref{eq:D*sum}.
We obtain
\begin{align}
	\label{eq:hatA2bd4}
	\sum_{x\in\Xi(t)} \hatAmg(t,x)^2
	\leq
	C\, \e^2 \sum_{s=0}^{t-1} (\e^{2}(t-s))^{-\frac12} (\e^2 s)^{-\frac12} \hat{E}(t)
	\leq
	C \hat{E}(t).
\end{align}
Combining~\eqref{eq:hatA2bd3}--\eqref{eq:hatA2bd4} gives
\begin{align}
	\label{eq:hatA2bd}
	\hatAmg(t,x)^2 \leq C \hat{E}(t) (\e^2t)^\frac12.
\end{align}

Inserting the bounds~\eqref{eq:hatA1bd} and~\eqref{eq:hatA2bd} into~\eqref{eq:stepmom2},
we arrive at $ \hat{E}(t) \leq C + C_*\, \hat{E}(t) (\e^{2}t)^\frac12 $, for some fixed constant $ C_*=C_*(n) $.
Fix $ \tau=\tau(n)>0 $ so that $ C_* \delta^\frac12 < \frac12 $,
we then have $ \hat{E}(t) \leq C $, for all $ t \leq \tau \e^{-2} $.
This conclude the desired moment bound~\eqref{eq:Zmomt:step} on $ \hat{Z}(t,x) $.

We now turn to showing~\eqref{eq:gZmomt:step}.
Multiply both sides of~\eqref{eq:stepmom1} by $ \den(1-\den)\e^{-\frac12} $ to get
\begin{align}
	\label{eq:stepmom3}
	\big\Vert \hat{Z}(t,x)-\hat{Z}(t,x) \big\Vert^{2}_{2n} \leq 2 \hatAgdr^2 + 2 \hatAgmg^2,
\end{align}
where $ \hatAgdr(t,x) := \sum_{y\in\Z} (\hk(t,x-y)-\hk(t,x'-y)) \hat{Z}(0,y) $,
and $ \hatAgmg(t,x) := \den(1-\den)\e^{-\frac12} \Vert (\Zmg(t,x)-\Zmg(t,x')) \Vert_{2n} $.
Using~\eqref{eq:hatZ0:sumbd}, in conjunction with \eqref{eq:hk:sum} and with \eqref{eq:hk:gd}, we bound
\begin{align*}
	|\hatAgdr| \leq C\,\e^{-1}|x-x'|^{2\alpha} (1+t)^{-\frac12-\alpha},
	\quad
	|\hatAgdr| \leq C\e^{-1}(1+t)^{-\frac12}.
\end{align*}
Multiplying the results gives $ \hatAgdr^2 \leq C |\e(x'-x)|^{2\alpha} (\e^2t)^{-1-\alpha} $.
As for $ \hatAgmg $, multiplying both sides of \eqref{eq:A'2} by $ (\den(1-\den)\e^{-\frac12})^2 $, here we have
\begin{align*}
	\hatAgmg^2
	\leq
	&|\e(x'-x)|^{2\alpha} C\, \e^2 \sum_{s=0}^{t-1} \sum_{y\in \Xi(s)} (\e^2(t-s))^{-\frac12-\alpha}
\\
	&\big( \hk(t-s-1,x_{-\mu}-y) + \hk(t-s-1,x'_{-\mu}-y) \big)\, \big\Vert\hat{Z}(s,y)\big\Vert^2_{2n}.
\end{align*}
Use~\eqref{eq:hatA2bd2} to bound the sum over $ y\in\Xi(s) $, noting that $ \hat{D}(s) \leq C $.
Then estimate the resulting sum over $ s\in[0,t-1] $ by comparison to an integral.
We obtain
$
	\hatAgmg^2
	\leq
	|\e(x'-x)|^{2\alpha} C\, (\e^2 t)^{-\frac12-\alpha}.
$
Inserting this and the preceding bounds on $ \hatAgdr^2 $ into~\eqref{eq:stepmom3} yields the desired result~\eqref{eq:gZmomt:step}.
\end{proof}

An immediate consequence of Proposition~\ref{prop:momt} is the tightness of $ \Zsv(\e^{-2}\Cdot,\e^{-1}\Cdot) $.
\begin{cor}\label{cor:tightness}
\begin{enumerate}[leftmargin=5ex, label=(\alph*)] %
\item[]
\item \label{cor:tight:nearStat} \textbf{(Near stationary initial conditions)}
Under the same assumptions in Proposition~\ref{prop:momt}\ref{prop:momt:nearStat},
The collection of processes $ \{\Zsv(\e^{-2}\Cdot,\e^{-1}\Cdot)\}_{\e>0} $ is tight in $ C([0,\infty),C(\R)) $.
\item \label{cor:tight:step} \textbf{(Step initial conditions)}
Under the same assumptions in Proposition~\ref{prop:momt}\ref{prop:momt:step},
The collection of processes $ \{\frac{\den(1-\den)}{\sqrt{\e}}\Zsv(\e^{-2}\Cdot,\e^{-1}\Cdot)\}_{\e>0} $ is tight in $ C((0,\infty),C(\R)) $.
\end{enumerate}
\end{cor}
\begin{proof}
Given Proposition~\ref{prop:momt}\ref{prop:momt:nearStat},
Part~\ref{cor:tight:nearStat} follows the Kolmogorov--Chentsov criterion (see, e.g., \cite[Theorem~1.4.1]{kunita97}).
We now turn to Part~\ref{cor:tight:step}.
The moment bounds from Proposition~\ref{prop:momt}\ref{prop:momt:step} asserts that,
the process $ \hat{Z} = \den(1-\den) \e^{-\frac12} Z $, when initiated from $ t=\e^{-2}\delta $, for small enough $ \delta $,
satisfies the near stationary properties~\eqref{eq:nearStat}.
This this being case, from Part~\ref{cor:tight:nearStat},
we infer that $ \{\hat\Zsv(\e^{-2}\Cdot,\e^{-1}\Cdot)\}_{\e>0} $ is tight in $ C(\delta,\infty,C(\R)) $.
Since this holds for all small enough $ \delta>0 $, we conclude the desired result.
\end{proof}

\subsection{Proof of Theorem~\ref{thm:S6V:}}
\label{sect:pfthmS6V}

\subsubsection{Part~\ref{thm:S6V:nearStat}: near stationary initial conditions}
Given the tightness result from Corollary~\ref{cor:tightness}, it remains to show that the limit points are the mild solution of \ac{SHE}.
We achieve this through martingale problems.
Recall from~\cite{Bertini1997} that,
we say a $ C([0,\infty), C(\R)) $-valued process $\limZ(t,x) $ solves \DefinText{the martingale problem} associated with the \ac{SHE}~\eqref{eq:SHE}
if, for any given $ T<\infty $, there exists $ C(T)<\infty  $ such that
\begin{align}
	\label{eq:MPmomt}
	\sup_{t\in[0,T]}\sup_{x\in\R}e^{-|x|C(T)}\Ex\left[\limZ^2(t,x)\right]<\infty,
\end{align}
and if, for any $ \phi\in C^\infty_c(\R) $, the processes $ \limM_\phi(t) $ and $ \limN_\phi(t) $, $ t\in\R_+ $,
\begin{align}
	\label{eq:linMP}
 	\limM_\phi(t)
 	&:=
 	\Big( \int_\R \phi(x) \limZ(s,x) dx \Big) \Big|^{s=t}_{s=0}
 	-
 	\frac{\nu_*}{2} \int^{t}_0\int_\R  \phi''(x) \limZ(s,x) ds dx,
\\
	\label{eq:quadMP}
	\limN_\phi(t) &:= \limM^2_\phi(t) - \frac{D_*\kappa_*^2}{\nu_*^2} \int^t_0 \int_\R \phi^2(x) \limZ^2(s,x) ds dx
\end{align}
are local martingales.
%We refer to~\eqref{eq:linMP}--\eqref{eq:quadMP}
%respectively as the \DefinText{linear-} and \DefinText{quadratic martingale problems.}
%
It is shown in \cite{Bertini1997} that any solution $ \limZ $ of the prescribed martingale problem
is a solution\footnote{In fact this is a weak solution. But solving~\eqref{eq:SHE} in the weak and mild senses are equivalent as shown in \cite[Proposition~4.11]{Bertini1997}.}
of the \ac{SHE}~\eqref{eq:SHE}. Moreover, they show that there is a unique such solution.

Hence, it suffices to show that any limit point of $ \Zsv(\e^{-2}\Cdot,\e^{-1}\Cdot) $ solves the martingale problem.
As mentioned earlier, the major technical step occurs in establishing~\eqref{eq:quadMP} (i.e., the quadratic martingale problem),
where we need self-averaging of the  quadratic variation.
We now state the desired estimate on such self-averaging.
To this end, recall the expressions $ \Theta_1, \Theta_2 $ from~\eqref{eq:Theta1}--\eqref{eq:Theta2},
which are associated with the quadratic variation of the martingale increment $ \mg $ in Proposition~\ref{prop:mSHE}.
\begin{prop}
\label{prop:qv}
Start the stochastic \ac{6V} model from near stationary initial conditions.
Given any fixed $ T<\infty $, we have that,
for all $ t \in [0,\e^{-2}T]\cap\Z $, $ x_\star\in\Z $, and all $ \e>0 $ small enough,
\begin{align}
	\label{eq:selfav}
	\Bigg\Vert \e^2 \sum_{s=0}^{t}
	\Big(\e^{-1}\Theta_1\Theta_2-\frac{2b_1\den(1-\den)}{1+b_1}\Zsv^2\Big)
		(s,x_\star - \mue s + \lfloor\mue s\rfloor) \Bigg\Vert_2
	\leq
	\e^{\frac14} C(T)e^{C\e |x_\star|}.
\end{align}
\end{prop}
\begin{rmk}
In~\eqref{eq:selfav}, we compensate the space variable $ x_\star\in\Z $ by $ \mue s-\mues \in [0,1) $
to ensure the resulting variables is in $ \Xi(s) $.
\end{rmk}
\begin{rmk}
\label{rem:self-av-mean}
Proposition~\ref{prop:qv} demonstrates a self-averaging upon integrating over long time interval, namely,
the quadratic variation of the martingale $\mg(t,x)$ subtracting the leading order term
(that is, a constant multiple of $Z^2$), vanishes as $\e \to 0$.
This is not obvious at all and is the linchpin of the analysis of the present paper.
The remainder of this subtraction is given in Lemma~\ref{lem:ThetaExpand},
which consists of terms of the form $(\e^{-\frac12} \nabla Z)(t,x_1) Z(t,x_2)$,
and $(\e^{-\frac12} \nabla Z)(t,x_1) (\e^{-\frac12} \nabla Z)(t,x_2)$ for $x_1<x_2$.
By the definition of $Z$, see \eqref{e:Gartner-eps}, $\nabla Z$ behaves as $\e^{\frac12} Z$,
so these remainder terms seem to be of the same order as the  leading order term. Self-averaging is key to showing that they are, in fact, of lower order.
The proof of Proposition~\ref{prop:qv} is
given in Section~\ref{sect:qv}, which
relies on duality argument in Section~\ref{sect:qv} and estimates of two-point transition kernels
given in Section~\ref{sect:SG}.
The heuristic on how duality and estimates of transition kernels lead to the proof of such a self-averaging is discussed in Appendix~\ref{sect:ASEP} with the simpler example of ASEP.
\end{rmk}

Postponing the proof of this proposition to Section~\ref{sect:qv},
we now finish the proof of Theorem~\ref{thm:S6V:}\ref{thm:S6V:nearStat}:
\begin{prop}\label{prop:martlimitprob}
Any limit point of $ \{\Zsv(\e^{-2}\Cdot,\e^{-1}\Cdot)\}_{\e>0} $ solves the martingale problems \eqref{eq:linMP}--\eqref{eq:quadMP}.
\end{prop}
\begin{proof}%[Sketch of Proof]
Fix a limit point $ \limZ $, and, after passing to a subsequence,
we assume $ \Zsv(\e^{-2}\Cdot,\e^{-1}\Cdot) $ converges in distribution to $ \limZ $.
The condition~\eqref{eq:MPmomt} is readily verified from the moment bounds in Proposition~\ref{prop:momt}.

We now turn to verifying the condition~\eqref{eq:linMP},
i.e., showing that $ \limM_\phi $ is a local martingale.
To this end, fixing a test function $ \phi\in C^\infty_c(\R) $,
we consider the discrete, microscopic analog of $ \limM_\phi $.
Recall from~\eqref{eq:hk} that $ \hk $ denote the one-step transition kernel.
Define the corresponding generator
\begin{align}
	\label{eq:gen}
	(\gen f)(x) := \sum_{y\in \Xi(t)} \big( \hk(x-y) - \ind_{\set{x+\mu=y}} \big) f(y),
	\quad
	x\in \Xi(t+1).
\end{align}
We now consider
\begin{align*}
	\mg_\phi(t)
	:=
	\e\sum_{x\in\Xi(s)} \phi(\e x) \Zsv(s,x) \Big|_{s=0}^{s=t}
	+
	\e\sum_{s=1}^{t}\sum_{x\in\Xi(s)} \phi(\e x)\big(\gen \Zsv(s-1) \big)(x),
\end{align*}
Recall the definition of $ \mg(t,x) $ from~\eqref{eq:mg}.
From Proposition~\ref{prop:mSHE}, we have
\begin{align*}
	\mg_\phi(t) = \e\sum_{s=0}^{t-1}\sum_{x\in\Xi(s)} \phi(\e x)\mg(s,x).
\end{align*}
%where $ \mg(s,x) $ is an $ \filt $-martingale increment.
Since $ \mg(s,x) $ is an $ \filt $ martingale increment (from Proposition~\ref{prop:mSHE}),
the process $ \mg_\phi(t) $, $ t\in\Z_{\geq 0} $, is a martingale.
Given the assume $ \Zsv(\e^{-2}\Cdot,\e^{-1}\Cdot) \Rightarrow \limZ(\Cdot,\Cdot) $,
with the aid of moment bounds from Proposition~\ref{prop:momt}\ref{prop:momt:nearStat},
it is standard (see for instance \cite[proof of Proposition~5.6]{corwin2016open}) to show that $ \mg_\phi(\e^{-2}\Cdot) \Rightarrow \limM_\phi(\Cdot) $, under the topology of uniform convergence over bounded intervals in $ [0,\infty) $.
This concludes that $ \limM_\phi(t) $ is a local martingale.
The factor $ \nu_* $ arises as the variance of the Brownian motion
which is the limit of the
random walk $R$ associated to the generator $\gen$.
More precisely, from~\eqref{eq:hk} and \eqref{eq:lambdae}--\eqref{eq:mue},
with $ b^\e_2=e^{-\sqrt\e}b_1 $, we  calculate
\begin{align}
	\text{Var}(\RW_\e)
	&=\mu_\e^2 \lambda_\e b_1 + \sum_{n\ge 1} (n-\mu_\e)^2 \lambda_\e (1-b_1)(1-b^\e_2)
	(b^\e_2)^{n-1} \tau_\e^{-n\den}  \notag \\
	&= \mu_\e^2 \lambda_\e b_1 +  \lambda_\e (1-b_1)(1-b^\e_2) \tau_\e^{-\den}
	\Big( \frac{\mu_\e^2}{1-b_2^\e \tau_\e^{-\den}}-\frac{2\mu_\e+1}{(1-b_2^\e \tau_\e^{-\den})^2}+\frac{2}{(1-b_2^\e \tau_\e^{-\den})^3}   \Big)  \notag \\
	 &\longrightarrow  \nu_* = \frac{2b_1}{1-b_1},
	\quad
	\text{as }
	\e \to 0.
		\label{e:VarR-nustar}
\end{align}
Here we used the fact that the sum over $n$ multiplied by a factor $(1-b_2^\e \tau_\e^{-\den})^2$ gives a quantity that can be summed as geometric series.

The proof of~\eqref{eq:quadMP} follows by a discrete-to-continuous scheme.
Specifically, the process
\begin{align*}
	\mg_\phi  (t) - \langle \mg_\phi \rangle (t),
	\quad
	t\in\Z_{\geq 0}
\end{align*}
is an $ \filt $-martingale, where $ \langle \mg_\phi \rangle (t) $ is the quadratic variation of $ \mg_\phi(t) $, given by
\begin{align*}
	\langle \mg_\phi \rangle (t)
	&:=
	\sum_{s=1}^{t} \Ex\big[ (\mg_\phi(s)-\mg_\phi(s-1))^2 \big| \filt(t) \big].
\end{align*}
%Given the standard argument from \cite[Proof of Proposition 2.13]{CT15}
%of approximating continuous-time processes with the discrete-time ones,
The major step here is to argue that $ \langle \mg_\phi \rangle (t) $
is well-approximated by a discrete analog of\\
$ \frac{D_*\kappa_*^2}{\nu_*^2} \int_0^t \int_\R (\mathcal{Z}^2\phi^2)(s,x)dsdx $. % \hao{$\nu_*^2$ in denominator?}
%\begin{align*}
%%	\label{eq:qv:goal}
%	\frac{D_*\kappa_*^2}{\nu_*} \e^{2}\sum_{s=0}^{t-1} \e \sum_{x\in\Xi(s)}\phi^2(\e x) \Zsv^2(s,x).
%\end{align*}
To this end, using~\eqref{eq:Quadvar}, we calculate $ \langle \mg_\phi \rangle (t) $ as
\begin{align*}
	\langle \mg_\phi \rangle (t)
	=
	\e^2\sum_{s=0}^{t-1} \Big(  \sum_{x,x'\in\Xi(s)} \phi(\e x)\phi(\e x') (b_1 e^{-\sqrt\e (1-\den)})^{|x-x'|} \Theta_1(t,x\wedge x')\Theta_2(t,x\wedge x') \Big).
\end{align*}
With $ b_1<1 $, the factor $ (b_1e^{-\sqrt\e(1-\den)})^{|x-x'|} $ introduces an exponential decay in $ |x-x'| $.
Since $ \phi\in C^\infty_c(\R) $,
one can bound  $|\phi(\e x) - \phi(\e x')| $ by a constant times $\e |x-x'|$,
so one can
  show that the previous expression is well-approximated by
the corresponding expression where $ \phi(\e x)\phi(\e x') $ is replaced by $ \phi^2(\e (x\wedge x')) $.
More precisely,
letting $ \mathcal{E}_\e(t) $ denote a generic process such that
\begin{align}
	\label{eq:ptwise}
	\lim_{ \e \to 0}
	\sup_{t\in\Z\cap[0,\e^{-2}T]} \Vert\mathcal{E}_\e(t)\Vert_2 =0,
	\quad
	\text{for any given } T<\infty,
\end{align}
the continuity of $ \phi $ gives that
\begin{align*}
	\langle \mg_\phi \rangle (t)
	=
	&\e^2\sum_{s=0}^{t-1} \Big(  \sum_{x,x'\in\Xi(s)} \phi^2(\e (x\wedge x')) (b_1e^{-\sqrt\e(1-\den)})^{|x-x'|} \Theta_1(t,x\wedge x')\Theta_2(t,x\wedge x') \Big)
	+
	\mathcal{E}_\e(t).
\end{align*}
With $ \sum_{y\in\Z} (b_1e^{-\sqrt\e(1-\den)})^{|y|} = \frac{1+b_1e^{-\sqrt\e(1-\den)}}{1-b_1e^{-\sqrt\e(1-\den)}} \to \frac{1+b_1}{1-b_1} $,
we now have
\begin{align}
	\label{eq:eq:qv:goal:}
	\langle \mg_\phi \rangle (t)
	-
	\frac{1+b_1}{1-b_1} \e^2 \sum_{s=0}^{t-1} \e \sum_{x\in\Xi(s)} \e^{-1}\Theta_1(t,x)\Theta_2(t,x) \phi^2(\e x)
	=
	\mathcal{E}_\e(t).
\end{align}
%\begin{align*}
%	\alpha_\e := \sum_{y\in\Z} (b_1e^{-\sqrt\e\den})^{|y|} = \frac{1+b_1e^{-\sqrt\e\den}}{1-b_1e^{-\sqrt\e\den}}.
%\end{align*}
%Now, with~$ \alpha_\e\to \frac{1+b_1}{1-b_1} $ and with $ \frac{1+b_1}{1-b_1} \frac{2b_1\den(1-\den)}{1+b_1} = D_*\kappa_*^2/\nu^2_* $,
Further, fixing some large enough $ L<\infty $ with $ \text{supp}(\phi) \subset [-L,L] $,
we have
\begin{align*}
%	\label{eq:eq:qv:goal:::}
	&\Big\Vert
		\e^2\sum_{s=0}^{t}
		\e\sum_{x\in\Xi(s)} \big(\e^{-1}\Theta_1\Theta_2-\tfrac{2b_1\den(1-\den)}{1+b_1}\Zsv^2\big)(s,x) \phi^2(\e x)
	\Big\Vert_2
\\
	=&	
	\Big\Vert
		\e\sum_{x_\star\in\Z}
		\e^2\sum_{s=0}^{t}
		\big(\e^{-1}\Theta_1\Theta_2-\tfrac{2b_1\den(1-\den)}{1+b_1}\Zsv^2\big)(s,x_\star+\mue s- \lfloor\mue s\rfloor) \phi^2(\e (x_\star+\mue s- \lfloor\mue s\rfloor))
	\Big\Vert_2	
\\
	\leq&	
	C(L,\phi)
	\sup_{x_\star\in[-\e L,\e L]\cap\Z}
	\Big\Vert
		\e^2\sum_{s=0}^{t}
		\big(\e^{-1}\Theta_1\Theta_2-\tfrac{2b_1\den(1-\den)}{1+b_1}\Zsv^2\big)(s,x_\star+\mue s- \lfloor\mue s\rfloor)
	\Big\Vert_2.
\end{align*}
The last expression, by Proposition~\ref{prop:qv},
is bounded by $ C(T,L,\phi)\e^{\frac14} $, for all $ t\in\Z\cap[0,\e^{-2}T] $,
for each fixed time horizon $ T<\infty $.
Consequently,
\begin{align*}
	\e^2\sum_{s=0}^{t}
	\e\sum_{x\in\Xi(s)} \big(\e^{-1}\Theta_1\Theta_2-\tfrac{2b_1\den(1-\den)}{1+b_1}\Zsv^2\big)(s,x) \phi^2(\e x)
	=
	\mathcal{E}_\e(t).
\end{align*}
Inserting this into~\eqref{eq:eq:qv:goal:},
together with
\[
 \frac{2b_1\den(1-\den)}{1+b_1}\frac{1+b_1}{1-b_1}=\frac{D_*\kappa_*^2}{\nu_*^2} ,
 \]
we now arrive at
\begin{align}
	\label{eq:eq:qv:goal::}
	\langle \mg_\phi \rangle (t)
	-
	\frac{D_*\kappa_*^2}{\nu_*^2} \e^{2}\sum_{s=0}^{t-1} \e \sum_{x\in\Xi(s)}\phi^2(\e x) \Zsv^2(s,x)
	=
	\mathcal{E}_\e(t).
\end{align}
So far, we have only shown that the expression~\eqref{eq:eq:qv:goal::} converges to zero (in $ L^2 $) \emph{pointwise in $ t $},
(i.e., \eqref{eq:ptwise}).
Given the moment bounds from Proposition~\ref{prop:momt},
a standard argument  (see for instance \cite[Section~4]{Bertini1997})  leverages such pointwise convergence to convergence at process level, yielding
\begin{align*}
	\sup_{t\in\Z\cap[0,\e^{-2}T]}
	\Big| \langle \mg_\phi \rangle (t) - \frac{D_*\kappa_*^2}{\nu^2_*} \e^{2} \sum_{s=0}^{t-1} \e \sum_{x\in\Xi(s)} \Zsv^2(s,x) \phi^2(\e x) \Big|
	\longrightarrow_\text{P}
	0.
\end{align*}
Given this, the rest of the proof is standard. We omit the details.
\end{proof}

\subsubsection{Part~\ref{thm:S6V:step}: step initial condition}
Consider $ \hat{\Zsv}(t,x) := \frac{\den(1-\den)}{\sqrt\e} \Zsv(t,x) $
under the step initial condition $ N(0,x) = (x)_+ $.
From~\eqref{eq:ColeHopfTransform},
\begin{align*}
	\hat{\Zsv}(0,x) = \left\{\begin{array}{l@{,}l}
		\frac{\den(1-\den)}{\sqrt\e} e^{-\sqrt\e(1-\den)x}	&\text{ for } x \geq 0,
		\\
		\frac{\den(1-\den)}{\sqrt\e} e^{-\sqrt\e\den x}	&\text{ for } x < 0.	
	\end{array}\right.
\end{align*}
In particular $\e \sum_{x\in\Z} \hat{Z}(0,x) = 	\frac{\den(1-\den)}{\sqrt\e} ( \frac{1}{1-e^{-\sqrt\e(1-\den)}} + \frac{e^{-\sqrt\e\den}}{1-e^{-\sqrt\e\den}} ) \to 1 $.
This together with the exponential decay (in $ |x| $) of $ \hat{\Zsv}(0,x) $
shows that $ \hat{\Zsv}(0,\e^{-1}x) $ converges to $ \delta(x) $.
Given this and the convergence result for near stationary initial conditions (i.e., part~\ref{thm:S6V:nearStat}),
Part~\ref{thm:S6V:step} follows by a procedure of two-step convergence:
first working on $ t\in [\e^{-2}\delta,\infty) $ and sending $ \e\to 0 $ with $ \delta>0 $, and then sending $ \delta\to 0 $.
This procedure is now standard, and is carried out in \cite[Section~3]{Amir11} for the ASEP, so we do not  repeat the argument here.

\subsection{Proof of Theorem~\ref{thm:6V}}
\label{sect:pfthm6V}
Recall that Proposition~\ref{prop:constrK} asserts an extension of the stationary solution of the \ac{SBE} to all values of $ t>-\infty $.
We being by giving this construction.
\begin{proof}[Proof of Proposition~\ref{prop:constrK}]
The construction of $ \mathcal{K} $ follows standard, Kolmogorov-type argument.
To begin with, given~\eqref{eq:stat:SBE}, we have that
\begin{align*}
%	\label{eq:Kstat}
	\big( \mathcal{H}_\text{stat}(t,\Cdot) - \mathcal{H}_\text{stat}(t,0) \big)_{t\geq 0} =: \big( \tilde{\mathcal{K}}(t,\Cdot) \big)_{t\geq 0}
\end{align*}
is a stationary (in $ t $) process.
Consider the space $ \mathcal{X} := \prod_{\R} C(\R) $,
endowed with the product $ \sigma $-algebra and with the product topology.
For each $ t_1<\ldots<t_n\in \R $, we define a probability distribution $ \Pr_{t_1,\ldots,t_n} $ on
$ \prod_{\set{t_1,\ldots,t_n}} C(\R) $ given by that of
\begin{align*}
	\Big( \tilde{\mathcal{K}}(0,\Cdot),
	\tilde{\mathcal{K}}(t_2-t_1,\Cdot),
	\ldots,
	\tilde{\mathcal{K}}(t_n-t_1,\Cdot)
	\Big).	
\end{align*}
Thanks to the stationarity of $ \mathcal{K}(t,\Cdot) $, the laws $ \Pr_{t_1,\ldots,t_n} $ are consistent among $ \{t_1<\ldots<t_n\} \in \R $.
Thus, the Kolmogorov extension theorem gives an $ \mathcal{X} $-valued process $ \hat{\mathcal{K}}(t,x) $,
such that, for any $ t_0\in\R $,
\begin{align}
	\label{eq:Kfdd}
	\hat{\mathcal{K}}(t-t_0,\Cdot) = \mathcal{H}_\text{stat}(t,\Cdot) - \mathcal{H}_\text{stat}(t,0),
	\quad
	\text{in finite dimensional (in }t\text{) distributions.}
\end{align}

The next step is to further construct a \emph{continuous version} of $ \hat{\mathcal{K}} $.
That is, a $ C(\R,C(\R)) $-valued process that shares the same finite dimensional (in $ t $) distributions as $ \hat{\mathcal{K}}(t,x) $.
To this end, for each $ n\in\Z_{>0} $, we construct a $ C(\R,C(\R)) $-valued process $ \mathcal{K}_n $ by setting
$ \mathcal{K}_n(\frac{i}{2^n},x) := \hat{\mathcal{K}}(\frac{i}{2^n},x) $, for $ i\in\Z $,
and linearly interpolate in $ t $.
For such dyadic approximations,
given any fixed $ [t_1,t_2]\times[x_1,x_2]:= D\subset \R^2 $, we have that
\begin{align*}
	\sup_{(t,x)\in D}
	\big| &\mathcal{K}_n(t,x)-\mathcal{K}_{n+m}(t,x) \big|
\\
	&\leq
	\sup
	\Big\{
		\big| \hat{\mathcal{K}}(t,x)-\hat{\mathcal{K}}(s,x) \big|
		:
		s,t\in[t_1,t_2]\cap 2^{-(m+n)}\Z,
		\
		|t-s| \leq 2^{-n},
		\
		x\in[x_1,x_2]
	\Big\}.
\end{align*}
As $ \mathcal{H}_\text{stat} $ is continuous,
with~\eqref{eq:Kfdd}, we see that the r.h.s.\ converges to zero in distribution (and hence converges to zero in probability)
as $ (n,m)\to(\infty,\infty) $.
This being the case, using the first Borel--Cantelli lemma,
it is standard to construct  a subsequence of $ \{\mathcal{K}_n\}_n $ that is almost surely Cauchy in $ C(\R,C(\R)) $.
The resulting limiting process $ \mathcal{K}\in C(\R,C(\R)) $ gives the desired continuous version of $ \hat{\mathcal{K}} $.
With $ \mathcal{K} $ and $ \mathcal{H}_\text{stat} $ both being continuous,
the desired property~\eqref{eq:KeqH} follows from \eqref{eq:Kfdd}.
\end{proof}

We now prove Theorem~\ref{thm:6V}.
\begin{proof}[Proof of Theorem~\ref{thm:6V}]
Recall the definition of $ \Vert \Cdot \Vert_{C^{-1}(\R^2),[-\ell,\ell]^2} $ from~\eqref{eq:C-1:}.
Referring to~\eqref{eq:C-1},
we see that $ U_\e \to U $ in $ C^{-1}(\R^2) $, if and only if, for any fixed $ \ell \in \Z_{>0} $,
$	
	\Vert U_\e-U \Vert_{C^{-1}(\R^2),[-\ell,\ell]^2} \to 0.
$
With this in mind, we henceforward fix $ \ell \in \Z_{>0} $.
Further, even though the relevant test functions in~\eqref{eq:C-1:} have support in $ [-\ell,\ell]^2 $,
with both $ \limU $ and stochastic Gibbs state being translation invariant in $ y $,
after a suitable translation, we assume without lost of generality that relevant test functions are supported in
$ (x,y) \in [-\ell,\ell]\times [0,3\ell] $.

The next step is to translate the statements regarding the symmetric 6V model into the context of the \emph{stochastic}  6V model.
Recall that, for a given $ (a,b,c) $-symmetric \ac{6V} model with $ \Delta>1 $,
defining $ b_1,b_2\in(0,1) $ by %inverting
the relation~\eqref{eq.bs},
the stochastic Gibbs state $ \Gibbs(b_1,b_2;h,v) $ for the $ (a,b,c) $-symmetric model is equivalent to
the $ (b_1,b_2) $-stochastic model in its stationary measure.
Here $ (h,v)\in(0,1)^2 $ is an one parameter family of parameters satisfying~\eqref{eq:slope},
and the corresponding stationary measure for the $ (b_1,b_2) $-stochastic model
is the product Bernoulli measure $ \bigotimes_{x\in\Z} $Ber$ (\den) $ with $ \den:=v $.
While for the symmetric model we have used coordinates $(x,y)$ for the $x$ and $y$ axes, for the stochastic model we will use $(x,t)$ with $y$ replaced by $t$ to represent the temporal axis. We will also tend to write these coordinates as $(t,x)$ with time first and then space.
%Under this mapping between the symmetric and stochastic models,
%the $ (x,y) $ coordinate translates into $ (t,x) $ ($ y\mapsto t $ and $ x $ stays).
The purpose of the shifting in $ y $ described above is to ensure that $ t\geq 0 $ for the stochastic model.

Recall from~\eqref{eq:vl:ind}--\eqref{eq:empirical}
that $ u(x,y) $ denote the indicator of an incoming vertical line,
and that $ U_\e $ is the corresponding empirical measure.
Under this mapping between the symmetric and stochastic models,
the former becomes the occupation variable
\begin{align*}
	u(x,y) = \ind_{\set{\text{having a particle at } (t=y,x)}} = \eta(y,x).
\end{align*}
Fix $ f \in C^\infty(\R^2) $ with support $ (x,y)\in[-\ell,\ell]\times[0,3\ell] $.
With $ \eta(y,x) := N(y,x) - N(y,x-1) $, we have
\begin{align}
	\label{eq:Thm6V:}
	\langle U_\e, f \rangle
	&=
	\e^{\frac52} \sum_{x,y\in\Z} \big(N(y,x) - N(y,x-1)- \den \big) f(\e^{-1}x-\mue \e^{-2}y,\e^{-2}y).
\end{align}
In order to apply Theorem~\ref{thm:S6V},
note that by direct calculation using \eqref{eq.bs} and \eqref{eq:projPara}
\begin{align}\label{e:b1b2-bax}
b_1= \frac{e^{\eta} \sinh(u)}{\sinh(u+\eta)},
\qquad
b_2= \frac{e^{-\eta} \sinh(u)}{\sinh(u+\eta)}
\end{align}
so $b_2/b_1=e^{-2\eta} = e^{-\sqrt\e}$.
One can also choose $u=u_\e=\frac12\zeta\sqrt\e+o(\sqrt\e)$ so that
$b_1\in(0,1)$ is independent of $\e$.
\footnote{One can also consider more explicit choice of parameter $(u_\e,\eta_\e)=(\frac12\zeta\sqrt\e,\frac12\sqrt\e)$ (without the lower order part in $u_\e$). This would lead to parameter $b_1$ which also depends on $\e$, though the relation $b_2/b_1 =e^{-\sqrt\e}$ still holds. Our proof should still go through with extra notational complexity; we do not pursue this direction here.}
We are thus in the scope of Theorem~\ref{thm:S6V}, from which we know that the centered scaled height function
\begin{align*}
	\tilde{N}(t,x) := \sqrt\e \Big( N_\e\big(\e^{-2}t,\e^{-1}x+\mue \e^{-2} t\big) - \den(\e^{-1}x+\mue \e^{-2} t) - \e^{-2}t\log\lambdae \Big).
\end{align*}
converges to solution of \ac{KPZ} equation with coefficients $\nu_* ,\kappa_*,D_* $ given by
\eqref{eq:ceffints}. With \eqref{e:b1b2-bax} and our choice of $(u_\e,\eta_\e)$
with matching $\rho=v$,
these coefficients are written as \eqref{e:SBEcoeff} in terms of $\zeta$.

In~\eqref{eq:Thm6V:}, we can substitute in $ \tilde{N} $ for $ N $,
switch from the $ (x,y) $-coordinates to $ (t,x) $, and apply summation by parts in $ x $.
This gives
\begin{align*}
	\langle U_\e, f \rangle
	&=
	\e^{2} \sum_{t\in \e^{2}\Z_{\geq 0}}
	\Big( \e \sum_{x\in \e\Xi(t)} \e^{-1}\big( \tilde{N}(t,x) -\tilde{N}(t,x-\e) \big) f(x,t) \Big)
\\
	&=
	-\e^{2} \sum_{t\in \e^{2}\Z_{\geq 0}}
	\Big( \e \sum_{x\in \e\Xi(t)} \tilde{N}(t,x) \big( \e (f(x+\e,t)-f(x,t)) \big) \Big).
\end{align*}
The last expression is indeed similar to $ \langle \limU, f \rangle $ defined in~\eqref{eq:SBE:},
with integrations replaced by sums, and derivative on $ f $ replaced by difference.
Recall that $ N(t,x) $ is linearly interpolated onto $ (t,x)\in\R_+\times\R $ to give a $ C(\R_+,C(\R)) $-valued process.
This being the case, we further write
\begin{align}
	\label{eq:Uef}
	\langle U_\e, f \rangle
	=
	-\int_0^\infty \int_\R \tilde{N}(t,x) \partial_x f(x,t) dxdt
	+
	A_\e(t,x),
\end{align}
where $ A_\e(t,x) $ denotes a residue term with
$ |A_\e(t,x)| \leq  \sqrt\e C(\ell) (\Vert f \Vert_\infty + \Vert \partial_x f \Vert_\infty) $.

Recall that, here, the stochastic model starts from Bernoulli initial condition
\begin{align*}
	(\eta(0,x))_x \sim \bigotimes_{x\in\Z} \text{Ber}(\den),
	\quad
	N(0,x) := \sum_{y\in(0,x]} (\eta(0,y)-\den).
\end{align*}
It is standard to check that such an initial condition indeed satisfies the conditions in Definition~\ref{def:nearStat}.
Further, as $ \e\to 0 $, we have $ \tilde{N}(0,\Cdot) \Rightarrow \sqrt{\den(1-\den)}B(\Cdot) $ in $ C(\R) $,
where $ B $ denotes a standard Brownian motion.
Given these properties, Theorem~\ref{thm:S6V} asserts that
\begin{align*}
	\tilde{N}(\Cdot,\Cdot) \Longrightarrow \limH_\text{stat}(\Cdot,\Cdot),
	\quad
	\text{in } C(\R_+,C(\R)).
\end{align*}
By Skorokhod's representation theorem,
we further assume that this convergence holds in probability under a suitable
coupling of $ \tilde{N} $ and $ \limH_\text{stat} $,
whereby
\begin{align}
	\sup_{t\in[0,3\ell]} \sup_{x\in[-\ell,\ell]} \big | \tilde{N}(t,x)-\limH_\text{stat}(t,x) |
	\longrightarrow_\text{P}
	0.
\end{align}

Recall that $ f_\delta(x,y) := f(\delta^{-1}x,y) $.
Now, under the aforementioned coupling,
take the difference of~\eqref{eq:SBE:} and~\eqref{eq:Uef},
and replace $ f $ with $ f_\delta $.
This gives
\begin{align*}
	\big| \langle U_\e-\limU, f_\delta \rangle \big|
	&\leq
	\Vert \partial_x f_\delta \Vert_\infty \, \sup_{t\in[0,3\ell]} \sup_{x\in[-\ell,\ell]} \big | \tilde{N}(t,x)-\limH_\text{stat}(t,x) |
	+
	C(\ell)
	\e \big( \Vert \partial_x f_\delta \Vert_\infty + \Vert f_\delta \Vert_\infty \big)
\\
	&=
	\delta^{-1} \Vert \partial_x f \Vert_\infty \, \sup_{t\in[0,3\ell]} \sup_{x\in[-\ell,\ell]} \big | \tilde{N}(t,x)-\limH_\text{stat}(t,x) |
	+
	C(\ell)
	\e \big( \delta^{-1} \Vert \partial_x f \Vert_\infty + \Vert f \Vert_\infty \big).
\end{align*}
As this holds true for all $ f\in C^\infty(\R^2) $ with supp$ (f)\subset [\ell,\ell]\times[0,3\ell] $,
referring to~\eqref{eq:C-1:}, we see that
\begin{align*}
	\Vert  U_\e-\limU \Vert_{C^{-1}(\R^2),[-\ell,\ell]^2}
	\leq
	\sup_{t\in[0,3\ell]} \sup_{x\in[-\ell,\ell]} \big | \tilde{N}(t,x)-\limH_\text{stat}(t,x) |
	+
	C(\ell)
	\e.
\end{align*}
Taking $ \e\to 0 $, we thus conclude $ \Vert U_\e-\limU \Vert_{C^{-1},[-\ell,\ell]^2} \to_\text{P} 0 $.
This being true for arbitrary $ \ell\in\Z_{>0} $, we conclude the desired result:
$ \metric(U_\e,\limU) \to_\text{P} 0 $.
\end{proof}

\section{Estimating the two-point semigroup}
\label{sect:SG}
Recall from~\eqref{eq:SG} that $ \SGe $ denotes
the semigroup for the two-point functions of $ \Zsv $, where we put $ \e $ in the notation of $ \SGe $ to emphasize the dependence.
In order to complete the proof of Theorem~\ref{thm:S6V:}, it remains to prove Proposition~\ref{prop:qv}.
The proof will be carried out in Section~\ref{sect:qv} with the aid of duality.
Key to this proof is certain estimates on $ \SGe $ and its gradients, which are the subjects of this section.

Recall that $ \nabla f(x) := f(x+1)-f(x) $ denotes discrete gradient.
In the sequel we use notation such as $ \nabla_{x} $ to highlight the variable on which the gradient acts.
Recall that $ \SGe\big((y_1,y_2),(x_1,x_2);t\big) $ is related to the stochastic \ac{6V} model
only within the Weyl chamber: $ x_1<x_2 $ and $ y_1<y_2 $.
Thus, for expressions such as
\begin{align*}
	\nabla_{x_1}\SGe\big((y_1,y_2),(x_1,x_2);t\big) = \SGe\big((y_1,y_2),(x_1+1,x_2);t\big) - \SGe\big((y_1,y_2),(x_1,x_2);t\big)
\end{align*}
to be relevant, we must impose an additional constraint $ x_1+1<x_2 $.
In this case we say $ (x_1,x_2,y_1,y_2) $ is in the \DefinText{$ \nabla $-Weyl chamber},
which is understood with respect to whichever gradient is taken.

\medskip
The goal of this section is to establish:
\begin{prop}\label{prop:SG}
For any $ \alpha,T\in(0,\infty) $, there exist constants $C(\alpha,T), C(\alpha)>0$ such that
\begin{align*}
	&\big|\SGe\big((y_1,y_2),(x_1,x_2);t\big)\big|
	\leq
	\frac{C(\alpha,T)}{t+1} e^{ \frac{-\alpha(|x_1-y_1|+|x_2-y_2|)}{\sqrt{t+1}+C(\alpha)} },
\\
	&\big|\nabla_{x_j}\SGe\big((y_1,y_2),(x_1,x_2);t\big)\big|,
	\
	\big|\nabla_{y_j}\SGe\big((y_1,y_2),(x_1,x_2);t\big)\big|	
	\leq
	\frac{C(\alpha,T)}{(t+1)^{3/2}} e^{ -\frac{\alpha(|x_1-y_1|+|x_2-y_2|)}{\sqrt{t+1}+C(\alpha)} },
\end{align*}
for all $ x_1<x_2\in\Xi(t+s) $, $ y_1<y_2\in\Xi(s) $,
$ s,t\in [0,\e^{-2}T]\cap\Z $, $ j=1,2 $,
and $ (x_1,x_2,y_1,y_2) $ in their respective Weyl or $ \nabla $-Weyl chamber.
\end{prop}

In proving Proposition~\ref{prop:SG}, it is convenient to consider
`small $ t $' and `large enough $ t $' separately.
More precisely, in the following we use the phrase \DefinText{for large enough $ t $} if the referred statement holds for all $ t\geq t_0 $,
for some generic threshold $ t_0<\infty $ that may change from line to line, but depends only on $ \alpha $ and $ T $.
This is \emph{not to be confused with} the global assumption $ t \leq \e^{-2}T $.

The case with $ t\leq t_0 $ is simple. Let us first settle it.
\begin{proof}[Proof of Proposition~\ref{prop:SG}, the case with $ t\leq t_0=t_0(\alpha,T) $]
Fix an arbitrary $ t_0<\infty $, and assume $ t\leq t_0 $ throughout the proof.
Since $ (t+1) $ is bounded away from zero and infinity, it suffices to show
\begin{align}
	\label{eq:SGbd:shrt}
	\big|\SGe((y_1,y_2),(x_1,x_2);t)\big|
	\leq
	C(t_0) e^{ \frac{-1}{C(t_0)}(|x_1-y_1|+|x_2-y_2|) }.
\end{align}
From this the desired estimates on $ |\SGe| $ and $ |\nabla\SGe| $ \emph{both} follow.

Instead of directly proving this bound for $ \SGe $, let us first consider $ \TrPrc $ and prove that
\begin{align}
	\label{eq:SGbd:shrttimeone}
	\big|\TrPrc((y_1,y_2),(x_1,x_2);t)\big|
	\leq
	C(t_0) e^{ \frac{-1}{C(t_0)}(|x_1-y_1|+|x_2-y_2|) }.
\end{align}
Recall from Proposition \ref{prop:TransitionProbability} that $ \TrPrc\big((y_1,y_2),(x_1,x_2);t\big) = \TrP\big((y_1,y_2)\to(x_1,x_2);t\big) $ denotes the transition probability of stochastic \ac{6V} particle system with two particles.
Here we will appeal to the probabilistic interpretation of $ \TrPrc = \TrP $, and not rely upon contour integral formulas.
Let $ (x_1(t)<x_2(t)) \in \Z^2 $ denote the time $ t $ locations of the particles, starting from $ x_i(0)=y_i $.
To show \eqref{eq:SGbd:shrttimeone}, it suffices to show such a statement with $t=1$.
%We claim that it suffices to show
To see this, observe that $\TrPrc\big((y_1,y_2),(x_1,x_2);t\big)$ can be written as a $t$-fold convolution of one-step transition probabilities. The convolution can be expanded into a sum over all trajectories $(x_1(\Cdot),x_2(\Cdot))$ with $x_i(0)=y_i$ and $x_i(t)=x_i$. The contribution to each trajectory can be bounded by $t$ products of the one-step bound, leading to the contribution $C^t  e^{ \frac{-1}{C}(|x_1-y_1|+|x_2-y_2|)}$ for some $C>0$. (Note that the exponential terms came from telescoping.) The total number of trajectories to sum over is upper-bounded by $\binom{|x_1-y_1| + t}{t}\binom{|x_2-y_2| + t}{t}$ which, for $t<t_0$, is bounded by $C(t_0) |x_1-y_1|^t |x_2-y_2|^t$. Combining these two bounds and using that $t<t_0$, we arrive at \eqref{eq:SGbd:shrttimeone}. The $t=1$ version of \eqref{eq:SGbd:shrttimeone} is easy shown directly from the definition of the dynamics of the stochastic \ac{6V} model.
%The particle $ x_1 $ evolves independently of $ x_2(t) $:
%in each time step $ x_1(t) $ attempts the first jump to the right with success probability $ 1-b_1 $,
%and subsequently geometric jumps each with probability $ b_2^\e $.
%With $ b_2^\e \to b_1\in (0,1) $ as $ \e\to 0 $, we have
%\begin{align*}
%	\Pr(x_1(t)-y_1=x_1) \leq C(t_0) e^{-\frac{1}{C(t_0)}|x_1-y_1|}.
%\end{align*}
%As for $ x_2(t) $, conditioned on $ x_1(t) $,
%we have that $ x_2(t)-y_2 $ is stochastically dominated by
%\begin{align*}
%	t + Y(1)+\ldots+Y(t),
%	\quad
%	(Y(s))_{s=1}^t \sim \bigotimes_{s=1}^t \text{geo}(b^\e_2).
%\end{align*}
%This is so because, during each update, $ x_2(t) $ either receives $ 1 $ unit push from $ x_1(t) $
%or makes it own attempt with success probability $ 1-b_1 $, which is subsequently followed geometric jumps.
%Consequently,
%\begin{align*}
%	\Pr\big(x_2(t)-y_2=x_2| x_1(t)\big) \leq C(t_0) e^{-\frac{1}{C(t_0)}|x_2-y_2|}.
%\end{align*}
%The preceding discussion gives
%\begin{align*}
%	|\TrPrc((y_1,y_2),(x_1,x_2);t)|
%	\leq
%	C(t_0) e^{ \frac{-1}{C(t_0)}(|x_1-y_1|+|x_2-y_2|) }.
%\end{align*}
Finally, recall that $ \SGe $ is related to $ \TrPrc $ through~\eqref{eq:SG:}.
Given that $ \lambdae \to 1 $, $ \mue\to 1 $, $ \taue \to 1 $, and $ t\leq t_0 $,
the preceding bound on $ |\TrPrc| $ immediately yields the desired result~\eqref{eq:SGbd:shrt}.
\end{proof}

Having settled Proposition~\ref{prop:SG} for short time,
we now turn to the case for large enough $ t $.
For this we appeal to the contour integral representation, and analyze the integrals therein.
To begin with, referring back the expression~\eqref{eq:SG},
we decompose $ \SGe = \SGFr - \SGIn $ into the difference of a `free part' and an `interacting part', where
%$ \SGFr=\SGFr((y_1,y_2),(x_1,x_2);t) $ and $ \SGIn=\SGIn((y_1,y_2),(x_1,x_2);t) $ are
\begin{align}
	\label{eq:SGFr}
	\SGFr\big((y_1,y_2),(x_1,x_2);t\big)
	&:=
	\prod_{i=1}^2\oint_{\Circ_r}
	 z_i^{x_i-y_i+(\mue t-\muet)}\frac{\SGte(t,z_i)dz_i}{2\pi\img z_i},
\\
	\label{eq:SGIn}
	\SGIn\big((y_1,y_2),(x_1,x_2);t\big)
	&:=
	\oint_{\Circ_r}\oint_{\Circ_r}
	z_1^{x_2-y_1+(\mue t-\muet)}z_2^{x_1-y_2+(\mue t-\muet)} \SGfre(z_1,z_2)
	\prod_{i=1}^2
	\frac{\SGte(t,z_i)dz_i}{2\pi\img z_i}.
\end{align}
Here $ \SGfre $ and $ \SGte $ are given by~\eqref{eq:SGfr} and \eqref{eq:SGt} under the weak asymmetry scaling.
Recall from~\eqref{eq:hk}--\eqref{eq:hkt} that $ \hk(t,x) $ denotes the one-particle transition kernel.
Comparing~\eqref{eq:SGFr} with~\eqref{eq:hk:contour}, we see that $ \SGFr $ is exactly the product of one-particle transition kernels, i.e.,
\begin{align}
	\label{eq:SGFr=hkhk}
	\SGFr\big((y_1,y_2),(x_1,x_2);t\big)
	=
	\hk(t,x_1-y_1) \hk(t,x_2-y_2).
\end{align}
Given this decomposition,
we breakdown the proof of Proposition~\ref{prop:SG} into proving:

\begin{prop}\label{prop:SG:}
For any  $ \alpha,T\in(0,\infty) $ and $ t_0=t_0(\alpha,T) $, there exist $C(\alpha,T), C(\alpha)>0$ such that
\begin{enumerate}[label=(\alph*)]%
\item\label{enu:SGFr}
$	
	\displaystyle
	\big|\SGFr\big((y_1,y_2),(x_1,x_2);t\big)\big|
	\leq
	\frac{C(\alpha,T)}{t+1} e^{ \frac{-\alpha(|x_1-y_1|+|x_2-y_2|)}{\sqrt{t+1}+C(\alpha)} },
$
\item\label{enu:SGFrg}
$	
	\displaystyle
	\big|\nabla_{x_j}\SGFr\big((y_1,y_2),(x_1,x_2);t\big)\big|,
	\
	\big|\nabla_{y_j}\SGFr\big((y_1,y_2),(x_1,x_2);t\big)\big|	
	\leq
	\frac{C(\alpha,T)}{(t+1)^{3/2}} 	
	e^{ \frac{-\alpha(|x_1-y_1|+|x_2-y_2|)}{\sqrt{t+1}+C(\alpha)} },
$
\item\label{enu:SGIn}
$
	\displaystyle
	\big|\SGIn\big((y_1,y_2),(x_1,x_2);t\big)\big|
	\leq
	\frac{C(\alpha,T)}{t+1} e^{ -\frac{\alpha(|x_2-y_1|+|x_1-y_2|)}{\sqrt{t+1}+C(\alpha)} },
$
\item\label{enu:SGIng}
$
	\displaystyle
	\big|\nabla_{x_j}\SGIn\big((y_1,y_2),(x_1,x_2);t\big)\big|,
	\
	\big|\nabla_{y_j}\SGIn\big((y_1,y_2),(x_1,x_2);t\big)\big|	
	\leq
	\frac{C(\alpha,T)}{(t+1)^{3/2}} e^{ -\frac{\alpha(|x_2-y_1|+|x_1-y_2|)}{\sqrt{t+1}+C(\alpha)} },
$
\end{enumerate}
for all $ x_1<x_2\in\Xi(t+s) $, $ y_1<y_2\in\Xi(s) $,
$ s\in[0,\e^{-2}T]\cap\Z $, $ t\in[t_0,\e^{-2}T]\cap\Z $, $ j=1,2 $,
and $ (x_1,x_2,y_1,y_2) $ in their respective Weyl or $ \nabla $-Weyl chamber.
\end{prop}

Note that in Proposition~\ref{prop:SG:}\ref{enu:SGIn}--\ref{enu:SGIng},
the pairing of $ x_i $'s and $ y_j $'s is swapped compared to Proposition~\ref{prop:SG}.
This arises naturally from the contour integral structure of $ \SGIn $,
and in fact gives a stronger bound than the one in the original pairing.
To see this, under the assumption $ x_1<x_2 $ and $ y_1<y_2 $,
considering separately the four cases distinguished by the signs of $ x_1-y_1 $ and $ x_2-y_2 $,
we check that
\begin{align*}
	&|x_1-y_1|+|x_2-y_2| \stackrel{(++)}{=} (x_1-y_1)+(x_2-y_2) = (x_1-y_2) + (x_2-y_1) \leq |x_1-y_2| + |x_2-y_1|,
\\
	&|x_1-y_1|+|x_2-y_2| \stackrel{(--)}{=} (y_1-x_1)+(y_2-x_2) = (y_1-x_2) + (y_2-x_1) \leq |y_1-x_2| + |y_2-x_1|,
\\
	&|x_1-y_1|+|x_2-y_2| \stackrel{(+-)}{=} (x_1-y_1)+(y_2-x_2) \leq (x_2-y_1) + (y_2-x_1) \leq |x_2-y_1| + |y_2-x_1|,
\\
	&|x_1-y_1|+|x_2-y_2| \stackrel{(-+)}{=} (y_1-x_1)+(x_2-y_2) \leq (y_2-x_1) + (x_2-y_1) \leq |y_2-x_1| + |x_2-y_1|.
\end{align*}

Throughout the rest of this section, we \emph{fix} an exponent $ \alpha\in(0,\infty) $, a time horizon $ T\in(0,\infty) $,
and assume $ t\leq \e^{-2}T $ is large enough.
In the sequel we will frequently use polar coordinates $ z=re^{\img\theta} $ to parametrize complex numbers.
Throughout this section we will operator under convention $ \theta \in (-\pi,\pi] $.

\subsection{Estimating the free part $ \SGFr $}
Let us explain the strategy before starting the estimate.
We plan to deform $ \Circ_r\times\Circ_r $ to some suitable contours,
along which we easily extract the spatial exponential decay.
To this end, for $ \beta\in\R $ set
\begin{align}
	\label{eq:radius}
	\rad(t,\beta) := \exp\big( \tfrac{\beta}{\sqrt{t+1}+|\beta| C_*} \big).
\end{align}
We \emph{fixed} the constant $ C_*\in(0,\infty) $ large enough so that
$ \rad(t,\beta) \geq \exp(-1/C_*) \geq \frac{1+b_1}{2} $.
This is to avoid the pole of $ \SGte(t,z) $ (given in~\eqref{eq:SGt}) at $ z=b_1e^{\sqrt{\e}(\den-1)} $.
Now, let $ \sgn(x):=\ind_{\set{x > 0}} $ denote the sign function,
and let
\[
 r_i = \rad(t,-\sgn(x_i-y_i)\alpha)
 \]
  where $\alpha\in(0,\infty)$ is the parameter given in Proposition~\ref{prop:SG:}.
Along the contour $ (z_1,z_2)\in \Circ_{r_1}\times \Circ_{r_2} $,
we have the desired exponential decay:
\begin{align*}
	|z_i|^{x_i-y_i} = \exp\big(-\tfrac{\alpha|x_i-y_i|}{\sqrt{t+1}+\alpha C_*}\big).
\end{align*}

Given the exponential decay,
we still need to show that each of the remaining integrals (for $i=1,2$)
\[
 \int_{-\pi}^{\pi} |\SGt(t,z_i(\theta_i))| \frac{d\theta_i}{2\pi|z_i(\theta_i)|}
 \]
  are bounded by  $ (t+1)^{-\frac12}C $.
This is achieved by steepest decent analysis.
Under weak asymmetry scaling, the function $ \SGte(t,z) $ (given in~\eqref{eq:SGt}) reads
\begin{align}
	\label{eq:SGte}
	\SGte(t,z)
	=
	z^{\muet}
	\Big(
		\frac{1-b_1e^{\sqrt\e(\den-1)}}{b_1+e^{\sqrt\e\den}-b_1e^{\sqrt\e\den}-b_1e^{\sqrt\e(\den-1)}}
		\frac{b_1+(e^{\sqrt\e\den}-b_1e^{\sqrt\e\den}-b_1e^{\sqrt\e(\den-1)})z^{-1}}{1-b_1e^{\sqrt\e(\den-1)}z^{-1}}
	\Big)^t.
\end{align}
As we show in Lemma~\ref{lem:SGtbd} below,
along the contour $ \Circ_{r_i} $, under the polar parametrization $ z_i=r_ie^{\img\theta_i} $,
\begin{itemize}[leftmargin=5ex]
\item $ |\SGte(t,z_i(\theta_i))| $ has Gaussian decay in $\theta_i$ of the form $ \exp(-\frac{1}{C}\theta_i^2(t+1)) $ in a neighborhood of $ \theta_i=0 $,
\item $ |\SGte(t,z_i(\theta_i))| $ has an exponential decay in $t$ of the form $ \exp(-\frac{1}{C}(t+1)) $ away from $ \theta_i=0 $.
\end{itemize}
The first bullet point is done by Taylor expansion,
and relies only on \emph{local} properties of $ \SGte(t,z_i) $ and $ \Circ_{r_i} $ near $ \theta_i=0 $.
The second bullet point holds because of \emph{global} properties of $ \SGt(t,z_i) $.
More precisely, set
\begin{align}
	\label{eq:SGteLim}
	\SGteLim(z) := \frac{b_1z+1-2b_1}{1-b_1/z}.
\end{align}
Referring to the definition~\eqref{eq:SGt} of $ \SGte(t,z) $,
with $ \mue\to 1 $ as $ \e \to 0 $,
we have that
\begin{align*}
	\lim_{(t,\e)\to(\infty,0)} |\SGte(t,z)|^{\frac{1}{t}}
	=
	|\SGteLim(z)|,
	\quad
	\text{uniformly over } z\in \Circ_{1}.
\end{align*}
Now, with $ r_i=\rad(t,\pm\alpha) \to 1 $ as $ t\to\infty $, we see that the second bullet point holds only if
\begin{align}
	\tag{SD.$ \Circ_1 $}
	\label{eq:steep:circ}
	|\SGteLim(z)| <1,
	\quad
	\forall z \in \Circ_1\setminus\set{1}.
\end{align}

Conditions of the type~\eqref{eq:steep:circ} will turn out to be decisive in showing that steepest decent analysis works.
The condition~\eqref{eq:steep:circ} can be verified by
interpreting $ \SGteLim(z) $ as a probability generating function $ \Ex[z^{X}] $ of a random variable $ X $.
We will, instead, verify~\eqref{eq:steep:circ} (Lemma~\ref{lem:SGtbd}) by viewing $ \SGteLim(z) $ as a rational function
and directly calculating its modulus along the unit circle $ \Circ_1 $.
This approach has the advantage of generalizing to the case for the interacting part $ \SGIn $.

We now begin the steepest-decent-like bound on $ |\SGte(t,z)| $.
\begin{lem}\label{lem:SGtbd}
Given any $ \beta\in\R $ and $ T<\infty $, there exists $C(\beta,T), C>0$ such that
\begin{align*}
	|\SGte(t,z)|
	\leq
	C(\beta,T)\exp\big(-\tfrac1C\theta^2(t+1) \big),
	\qquad \textrm{with}\quad
	z=\rad(t,\beta)e^{\img\theta} \in \Circ_{\rad(t,\beta)},
\end{align*}
for all $ \theta\in (-\pi,\pi] $, large enough $ t\leq \e^{-2}T $, and small enough $ \e>0 $.
\end{lem}
\begin{proof}
Our first step is to recognize $ \SGte(t,z) $ as the $ t $-th power of a given function.
To this end, referring to~\eqref{eq:SGt}, observe that
\begin{align*}
	\SGtee(z)
	:= \SGte(t,z)^{\frac1t}
%	\label{eq:SGtz}
	=
	z^{\frac{\muet}{t}}
	\lambdae \frac{b_1+(1-b_1-b^\e_2)(\taue^\den z)^{-1}}{1-b^\e_2(\taue z)^{-1}}.
\end{align*}
%Referring to~\eqref{eq:rw}--\eqref{eq:Rw}, with $ \RW_\e := \Rwe-\mue $, we further express this as
%\begin{align}
%	\label{eq:SGtz}
%	\SGtee(z)
%	=
%	z^{\frac{\muet}{t}} \lambdae \Ex((\taue z)^{-\rw_\e})
%	=
%	z^{\frac{\muet}{t}} \Ex(z^{-\Rw_\e})
%	=
%	z^{\frac{\muet}{t}-1} \Ex(z^{-\RW_\e}),
%	\qquad
%	|z| > b^\e_2/\taue.
%\end{align}
Indeed, $ \SGtee(z) $ has a $ t $-dependence through $ z^{\frac{\muet}{t}} $,
but since $ \mue\to 1 $ as $ \e\to 0 $, we expect the $ t $-dependence to be `weak' and hence suppress it in notation.
Due to the non-integer power $ z^{\frac{\muet}{t}} $, the function $ \SGtee(z) $ is not meromorphic on $ \C $.
However,  since $ \mue\to 1 $ as $ \e\to 0 $, there exists a fixed neighborhood $ O $ of $ z=1 $,
such that $ \SGtee (z)$ is analytic on $ z\in O $. %\hao{and on $ z\in O $, we have $ \SGte(t,z) = \SGtee(z)^t $, for all $ t \geq 1 $. (remove this sentence)}
Throughout the proof we will operate on $ O $ whenever we refer to the function $ \SGtee(z) $.

As in the statement of Lemma \ref{lem:SGtbd}, set $ z(\theta)=\rad(t,\beta)e^{\img\theta} $.
%and we say `for all $ t $ large enough' if the refereed statement holds for all $ t \geq t_0 $,
%for some threshold depending only on $ \beta $.
The proof follows a three-step procedure:
\begin{enumerate}%
\myitem{(Zero $ \theta $)} \label{enu:3step:theta=0}
%Establish estimate at $ \theta=0 $.\\{}
Show that $ |\SGtee(z(0))| \leq \exp(C(\beta,T)\frac{1}{t+1}) $, for all $ t \leq \e^{-2}T $ large enough and $ \e>0 $ small enough.
Note that the right hand side of this bound also `weakly' depends on $t$ for $t$ sufficiently large.
\myitem{(Small $ \theta $)} \label{enu:3step:small}
%Establish Gaussian decay in $ \theta $ for small $ \theta $.\\{}
Show that there exists $ \theta_0>0 $, such that
$ |\SGtee(z(\theta))| \leq |\SGtee(z(0))|\exp(-\frac{\theta^2}{C}) $, for all $ |\theta|\leq\theta_0 $, and $ \e>0 $ small enough.
\myitem{(Large $ \theta $)} \label{enu:3step:large}
%Establish exponential decay in $ t $ for $ \theta$ away from $0 $.\\{}
Show that $ |\SGte(t,z(\theta))| \leq \exp(-\frac{t}{C}) $, for $ |\theta| > \theta_0 $, $ t \geq 0 $ large enough, and $ \e>0 $ small enough.
\end{enumerate}

Once these have been established, with $ \SGte(t,z) = \SGtee(z)^t $, the desired result follows immediately.
Our task is hence to carry out the steps \ref{enu:3step:theta=0}, \ref{enu:3step:small}, and \ref{enu:3step:large}.

\medskip
\ref{enu:3step:theta=0}:
First, since the function $ \SGtee(z) $ is invoked here,
let us check that the claimed assumption $ z(0)\in O $ holds.
Indeed, with $ \rad(t,\beta)\to 1 $ as $ t\to\infty $, we have that $ z(0)\in O $, for all $ t $ large enough.

Recall that $ \RW_\e := \Rw_\e -\mue $,
and that $ \Rw_\e $ is defined in~\eqref{eq:rw}--\eqref{eq:Rw} with $\mue=\Ex(\Rw_\e)$.
One readily checks that
%\begin{align}
%	\label{eq:SGtz}
$	\SGtee(z) = z^{\frac{\muet}{t}-\mue}\Ex[z^{-\RW_\e}]$, $z\in O$.
%\end{align}
Given this, it is straightforward to calculate
\begin{subequations}
\begin{align}
	\label{eq:SGte:D}
	\partial_z \big(\log\SGte(z)\big) &= \frac{\muet}{t}-\mue - \frac{\Ex[\RW_\e z^{-\RW_\e-1}]}{\Ex[z^{-\RW_\e}]},
\\
	\label{eq:SGte:DD}
	\partial^2_z \big(\log\SGte(z)\big) &=
		\frac{\Ex[\RW_\e(\RW_\e+1) z^{-\RW_\e-1}]}{\Ex[z^{-\RW_\e}]}
		-\Big(\frac{\Ex[\RW_\e z^{-\RW_\e-1}]}{\Ex[z^{-\RW_\e}]}\Big)^2,
\\
	\label{eq:SGte:DDD}
	\big|\partial^3_z \big(\log\SGte(z)\big)\big| & \leq C,
\end{align}
\end{subequations}
for all $ z\in O $.
Using~\eqref{eq:SGte:D}--\eqref{eq:SGte:DDD} we see that $\big|\partial_z \big(\log\SGte(z)\big)\vert_{z=1}\big| \leq t^{-1}$ and $\big|\partial^2_z \big(\log\SGte(z)\big)\vert_{z=1}\big| \leq C$ for some $C>0$. Using this, along with $ \log\SGtee(1)=0 $, we may Taylor expand around $z=1$ and bound $\big|\log\SGtee(z)\big| \leq t^{-1}|z-1| + C|z-1|^2 $.
Now, set $ z=z(0)=\rad(t,\beta) $, and use the fact that $|\rad(t,\beta)-1|\leq C(\beta,T)(t+1)^{-1/2}$ to bound (after exponentiating)
\begin{align*}
	| \SGtee(z(0)) |
	\leq
	\exp\big(t^{-1}|\rad(t,\beta)-1|+C|\rad(t,\beta)-1|^2\big)
	\leq
	\exp\big( C(\beta,T) \tfrac{1}{t+1} \big).
\end{align*}

\ref{enu:3step:small}:
First, with $ \rad(t,\beta)\to 1 $ as $ t\to\infty $,
it is readily verified that there exists a small enough $ \theta_0>0 $
such that the assumption $ z(\theta)\in O $ holds for all $ |\theta|\leq \theta_0 $ and $ t $ large enough.
From~\eqref{eq:SGte:D}--\eqref{eq:SGte:DDD}, we calculate (recall $\nu_*$ from \eqref{eq:ceffints})
\begin{align*}
	&
	\partial_\theta (\log\SGtee(z(\theta)))|_{\theta=0} \in \img\R,
\\
	&\lim_{ \e \to 0} \partial^2_\theta (\log\SGtee(z(\theta)))|_{\theta=0}
	=
	-\rad(t,\beta)^2\lim_{\e\to 0} \text{Var}(\RW_\e) \leq -\tfrac1C \nu_*,
\\
	&\big|\partial^3_\theta (\log\SGtee(z(\theta)))\big|
	\leq
	C.
\end{align*}
Given these properties,
Taylor expanding $ \log \SGte(t,z(\theta)) $ in $ \theta $ around $ \theta=0 $ to the second order yields
\begin{align*}
	\text{Re}\big[\log \SGte(t,z(\theta))-\log \SGte(t,z(0))\big]
	\leq
	-\tfrac{1}{C}\theta^2,
	\qquad
	|\theta| \leq \theta_0,
\end{align*}
for some fixed $ \theta_0>0 $.
Further exponentiating this gives the desired result
\begin{align*}
	\big|\SGtee(z(\theta))\big| \leq |\SGtee(z(0))| e^{ -\frac{1}{C}\theta^2},
	\quad
	\forall |\theta| \leq \theta_0,
\end{align*}
and $ \e>0 $ small enough.

\ref{enu:3step:large}:
Recall the definition of $ \SGteLim(z) $ from~\eqref{eq:SGteLim}.
With $ \mue\to 1 $ as $ \e\to 0 $, referring to the expression~\eqref{eq:SGte} for $ \SGte(t,z) $,
we readily verify that
\begin{align}
	\label{eq:SGtelim:1}
	\lim_{(t,\e)\to(\infty,0)} |\SGte(t,z(\theta))|^{\frac{1}{t}}
	=
	|\SGteLim(e^{\img\theta})|,
	\quad
	\text{uniformly over } \theta\in(-\pi,\pi].
\end{align}
The r.h.s.\ of~\eqref{eq:SGtelim:1} leads us to want to show \eqref{eq:steep:circ}.
To verify~\eqref{eq:steep:circ}, we calculate
\begin{align*}
	\big|\SGteLim(e^{\img\theta})\big|^2
	&=
	\Big(1+ \frac{b_1(w+w^{-1}-2)}{1-b_1w^{-1}} \Big)	
	\Big(1+ \frac{b_1(w^{-1}+w-2)}{1-b_1w} \Big)\Big|_{w=e^{\img\theta}}
\\
	&=
	1 + \frac{(w+w^{-1}-2)(2b_1+2-(b_1^2+1)(w+w^{-1}))}{|1-b_1w|^2}\Big|_{w=e^{\img\theta}}
\\
	&=
	1 - \frac{4(1-\cos\theta)(1+b_1-(1+b_1^2)\cos\theta)}{|1-b_1e^{\img\theta}|^2}
	< 1,
	\quad
	\forall \theta \in (-\pi,\pi]\setminus \{0\}.
\end{align*}
This calculation shows $ |\SGteLim(e^{\img\theta})| < 1- \frac{1}{C} $ for $ |\theta| > \theta_0 $.
Combining  with~\eqref{eq:SGtelim:1} gives the desired result:
\begin{align*}
	|\SGte(t,z(\theta))|^{\frac1t} \leq  1 - \tfrac{1}{C},
	\quad
	\forall |\theta| > \theta_0,
\end{align*}
for $ t\leq \e^{-2}T $ large enough, and $ \e>0 $ small enough.
\end{proof}

\begin{proof}[Proof of Proposition~\ref{prop:SG:}\ref{enu:SGFr}--\ref{enu:SGFrg}]
Given the expression~\eqref{eq:SGFr=hkhk}, it suffices to analyze each piece of $ \hk(t,x_i-y_i) $. We will do son using the contour integral expression given in~\eqref{eq:hk:contour}.
To begin with, we deform the contours $ \Circ_r \mapsto \Circ_{r_1}\times\Circ_{r_2} $,
where $ r_i := \rad(-\sgn(x_i-y_i)\alpha) $.
With $ r_i \geq \frac{1+b_1}{2} $ as explained below \eqref{eq:radius}, the deformation does not cross any pole, and gives
\begin{align}
	\label{eq:SGFr:}
	\hk(t,x_i-y_i)
	&=
	\oint_{\Circ_{r_i}} z_i^{x_i-y_i+(\mue t-\muet)}\frac{\SGte(t,z_i)dz_i}{2\pi\img z_i}.
\end{align}
Along the new contour $ \Circ_{r_i} $,
we have the desired exponential decay $ |z_i|^{x_i-y_i} = \exp\big( -\frac{\alpha|x_i-y_i|}{\sqrt{t+1}+\alpha C_*} \big) $.
Hence, under the parametrization $ z_i = r_ie^{\img\theta_i} $, we have
\begin{align*}
%	|\SGFr|
%	\leq
%	e^{ \frac{-\alpha(|x_1-y_1|+|x_2-y_2|)}{\sqrt{t+1}+C(\alpha)} }
%	\prod_{i=1}^2\int_{-\pi}^{\pi}\frac{|\SGte(t,z_i)|d\theta_i}{2\pi r_i}.
%
	\big| \hk(t,x_i-y_i) \big|
	\leq
	e^{ \frac{-\alpha(|x_i-y_i|)}{\sqrt{t+1}+C(\alpha)} }
	\int_{-\pi}^{\pi}\frac{|\SGte(t,z_i)|d\theta_i}{2\pi r_i}.	
\end{align*}
Now, using the bound on $ \SGte(t,z_i) $ from Lemma~\ref{lem:SGtbd}, we have
\begin{align}
	\label{eq:1contourbd}
%	\big|\SGFr\big|
%	\leq
%	C(\alpha,T) e^{ \frac{-\alpha(|x_1-y_1|+|x_2-y_2|)}{\sqrt{t+1}+C(\alpha)} }
%	\prod_{i=1}^2 \int_{\R}
%	e^{-\frac1C(t+1)\theta_i^2}d\theta_i
%	=
%	C(\alpha,T) e^{ \frac{-\alpha(|x_1-y_1|+|x_2-y_2|)}{\sqrt{t+1}+C(\alpha)} } \frac{1}{t+1}.
	\big| \hk(t,x_i-y_i) \big|
	\leq
	C(\alpha,T) e^{ \frac{-\alpha(|x_i-y_i|)}{\sqrt{t+1}+C(\alpha)} }
	\int_{\R} e^{-\frac1C(t+1)\theta_i^2}d\theta_i
	=
	C(\alpha,T) e^{ \frac{-\alpha(|x_i-y_i|)}{\sqrt{t+1}+C(\alpha)} } \frac{1}{\sqrt{t+1}}.
\end{align}
Inserting this bound for $ i=1,2 $ into~\eqref{eq:SGFr=hkhk} yields desired estimate on $ |\SGFr| $.

Turning to the gradients,
since the expression~\eqref{eq:SGFr=hkhk} is symmetric in the indices $ i=1,2 $,
without lost of generality we assume $ j=1 $.
Taking gradient $ \nabla_{x_1}, \nabla_{y_1} $ in \eqref{eq:SGFr=hkhk} gives
\begin{align}
	\label{eq:SGFr=hkhk:nabla}
	\nabla_{x_1} \SG = (\nabla \hk(t,x_1-y_1)) \hk(x_1-y_1),
	\quad
	\nabla_{y_1} \SG = (-\nabla \hk(t,x_1-y_1-1)) \hk(x_1-y_1).
\end{align}
Given this expression, it suffices to analyze $ \nabla \hk(t,x_i-y_i) $.
To this end, take $ \nabla $ in~\eqref{eq:hk:contour} to get
\begin{align*}
	\nabla\hk(t,x_i-y_i)
	&=
	\oint_{\Circ_{r_i}}
	(z_i-1)\prod_{i=1}^2 z_i^{x_i-y_i+(\mue t-\muet)}\frac{\SGte(t,z_i)dz_i}{2\pi\img z_i},
%\\
%	\nabla_{y_i}\hk(t,x_i-y_i)
%	&=
%	\oint_{\Circ_{r_i}}
%	(z^{-1}_i-1)\prod_{i=1}^2 z_i^{x_i-y_i+(\mue t-\muet)}\frac{\SGte(t,z_i)dz_i}{2\pi\img z_i}.
\end{align*}
With $ r_i=\rad(t,\pm\alpha) $, we have $ |z^\pm_i-1| \leq \frac{C(\alpha)}{\sqrt{t+1}} + \theta_i $
for $ z_i = r_ie^{\img\theta_j} $.
Using this bound and the preceding procedure for bounding $ |\hk(t,x_i-y_i)| $,
we obtain
\begin{align}
\label{eq:1contourbd:gd}
\begin{split}
	\big|\nabla\hk(t,x_i-y_i)\big|
%	\
%	\big|\nabla_{y_i}\hk(t,x_i-y_i)\big|
	&\leq
	C(\alpha,T)
	e^{ \frac{-\alpha(|x_i-y_i|)}{\sqrt{t+1}+C(\alpha)} }
	\int_{\R} \Big(\frac{1}{\sqrt{t+1}} + \theta_i\Big) e^{-\frac1C (t+1)\theta_i^2}d\theta_i
\\
	&=
	C(\alpha,T) e^{ \frac{-\alpha(|x_i-y_i|)}{\sqrt{t+1}+C(\alpha)} } \frac{1}{t+1}.
\end{split}
\end{align}
Inserting~\eqref{eq:1contourbd:gd} for $ i=1 $ and~\eqref{eq:1contourbd} for $ i=2 $ into~\eqref{eq:SGFr=hkhk:nabla}
yields the desired bound on $ \nabla_{x_1} \SGFr $ and $ \nabla_{y_1}\SGFr $.
\end{proof}

\subsection{Estimating the interacting part~$ \SGIn $, an overview}
\label{sect:SGIn:ov}
In this subsection, we give an overview of the strategy for estimating~$ \SGIn $.
Compared to the estimate for $ \SGFr $,
the major difference is that the expression $ \SGfre(z_1,z_2) $
introduces a pole during contour deformations.
More explicitly, under weak asymmetry scaling, $ \SGfre(z_1,z_2) $ (defined in~\eqref{eq:SGfr}) reads
\begin{align}
	\label{eq:SGfre}
	\SGfre(z_1,z_2)
	=
	\frac{
		1+e^{\sqrt{\e}(1-2\den)}z_1z_2-(e^{-\sqrt{\e}\den}+e^{\sqrt{\e}(1-\den)}) z_2
	}{
		1+e^{\sqrt{\e}(1-2\den)}z_1z_2-(e^{-\sqrt{\e}\den}+e^{\sqrt{\e}(1-\den)}) z_1
	}.
\end{align}
This expression has a pole at $ z_2=\pole(z_1) $, where
\begin{align}
	\label{eq:pole}
	\pole(z) := (e^{\sqrt{\e}(\den-1)}+e^{\sqrt{\e}\den})-e^{\sqrt{\e}(2\den-1)}z^{-1}.
\end{align}

For the variable $ z_1 $, we will devise a suitable contour $ \zoneC $,
on a case-by-case basis depending on the signs of $ x_2-y_1 $.
Starting with the expression~\eqref{eq:SGIn}, we deform the contours in two steps.
First, with $ z_2\in\Circ_{r} $ being fixed,
we deform the contour of $ z_1 $: $ \Circ_r \mapsto \zoneC $.
For the suitable $ \zoneC $ so constructed in the sequel, we will check that
\begin{align}
	\tag{No Pole}
	\label{eq:nopole}
	\text{no pole is crossed during the deformation } z_1\in\Circ_r\longmapsto \zoneC,
	\text{ if } r \text{ is large enough.}
\end{align}
In particular, here $ r $ must be so large that $ \Circ_{r} $ contains $ \pole(\zoneC) $.
Next, for the $ z_2 $-contour, consider
\begin{align}
	\label{eq:r2s}
	r_2 := \rad(t,\sgn(x_1-y_2)k_2\alpha),
	\
	r'_2 := \rad(t,\sgn(x_1-y_2)2k_2\alpha),
	\
	r''_2 := \rad(t,\sgn(x_1-y_2)3 k_2 \alpha),	
\end{align}
where $ k_2\in\Z_{>0} $ is an auxiliary parameter, irrelevant for the general discussion in this subsection.
With $ z_1\in\zoneC $ being fixed, we shrink the contour of $ z_2 $ from the large circle $ \Circ_r $ to $ \Circ_{\tilde{r_2}(z_1)} $,
where the radius $ \tilde{r}_2(z_1) $ depends on the location of $ \pole(z_1) $, given by
\begin{align}
	\label{eq:tilr2}
	\tilde{r}_2(z_1) := \ind_{\set{|\pole(z_1)| \leq r'_2}} (r_2\vee r''_2) + \ind_{\set{|\pole(z_1)| >r'_2}} (r_2\wedge r''_2).
\end{align}
That is, for a \emph{fixed} $ z_1\in\zoneC $, we examine the location of $ \pole(z_1) $,
and if it sits outside of $ \Circ_{r'_2} $, we shrink the large circle $ z_2\in\Circ_r $ to a smaller circle with radius $ r_2\wedge r''_2 \leq r'_2 $,
otherwise shrink $ \Circ_r $ to a circle with radius $ r_2\vee r''_2 > r'_2 $.

During the second deformation $ z_2\in\Circ_r \mapsto \Circ_{\tilde{r}_2(z_1)} $,
we cross a pole at $ z_2=\pole(z_1) $ if $ r'_2<|\pole(z_1)| $.
This is a simple pole from the term $ \SGfre(z_1,z_2) $, with
\begin{align*}
	\underset{z_2=\pole(z_1)}{\text{Res}} \SGfre(z_1,z_2)
	=
	\big(e^{\sqrt{\e}(\den-1)}+e^{\sqrt{\e}\den}\big)
	\big( \tfrac{\pole(z_1)}{z_1}-1 \big).
\end{align*}
Set
\begin{align}
	\label{eq:SGtt}
	\SGtt(t,z) :=& \SGte(t,z)\SGte(t,\pole(z))
\\
\begin{split}
	\label{eq:zp}
	\zp(z_1) :=&
	z_1^{x_2-y_1+(\mue t-\muet)} \pole(z_1)^{x_1-y_2+(\mue t-\muet)}
\\
	=&
	z_1^{x_2-y_1-1+(\mue t-\muet)} \pole(z_1)^{x_1-y_2+1+(\mue t-\muet)}
\\
	&\qquad
	- z_1^{x_2-y_1+(\mue t-\muet)} \pole(z_1)^{x_1-y_2+(\mue t-\muet)}.
\end{split}
\end{align}
For each \emph{fixed} $ z_1\in\zoneC $, applying the residue theorem to calculate the resulting expression
after the deformation $ z_2\in\Circ_{r}\mapsto \Circ_{\tilde{r_2}(z_1)} $, we have
\begin{align*}
	\SGIn = \SGbk + \SGres,
\end{align*}
where $ \SGbk $ and $ \SGres $ respectively contribute the `bulk' and `residue' parts of the deformed integral:
\begin{align}
	\label{eq:SGbk}
	\SGbk &:=
	\oint_{\zoneC}
	\Bigg(
		\oint_{\Circ_{\tilde{r}_2(z_1)}}
		\hspace{-10pt}
		z_1^{x_2-y_1+(\mue t-\muet)}z_2^{x_1-y_2+(\mue t-\muet)} \SGfre(z_1,z_2)
		\frac{\SGte(t,z_2)dz_2}{2\pi\img z_2}
	\Bigg)
	\frac{\SGte(t,z_1)dz_1}{2\pi\img z_1},
\\
	\label{eq:SGres}
	\SGres &:=
	\oint_{\zoneC} \ind_{\set{|\pole(z_1)|>r'_2}} \big(e^{\sqrt{\e}(\den-1)}+e^{\sqrt{\e}\den}\big) \zp(z_1) \frac{\SGtt(t,z_1)dz_1}{2\pi\img z_1 \pole(z_1)}.
\end{align}
The integral in~\eqref{eq:SGbk} is \emph{iterated} because $ \tilde{r_2}(z_1) $ depends on $ z_1 $.

Recall that $ |\SGfre(z_1,z_2)|=\infty $ at $ z_2=\pole(z_1) $.
By having $ \tilde{r_2}(z_1) $ as in~\eqref{eq:tilr2}, we avoid the point $ z_2=\pole(z_1) $ in the integral~\eqref{eq:SGbk}.
More precisely, from~\eqref{eq:tilr2}, together with \eqref{eq:radius}, we have that
\begin{align}
	\label{eq:z2-polez1}
	|z_2-\pole(z_1)| \geq (|r''_2-r'_2|\wedge|r'_2-r_2|) \geq \tfrac{1}{C\sqrt{t+1}},
	\quad
	(z_1,z_2) \in \zoneC\times\Circ_{\tilde{r}_2(z_1)}.
\end{align}
(\emph{Alternatively}, one could also fix the radius $ \tilde{r_2}(z_1)=r'_2 $ for the $ z_2 $ contour.
The resulting integrand in \eqref{eq:SGbk} in this case has a singularity at $ z_2=\pole(z_1) $,
which is integrable over $ (z_1,z_2)\in\zoneC\times\Circ_{r'_2} $.
Proceeding this way however, requires elaborated estimates near the singularly jointly as $ (t,\e) $ varies.
We avoid doing so by constructing $ \tilde{r}_2(z_1) $ in such a way that \eqref{eq:z2-polez1} holds.)

The contour $ \zoneC $ needs be constructed in such a way that both $ \SGbk $ and $ \SGres $
are controlled by steepest decent analysis.
In particular, a steepest decent condition analogous to~\eqref{eq:steep:circ} needs to hold here.
To formulate the condition, assume that $ \zoneC $ converges to a limiting contour $ \zoneCLim $ as $ (t,\e)\to(\infty,0) $.
Given $ \lim_{\e\to 0}\pole(z) =2-z^{-1} $ from~\eqref{eq:pole},
we define
\begin{align}
	\label{eq:SGttLim}
	\SGttLim(z):= \SGteLim(z)\SGteLim(2-z^{-1})
	=
	\frac{b_1z+1-2b_1}{1-b_1z^{-1}}\frac{b_1(2-z^{-1})+1-2b_1}{1-b_1/(2-z^{-1})}.
\end{align}
The analogous steepest decent condition we must check here is
\begin{align*}
	|\SGteLim(z)| < 1 \text{ for all } z\in \zoneCLim\setminus\set{1},
	\quad
	|\SGttLim(z)| < 1 \text{ for all } z\in \zoneCLim\setminus\set{1}.
\end{align*}

\begin{figure}[ht]
	\begin{subfigure}{0.4\textwidth}
	\includegraphics[width=\linewidth]{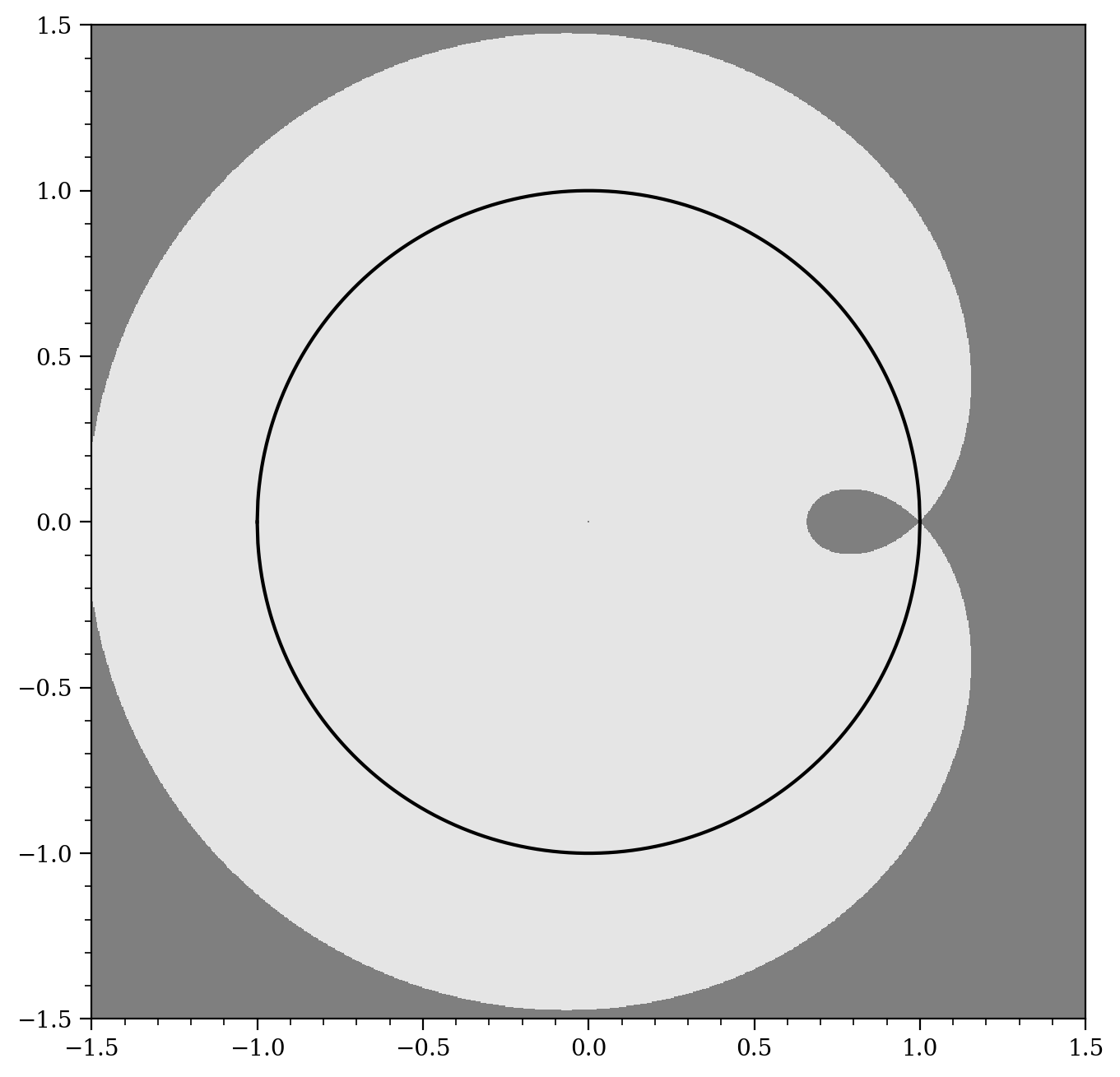}
	\caption{The function $ \SGteLim $ }
	\end{subfigure}
  	\hfil
	\begin{subfigure}{0.4\textwidth}
	\includegraphics[width=\linewidth]{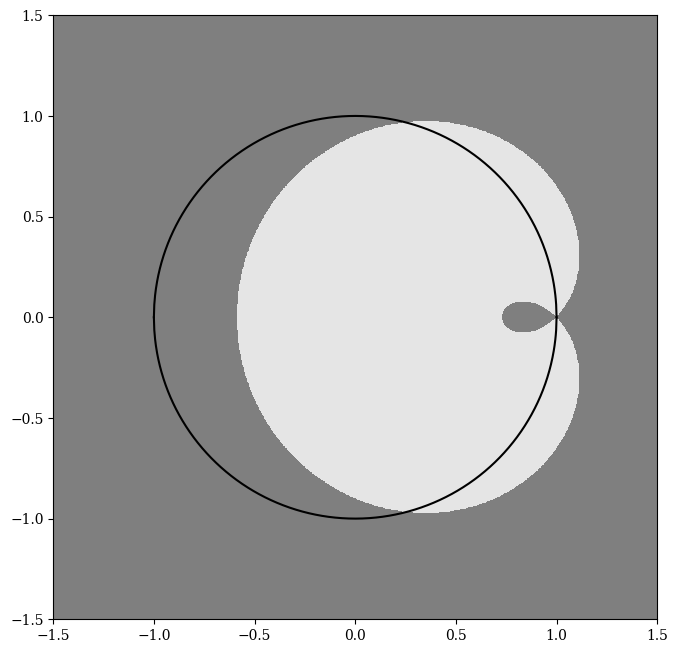}
	\caption{The function $ \SGttLim $}
	\end{subfigure}
	\caption{The figures show where the designated function is larger (darker) or smaller (lighter) than 1 in absolute value,
		for $ b_1=0.7 $.
		The unit circle is  shown for comparison.
	}
	\label{fig:SGLim}
\end{figure}

Figure~\ref{fig:SGLim} shows the region in $ \C $ where $ |\SGteLim(z)| < 1 $ and where $ |\SGttLim(z)| < 1 $, for $ b_1=0.7 $.
In particular, we see that $ |\SGttLim(z)| < 1 $ \emph{fails} for a portion of the unit circle $ \Circ_1 $.
This being the case, we need to devise a different type of contour than the contour $ \Circ_{r_1} $ used in the preceding subsection.
We begin with a prototype
\begin{align*}
	\Magic := \{z: |z-\tfrac12|=\tfrac12\}.
\end{align*}
This contour $ \Magic $ satisfies the steepest decent condition
\begin{align*}
	\tag{SD.$ \Magic $}
	\label{eq:steep:Magic}
	|\SGteLim(z)| < 1 \text{ for all } z\in \Magic\setminus\set{1},
	\quad
	|\SGttLim(z)| < 1 \text{ for all } z\in \Magic\setminus\set{1},
\end{align*}
which we verify now.
\begin{proof}[Proof of \eqref{eq:steep:Magic}]
First, express $ \SGteLim(z) $ and $ \SGttLim(z) $ (defined in~\eqref{eq:SGteLim} and \eqref{eq:SGttLim}) as
\begin{align*}
	\SGteLim(z) &= \frac{b_1z+1-2b_1}{1-b_1z^{-1}}
	=
	1 + \frac{b_1z+b_1z^{-1}-2b_1}{1-b_1z^{-1}}.
\\
	\SGttLim(z) &= \frac{b_1z+1-2b_1}{1-b_1z^{-1}}\frac{b_1(2-z^{-1})+1-2b_1}{1-b_1/(2-z^{-1})}
	=
	1 + \frac{2b_1z+2b_1z^{-1}-4b_1}{2-b_1-z^{-1}}.
\end{align*}
under the parametrization $ z(\theta) := \frac{1+e^{\img\theta}}{2}\in\Circc $,
we calculate
\begin{subequations}
\label{eq:SGexact}
\begin{align*}
	\Big|\SGteLim\Big(\frac{1+e^{\img\theta}}{2}\Big)\Big|^2
	&=
	\Big(1+ \frac{b_1(w-1)^2}{2(w+1-2b_1)}\Big)	
	\Big(1+ \frac{b_1(w^{-1}-1)^2}{2(w^{-1}+1-2b_1)}\Big)\Big|_{w=e^{\img\theta}}
\\
	&=
	1 + \frac{b_1(w-2+w^{-1})((2-3b_1)(w+w^{-1})+4-2b_1)}{|2(w^{-1}+1-2b_1)|^2}\Big|_{w=e^{\img\theta}}
\\
	&=
	1 - \frac{b_1(1-\cos\theta)(2-b_1+(2-3b_1)\cos\theta)}{|(w^{-1}+1-2b_1)|^2}.
\\
	\Big|\SGttLim\Big(\frac{1+e^{\img\theta}}{2}\Big)\Big|^2
	&=
	\Big(1+ \frac{b_1(w-1)^2}{(2-b_1)w-b_1}\Big)	
	\Big(1+ \frac{b_1(w^{-1}-1)^2}{(2-b_1)w^{-1}-b_1}\Big)\Big|_{w=e^{\img\theta}}
\\
	&=
	1 + \frac{4b_1(1-b_1)(w-2+w^{-1})}{|(2-b_1)w-b_1|^2}\Big|_{w=e^{\img\theta}}
\\
	&=
	1 - \frac{8b_1(1-b_1)(1-\cos\theta)}{|(2-b_1)e^{\img\theta}-b_1|^2}.
\end{align*}
\end{subequations}
It is now readily checked that these expressions are strictly less than $ 1 $ for all $ \theta\in(\pi,\pi]\setminus\set{0} $ (and $ b_1\in(0,1) $),
which gives exactly the desired properties.
\end{proof}

\begin{figure}[h]
\includegraphics[width=.85\linewidth]{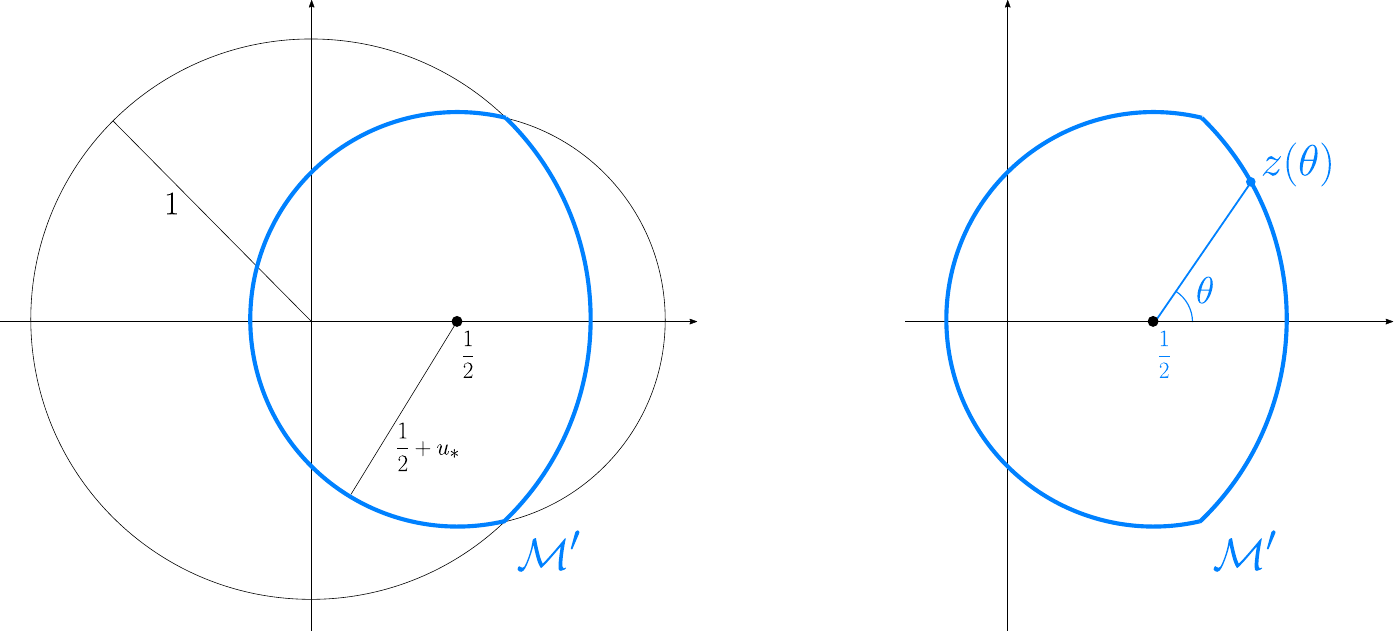}
\caption{The contour $ \MagicC $ and its parametrization.}
\label{fig:MagicC}
\end{figure}
\begin{figure}[h]
\includegraphics[width=.85\linewidth]{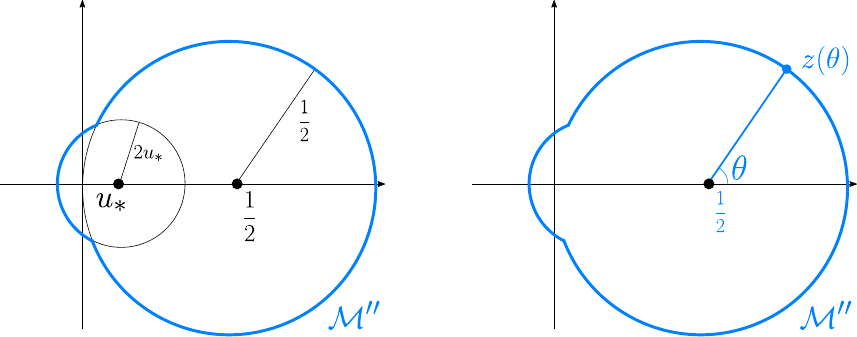}
\caption{The contour $ \Magicc $ and its parametrization.}
\label{fig:Magicc}
\end{figure}

Even though $ \Magic $ enjoys the desired property~\eqref{eq:steep:Magic}, it cuts through the point $ z=0 $.
This could cause issues, as the integrals~\eqref{eq:SGbk}--\eqref{eq:SGres} generally contain poles at $ z_1=0 $.
To circumvent this problem, we consider modifications $ \MagicC $ and $ \Magicc $ of $ \Magic $:
\begin{align}
	\label{eq:MagicC}
	\MagicC=\MagicC(u_*) &:= \partial \big( \big\{ |z| \leq 1 \big\}\cap\big\{ |z-\tfrac12| \leq \tfrac{1}{2}+u_* \big\} \big),
\\
	\label{eq:Magicc}
	\Magicc=\Magicc(u_*) &:= \partial \big( \big\{ |z-\tfrac12| \leq \tfrac12 \big\}\cup\big\{ |z-u_*| \leq 2u_* \big\} \big),
\end{align}
counterclockwise oriented; see Figures~\ref{fig:MagicC}--\ref{fig:Magicc}.
Here $ u_*\in(0,\frac1{12}\wedge b_1) $ is a parameter, which we \emph{fix} in Lemma~\ref{lem:magicc}
so that the resulting contours~$ \MagicC $ and $ \Magicc $ also enjoy the steepest decent condition.
We now verify the steepest decent condition for $ \MagicC $ and $ \Magicc $.
\begin{lem}
\label{lem:magicc}
There exists $ u_*\in(0,\frac1{12}\wedge b_1) $ such that, for the contours $ \MagicC(u_*)$ and $ \Magicc(u_*)$  we have
\begin{align}
	\tag{SD.$ \MagicC $}
	\label{eq:steep:MagicC}
	&|\SGteLim(z)| < 1, \ |\SGttLim(z)| < 1
	\quad
	z\in \MagicC\setminus\set{1},
\\
	\tag{SD.$ \Magicc $}
	\label{eq:steep:Magicc}	
	&|\SGteLim(z)| < 1, \ |\SGttLim(z)| < 1
	\quad
	z\in \Magicc\setminus\set{1}.
\end{align}
\end{lem}
\begin{proof}
%Throughout this proof, we write $ \MagicC(u) := \partial ( \{ |z| \leq 1 \big\}\cap\{ |z-\tfrac12| \leq \tfrac{1}{2}+u \} ) $
%and $ \Magicc(u) := \partial ( \{ |z-\frac12| \leq \frac12 \}\cap\{ |z-u| \leq 2u \} ) $
%for the corresponding contours with a \emph{varying} parameter $ u\in(0,\infty) $.
We will show that for all small enough $ u>0 $,
\begin{align*}
	&|\SGteLim(z)| < 1, \ |\SGttLim(z)| < 1
	\quad
	z\in \MagicC(u)\setminus\set{1},
\\
	&|\SGteLim(z)| < 1, \ |\SGttLim(z)| < 1
	\quad
	z\in \Magicc(u)\setminus\set{1}.
\end{align*}

We begin with the statement for $ \Magicc(u) $.
Indeed, this contour differs from $ \Magic $ only in the neighborhood $ O(3u):=\{z\in\C: |z| < 3u \} $ of $ z=0 $.
This being the case, instead of the entire contour $ \Magicc(u) $, we need only to consider the part $ \Magicc(u)\cap O(3u) $.
We already know from~\eqref{eq:steep:Magic} that $ |\SGteLim(0)|<1 $ and $ |\SGttLim(0)|<1 $.
It is readily checked from~\eqref{eq:SGteLim} and \eqref{eq:SGttLim} that $ \SGteLim(z) $ and $ \SGttLim(z) $
are continuous at $ z=0 $, hence we see that
$ |\SGteLim(z)| < 1, \ |\SGttLim(z)| < 1 $ holds on $ z\in \Magicc(u)\cap O(3u)  $ for all small enough $ u>0 $.
	
We now turn to $ \MagicC(u) $.
Let us first analyze the local behaviors of $ \SGteLim(z) $ and $ \SGttLim(z) $ near $ z=1 $.
Straightforward calculation gives
\begin{align*}
	\SGteLim(1)=1,& &\partial_z\SGteLim(1)=0,& &\partial^2_z\SGteLim(1)=\nu_*,&
	&
	\SGttLim(1)=1,& &\partial_z\SGttLim(1)=0,& &\partial^2_z\SGttLim(1)=2\nu_*,
\end{align*}
so Taylor expansion of $ \SGteLim(z) $ around $ z=1 $ gives $ 1 + \frac12\nu_*(z-1)^2 $ up the second order,
and Taylor expansion of $ \SGttLim(z) $ around $ z=1 $ gives $ 1 + \nu_*(z-1)^2 $ up the second order.
The expression $ \nu_*(z-1)^2 $ is real and negative along the vertical direction: $ z-1 \in \img\R $.
Since $ \SGteLim(z) $ and $ \SGttLim(z) $ are analytic in a neighborhood of $ z=1 $, we have
\begin{align*}
	\big| \SGteLim(z) \big|, \
	\big| \SGttLim(z) \big| \leq 1 - \tfrac{1}{C}|z-1|^2,
	\quad
	\forall \, z \in \mathcal{A},
\end{align*}
where $ \mathcal{A} := \{ z=ve^{\img\phi}: v\in[0,v_0], |\phi\pm\tfrac{\pi}{2}| \leq \phi_0 \} $
is an `hourglass-shape' region centered at $ z=1 $,
and $ v_0,\phi_0 >0 $ are fixed. See Figure~\ref{fig:hoursglass}.
This property ensures that $ |\SGteLim|, |\SGttLim|< 1 $ within $ \mathcal{A}\setminus\set{1} $,
so instead of the entire contour $ \MagicC(u) $, it suffices to consider the part $ (\MagicC(u)\setminus\mathcal{A}) $.

\begin{figure}
\includegraphics[width=.45\linewidth]{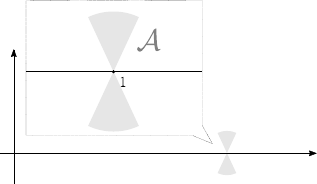}
\caption{The hourglass-shape region $ \mathcal{A} $.}
\label{fig:hoursglass}
\end{figure}

Instead of $ (\MagicC(u)\setminus\mathcal{A}) $, let us first consider $ (\Magic\setminus\mathcal{A}) $.
Since the contour $ \Magic $ passes through the point $ z=1 $ vertically, under the parametrization $ z(\theta) = \frac{1+e^{\img\theta}}{2} $,
the part $ (\Magic\setminus\mathcal{A}) $ avoids a neighborhood of $ \theta=0 $.
This being the case, referring to the calculations~\eqref{eq:SGexact}, we see that
\begin{align*}
%	\label{eq:magic:hoursglass}
	\sup_{z\in\Magic\setminus\mathcal{A}} |\SGteLim(z)| <1,
	\quad
	\sup_{z\in\Magic\setminus\mathcal{A}} |\SGttLim(z)| <1.
\end{align*}
Let dist$ (A,B) := \inf\{|z_1-z_2|:z_1\in A, z_2\in B\} $ denotes the distance of two sets $ A,B\subset \C $.
Referring to the definition~\eqref{eq:MagicC} of $ \MagicC(u) $, we see that
\begin{align*}
	\lim_{u\downarrow 0} \text{dist}\big( (\Magic\setminus\mathcal{A}) ,(\MagicC(u)\setminus\mathcal{A})\big) =0.
\end{align*}
Further, it is readily verified (from~\eqref{eq:SGteLim} and \eqref{eq:SGttLim}) that $ \SGteLim $ and $ \SGttLim $ are uniformly continuous $ \Magic $.
These properties together give
\begin{align*}
	&\lim_{u\downarrow 0}
	\Big( \sup_{z\in\MagicC(u)\setminus\mathcal{A}} |\SGteLim(z)| \Big)
	=
	\sup_{z\in\Magic\setminus\mathcal{A}} |\SGteLim(z)| <1,
&
	&\lim_{u\downarrow 0}
	\Big( \sup_{z\in\MagicC(u)\setminus\mathcal{A}} |\SGttLim(z)| \Big)
	=
	\sup_{z\in\Magic\setminus\mathcal{A}} |\SGttLim(z)| <1,
\end{align*}
which concludes the proof.
\end{proof}

In the following subsections we prove Proposition~\ref{prop:SG:}\ref{enu:SGIn}--\ref{enu:SGIng},
namely establishing the desired estimates on $ \SGIn $ and its gradients.
To this end, we treat separately the cases distinguished by the signs of $ x_2-y_1 $ and $ x_1-y_2 $,
which we refer to as the $ (+-) $, $ (--) $, and $ (++) $-cases:
\begin{itemize}
\item[] $ x_2-y_1 >0 $ and $ x_1-y_2 \leq 0 $, the $ (+-) $-case;
\item[] $ x_2-y_1 \leq 0 $ and $ x_1-y_2 \leq 0 $, the $ (--) $-case;
\item[] $ x_2-y_1 >0 $ and $ x_1-y_2 >0 $, the $ (++) $-case.
\end{itemize}
The $ (-+) $-case (i.e., $ x_2-y_1 \leq 0 $ and $ x_1-y_2 >0 $)
is irrelevant due the assumption $ x_1<x_2 $ and $ y_1<y_2 $.

Let us introduce one more convention about Taylor expansion which will be used in the subsequent arguments.
Recall the assumption $ t \leq \e^{-2}T $ which ensures that $ \e \leq C(T)(t+1)^{-1/2} $.
At times we will Taylor expand expressions in the variables $ (\sqrt{\e}, \frac{1}{\sqrt{t+1}}) $.
In the course of doing so, we adopt the following ordering convention in light of the aforementioned condition on $ \e $.
\begin{defn}
\label{def:Taylor}
To Taylor expand a given expression $ f(\sqrt{\e},\frac{1}{\sqrt{t+1}}) $,
we assign $ \sqrt{\e} $ the order of $ (t+1)^{-1/4} $.
For example, Taylor expansion of $ f(\sqrt{\e},\frac{1}{\sqrt{t+1}}) $ up to order $ \frac{1}{\sqrt{t+1}} $ reads
\begin{align*}
	f(0,0) + \partial_1 f(0,0) \sqrt{\e} + \tfrac12 \partial^2_{1} f(0,0) \e + \partial_{2}f(0,0) \tfrac{1}{\sqrt{t+1}}.
\end{align*}
\end{defn}

\subsection{Estimating the interacting part $ \SGIn $, the $ (+-) $-case}
\label{sect:+-}
We begin by constructing the contour $ \zoneC $.
For the $ (+-) $-case considered here, $ \zoneC $ is constructed as perturbation of $ \MagicC $.
More precisely, recall the definition of $ \rad(t,\beta) $ from~\eqref{eq:radius}. For $ \beta\in\R $, set
\begin{align}
	\label{eq:magicC}
	\magicC(t,\beta) := \partial \big( \set{|z| \leq \rad(t,\beta)} \cap \set{|z-\tfrac12|\leq\tfrac12} \big),
\end{align}
counterclockwise oriented; see Figure~\ref{fig:magicC} and compare it with Figure~\ref{fig:MagicC}.
Under these notation, we set\footnote{%
	Here $ \zoneC $ does not depend on $ \e $, but we keep this notation to be consistent throughout all cases.}
\begin{align*}
	\zoneC := \magicC(t,-k_1\alpha),
\end{align*}
where $ k_1=k_1(\alpha,T)\in\Z_{>0} $ is an auxiliary parameter to be fixed later.

\begin{figure}[h]
\includegraphics[width=.85\linewidth]{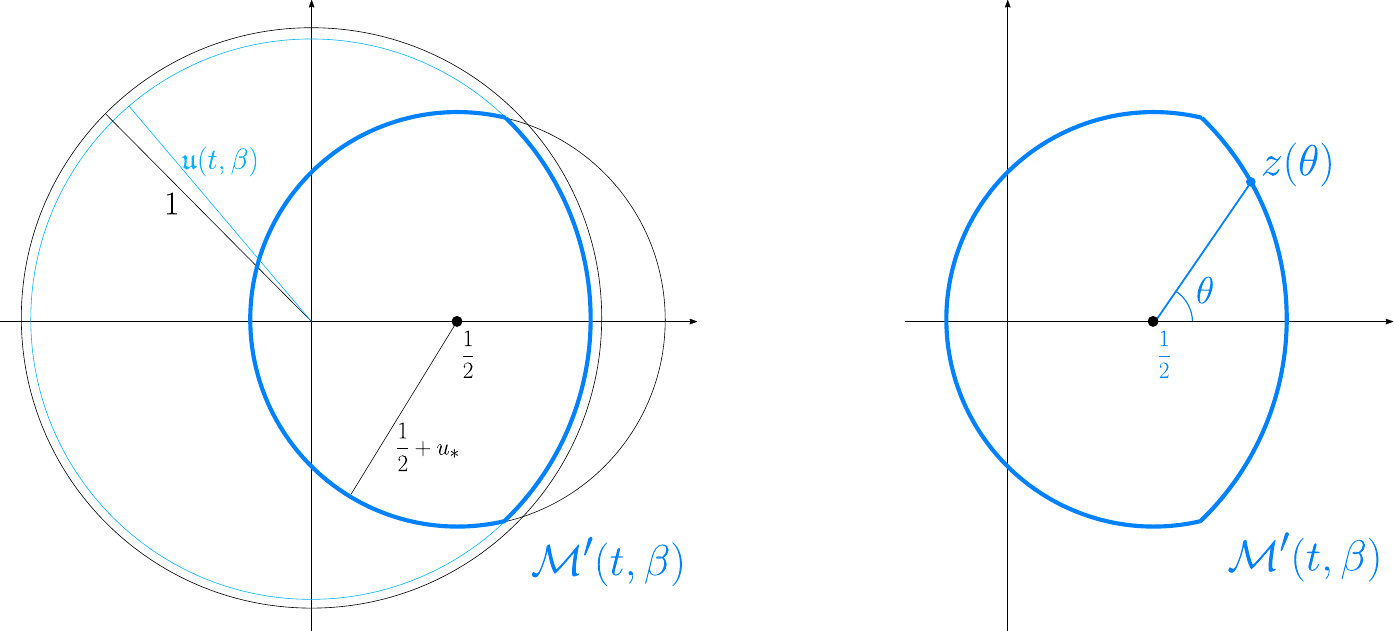}
\caption{The contour $ \magicC(t,\beta) $ and its parametrization.
The figure shows the case $ \beta<0 $.
}
\label{fig:magicC}
\end{figure}

Hereafter we parametrize $ z_1=z_1(\theta_1)\in\magicC(t,-k_1\alpha) $ as depicted in Figure~\ref{fig:magicC}.
As for the $ z_2 $-contour, we fix $ k_2:=1 $ in~\eqref{eq:r2s}.
Recalling  $ \tilde{r}_2(z_1) $ from~\eqref{eq:tilr2},
we parametrize $ z_2(\theta):=\tilde{r}_2(z_1)e^{\img\theta_2} \in \Circ_{\tilde{r}_2(z_1)} $.

The parameter $ k_1\in\Z_{>0} $ is to ensures that
\begin{align}
	\label{eq:cntur:intersect}
	r'_2 \geq \pole(z_1(0)) + \tfrac{1}{\sqrt{t+1}} \in \R.
\end{align}
To see why this holds for large enough $ k_1 $, recall from Definition~\ref{def:Taylor} the announced convention on Taylor expansion,
and expand the expression $ r'_2-\pole(z_1(0)) = \rad(t,2\alpha) - \pole(\rad(t,-k_1\alpha)) $ in $ (\sqrt\e,\frac{1}{\sqrt{t+1}}) $
to the leading order in $ \frac{1}{\sqrt{t+1}} $ to get
\begin{align*}
	z_2(0)-\pole(z_1(0)) = 0\cdot\sqrt\e - \den(1-\den)\e + \tfrac{(k_1+2)\alpha}{\sqrt{t+1}} + \ldots.
\end{align*}
From this, together with $ \e \leq \frac{C(T)}{\sqrt{t+1}} $ under current assumptions,
we see that the condition~\eqref{eq:cntur:intersect} holds for a large enough $ k_1=k_1(\alpha,T) $,
and we fix such a $ k_1 \in\Z_{>0} $ hereafter.

The purpose of imposing the condition~\eqref{eq:cntur:intersect} is to control the region $ \{z_1:|\pole(z_1)| >r'_2\} $,
as will be relevant toward controlling the integral $ \SGres $~\eqref{eq:SGres}.
Under the aforementioned parametrization $ z_1=z_1(\theta) $,
the condition~\eqref{eq:cntur:intersect} ensures a lower bound on $ |\theta_1| $
for which $ |\pole(z_1(\theta))| >r'_2 $.
That is,
\begin{align}
	\label{eq:theta1:+-}
	|\pole(z_1(\theta_1))| > r'_2
	\text{ holds only if }
	|\theta_1| \geq \frac{1}{C(\alpha)(t+1)^{1/4}}.
\end{align}
\begin{proof}[Proof of~\eqref{eq:theta1:+-}]
Set $ f(\theta_1) := |\pole(z_1(\theta_1))| -r'_2 $.
Our goal is to obtain a lower bound on those $ |\theta_1| $ such that $ f(\theta_1) \geq 0 $.
Given the explicit expression $ \pole(z_1(\theta_1)) = e^{\sqrt{\e}(\den-1)}+e^{\sqrt\e\den} - e^{\sqrt\e(2\den-1)} \rad(t,k_1\alpha)e^{-\img\theta_1} $,
one readily checks that $ \frac{d~}{d\theta_1}f(0) = 0 $, and that $ |\frac{d~}{d\theta_1}f(\theta_1)| \leq C(\alpha) $,
for all $ (\theta_1,t,\e)\in (-\pi,\pi]\times\Z_{\geq 0}\times(0,1) $.
Taylor expanding $ f(\theta_1) $ accordingly as
\begin{align*}
	f(\theta_1) = f(0) + \int_0^{\theta_1} (\theta_1-\theta)\tfrac{d~}{d\theta}f(\theta) d\theta,
\end{align*}
we see $ f(\theta_1) \geq 0 $ only if $ f(0) + C(\alpha)\theta_1^2 \geq 0 $.
Now, the condition~\eqref{eq:cntur:intersect} ensures that $ f(0) \leq -\frac{1}{\sqrt{t+1}} $.
From this we conclude~\eqref{eq:theta1:+-}.
\end{proof}

Recall that $ \SGtt(t,z) := \SGte(t,z) \SGte(t,\pole(z)) $.
Let us check that, along the contour $ \magicC(t,-k_1\alpha) $,
we do have the desired Gaussian decay of $ |\SGte| $ and $ |\SGtt| $.
\begin{lem}
\label{lem:SGtbd:+-}
Given any $ T\in(0,\infty) $ and $ \beta\in\R $,
\begin{align*}
	\big|\SGte(t,z)\big|,
	\
	\big|\SGtt(t,z)\big|
	\leq C(\beta,T) \exp(-\tfrac{\theta^2}{C}(t+1)),
	\quad
	z=z(\theta) \in\magicC(t,\beta),
\end{align*}
for all $ \theta\in (-\pi,\pi] $, large enough $ t\leq \e^{-2}T $, and small enough $ \e>0 $.
\end{lem}
\begin{proof}
The proof follows the same three-step procedure as the proof of Lemma~\ref{lem:SGtbd}.
Given the identities~\eqref{eq:SGte:D}--\eqref{eq:SGte:DDD},
the proof of the first two steps~\ref{enu:3step:theta=0}--\ref{enu:3step:small}
follows the same argument via Taylor expansion as in Lemma~\ref{lem:SGtbd},
and we do not repeat it here.

We now focus on establishing the last step~\ref{enu:3step:large}.
First, the contour $ \magicC(t,\beta) $ converges, as $ t\to\infty $, to $ \magicC $.
More precisely, write $ z_{\magicC(t,\beta)}(\theta;t,\beta) $ and $ z_{\magicC}(\theta) $
for the respectively polar parametrization as depicted in Figures~\ref{fig:magicC} and \ref{fig:MagicC}.
We have $ \lim_{t\to\infty }z_{\magicC(t,\beta)}(\theta;t,\beta) = z_{\magicC}(\theta) $, uniformly over $ \theta\in(-\pi,\pi] $.
This being the case,
from the given expressions~\eqref{eq:SGte}--\eqref{eq:SGteLim}, \eqref{eq:SGtt} and \eqref{eq:SGttLim} of $ \SGte(t,z) $, $ \SGteLim(z) $, $ \SGtt(z) $, and $ \SGttLim(z) $,
it is readily checked that
\begin{align*}
	\lim_{t\to\infty } \big|\SGte(t,z_{\magicC(t,\beta)}(\theta))\big|^{\frac1t}
	&=
	\big| \SGteLim(z_{\magicC}(\theta)) \big|,
\\	
\lim_{t\to\infty } \big| \SGtt(t,z_{\magicC(t,\beta)}(\theta)) \big|^{\frac1t}
	&=
	\big|\SGttLim(z_{\magicC}(\theta))\big|,
\end{align*}
uniformly over $ \theta\in(-\pi,\pi] $.
The limiting expressions on the r.h.s.\ put us into the considerations of the steepest decent condition~\eqref{eq:steep:MagicC},
which has been verified in Lemma~\ref{lem:magicc}.
From this we conclude the desired conclusion: there exists $ t_0<\infty $ such that, for any given $ \theta_0>0 $,\\
$
	\hphantom{aaaaa}
	\displaystyle
	\big|\SGte(t,z)\big|^{\frac1t}
	\leq
	1 -\tfrac1{C(\theta)},
	\quad
	\big| \SGte(t,z) \big|^{\frac1t}
	\leq
	1 - \tfrac1{C(\theta)},
	\quad
	\forall z= z_{\magicC(t,\beta)}(\theta) \in \magicC(t,\beta),
	\
	|\theta| \geq \theta_0.
$
\end{proof}

We have all the necessary ingredients for estimating $ \SGIn $.
\begin{proof}[Proof of Proposition~\ref{prop:SG:}\ref{enu:SGIn}--\ref{enu:SGIng}, the $ (+-) $-case, with large enough $ t $]
The proof begins with the contour deformation described in Section~\ref{sect:SGIn:ov}.
Let us check the condition~\eqref{eq:nopole}.
For a fixed $ z_2\in\Circ_r $, the integrand in~\eqref{eq:SGIn} has poles in $ z_1=0 $, $ z_1=e^{\sqrt\e(\den-1)}b_1 $,
and $ \pole(z_1)=z_2 $.
Referring to the definition~\eqref{eq:magicC} of $ \magicC(t,-k_1\alpha) $
(or Figure~\ref{fig:magicC}), we see that the first two poles are contained in $ \magicC(t,-k_1\alpha) $.
As for the pole $ \pole(z_1)=z_2 $, the function $ \pole(z) $ (defined in~\eqref{eq:pole})
is uniformly bounded (in $ (\e,z) $) away from $ z=0 $.
This being the case, by making $ r $ large enough, we ensure that $ |\pole(z_1)| < r=|z_2| $
throughout the contour deformation $ z_1\in\Circ_{r} \mapsto \magicC(t,-k_1\alpha) $.
Having checked the condition~\eqref{eq:nopole},
we are now given the decomposition $ \SGIn = \SGbk + \SGres $.
The proof amounts to bounding $ \SGbk $ and $ \SGres $, as well as their gradients.

We begin with $ \SGbk $~\eqref{eq:SGbk}.
The proof consists of a sequence of bounds on terms appearing in the integrand~\eqref{eq:SGbk}.
In the following we  assume $ z_1=z_1(\theta_1)\in\magicC(t,-k_1\alpha) $ and $ z_2=z_1(\theta_2)\in\Circ_{\tilde{r}_2(z_1)} $.
\medskip
\begin{itemize}
\setlength\itemsep{2pt}
\myitem{($ \SGbk $.$ z_1 $)} \label{enu:+-SGbk:1}
	Show that $ |z_1|^{x_2-y_1+\mue t-\muet} \leq \exp(-\frac{\alpha|x_2-y_1|}{\sqrt{t+1}+C(\alpha)}) $:\\
	Referring to the definition~\eqref{eq:magicC} of $ \magicC(t,-k_1\alpha) $
	(or Figure~\ref{fig:magicC}),
	we see that $ \magicC(t,-k_1\alpha) $ is contained in $ \Circ_{\rad(t,-k_1\alpha)} $, so
	\begin{align*}
		|z_1|^{x_2-y_1+\mue t-\muet}
		\leq
		\rad(t,-k_1\alpha)^{|x_2-y_1|}
		\leq
		C(\alpha) e^{-\frac{\alpha|x_2-y_1|}{\sqrt{t+1}+C(\alpha)}}.
	\end{align*}	
\myitem{($ \SGbk $.$ z_2 $)} Show that $ |z_2|^{x_1-y_2+\mue t-\muet} \leq C(\alpha)\exp(-\frac{\alpha|x_1-y_2|}{\sqrt{t+1}+C(\alpha)}) $:\\
	Recall the current assumption $ x_1-y_2\leq 0 $.
	The power $ x_1-y_2+\mue t-\muet $ would have a definitive sign (i.e., non-positive) if we offset it by $ -(\mue t-\muet) $.
	Since $ |z_2| \leq C(\alpha) $ is bounded along its contour $ z_2\in\Circ_{\tilde{r}_2(z_1)} $,
	offsetting the exponent costs only a factor of $ C(\alpha) $:
	\begin{align*}
		|z_2|^{x_1-y_2+\mue t-\muet}	\leq C(\alpha) |z_2|^{x_1-y_2} = C(\alpha)|z_2|^{-|x_1-y_2|}.
	\end{align*}
	Recall the definitions of the $ r_2 $'s and of $ \tilde{r}_2 $ from~\eqref{eq:r2s}--\eqref{eq:tilr2},
	and recall that $ k_2:=1 $ here.
	We see that $ \tilde{r}_2(z_2) \geq \rad(t,\alpha) $, so
	\begin{align*}
		|z_2|^{x_1-y_2+\mue t-\muet}	
		\leq
		C(\alpha) \rad(t,\alpha)^{-|x_1-y_2|}
		\leq
		C(\alpha) e^{\frac{-\alpha|x_1-y_2|}{\sqrt{t+1}+C(\alpha)}}.
	\end{align*}
\myitem{($ \SGbk $.$ \SGfre $)} Show that $ |\SGfre(z_1,z_2)| \leq C(\alpha)(1+|\theta_1|\sqrt{t+1}+|\theta_2|\sqrt{t+1}) $:\\
Recall the expression~\eqref{eq:SGfre} of $ \SGfre $, and rewrite it as
\begin{align}
	\label{eq:SGfre:}
	\SGfre(z_1,z_2)
	=
	1+
	\big(e^{\sqrt{\e}(\den-1)}+e^{\sqrt{\e}\den}\big)
	\frac{ z_2/z_1-1 }{ z_2 -\pole(z_1) }.
\end{align}
Referring to the definition~\eqref{eq:magicC} of $ \magicC(t,-k_1\alpha) $ (or Figure~\ref{fig:magicC}),
we see that $ \magicC(t,-k_1\alpha) $ coincides with the circle $ \Circ_{\rad(t,-k_1\alpha)} $ for small $ \theta_1 $,
i.e., $ |\theta_1|\leq\phi_1^*  $, fixed $ \phi^*_1>0 $.
Also $ z_2(\theta_2) = \tilde{r}_2(z_1)e^{\img\theta_2} = \rad(t,(2\pm 1)\alpha)e^{\img\theta_2} $,
where the $ \pm $ depends on whether $ \pole(z_1(\theta_1))>r'_2 $ or not.
Taylor expanding $ (z_2/z_1-1) $ in $ \theta_1,\theta_2 $ then yields
\begin{align}
	\notag
	|z_2/z_1-1| &\leq \tfrac{C(\alpha)}{\sqrt{t+1}} + C(\alpha)|\theta_2-\theta_1|
\\
	\label{eq:z2ovz1-1}
	&\leq \tfrac{C(\alpha)}{\sqrt{t+1}} + C(\alpha)|\theta_1|+C(\alpha)|\theta_2|.
\end{align}
for all $ \theta_1 $ and $ \theta_2 $ small enough.
Further, since both $ |z_1| $ and $ |z_2| $ are bounded away from $ 0 $ and $ \infty $ along their relevant contours,
the bound~\eqref{eq:z2ovz1-1} actually extends to all values of $ \theta_1,\theta_2 $.
Using~\eqref{eq:z2ovz1-1} and~\eqref{eq:z2-polez1} on the r.h.s.\ of~\eqref{eq:SGfre:},
we conclude the desired bound on $ |\SGfre(z_1,z_2)| $.
\myitem{($ \SGbk $.$ \SGte $)} \label{enu:+-SGbk:4}
	Show that $ |\SGte(z_i)| \leq C(\alpha,T)\exp(-\frac{\theta^2_i}{C}(t+1)) $:\\
	This is the content of Lemma~\ref{lem:SGtbd:+-}.
\end{itemize}
Expressing~\eqref{eq:SGbk} as an integral over $ (\theta_1,\theta_2)\in (-\pi,\pi]^2 $,
and inserting the bounds from~\ref{enu:+-SGbk:1}--\ref{enu:+-SGbk:4} into the resulting expression, we arrive at
\begin{align*}
	|\SGbk| \leq
	C(\alpha,T)
	\int_{-\pi}^{\pi} \int_{-\pi}^{\pi}
	 e^{ -\frac{\alpha(|x_2-y_1|+|x_1-y_2|)}{\sqrt{t+1}+C(\alpha)} }
	 (1+\sqrt{t+1}|\theta_1|+\sqrt{t+1}|\theta_2|)
	e^{-\frac1C(t+1)\theta_i^2}d\theta_i.
\end{align*}
Performing the change of variables $ \sqrt{t+1}\theta_i \mapsto \theta_i $,
and extending the integration domain to $ \R^2 $ (which only increases its value) we obtain the desired bound on $ |\SGbk| $:
\begin{align*}
	|\SGbk| \leq
	C(\alpha,T)
	e^{ -\frac{\alpha(|x_2-y_1|+|x_1-y_2|)}{\sqrt{t+1}+C(\alpha)} }
	\frac{1}{t+1}\int_{\R^2}
	(1+|\theta_1| +|\theta_2|) e^{-\frac1C\theta_i^2}d\theta_i	
	=
	\frac{C(\alpha,T)}{t+1}
	 e^{ -\frac{\alpha(|x_2-y_1|+|x_1-y_2|)}{\sqrt{t+1}+C(\alpha)} }.
\end{align*}

As for $ \SGres $, the proof similarly consists of bounds on terms involved in the integral~\eqref{eq:SGres}.
In the following we always assume $ z_1=z_1(\theta_1)\in\magicC(t,-k_1\alpha) $.
\medskip
\begin{itemize}
\setlength\itemsep{2pt}
\myitem{($ \SGres $.$ \frac{1}{z_1\pole} $)} \label{enu:+-SGres:1}
	Show that $ \frac{1}{|\pole(z_1)||z_1|} \leq C(\alpha) $:\\
	Referring to the definition~\eqref{eq:magicC} of $ \magicC(t,-k_1\alpha) $ (or Figure~\ref{fig:magicC}),
	we see that $ |z_1| $ is bounded away from $ 0 $ and $ \infty $ along $ \magicC(t,-k_1\alpha) $.
	This being the case, referring to the definition~\eqref{eq:pole} of $ \pole(z) $, we see that the same holds for $ |\pole(z_1)| $.
	Hence the claim follows.
\myitem{($ \SGres $.$ \zp $)}\label{enu:+-SGres:2}
	Show that $ \ind_{\set{|\pole(z_1)|>r'_2}} |\zp(z_1)| \leq \exp(-\frac{\alpha(|x_2-y_1|+|x_1-y_2|)}{\sqrt{t+1}+C(\alpha)}) $:\\
	Recall from~\eqref{eq:zp} that $ \zp(z_1) $ consists of products of powers of $ z_1 $ and $ \pole(z_1) $.
	As argued in the previous step~\ref{enu:+-SGres:1},
	the terms $ |z_1|,|z_1|^{-1}, |\pole(z_1)|, |\pole(z_1)|^{-1} \leq C(\alpha) $ are bounded along $ \magicC(t,-k\alpha_1) $.
	This being the case, we alter the powers in~\eqref{eq:zp} by some fixed amount, at the cost of $ C(\alpha) $,
	and write
	\begin{align*}
		\ind_{\set{|\pole(z_1)|>r'_2}} |\zp(z_1)|
		\leq
		C(\alpha) \ind_{\set{|\pole(z_1)|>r'_2}} |z_1|^{|x_2-y_1|} |\pole(z_1)|^{-|x_1-y_2|}.
	\end{align*}
	In the last expression, using $ |z_1|\leq \rad(t,-k_1\alpha) $ (as argued in~\ref{enu:+-SGbk:1}) and the given constraint $ |\pole(z_1)| > r'_2 = \rad(t,2\alpha) $,
	we obtained the desired property:
	\begin{align*}
		\ind_{\set{|\pole(z_1)|>r'_2}} |\zp(z_1)|
		\leq	
		C(\alpha) \rad(t,-k_1\alpha)^{|x_2-y_1|} \rad(t,\alpha)^{-|x_1-y_2|}
		\leq
		C(\alpha) e^{-\frac{\alpha(|x_2-y_1|+|x_1-y_2|)}{\sqrt{t+1}+C(\alpha)}}.	
	\end{align*}
\myitem{($ \SGres $.$ \SGtt $)} \label{enu:+-SGres:3}
	Show that $ |\SGtt(z_1)| \leq C(\alpha,T)\exp(-\frac{\theta^2_1}{C}(t+1)) $:\\
	This is the content of Lemma~\ref{lem:SGtbd:+-}.
\end{itemize}
Express~\eqref{eq:SGres} as an integral over $ \theta_1\in (-\pi,\pi] $,
and insert the bounds from~\ref{enu:+-SGres:1}--\ref{enu:+-SGres:3} into the resulting expression.
This together the derived lower bound~\eqref{eq:theta1:+-} on $ |\theta_1| $ gives
\begin{align*}
	|\SGres|
	\leq
	C(\alpha,T) e^{-\frac{\alpha(|x_2-y_1|+|x_1-y_2|)}{\sqrt{t+1}+C(\alpha)}}
	\int_{(-\pi,\pi]}
	\ind\{|\theta_1|\geq \tfrac{1}{C(\alpha)(t+1)^{1/4}}\} e^{-\frac{1}{C}(t+1)\theta_1^2} d\theta_1.
\end{align*}
Extending the integration domain to $ \R $, and performing a change of variable $ \sqrt{t+1}\theta_1\mapsto\theta_1 $ yields
\begin{align*}
	|\SGres|
	\leq
	C(\alpha,T) e^{-\frac{\alpha(|x_2-y_1|+|x_1-y_2|)}{\sqrt{t+1}+C(\alpha)}}
	\frac{1}{\sqrt{t+1}}\int_{\R}
	\ind\{|\theta_1|\geq \tfrac{(t+1)^{1/4}}{C(\alpha)}\} e^{-\frac{1}{C}\theta_1^2} d\theta_1.
\end{align*}
Here, unlike in the case for $ \SGFr $,
we get $ \frac{1}{\sqrt{t+1}} $ instead of $ \frac{1}{t+1} $ in front of the integral.
This insufficiency is compensated by having the constraint $|\theta_1|\geq (t+1)^{1/4}/C(\alpha)$.
Indeed,
\begin{align*}
	\int_{\R}
	\ind\{|\theta_1|\geq \tfrac{(t+1)^{1/4}}{C(\alpha)}\} e^{-\frac{1}{C}\theta_1^2} d\theta_1
	\leq
	\exp\big(-\tfrac{1}{C(\alpha)}(t+1)^{1/4}\big),
\end{align*}
and fractional exponentials such as $ \exp(-\frac1{C(\alpha)}(t+1)^{1/4}) $ decay faster than any power $ (t+1)^{-n} $.
From this we conclude the desired bound on $ |\SGres| $:
\begin{align*}
	|\SGres|
	\leq
	\frac{C(\alpha,T)}{t+1} e^{-\frac{\alpha(|x_2-y_1|+|x_1-y_2|)}{\sqrt{t+1}+C(\alpha)}}.
\end{align*}

So far we have derived bounds on $ |\SGbk| $ and $ |\SGres| $, and this concludes
the proof of Proposition~\ref{prop:SG:}\ref{enu:SGIn}.
Part~\ref{enu:SGIng} amounts to performing similar estimates on the gradients, e.g., $ |\nabla_{x_j}\SGbk| $ and $ |\nabla_{x_j}\SGres| $.
Taking a gradient merely introduces a factor of $ (z_j^\pm-1) $ in the contour integrals~\eqref{eq:SGbk}--\eqref{eq:SGres}.
It is straightforward to check that
\begin{align}
	\label{eq:zpm-1}
	 |z_j^\pm-1| \leq \tfrac{1}{\sqrt{t+1}} + |\theta_j|,
	 \quad
	 z_1=z_1(\theta) \in \magicC(t,-k_1\alpha),
	 \quad
	 z_2=z_2(\theta) \in \Circ_{\tilde{r}_2(z_1)}.	
\end{align}
Incorporate this bound into the preceding analysis gives the desired bounds on the gradients.
Compared to the bounds on~$ |\SGbk| $ and $ |\SGres| $,
an additional factor of $ \frac{1}{\sqrt{t+1}} $ arises due to~\eqref{eq:zpm-1}.
\end{proof}

\subsection{Estimating the interacting part $ \SGIn $, the $ (--) $-case}
\label{sect:--}
The case considered here is more involved than the $ (+-) $-case:
we face a conflict in the choice of the $ z_1 $-contour.
As discussed in Section~\ref{sect:SGIn:ov}, in order to control the term $ \SGtt(t,z) $ in $ \SGres $ by steepest decent analysis,
we favor contours of the type $ \magicC(t,\beta) $.
On the other hand, with $ x_2-y_1 \leq 0 $ under current assumptions,
we need $ |z_1| >1 $ in $ \SGbk $ to obtain the desired spatial exponential decay $ \exp(-\frac{\alpha|x_2-y_1|}{\sqrt{t+1}+C(\alpha)}) $.
Referring to the definition~\eqref{eq:magicC} of $ \magicC(t,-k_1\alpha) $ (or Figure~\ref{fig:magicC}),
we see that $ |z_1|>1 $ \emph{fails} for a portion of $ \magicC(t,\beta) $, regardless of the sign of $ \beta $---i.e., the bulk part $ \SGbk $ and the residue part $ \SGres $ favor different contours.

In view of the preceding discussion, we choose
\begin{align*}
	\zoneC := \Circ_{\rad(t,3\alpha)},
\end{align*}
which is preferred for controlling $ \SGbk $ but not $ \SGres $,
and then, \emph{re-deforming} contour $ \Circ_{\rad(t,3\alpha)} \mapsto \magicC(t,3\alpha) $ in $ \SGres $.
Let us check that doing so does not cross a pole.
\begin{lem}
\label{lem:redeform:--}
For all $ t>0 $ large enough and $ \e>0 $ small enough, we have $ \SGres=\SGres' $, where
\begin{align}
	\label{eq:SGres'}
	\SGres' &:=
	\oint_{\magicC(t,3\alpha)} \ind_{\set{|\pole(z_1)|>r'_2}}  	 \zp(z_1) \frac{1}{ z_1 \pole(z_1)} \SGtt(t,z_1)dz_1
\end{align}
is the same as $ \SGres $ except the contour is replaced by $ \magicC(t,3\alpha) $.
\end{lem}
\begin{proof}
Referring to the definition~\eqref{eq:magicC} of $ \magicC(t,3\alpha) $ (or Figure~\ref{fig:magicC}),
we see that the difference $ \Circ_{\rad(t,3\alpha)}-\magicC(t,3\alpha) $ is the boundary of the crescent
\begin{align*}
	\mathcal{G}(t) := \big\{z\in\C: |z|\leq\rad(t,3\alpha)\big\} \setminus \big\{z\in\C:|z-\tfrac12|<\tfrac12+u_*\big\}.
\end{align*}
See Figure~\ref{fig:crescent}.
We write $ \partial\mathcal{G}(t) $ for the boundary, counterclockwise oriented.
This gives
\begin{align*}
	\SGres-\SGres' =
	\frac{e^{\sqrt{\e}(\den-1)}+e^{\sqrt{\e}\den}}{2\pi\img}
	\oint_{\partial\mathcal{G}(t)}
 	\ind_{\set{|\pole(z_1)|>r'_2}} \zp(z_1) \frac{1}{ z_1 \pole(z_1)} \SGtt(t,z_1)dz_1.
\end{align*}
Along $ \partial\mathcal{G}(t) $, the indicator $ \ind_{\set{|\pole(z_1)|>r'_2}} $ is in fact irrelevant.
More precisely, setting
\begin{align}
	\label{eq:calH}
	\mathcal{H}(t,\e,\beta) := \{ |\pole(z)| \leq \rad(t,\beta) \},
\end{align}
let us check that
\begin{align}
	\label{eq:redeform:--:claim}
	\text{ given any } \beta\in\R \text{ and } u>0,
	\
	\mathcal{H}(t) \subset \big\{ |z-\tfrac12| \leq \tfrac12 +u \big\},
\end{align}
for all $ t $ large enough and $ \e $ small enough.
Referring to the definition~\eqref{eq:pole} of $ \pole(z) $, we have $ \lim_{\e\to 0} \pole(z) := \poleLim(z) = 2-z^{-1} $.
Consider $ \mathcal{H}_* := \{z\in\C: |\poleLim(z)| \leq 1 \} $, which is the $ (t,\e)\to(\infty,\e) $ limit of $ \mathcal{H}(t,\e,\beta) $.
Indeed, along the contour $ \Magic:=\frac12+\Circ_\frac12 $,we have $ |z\poleLim(z)|=2|z-\frac12|=1 $ and $ |z|>1 $ except when $ z=1 $.
Consequently, $ \Magic\cap\mathcal{H}_* = \{1\} $.
Also, it is readily checked that $ \frac12 \in \mathcal{H}_* $ and that $ \mathcal{H}_* $ is connected.
From these properties, we deduce that $ \mathcal{H}_* \subset \{|z-\frac12| \leq \frac12 \} $.
Since $ \mathcal{H}_* $ is the $ (t,\e)\to(\infty,0) $ limit of $ \mathcal{H}(t,\e,\beta) $,
for all $ t $ large enough and $ \e $ small enough, the claim~\eqref{eq:redeform:--:claim} follows.

\begin{figure}[h]
\includegraphics[width=.35\linewidth]{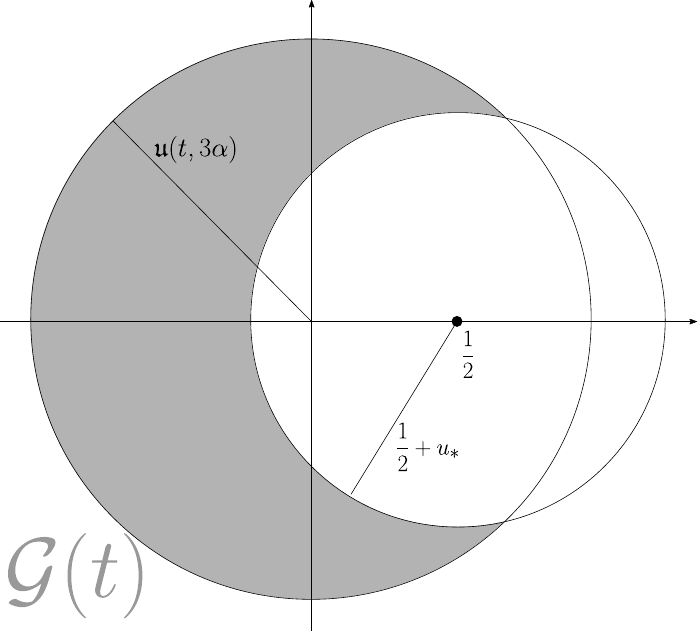}
\caption{The region $ \mathcal{G}(t) $.}
\label{fig:crescent}
\end{figure}

Given~\eqref{eq:redeform:--:claim}, we drop the indicator $ \ind_{\set{|\pole(z_1)|>r'_2}} $ and write
\begin{align*}
	\SGres'-\SGres =
	\frac{e^{\sqrt{\e}(\den-1)}+e^{\sqrt{\e}\den}}{2\pi\img}
	\oint_{\partial\mathcal{G}(t)}
 	  \frac{\zp(z_1)}{ z_1 \pole(z_1)} \SGtt(t,z_1)dz_1.
\end{align*}
Our goal is to show that the integral is zero.
To this end, set $ \polee(z) := z\pole(z) := (e^{\sqrt\e\den}+e^{\sqrt\e(\den-1)})z - e^{\sqrt\e(2\den-1)} $,
recall the definition of $ \zp(z_1) $ from~\eqref{eq:zp}, that $ \SGtt(t,z_1):= \SGte(t,z_1) \SGte(t,\pole(z_1)) $,
and recall the definition of $ \SGte(t,z) $ from~\eqref{eq:SGt}.
We express the integrand of the last integral as
\begin{align}
	\notag
	\frac{\zp(z_1)}{ z_1 \pole(z_1)} \SGtt(t,z_1)
	=&
	\big( z_1^{(x_2-y_1)-(x_1-y_2)-2+(\mue t - \muet)} - z_1^{(x_2-y_1)-(x_1-y_2)+(\mue t - \muet)} \big) \polee(z_1)^{x_1-y_2-1+\muet},
\\
	\label{eq:zp:}
	&\Big(
			\lambdae \Big(\frac{z_1b_1+(1-b_1-b^\e_2)\taue^{-\den}}{z_1-b^\e_2\taue^{-\den}}
		\Big)\Big)^t
	\Big(
		\lambdae \Big(\frac{\pole(z_1)b_1+(1-b_1-b^\e_2)\taue^{-\den}}{\pole(z_1)-b^\e_2\taue^{-\den}}
	\Big)\Big)^t.
\end{align}
It suffices to check that this expression has no poles within $ z_1\in\mathcal{G}(t) $.
The assumption $ x_1<x_2 $ and $ y_1<y_2 $ ensures that $ (x_2-y_1)-(x_1-y_2) \geq 2 $.
Thus,
the expression~\eqref{eq:zp:} can only have poles at $ \polee^{-1}(0) $, $ b^\e_2\taue^{-\den} $, or $ \pole^{-1}(b^\e_2\taue^{-\den}) $.
With $ b^\e_2\to b_1 $, $ \taue\to 1 $, and $ \pole(z)\to 2-z^{-1} $, we have
\begin{align*}
	\polee^{-1}(0) \longrightarrow \tfrac12,
	\quad
	b^\e_2\taue^{-\den} \longrightarrow b_1,
	\quad
	\pole^{-1}(b^\e_2\taue^{-\den}) \longrightarrow \tfrac{1}{2-b_1},
	\quad
	\text{as }\e\to 0.
\end{align*}
Referring to Figure~\ref{fig:crescent}, we see that $ \frac12 $, $ b_1 $, and $ \frac{1}{2-b_1} $,
all sit strictly outside of $ \mathcal{G}(t) $.
Consequently, no poles enter into $ \mathcal{G}(t) $ as long as $ t>0 $ is large enough and $ \e>0 $ is small enough.
\end{proof}

Having introduced the contours $ \Circ_{\rad(t,3\alpha)} $ and $ \magicC(t,3\alpha) $,
hereafter we write $ z_1(\theta_1) = \rad(t,3\alpha)e^{\img\theta_1} \in \Circ_{\rad(t,3\alpha)} $,
and write $ \tilde{z}_1(\theta_1)\in\magicC(t,3\alpha) $ for the parametrization depicted in Figure~\ref{fig:magicC}.

To control $ \SGres $ in the following,
similarly to the $ (+-) $-case done previously, we need the analogous condition~\eqref{eq:cntur:intersect} to hold:
\begin{align}
	\tag{\ref{eq:cntur:intersect}'}
	\label{eq:cntur:intersect:--}
	r'_2 \geq \pole(\tilde{z}_1(0)) + \tfrac{1}{\sqrt{t+1}} \in \R.
\end{align}
We achieve this by making the auxiliary parameter $ k_2\in\Z_{>0} $ in~\eqref{eq:r2s} large enough.
Recall from Definition~\ref{def:Taylor} the announced convention on Taylor expansion,
and expand the expression $ r'_2-\pole(\tilde{z}_1(0)) = \rad(t,2k_2\alpha) - \pole(\rad(t,3\alpha)) $ in $ (\sqrt\e,\frac{1}{\sqrt{t+1}}) $
to the leading order in $ \frac{1}{\sqrt{t+1}} $ to get
\begin{align*}
	z_2(0)-\pole(\tilde{z}_1(0)) = 0\cdot\sqrt\e - \den(1-\den)\e + \tfrac{(k_2-3)\alpha}{\sqrt{t+1}} + \ldots.
\end{align*}
With $ \e \leq \frac{C(T)}{\sqrt{t+1}} $,
from the expansion we see that~\eqref{eq:cntur:intersect:--} does hold for some large enough $ k_2=k_2(\alpha,T) $,
and we fix such a $ k_2\in\Z_{>0} $ hereafter.
Given this condition, following the same procedure of deriving~\eqref{eq:theta1:+-} as in the $ (+-) $-case, here we have
\begin{align}
	\tag{\ref{eq:theta1:+-}'}
	\label{eq:theta1:--}
	|\pole(\tilde{z}_1(\theta_1))| > r'_2
	\text{ holds only if }
	|\theta_1| \geq \frac{1}{C(\alpha)(t+1)^{1/4}}.
\end{align}

\begin{proof}[Proof of Proposition~\ref{prop:SG:}\ref{enu:SGIn}--\ref{enu:SGIng}, the $ (--) $-case, with large enough $ t $]
The proof begins with the contour deformation described in Section~\ref{sect:SGIn:ov}.
The condition~\eqref{eq:nopole} is checked the same way as in the $ (+-) $-case,
which gives the decomposition $ \SGIn = \SGbk + \SGres $.
We next perform the aforementioned re-deformation $ \Circ_{\rad(t,3\alpha)} \mapsto \magicC(t,3\alpha) $ in $ \SGres $.
Lemma~\ref{lem:redeform:--} ensures that no pole is crossed during this step, giving $ \SGIn = \SGbk + \SGres' $.

The proof amounts to bounding $  \SGbk $, $ \SGres' $, and their gradients.
We begin with $ \SGbk $, given by the integral expression~\eqref{eq:SGbk}.
In the following we check a sequence of bounds on terms involved in~\eqref{eq:SGbk},
and we always assume $ z_1=z_1(\theta_1)\in\Circ_{\rad(t,3\alpha)} $ and $ z_2=z_1(\theta_2)\in\Circ_{\tilde{r}_2(z_1)} $
in the course of doing so.
\medskip
\begin{itemize}
\setlength\itemsep{2pt}
\myitem{($ \SGbk $.$ z_1 $)} \label{enu:--SGbk:1}
	Show that $ |z_1|^{x_2-y_1+\mue t-\muet} \leq \exp(-\frac{\alpha|x_2-y_1|}{\sqrt{t+1}+C(\alpha)}) $:\\
	This is so because $ |z_1| = \rad(t,3\alpha) $ and $ x_2-y_1 \leq 0 $ under current assumptions.
\myitem{($ \SGbk $.$ z_2 $)}
Show that $ |z_2|^{x_1-y_2+\mue t-\muet} \leq C(\alpha)\exp(-\frac{\alpha|x_1-y_2|}{\sqrt{t+1}+C(\alpha)}) $:\\
	With $ k_2 \geq 1 $ and with $ \tilde{r}_2 $ defined in \eqref{eq:tilr2},
	we have $ |z_2| \geq \rad(t, k_2\alpha) \geq \rad(t,\alpha) $.
	This and the assumption $ x_1-y_2 \leq 0 $ gives the desired claim.
\myitem{($ \SGbk $.$ \SGfre $)} Show that $ |\SGfre(z_1,z_2)| \leq C(\alpha)(1+|\theta_1|\sqrt{t+1}+\theta_2\sqrt{t+1}) $:\\
	This is established by the same argument as in the $ (+-) $-case.
	We do not repeat it here.
\myitem{($ \SGbk $.$ \SGte $)} \label{enu:--SGbk:4}
	Show that $ |\SGte(z_i)| \leq C(\alpha,T)\exp(-\frac{\theta^2_i}{C}(t+1)) $:\\
	This is the content of Lemma~\ref{lem:SGtbd}.
\end{itemize}
Given \ref{enu:--SGbk:1}--\ref{enu:--SGbk:4},
the desired bound on $ \SGbk $ follows by inserting the bounds into~\eqref{eq:SGbk}, and integrating the result.
The procedure is the same as the $ (+-) $-case, and we do not repeat it here.

We now turn to $ \SGres $. In the following we always assume $ \tilde{z}_1=\tilde{z}_1(\theta_1)\in\magicC(t,3\alpha) $.
\medskip
\begin{itemize}
\setlength\itemsep{2pt}
\myitem{($ \SGres' $.$ \frac{1}{z_1\pole} $)} \label{enu:--SGres:1}
	Show that $ \frac{1}{|\pole(\tilde{z}_1)\tilde{z}_1|} \leq C(\alpha) $:\\
	Referring to the definition~\eqref{eq:magicC} of $ \Magic(t,3\alpha) $ (or Figure~\ref{fig:magicC}),
	we see that $ |\tilde{z}_1| $ is bounded away from $ 0 $ and $ \infty $ along $ \magicC(t,3\alpha) $.
	This being the case, referring to the definition~\eqref{eq:pole} of $ \pole(z) $,
	the same holds for $ |\pole(\tilde{z}_1)| $.
\myitem{($ \SGres' $.$ \zp $)}\label{enu:--SGres:2}
	Show that $ |\zp(\tilde{z}_1)| \leq C(\alpha)\exp(-\frac{\alpha(|x_2-y_1|+|x_1-y_2|)}{\sqrt{t+1}+C(\alpha)}) $:\\
	Recall from~\eqref{eq:zp} that $ \zp(z_1) $ consists of products of powers of $ \tilde{z}_1 $ and $ \pole(\tilde{z}_1) $.
	As argued in the previous step~\ref{enu:--SGres:1},
	the terms $ |\tilde{z}_1|,|\tilde{z}_1|^{-1}, |\pole(\tilde{z}_1)|, |\pole(\tilde{z}_1)|^{-1} \leq C(\alpha) $ are bounded along $ \magicC(t,3\alpha) $.
	This being the case, we alter the powers~\eqref{eq:zp} in by some fixed amount, at the cost of $ C(\alpha) $,
	and write
	\begin{align*}
		|\zp(\tilde{z}_1)|
		\leq
		C(\alpha) |\tilde{z}_1|^{-|x_2-y_1|} |\pole(\tilde{z}_1)|^{-|x_1-y_2|}.
	\end{align*}
	
	Set $ n_1:=|x_2-y_1| $ and $ n_2:=|x_1-y_2| $.
	Instead of bounding $ |\tilde{z}_1|^{-n_1} $ and $ |\pole(\tilde{z}_1)|^{-n_2} $ separately,
	here we need to `bundle' part of them together.
	The assumption $ y_1<y_2 $, $ x_1<x_2 $ in the $ (--) $-case yields $ n_2>n_1 $.
	Given this, we write
	\begin{align*}
		|\zp(\tilde{z}_1)|
		\leq
		C(\alpha) |\tilde{z}_1|^{-n_1} |\pole(\tilde{z}_1)|^{-n_2}
		=
		C(\alpha) |\tilde{z}_1\pole(\tilde{z}_1)|^{-n_1} |\pole(\tilde{z}_1)|^{-(n_2-n_1)}.
	\end{align*}
	We claim that, for all $ t\leq \e^{-2}T $ large enough and $ \e>0 $ small enough,
	\begin{align}
		\label{eq:--sGres:claim}
		|\tilde{z}_1\pole(\tilde{z}_1)| \geq \rad(t,\alpha),\quad |\pole(\tilde{z}_1)| \geq \rad(t,2\alpha),
		\quad
		\tilde{z}_1 \in \magicC(t,3\alpha).
	\end{align}	
	Once these bounds are established, it follows that
	\begin{align*}
		|\zp(\tilde{z}_1)|
		\leq
		C(\alpha)  \rad(t,\alpha)^{-n_1} \rad(t,2\alpha)^{-(n_2-n_1)}
		\leq
		C(\alpha) e^{-\frac{2\alpha n_1}{\sqrt{t+1}+C(\alpha)}} e^{-\frac{\alpha(n_2-n_1)}{\sqrt{t+1}+C(\alpha)}}.
	\end{align*}
	This concludes the desired bound on $ |\zp(\tilde{z}_1)| $, and it hence suffices to verify the claim~\eqref{eq:--sGres:claim}.
	
	Recall from~\eqref{eq:magicC} that $ \magicC(t,3\alpha) $ is given by $ \Circ_{\rad(t,3\alpha)} $ near $ z=1 $,
	and the rest by $ \tilde{\Magic}:=\{|z-\frac12|=\frac12+u_*\} $.
	With this in mind, let us check the bounds separately on $ \Circ_{\rad(t,3\alpha)} $ and $ \tilde{\Magic} $.

	We begin with  $ \Circ_{\rad(t,3\alpha)} $.
	Adopt the parametrization $ \Circ_{\rad(t,3\alpha)}\ni \tilde{z}_1(\theta_1) = \rad(t,3\alpha) e^{\img\theta_1} $ and write
	\begin{align}
		\label{eq:zpz1}
		\tilde{z}_1\pole(\tilde{z}_1) &= \rad(t,3\alpha) e^{\img\theta_1}(e^{\sqrt{\e}\den}+e^{\sqrt{\e}(\den-1)}) - e^{\sqrt{\e}(2\den-1)},
	\\
		\label{eq:zpz2}
		\pole(\tilde{z}_1) &= (e^{\sqrt{\e}\den}+e^{\sqrt{\e}(\den-1)}) - e^{\sqrt{\e}(2\den-1)}\rad(t,-3\alpha) e^{-\img\theta_1}.	
	\end{align}
	As $ \theta_1 $ varies, the r.h.s.\  of~\eqref{eq:zpz1}--\eqref{eq:zpz2} trace out circles,
	denoted by $ \tilde{\Circ}(t,\e) $ and $ \tilde{\Circ}'(t,\e) $ respectively.
	The circle $ \tilde{\Circ}(t,\e) $ is centered at a point in $ (-\infty,0) $.
	For such circles, the nearest point to the origin occurs at the right-end.
	This gives
	\begin{align*}
		\inf_{\tilde{z}_1\in\Circ_{\rad(t,3\alpha)}} |\tilde{z}_1\pole(\tilde{z}_1)|
		=
		\rad(t,3\alpha) (e^{\sqrt{\e}\den}+e^{\sqrt{\e}(\den-1)}) - e^{\sqrt{\e}(2\den-1)}.
	\end{align*}
	A similarly geometric reasoning gives
	\begin{align*}
		\inf_{\tilde{z}_1\in\Circ_{\rad(t,3\alpha)}} |\pole(\tilde{z}_1)|
		=
		(e^{\sqrt{\e}\den}+e^{\sqrt{\e}(\den-1)}) - e^{\sqrt{\e}(2\den-1)} \rad(t,-3\alpha).
	\end{align*}	
	To bound the r.h.s., under the convention announced in Definition~\ref{def:Taylor},
	we Taylor expand the r.h.s.\ in $ (\sqrt\e,\frac{1}{\sqrt{t+1}}) $ up to the leading order in $ \frac{1}{\sqrt{t+1}} $ to get
	\begin{align*}
		\rad(t,3\alpha) (e^{\sqrt{\e}(\den-1)}-e^{\sqrt{\e}(\den-1)}) - e^{\sqrt{\e}(2\den-1)}
		=
		1+0\cdot\sqrt{\e} + \den(1-\den)\e + \tfrac{3\alpha}{\sqrt{t+1}}+\ldots,
	\\
		(e^{\sqrt{\e}(\den-1)}-e^{\sqrt{\e}(\den-1)}) - e^{\sqrt{\e}(2\den-1)} \rad(t,-3\alpha)
		=
		1+0\cdot\sqrt{\e} + \den(1-\den)\e + \tfrac{3\alpha}{\sqrt{t+1}}+\ldots.	
	\end{align*}
	From this, together with $ \e \leq \frac{C(T)}{\sqrt{t+1}} $ (because $ t \leq \e^{-2}T $),
	we see that the desired bounds
	$ |\tilde{z}_1\pole(\tilde{z}_1)| \geq \rad(t,\alpha) $, $ |\pole(\tilde{z}_1)| \geq \rad(t,2\alpha) $
	hold on $ \Circ_{\rad(t,3\alpha)} $, for all large enough $ t \leq \e^{-2}T $ and small enough $ \e>0 $.

	We now turn to $ \tilde{\Magic} $.
	Recall that $ \poleLim(z):=2-z^{-1} $ denotes the $ \e\downarrow 0 $ limit of $ \pole(z) $.
	Along the contour $ \tilde{\Magic}:=\{|z-\frac12|=\frac12+u_*\} $ we have $ |z\poleLim(z)|=1+2u_* > 1 $.
	This being the case, the bound $ |\tilde{z}_1\pole(\tilde{z}_1)| \geq \rad(t,\alpha) $ holds on $ \tilde{\Magic} $
	for large enough $ t $.
	The other bound $ |\pole(\tilde{z}_1)| \geq \rad(t,2\alpha) $ follows from~\eqref{eq:redeform:--:claim}.
%	The other bound $ |\pole(\tilde{z}_1)| \geq \rad(t,2\alpha) $ actually does not hold on the entire $ \tilde{\Magic} $,
%	but relevant to our purpose is the part $ \tilde{\Magic}\cap\magicC(t,\beta) := \tilde{\Magic}'(t,\beta) $.
%	Referring to the definition~\eqref{eq:magicC} of $ \magicC(t,\beta) $,
%	we see that $ \tilde{\Magic}'(t,\beta) \subset \{ z:\text{Re}(z) <1-\delta \} $ holds for some fixed $ \delta>0 $,
%	for all $ t>0 $ large enough.
%	This and~\eqref{eq:zplwbd} together gives the desired bound $ |\pole(\tilde{z}_1)| \geq \rad(t,2\alpha) $ on
%	$ \tilde{\Magic}'(t,\beta) $.
	
\myitem{($ \SGres' $.$ \SGtt $)} \label{enu:--SGres:3}
	Show that $ |\SGtt(\tilde{z}_1)| \leq C(\alpha,T)\exp(-\frac{\theta^2_1}{C}(t+1)) $:\\
	This is the content of Lemma~\ref{lem:SGtbd:+-}.
\end{itemize}
Given \ref{enu:--SGres:1}--\ref{enu:--SGres:3}, and the derived constraint~\eqref{eq:theta1:--} on $ |\theta_1| $,
the desired bound on $ \SGres $ follows the same integration procedure as the $ (+-) $-case.

As for the gradient, similarly to the $ (+-) $-case, here we have
\begin{align}
	\tag{\ref{eq:zpm-1}'}
	\label{eq:zpm-1:--}
	 |z_j^\pm-1| \leq \tfrac{1}{\sqrt{t+1}} + |\theta_j|,
	 \quad
	 z_1=z_1(\theta) \in \Circ_{\rad(t,3\alpha)} \text{ or } \magicC(t,3\alpha),
	 \quad
	 z_2=z_2(\theta) \in \Circ_{\tilde{r}_2(z_1)}.
\end{align}
Incorporate this bound into the preceding analysis gives the desired bounds on the gradients.
\end{proof}

\subsection{Estimating the interacting part $ \SGIn $, the $ (++) $-case}
\label{sect:++}
Before heading to the construction of $ \zoneC $, we begin with some general discussion that motivates the construction.
As it turns out, the analysis of the residue part $ \SGres $ favors contours of the type:
\begin{align}
	\label{eq:zplv}
	\zplv(t,\e,\beta)
	:= \Big\{
		\Big| z- \frac{e^{\sqrt\e(2\den-1)}}{e^{\sqrt\e\den}+e^{\sqrt\e(\den-1)}} \Big|
		=
		\frac{\rad(t,\beta)}{e^{\sqrt\e\den}+e^{\sqrt\e(\den-1)}}	
	\Big\}.
\end{align}
First, it is readily checked that $ \zplv(t,\e,\beta) $ is the $ \rad(t,\beta) $-level set of $ |z\pole(z)| $, i.e.,
\begin{align*}
	\zplv(t,\e,\beta)
	= \big\{ |z\pole(z)|=\rad(t,\beta)\big\}.
\end{align*}
This property is useful toward extracting the spatial exponential decay $ \exp(-\frac{\alpha(|x_2-y_1|+|x_1-y_2|)}{\sqrt{t+1}+C(\alpha)}) $
from  $ \SGres $.
Further, $ \zplv(t,\e,\beta) $ is itself a circle,
and, as $ (t,\e)\to(\infty,0) $, converges to $ \Magic := \{|z-\frac12|=\frac12\} $.
With $ \Magic $ satisfying the steepest decent condition~\eqref{eq:steep:Magic},
it is conceivable that $ \SGtt(t,z_1) $ will be controlled along the contour $ \zplv(t,\e,\beta) $.

However, if we choose $ \zoneC $ to be $\zplv(t,\e,\beta) $ (with $ \beta\in\R $),
for all $ \den>\frac12 $, the first stage of contour deformation $ \Circ_r\mapsto \zoneC $
will \emph{inevitably} cross a pole at $ \pole(z_1)=z_2 $ no matter how large $ r $ is.
To avoid this issue, we consider a modification $ \magicc(t,\e,\beta) $ of $ \zplv(t,\e,\beta) $.
This modification is similar to how we modified $ \Magic $ to get $ \Magicc $.
Recall that $ u_*>0 $ is a fixed parameter in the definition of $ \MagicC $ and $ \Magicc $ (see~\eqref{eq:MagicC}--\eqref{eq:Magicc}).
We set
\begin{align}
	\label{eq:magicc}
	\magicc(t,\e,\beta) :=
	\partial \Big(
		\big\{|z-u_*| \leq 2u_*\big\}
		\cup
		\Big\{
		\Big| z- \frac{e^{\sqrt\e(2\den-1)}}{e^{\sqrt\e\den}+e^{\sqrt\e(\den-1)}} \Big|
		\leq
		\frac{\rad(t,\beta)}{e^{\sqrt\e\den}+e^{\sqrt\e(\den-1)}}
		\Big\}
	\Big),
\end{align}
counterclockwise oriented. See Figure~\ref{fig:magicc}.

\begin{figure}[h]
\includegraphics[width=\linewidth]{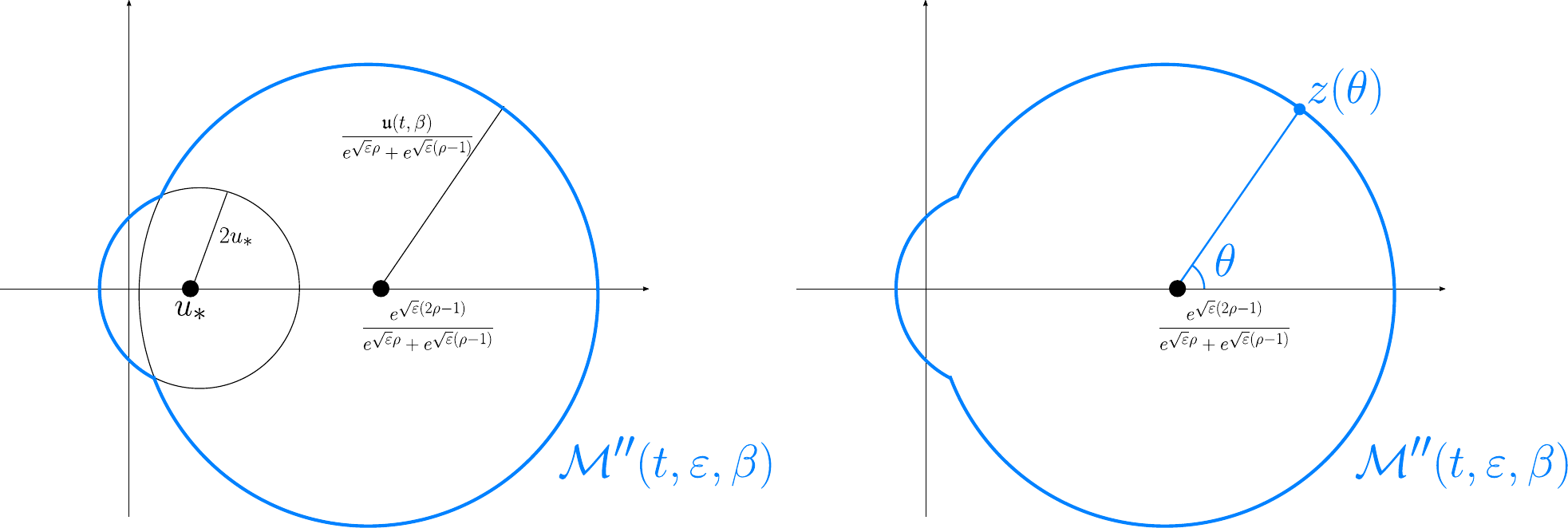}
\caption{The contour $ \magicc(t,\e,\beta) $ and its parametrization.}
\label{fig:magicc}
\end{figure}

We now define the $ z_1 $-contour
\begin{align*}
	\zoneC := \magicc(t,\e,-k_1\alpha).
\end{align*}
As for the $ z_2 $-contour, we fix $ k_2:=1 $ in~\eqref{eq:r2s}.
Recall the definition of $ \tilde{r}_2(z_1) $ from~\eqref{eq:tilr2},
we parametrize $ z_2(\theta):=\tilde{r}_2(z_1)e^{\img\theta_2} \in \Circ_{\tilde{r}_2(z_1)} $.

The auxiliary parameter $ k_1=k_1(\alpha)\in\Z_{\geq 2} $ is in place for technical purpose.
We delay specifying $ k_1 $, and first explain the contour deformation we need here.
Similar to the $ (--) $-case, here we need a re-deformation
$ \magicc(t,\e,-k_1\alpha) \mapsto \zplv(t,\e,-k_1\alpha) $ for $ \SGres $.
As explained earlier, the analysis of $ \SGres $ favors the contour $ \zplv(t,\e,-k_1\alpha) $.
Unfortunately, we could not have chosen $ \zoneC $ to be $ \zplv(t,\e,-k_1\alpha) $ in the first place,
because the bulk part $ \SGbk $ is sensitive to crossing $ z_1=0 $ (due to the pole at $ \pole(z_1)=z_2 $).
On the other hand, $ \SGres $ is \emph{not}.
We utilize this fact to deliver the desired contour to $ \SGres $ via re-deformation.
Let us verify that re-deformation for $ \SGres $ does not cross a pole.
\begin{lem}
\label{lem:redeform:++}
For all $ t>0 $ large enough and $ \e>0 $ small enough, we have $ \SGres=\SGres' $, where
\begin{align}
	\label{eq:SGres''}
	\SGres'' :=
	\oint_{\zplv(t,\e,-k_1\alpha)} \ind_{\set{|\pole(z_1)|>r'_2}} \big(e^{\sqrt{\e}(\den-1)}+e^{\sqrt{\e}\den}\big) \zp(z_1) \frac{\SGtt(t,z_1)dz_1}{2\pi\img z_1 \pole(z_1)}
\end{align}
is the same as $ \SGres $ except the $ z_1 $-contour is replaced by  $ \zplv(t,\e,-k_1\alpha) $.
\end{lem}
\begin{proof}
Referring to the definitions~\eqref{eq:zplv}--\eqref{eq:magicc} of $ \zplv(t,\e,-k_1\alpha) $ and $ \magicc(t,\e,-k_1\alpha) $
(see also Figure~\ref{fig:magicc}),
we see that the difference $ \zplv(t,\e,-k_1\alpha)-\magicc(t,\e,-k_1\alpha) $ is the boundary of the crescent
\begin{align*}
	\mathcal{G}(t,\e)
	:=
	\big\{z\in\C: |z-u_*|\leq 2u_*\big\}
	\setminus
	\Big\{
		\Big| z- \frac{e^{\sqrt\e(2\den-1)}}{e^{\sqrt\e\den}+e^{\sqrt\e(\den-1)}} \Big|
		\leq
		\frac{\rad(t,-k_1\alpha)}{e^{\sqrt\e\den}+e^{\sqrt\e(\den-1)}}
	\Big\}.
\end{align*}
See Figure~\ref{fig:crescent1}.
With $ \partial\mathcal{G}(t,\e) $ denoting the boundary, counterclockwise oriented, we have
\begin{align*}
	\SGres-\SGres'' =
	\frac{e^{\sqrt{\e}(\den-1)}+e^{\sqrt{\e}\den}}{2\pi\img}
	\oint_{\partial\mathcal{G}(t,\e)}
 	\ind_{\set{|\pole(z_1)|>r'_2}} \zp(z_1) \frac{1}{ z_1 \pole(z_1)} \SGtt(t,z_1)dz_1.
\end{align*}
Recall that $ \poleLim(z) = 2-z^{-1} $ denotes the $ \e\to 0 $ limit of $ \pole $.
Since $ \mathcal{G}(t,\e) \subset \{ |z|\leq 3u_* \} $ and $ u_*< \frac1{12} $,
on $ \mathcal{G}(t,\e) $ we have $ |\poleLim(z)| \geq |z|^{-1}-2 \geq 2 $.
Consequently, $ |\pole(z)| >r'_2 $, $ z\in z\in\mathcal{G}(t,\e) $,
%\begin{align*}
%	|\pole(z)| >r'_2,\quad z\in\mathcal{G}(t,\e),
%\end{align*}
for all $ t>0 $ large enough and $ \e>0 $ small enough.
We hence drop the indicator $ \ind_{\set{|\pole(z_1)|>r'_2}} $ and write
\begin{align*}
	\SGres-\SGres'' =
	\frac{e^{\sqrt{\e}(\den-1)}+e^{\sqrt{\e}\den}}{2\pi\img}
	\oint_{\partial\mathcal{G}(t)}
 	 \frac{\zp(z_1)}{ z_1 \pole(z_1)} \SGtt(t,z_1)dz_1.
\end{align*}

\begin{figure}[h]
\includegraphics[width=.55\linewidth]{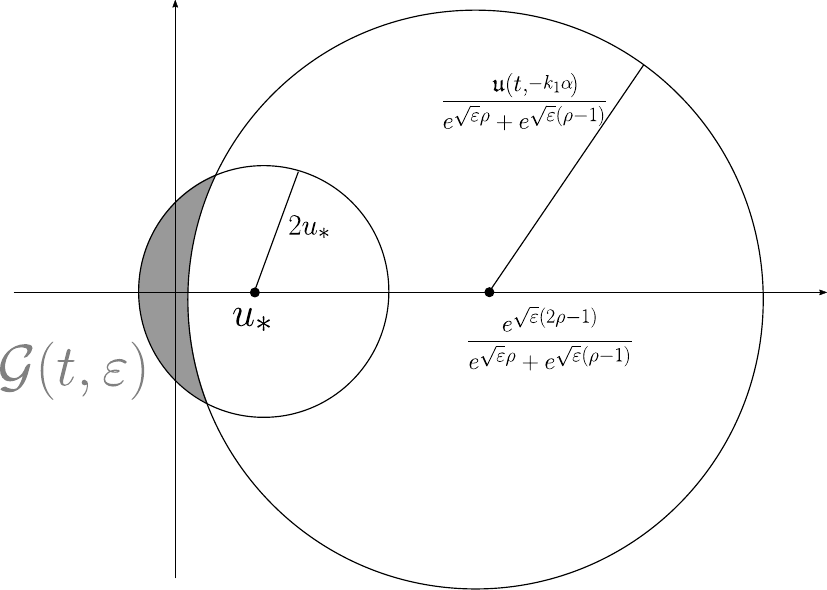}
\caption{The region $ \mathcal{G}(t,\e) $.}
\label{fig:crescent1}
\end{figure}

It suffices to check that the integrand $ \frac{\zp(z_1)}{ z_1 \pole(z_1)} \SGtt(t,z_1) $
has no poles within $ z_1 \mathcal{G}(t,\e) $.
This was carried out in the proof of Lemma~\ref{lem:redeform:--} already.
There we found that the $ \e\to 0 $ limit of the poles occurs at $ \frac12 $, $ b_1 $, and $ 2-b_1 $.
With $ u_*<\frac1{12}\wedge b_1 $,
these points sit strictly outside of $ \mathcal{G}(t,\e) $ for all $ t>0 $ large enough and $ \e>0 $ small enough.
Hence, no poles of $ \frac{\zp(z_1)}{ z_1 \pole(z_1)} \SGtt(t,z_1) $ enters $  \mathcal{G}(t,\e) $,
as long as $ t>0 $ is large enough and $ \e>0 $ is small enough.
\end{proof}

Having introduce the contours $ \magicc(t,\e,-k_1\alpha) $ and $ \zplv(t,\e,-k_1\alpha) $,
hereafter we write $ z_1(\theta_1) \in \magicc(t,\e,-k_1\alpha) $ for the parametrization depicted in Figure~\ref{fig:magicc},
and write $ \tilde{z}_1(\theta_1)\in\zplv(t,\e,-k_1\alpha) $ for the parametrization give in~\eqref{eq:zplv}.
We now turn to the auxiliary parameter $ k_1=k_1(\alpha)\in\Z_{\geq 2} $.
Similar to previous cases, the parameter $ k_1 $ is chosen large enough to ensure that
\begin{align*}
	r'_2= \rad(-2\alpha) \geq \pole(\tilde{z}_1(0)) + \tfrac{1}{\sqrt{t+1}} \in \R.
\end{align*}
Such a condition holds for a large enough $ k_1=k_1(\alpha,T) $,
as can be checked by the same calculations by Taylor expansion as in the $ (+-) $-case.
We do not repeat the calculations, and fix such $ k_1 \in \Z_{\geq 2} $.
Given this condition, using the same argument for obtaining~\eqref{eq:theta1:+-} in the $ (+-) $-case,
here we have
\begin{align}
	\tag{\ref{eq:theta1:+-}''}
	\label{eq:theta1:++}
	|\pole(\tilde{z}_1(\theta_1))| > r'_2
	\text{ holds only if }
	|\theta_1| \geq \frac{1}{C(\alpha)(t+1)^{1/4}}.
\end{align}

Let us check that, along the contours $ z_1\in\magicc(t,\e,-k_1\alpha) $ and $ z_1\in\zplv(t,-k\alpha) $, and
we do have the desired Gaussian decay of $ |\SGte| $ and $ |\SGtt| $.
\begin{lem}
\label{lem:SGtbd:++}
Given any $ T\in(0,\infty) $ and $ \beta\in\R $,
\begin{align*}
	&\big|\SGte(t,z)\big|,
	\
	\big|\SGtt(t,z)\big|
	\leq C(\beta,T) \exp(-\tfrac{\theta^2}{C}(t+1)),
	\quad
	z=z(\theta) \in\magicc(t,\e,\beta),
\\
	&\big|\SGte(t,z)\big|,
	\
	\big|\SGtt(t,z)\big|
	\leq C(\beta,T) \exp(-\tfrac{\theta^2}{C}(t+1)),
	\quad
	z=z(\theta) \in\zplv(t,\e,\beta),
\end{align*}
for all $ \theta\in (-\pi,\pi] $, large enough $ t\leq \e^{-2}T $, and small enough $ \e>0 $.
\end{lem}
\begin{proof}
The proof follows the same three-step procedure as the proof of Lemma~\ref{lem:SGtbd}.
Given the identities~\eqref{eq:SGte:D}--\eqref{eq:SGte:DDD},
the proof of the first two steps~\ref{enu:3step:theta=0}--\ref{enu:3step:small}
follows the same argument via Taylor expansion as in Lemma~\ref{lem:SGtbd},
and we do not repeat it here.
As for the last step~\ref{enu:3step:large},
as argued in the proof of Lemma~\ref{lem:SGtbd:+-}, it amounts to checking the corresponding limiting condition.
Recall that $ \Magic = \{|z-\frac12|=\frac12\} $ and recall the definition of $ \Magicc $ from~\eqref{eq:Magicc}.
It is readily checked that $ \magicc(t,\e,\beta) $ converges uniformly to $ \Magicc $ as $ (t,\e)\to(\infty,0) $, under their respective polar parametrization,
and similarly $ \zplv(t,\e,\beta) $ converges uniformly to $ \Magic $ as $ (t,\e)\to(\infty,0) $.
This being the case, the proof reduces to checking the steepest decent condition~\eqref{eq:steep:Magic} and \eqref{eq:steep:Magicc},
which have been verified.
\end{proof}

We have all the necessary ingredients for estimating $ \SGIn $.
\begin{proof}[Proof of Proposition~\ref{prop:SG:}\ref{enu:SGIn}--\ref{enu:SGIng}, the $ (++) $-case, with large enough $ t $]
The proof begins with the contour deformation described in Section~\ref{sect:SGIn:ov}.
The condition~\eqref{eq:nopole} is checked by the same argument in the $ (+-) $-case,
which gives the decomposition $ \SGIn = \SGbk + \SGres $.
We next perform the aforementioned re-deformation $ \magicc(t,\e,-k_1\alpha) \mapsto \zplv(t,\e,-k_1\alpha) $ in $ \SGres $.
Lemma~\ref{lem:redeform:++} ensures that no pole is crossed during this step, giving $ \SGIn = \SGbk + \SGres'' $.

The proof amounts to bounding $  \SGbk $, $ \SGres'' $, and their gradients.
We begin with $ \SGbk $.
In the following we check a sequence of bounds on terms involved in~\eqref{eq:SGbk},
and we always assume $ z_1=z_1(\theta_1)\in\magicc(t,\e,-k_1\alpha) $ and $ z_2=z_1(\theta_2)\in\Circ_{\tilde{r}_2(z_1)} $
in the course of doing so.
\medskip
\begin{itemize}
\myitem{($ \SGbk $.$ z_1 $)} \label{enu:++SGbk:1}
	Show that $ |z_1|^{x_2-y_1+\mue t-\muet} \leq \exp(-\frac{\alpha|x_2-y_1|}{\sqrt{t+1}+C(\alpha)}) $:\\
	With $ x_2-y_1>0 $ under current assumptions, we need an upper bound on $ |z_1| $.
	To this end, instead of $ z_1\in\magicc(t,\e,-k_1\alpha) $,
	let us first consider $ \tilde{z}_1 \in \zplv(t,\e,-k_1\alpha) $.
	This contour $ \zplv(t,\e,-k_1\alpha) $ is a circle with a center in $ (0,\infty) $.
	For such circles, the farthest point to the origin occurs at the right-end.
	This gives
	\begin{align*}
		\sup_{\tilde{z}_1(\theta_1) \in \zplv(t,\e,-k_1\alpha) } |\tilde{z}_1(\theta_1)| = \tilde{z}_1(0)
		=
		\frac{e^{\sqrt\e(2\den-1)}+\rad(t,-k_1\alpha)}{e^{\sqrt\e\den}+e^{\sqrt\e(\den-1)}}.
	\end{align*}
	Recall from Definition~\ref{def:Taylor} the announced convention on Taylor expansion,
	and expand the last expression in $ (\sqrt\e,\frac{1}{\sqrt{t+1}}) $ up to the leading order in $ \frac{1}{\sqrt{t+1}} $.
	This gives
	\begin{align*}
		\sup_{\tilde{z}_ \in \zplv(t,\e,-k_1\alpha) } |\tilde{z}_1| = 1 + 0\cdot\sqrt\e - \tfrac12\den(1-\den) \e - \frac{k_1\alpha}{\sqrt{t+1}} + \ldots,
	\end{align*}
	With $ k_1 \geq 2 $, and $ \e\leq\frac{C(T)}{\sqrt{t+1}} $ under current assumptions,
	we have
	\begin{align}
		\label{eq:++:z1bd}
		\sup_{\tilde{z}_1 \in \zplv(t,\e,-k_1\alpha) } |\tilde{z}_1| \leq  \rad(t,-\alpha),
	\end{align}
	for all large enough $ t $.
	
	Now, recall from~\eqref{eq:magicc} that $ \magicc(t,\e,-k_1\alpha) $ differs from $ \zplv(t,\e,-k_1\alpha) $ only in $ \{|z-u_*|\leq 2u_*\} \subset \{|z|\leq 3u_*\} $.
	With $ 3u_*<1 $, the bound~\eqref{eq:++:z1bd} readily implies
	\begin{align*}
		\sup_{z_1(\theta_1) \in \magicc(t,\e,-k_1\alpha) } |\tilde{z}_1(\theta_1)| \leq  \rad(t,-\alpha)\vee (3u_*) =\rad(t,-\alpha),
	\end{align*}
	for all $ t $ large enough.
	Consequently, $ |z_1|^{x_2-y_1+\mue t-\muet} \leq \rad(t,-\alpha)^{|x_2-y_1|} \leq \exp(-\frac{\alpha|x_2-y_1|}{\sqrt{t+1}+C(\alpha)}). $
\myitem{($ \SGbk $.$ z_2 $)} Show that $ |z_2|^{x_1-y_2+\mue t-\muet} \leq C(\alpha)\exp(-\frac{\alpha|x_1-y_2|}{\sqrt{t+1}+C(\alpha)}) $:\\
	With $ k_2 := 1 $ and with $ \tilde{r}_2 $ defined in \eqref{eq:tilr2},
	we have $ |z_2| \leq \rad(t,-\alpha) $.
	This and the assumption $ x_1-y_2 >0 $ gives the desired claim.
\myitem{($ \SGbk $.$ \SGfre $)} Show that $ |\SGfre(z_1,z_2)| \leq C(\alpha)(1+|\theta_1-\theta_2|\sqrt{t+1}) $:\\
	This bound is establish by the same argument as in the $ (+-) $-case.
	We do not repeat it here.
\myitem{($ \SGbk $.$ \SGte $)} \label{enu:++SGbk:4}
	Show that $ |\SGte(z_i)| \leq C(\alpha,T)\exp(-\frac{\theta^2_i}{C}(t+1)) $:\\
	This is the content of Lemma~\ref{lem:SGtbd:++}.
\end{itemize}
Given \ref{enu:++SGbk:1}--\ref{enu:++SGbk:4},
the desired bound on $ \SGbk $ follows by inserting the bounds into~\eqref{eq:SGbk}, and integrating the result.
The procedure is the same as the $ (+-) $-case, and we do not repeat it here.

We now turn to $ \SGres'' $. In the following we always assume $ \tilde{z}_1=\tilde{z}_1(\theta_1)\in\zplv(t,\e,-k_1\alpha) $.
\medskip
\begin{itemize}
\setlength\itemsep{2pt}
\myitem{($ \SGres'' $.$ \frac{1}{z_1\pole} $)} \label{enu:++SGres:1}
	Show that $ \frac{1}{|\pole(\tilde{z}_1)\tilde{z}_1|} \leq C(\alpha) $:\\
	This is true because $ |\pole(\tilde{z}_1)\tilde{z}_1|= \rad(t,-k_1\alpha) $.
\myitem{($ \SGres'' $.$ \zp $)}\label{enu:++SGres:2}
	Show that $ |\zp(\tilde{z}_1)| \leq C(\alpha)\exp(-\frac{\alpha(|x_2-y_1|+|x_1-y_2|)}{\sqrt{t+1}+C(\alpha)}) $:\\
	Set $ n_1:=|x_2-y_1| $ and $ n_2:=|x_1-y_2| $.
	The assumption $ y_1<y_2 $, $ x_1<x_2 $ in the $ (++) $-case yields $ n_1-2\geq n_2 >0 $.
	Given this, recalling the definition of $ \zp $ from~\eqref{eq:zp}, we write
	\begin{align*}
		|\zp(\tilde{z}_1)|
		\leq
		(|\tilde{z}_1|^{n_1-n_2-2} + |\tilde{z}_1|^{n_1-n_2})|\tilde{z}_1\pole(\tilde{z}_1))|^{n_2}.
	\end{align*}
	Given the bound~\eqref{eq:++:z1bd} on $ |\tilde{z}_1| $
	and given that $ |\pole(\tilde{z}_1)\tilde{z}_1|= \rad(t,-k_1\alpha) $,
	we have
	\begin{align*}
		|\zp(\tilde{z}_1)|
		\leq
		2 \rad(t,-\alpha)^{n_1-n_2-2} \rad(t,-k_1\alpha)^{n_2}.
	\end{align*}
	With $ k_1 \geq 2 $, the desired result follows:
	\begin{align*}
		|\zp(\tilde{z}_1)|
		\leq
		C(\alpha) e^{\frac{-\alpha(n_1-n_2)}{\sqrt{t+1}+C(\alpha)}} e^{-\frac{2\alpha n_2}{\sqrt{t+1}+C(\alpha)}}
		=
		C(\alpha) e^{-\frac{\alpha(n_1+n_2)}{\sqrt{t+1}+C(\alpha)}}.
	\end{align*}

\myitem{($ \SGres'' $.$ \SGtt $)} \label{enu:++SGres:3}
	Show that $ |\SGtt(\tilde{z}_1)| \leq C(\alpha,T)\exp(-\frac{\theta^2_1}{C}(t+1)) $:\\
	This is the content of Lemma~\ref{lem:SGtbd:++}.
\end{itemize}
Given \ref{enu:++SGres:1}--\ref{enu:++SGres:3}, and the derived constraint~\eqref{eq:theta1:++} on $ |\theta_1| $,
the desired bound on $ \SGres $ follows the same integration procedure is the same as the $ (+-) $-case.

As for the gradient, similarly to the $ (+-) $-case, here we have
\begin{align}
	\tag{\ref{eq:zpm-1}''}
	\label{eq:zpm-1:++}
	 |z_j^\pm-1| \leq \tfrac{1}{\sqrt{t+1}} + |\theta_j|,
	 \
	 z_1=z_1(\theta) \in \magicc(t,\e,-k_1\alpha) \text{ or } \zplv(t,\e,-k_1\alpha),
	 \
	 z_2=z_2(\theta) \in \Circ_{\tilde{r}_2(z_1)}.
\end{align}
Incorporate this bound into the preceding analysis gives the desired bounds on the gradients.
\end{proof}

\section{Controlling the quadratic variation: Proof of Proposition~\ref{prop:qv}}
\label{sect:qv}
Based on the estimates from Section~\ref{sect:SG}
and the duality of the stochastic \ac{6V} model from Section~\ref{sect:duality},
here we prove Proposition~\ref{prop:qv}.
\subsection{Expanding the quadratic variation}
The first step toward proving Proposition~\ref{prop:qv} is to find
an expression for $ \e^{-1}\Theta_1(t,x)\Theta_2(t,x) $
that exposes the limiting behavior
$ \frac{2b_1 \den(1-\den)}{1+b_1}\Zsv^2(t,x) $.
Recall the definition of $ \Theta_1(t,x) $ and $ \Theta_2(t,x) $ from~\eqref{eq:Theta1}--\eqref{eq:Theta2}.
With $ \sum_{i=0}^\infty \hke(i-\mu) =1 $, we rewrite them as
\begin{align}
	\label{eq:Theta1:}
	\e^{-\frac12}\Theta_1(t,x)
	&=
	\e^{-\frac12}(\lambdae\taue^{-1} -1)\Zsv(t,x) - \e^{-\frac12}\sum_{i=0}^\infty \hke(i-\mu) \big( \Zsv(t,x-i)- \Zsv(t,x)\big),
\\
	\label{eq:Theta2:}
	\e^{-\frac12}\Theta_2(t,x)
	&=
	\e^{-\frac12}(1-\lambdae)\Zsv(t,x) + \e^{-\frac12}\sum_{i=0}^\infty \hke(i-\mu) \big( \Zsv(t,x-i)- \Zsv(t,x)\big).
\end{align}
In order the extract the relevant limiting behaviors,
in the sequel we will perform a sequence of expansions on the r.h.s.\ of~\eqref{eq:Theta1:}--\eqref{eq:Theta2:}.
Here, let us prepare some notation to express various error terms throughout the subsequent expansions.
We use $ \Decay(t,x_1,\ldots,x_n;x) $
to denote a \emph{generic} (random) process that has a uniform exponential decay off the point $ x $;
and use $ \Bdd(t,x_1,\ldots,x_n) $ to denote a generic uniformly bounded (random) process.
More precisely, there exists deterministic $ a>0 $, $ C<\infty $ such that,
for all $ \e\in (0,1) $, $ t\in\Z_{\geq 0} $ , $x_1,\ldots,x_n,x\in\Xi(t) $,
\begin{align*}
	|\Decay(t,x_1,\ldots,x_n;x)| &\leq C \exp(-a|x_1-x|-\ldots-a|x_1-x|),
\\
	|\Bdd(t,x_1,\ldots,x_n)| &\leq C.
\end{align*}
With these notation
we write $ \Xbd(t,x) $ for a \emph{generic} expression of the form
\begin{align}
	\label{eq:Xbd}
	\Xbd(t,x)  = \sum_{x_1,x_2\in\Xi(t)} \Decay(t,x_1,x_2;x) \Zsv(t,x_1)\Zsv(t,x_2),
\end{align}
where `bdd' stands for `bounded'.
In the sequel $ \Decay $, $ \Bdd $ and $ \Xbd $ may differ from line to line,
as they refer to generic expressions of the declared \emph{type}.
Under this notation, we view expression of the type $ \e^{u} \Xbd(t,x) $, $ u>0 $, small and negligible.

We will also consider expressions that involve gradients.
To motivate the definitions of the following expressions,
let us first consider an expansion of $ \nabla\Zsv(t,x) $.
Recall that $ \nabla f(x) := f(x+1)-f(x) $ denotes the (forward) discrete gradient,
and recall from~\eqref{eq:etaCentered} that $ \etacnt(t,x)\in\{0,1\} $, $ x\in\Xi(t) $, denote the centered occupation variable.
Referring back to the definition~\eqref{eq:ColeHopfTransform} of $ \Zsv $,
with $ \taue=\exp(-\sqrt\e) $, we see that
$
	\nabla \Zsv(t,x)
	 %= (\exp(-\sqrt\e(\etacnt(t,x+1)-\den))-1) \Zsv(t,x).
	 = (e^{-\sqrt\e(\etacnt^{+}(t,x)-\den)}-1) \Zsv(t,x).
$
%where $ \etacnt^{+}(t,x)$ is defined in~\eqref{eq:etaCentered}.
Taylor expanding the exponential gives
\begin{align}
	\label{eq:gZtaylor}
	\e^{-\frac12}\nabla \Zsv(t,x)
	&=
	-\big(\etacnt^+\Zsv\big)(t,x) + \den \Zsv(t,x) + \sqrt{\e} \Bdd(t,x)\Zsv(t,x),
\end{align}
    %:=\eta(t,x+1) $ indicates that the $ x $-variable therein is shifted by $ 1 $.
In particular,
\begin{align}
	\label{eq:gZ:basic}
	\e^{-\frac12}\nabla \Zsv(t,x) = \Bdd(t,x) \Zsv(t,x).
\end{align}
Such a bound~\eqref{eq:gZ:basic} is \emph{pointwise}.
As it turn out, after a suitable time averaging,
expressions that involves $ \e^{-\frac12}\nabla $ acting on $ \Zsv $ \emph{decay} to zero
(except for a product of two $ \e^{-\frac12}\nabla  \Zsv $ evaluated at the same site, see \eqref{eq:Ygg} and Lemma~\ref{lem:gZ:diag} below).
The underlying mechanism arises from the structure for the semigroup $ \SGe $:
referring to Proposition~\ref{prop:SG}, we see that $ \SGe $ gains an extra factor $ (t+1)^{-\frac12} $ upon taking gradient.
This being the case, we view expressions of the type
\begin{align*}
	\Zg(t,x_1,x_2) := (\e^{-\frac12}\nabla \Zsv(t,x_1)) \Zsv(t,x_2)
\end{align*}
as small,
and consider {\it generic} linear combinations of them
\begin{align}
	\label{eq:Yg}
	\Yg(t,x) &=  \sum_{x_1,x_2\in\Xi(t)} \decay(t,x_1,x_2;x) \Zg(t,x_1,x_2),
\end{align}
with some \emph{deterministic} coefficients $ \decay(t,x_1,x_2;x) $ that decay exponentially off $ x $:
\begin{align}
	\label{eq:decay}
	|\decay(t,x_1,x_2;x)| \leq C \exp(-a|x_1-x|-a|x_2-x|).
\end{align}

We will also consider {\it generic} expressions that involves {\it two pieces} of gradient:
\begin{align}
	\label{eq:Ygg}
	\Ygg(t,x) = \sum_{x_1<x_2\in\Xi(t)} \decay(t,x_1,x_2;x) (\e^{-\frac12}\nabla \Zsv)(t,x_1) (\e^{-\frac12}\nabla \Zsv)(t,x_2),
\end{align}
for some generic deterministic coefficients $ \decay(t,x_1,x_2;x) $ satisfying \eqref{eq:decay}, (and may differ from line to line in the sequel).
%In the sequel, the coefficients $ \decay(t,x_1,x_2;x) $ in the quantities written as $\Yg$ and $\Ygg$ may differ from line to line.

Note that in \eqref{eq:Ygg}, the sum ranges over {\it distinct} $ x_1 $ and $ x_2 $.
In fact, diagonal terms $ x_1=x_2 $ contains non-negligible contributions:
\begin{lem}
\label{lem:gZ:diag}
We have that
\begin{align*}
	(\e^{-\frac12}\nabla \Zsv)^2(t,x)
	-
	\den(1-\den) \Zsv^2(t,x)
	=
	-
	\Zg(t,x,x+1)
	+
	\e^\frac12 \Bdd(t,x)\Zsv^2(t,x).
\end{align*}
\end{lem}
\begin{proof}
To expose the relevant contribution from this expression,
we appeal the expansion~\eqref{eq:gZtaylor} of $ \e^{-\frac12}\nabla \Zsv $,
square it, followed by using $ \etacnt^2=\etacnt $.
This gives (recall $\etacnt$ from \eqref{eq:etaCentered})
\begin{align*}
	(&\e^{-\frac12}\nabla \Zsv(t,x))^2
	=
	\Big(-\etacnt(t,x+1)\Zsv(t,x)+\den\Zsv(t,x)+ \e^{\frac12} \Bdd(t,x)\Zsv(t,x) \Big)^2
\\
	&=
	\Big(\etacnt(t,x+1)\Zsv^2(t,x) - 2\den\etacnt(t,x+1)\Zsv^2(t,x) + \den^2\Zsv^2(t,x)\Big)
	+
	\e^{\frac12} \big(\Bdd(t,x)\Zsv^2(t,x)\big)
\\
	&=
	\Big( \big( (1-2\den)\etacnt^+\Zsv^2  +\den^2\Zsv^2\big) + \e^{\frac12} \Bdd\Zsv^2 \Big)\Big|_x.
\end{align*}
Use~\eqref{eq:gZtaylor} in reverse:
$ \etacnt^{+}\Zsv = -\e^{-\frac12} \nabla \Zsv + \den \Zsv + \e^{\frac12} \Bdd \Zsv $,
we rewrite the expression $ \etacnt^+\Zsv^2 $ as
$  -(\e^{-\frac12}\nabla \Zsv)\Zsv + \den \Zsv^2 + \e^{\frac12} \Bdd \Zsv^2 $.
Inserting this into the last displayed equation gives the desired result.
\end{proof}

Having introduced the necessary notation and tools,
we now begin to expand $ \Theta_1 $ and $ \Theta_2 $.
\begin{lem}
\label{lem:ThetaExpand}
We have that
\begin{align*}
%	\label{eq:ThetaExp}
	\e^{-1}\Theta_1(t,x)\Theta_2(t,x)-\tfrac{2b_1\den(1-\den)}{1+b_1}&\Zsv^2(t,x)
	=
	\sqrt\e \Xbd(t,x) + \Yg(t,x) + \Ygg(t,x).
\end{align*}
\end{lem}
\begin{proof}
The starting point of the proof is
the expressions~\eqref{eq:Theta1:}--\eqref{eq:Theta2:} for $ \Theta_1(t,x) $ and $ \Theta_2(t,x) $.
First, from~\eqref{eq:lambdae} and $ \taue^{-1}= e^{\sqrt\e} $,
we have that $  \e^{-\frac12}(\lambdae\taue^{-1} -1) = (1-\den) + \mathcal{O}(\e^\frac12) $
and that $ \e^{-\frac12}(1-\lambdae) = \den + \mathcal{O}(\e^\frac12) $.
Given this, in~\eqref{eq:Theta1:}--\eqref{eq:Theta2:}
we replace $ \e^{-\frac12}(\lambdae\taue^{-1} -1) $ with $ (1-\den) $
and replace $ \e^{-\frac12}(1-\lambdae) $ with $ \den $, up to errors of the form $ \e^\frac12 \Bdd(t,x) $.
Further, telescope the expression $ \Zsv(t,x-i)- \Zsv(t,x) $ into
$ -\nabla\Zsv(t,x-i)-\nabla\Zsv(t,x-i+1) - \ldots -\nabla\Zsv(t,x-1) $.
This, combined with \eqref{eq:lambdae}, gives
\begin{align*}
	\e^{-\frac12}\Theta_1(t,x)
	&=
	(1-\den)\Zsv(t,x) + \sum_{i=0}^\infty \sum_{0<j\leq i} \hke(i-\mue)\e^{-\frac12}\nabla\Zsv(t,x-j)  + \e^\frac12 \Bdd(t,x) \Zsv(t,x),
\\
	\e^{-\frac12}\Theta_2(t,x)
	&=
	\den\Zsv(t,x) - \sum_{i=0}^\infty \sum_{0<j\leq i} \hke(i-\mue)\e^{-\frac12}\nabla\Zsv(t,x-j)  + \e^\frac12 \Bdd(t,x) \Zsv(t,x).
\end{align*}
To simplify notation, set $ u_{\e}(j) := \sum_{i=j}^{\infty}\hke(i-\mu_\e)$, we write
\begin{align}
	\label{eq:Theta1:expand}
	\e^{-\frac12}\Theta_1(t,x)
	&=
	(1-\den)\Zsv(t,x)
	+ \sum_{j=1}^\infty u_{\e}(j) \e^{-\frac12}\nabla \Zsv(t,x-j)
	+ \e^\frac12 \Bdd(t,x)\Zsv(t,x).
\\	
	\label{eq:Theta2:expand}
	\e^{-\frac12}\Theta_2(t,x)
	&=
	\den\Zsv(t,x)
	- \sum_{j=1}^\infty u_{\e}(j) \e^{-\frac12} \nabla \Zsv(t,x-j)
	+ \e^{\frac12} \Bdd(t,x)\Zsv(t,x).
\end{align}

The next step is to take the product of~\eqref{eq:Theta1:expand}--\eqref{eq:Theta2:expand}.
Let $ A_{1,\Zsv}, A_{1,\nabla}, A_{1,\text{err}} $ denote the respective terms on the r.h.s.\ of \eqref{eq:Theta1:expand},
and similarly $ A_{2,\Zsv}, A_{2,\nabla}, A_{2,\text{err}} $ for \eqref{eq:Theta2:expand}.
In the following we expand
\begin{align*}
	\e^{-1}\Theta_1(t,x)\Theta_2(t,x) = \big( A_{1,\Zsv}+A_{1,\nabla}+A_{1,\text{err}} \big)\big( A_{2,\Zsv}+A_{2,\nabla}+A_{2,\text{err}} \big),
\end{align*}
and analyze the resulting terms.

\medskip
\begin{itemize}[leftmargin=2ex]
\setlength\itemsep{2pt}
\item Indeed, $ A_{1,\Zsv}A_{2,\Zsv} = \den(1-\den) \Zsv^2(t,x) $.
\item Next, the term $ A_{1,\Zsv}A_{2,\nabla}+A_{1,\nabla}A_{2,\Zsv} $.
	
	Indeed $ A_{1,\Zsv}A_{2,\nabla}+A_{1,\nabla}A_{2,\Zsv} $ is a linear combination of $ \Zsv(t,x)\e^{-\frac12}\nabla\Zsv(t,x-j) $,
	with coefficients $ (2\den-1) u_\e(j) $.
	Let us check that $ u_\e(j) $ decays exponentially in $ |j| $.
	Referring back to~\eqref{eq:hk}, with $ \mue,\lambdae\to 1 $ as $ \e\to 0 $,
	the kernel $ \hke $ decays geometrically, uniformly over $ \e\in(0,1) $:
	\begin{align}
		\label{eq:hkdecay}
		\hke(x) \leq C b_1^{-|x|}.
	\end{align}
	From this we see that
	\begin{align}
		\label{eq:qvexpandbd}
		|u_{\e}(j)| = \sum_{i\in\Z_{\geq j}}\hke(i-\mue) \leq C |j| b_1^{|j|} \leq C e^{-\frac12|\log b_1||j|}.
	\end{align}	
	
	Given this property~\eqref{eq:qvexpandbd},
	we conclude that $ A_{1,\Zsv}A_{2,\nabla}+A_{1,\nabla}A_{2,\Zsv} $
	is a linear combination of $ \Zsv(t,x)\e^{-\frac12}\nabla\Zsv(t,x-j) $,
	with deterministic coefficients that decay exponentially in $ |j| $, whereby
	\begin{align*}
		A_{1,\Zsv}A_{2,\nabla}+A_{1,\nabla}A_{2,\Zsv} = \Yg(t,x).
	\end{align*}	
\item We now turn to $ A_{1,\nabla} A_{2,\nabla} $.

	With $ A_{1,\nabla} $ and $ A_{2,\nabla} $ both being sums,
	in the produce of $ A_{1,\nabla} A_{2,\nabla} $, we separate the diagonal and off-diagonal term.
	Off-diagonal terms form a linear combination of $ \e^{-\frac12}\nabla\Zsv(x-j)\e^{-\frac12}\nabla\Zsv(x-j') $, $ j\neq j' $,
	with coefficient $ u_\e(j) u_\e(j') $.
	Thanks to~\eqref{eq:qvexpandbd}, this coefficient decays exponentially in $ |j|+|j'| $.
	This being the case, off-diagonal terms jointly contribute an expression of the type $ \Ygg(t,x) $.
	We hence keep track of only the diagonal terms, and write
	\begin{align*}
		A_{1,\nabla} A_{2,\nabla} = - \sum_{j=1}^\infty u_\e(j)^2 (\nabla \Zsv(t,x-j))^2 + \Ygg(t,x).
	\end{align*}	
\item Lastly, everything else: $ (A_{1,\Zsv}+A_{1,\nabla})A_{2,\text{err}}+A_{1,\text{err}}(A_{2,\Zsv}+A_{2,\nabla}) + A_{1,\text{err}}A_{2,\text{err}} $.

	First, by~\eqref{eq:gZ:basic}, in~$ A_{i,\nabla} $ we replace each $ \e^{-\frac12}\nabla\Zsv(t,x-j) $ with $ \Bdd(t,x-j)\Zsv(t,x-j) $.
	Once this is done,
	expanding the expression $ (A_{1,\Zsv}+A_{1,\nabla})A_{2,\text{err}}+A_{1,\text{err}}(A_{2,\Zsv}+A_{2,\nabla}) + A_{1,\text{err}}A_{2,\text{err}} $ gives
	\begin{align*}
		\e^{\frac12} \big(\ \text{linear combination of } \Bdd(t,x,x-j)\Bdd(t,x,x-j') \ \big)\Zsv(t,x)^2.
	\end{align*}	
	Thanks to~\eqref{eq:qvexpandbd}, the coefficients within the linear combination decays exponentially in $ |j|+|j'| $.
	This gives
	\begin{align*}
		(A_{1,\Zsv}+A_{1,\nabla})A_{2,\text{err}}+A_{1,\text{err}}(A_{2,\Zsv}+A_{2,\nabla}) + A_{1,\text{err}}A_{2,\text{err}}
		=
		\e^{\frac12} \Xbd(t,x).
	\end{align*}	
\end{itemize}
\medskip
Given the preceding discussion, we now have
\begin{align}
\label{eq:ThetaExpand}
\begin{split}
	\e^{-1} \Theta_1(t,x) \Theta_2(t,x)
	=
	&\den(1-\den) \Zsv^2(t,x) + \sqrt\e \Xbd(t,x) + \Yg(t,x) + \Ygg(t,x)
\\
	&-\sum_{j=1}^\infty u_\e(j)^2 (\e^{-\frac12}\nabla \Zsv(t,x-j))^2.
\end{split}
\end{align}

As shown in Lemma~\ref{lem:gZ:diag},
the last term in~\eqref{eq:ThetaExpand} contains a non-negligible contribution to $ \Zsv^2(t,x) $.
The rest of the proof consists of extracting this contribution.
First, using Lemma~\ref{lem:gZ:diag}, we write
\begin{align}
	\label{eq:theta:A}
	\sum_{j=1}^\infty u_\e(j)^2 (\e^{-\frac12}\nabla \Zsv(t,x-j))^2
	-
	\den(1-\den) A
	=
	\Yg(t,x) + \e^\frac12 \Xbd(t,x),
\end{align}
where $ A:=\sum_{j=1}^\infty u_\e(j)^2 \Zsv^2(t,x-j) $.
The focus now is on the term $ A $.
We argue that, replacing $ \Zsv(t,x-j) $ with $ \Zsv(t,x) $ in $ A $ only produces an affordable error.
To see this, write
\begin{align}
	\label{eq:A1est}
	\big|\Zsv(t,x-j) - \Zsv(t,x)\big| = \big|e^{\sqrt{\e}\sum_{i=0}^{j-1}(\etacnt(t,x-i)-\den)}-1\big| \Zsv(t,x)
	\leq
	\sqrt{\e}|j| e^{\sqrt{\e} |j|} \Zsv(t,x).
\end{align}
Now, write $ \Zsv(t,x-j) $ as $ \Zsv(t,x) + \Zsv(t,x-j) - \Zsv(t,x) $, with the aid of~\eqref{eq:A1est} and~\eqref{eq:qvexpandbd}, we have
\begin{align}
	\label{eq:theta:A:}
	A = \Zsv^2(t,x)\sum_{j=1}^\infty u^2_\e(j)
	+
	\e^{\frac12} \Bdd(t,x) \Zsv^2(t,x).
\end{align}
%Recall that $ u_{\e}(j):= \sum_{i:i\geq j}\hk(\mue-j) $.
With \eqref{eq:lambdae}, and $ b_2=e^{-\sqrt\e}b_1 $,
a straightforward calculation from~\eqref{eq:hk} gives
\begin{align*}
	\sum_{j=1}^\infty u^2_\e(j)
	 = \frac{1-b_1}{1+b_1} + \mathcal{O}(\sqrt{\e}).
\end{align*}
Using this in~\eqref{eq:theta:A:}, and inserting the result back into~\eqref{eq:theta:A}, we conclude
\begin{align*}
	\sum_{j=1}^\infty u_\e(j)^2 (\e^{-\frac12}\nabla \Zsv(t,x-j))^2
	-
	\den(1-\den) \frac{1-b_1}{1+b_1} \Zsv^2(t,x)
	=
	\Yg(t,x) + \e^\frac12 \Xbd(t,x).
\end{align*}
This together with~\eqref{eq:ThetaExpand} gives the desired result.
\end{proof}

Lemma~\ref{lem:ThetaExpand} provides the relevant decomposition of $ \e^{-1}\Theta_1\Theta_2 $
into its limiting expression and residual terms.
While we do expect the residual terms $ \e^{\frac12} \Xbd $, $ \Yg $, and $ \Ygg $ to tend to zero,
bounds on the last two terms are not immediate.
%and in the subsequent we will explain how duality can be used to make these desired estimates.
%
To see this, recall from Proposition~\ref{prop:Zduality} that the duality functions for the stochastic \ac{6V} model are $ \Zsv(s,x_1)\Zsv(s,x_2) $
and $ (\etacnt^+\Zsv)(s,x_1)(\etacnt^+\Zsv)(s,x_2) $, for $ x_1< x_2 $.
On the other hand, the expressions $ \Yg $ and $ \Ygg $ (as in~\eqref{eq:Yg} and \eqref{eq:Ygg}) are linear combinations of $ \Zsv(s,x_1)\Zsv(s,x_2) $.
that generally involve $ x_1=x_2 $.

To circumvent this `diagonal' issue,
recalling from~\eqref{eq:decay} that $ \decay $ denotes generic deterministic coefficients with an exponential decay,
we consider a slight modification $ \Xg $ of $ \Yg $,
which is the same type of expressions with an additional constraint $ |x_1-x_2| > 1 $:
\begin{align}
	\notag
	\Xg(t,x) = \sum_{x_1,x_2\in\Xi(t), |x_2-x_1|>1}
	\decay(t,x_1,x_2;x) \Zg(s,x_1,x_2).
\end{align}
Next, set
\begin{align}
	\label{eq:tilZ}
	\tilZ(t,x_1,x_2) := \big(\etacnt^{+}\Zsv\big)(t,x_1)\big(\etacnt^{+}\Zsv\big)(t,x_2) - \den^2 \Zsv(t,x_1)\Zsv(t,x_2).
\end{align}
In place of $ \Ygg $, we consider expressions $ \Xgg $ of the type
\begin{align}
	\label{eq:Xgg}
	\Xgg(t,x) &= \sum_{x_1<x_2\in\Xi(t)} \decay(t,x_1,x_2;x) \tilZ(t,x_1,x_2).
\end{align}
The next lemma allows us to trade in $ \Yg $ and $ \Ygg $ for $ \Xg $ and $ \Xgg $.
\begin{lem}
\label{lem:trade}
We have that
\begin{align}
	\label{eq:ThetaExp:1}
	\Ygg(t,x) &= \Xgg(t,x)  + \Yg(t,x) + \e^{\frac12}\Xbd(t,x),
\\
	\label{eq:ThetaExp:2}
	\Yg(t,x) &= \Xg(t,x) + \e^{\frac12}\Xbd(t,x).
\end{align}
\end{lem}
\begin{proof}
Indeed, $ \Ygg(t,x) $ denotes a generic linear combination of
\begin{align*}
	A := (\e^{-\frac12}\nabla\Zsv)(t,x_1)(\e^{-\frac12}\nabla\Zsv)(t,x_2),\quad x_1<x_2,
\end{align*}
and $ \Xgg(t,x) $ denotes a generic linear combination of $ \tilZ(t,x_1,x_2) $, $ x_1<x_2 $.
This being the case, to prove~\eqref{eq:ThetaExp:1},
it suffices to show that $ A-\tilZ(t,x_1,x_2) $ is written as a linear combination
of $ \Zg(t,x_1,x_2) $ and negligible terms that carry an outstanding $ \e^{\frac12}$ factor.
To this end, we use~\eqref{eq:gZtaylor} to expand
\begin{align}
	\notag
	A
	=&
	\big(-\etacnt^{+}\Zsv+\den\Zsv + \e^{\frac12}\Bdd \Zsv \big)(t,x_1)
	\big(-\etacnt^{+}\Zsv+\den\Zsv + \e^{\frac12}\Bdd \Zsv \big)(t,x_2)	
\\
	\label{eq:tilZexpand}
=&
	\tilZ(t,x_1,x_2)
	+
	\den
	\big(-\etacnt^{+}\Zsv+\den\Zsv\big)(t,x_1) \Zsv(t,x_2)
	+
	\den
	\Zsv(t,x_1) \big(-\etacnt^{+}\Zsv+\den\Zsv\big)(t,x_2)
\\
	\notag
	&+
	\e^{\frac12}\Bdd(t,x_1,x_2) \Zsv(t,x_1)\Zsv(t,x_2).
\end{align}
In~\eqref{eq:tilZexpand},
further use~\eqref{eq:gZtaylor} in reverse to write
$ -\etacnt^{+}\Zsv+\den\Zsv = \e^{-\frac12}\nabla\Zsv + \e^{\frac12}\Bdd\Zsv $.
We get
\begin{align*}
	A-\tilZ(t,x_1,x_2)
	=
	\den \Zg(t,x_1,x_2)
	+
	\den \Zg(t,x_2,x_1)
	+
	\e^{\frac12}\Bdd(t,x_1,x_2) \Zsv(t,x_1)\Zsv(t,x_2).
\end{align*}
This gives the desired result \eqref{eq:ThetaExp:1}.

As for \eqref{eq:ThetaExp:2},
recall that both $ \Xg $ and $ \Yg $ denote generic linear combinations of the same terms.
The only difference is in that the former misses those terms with $ |x_1-x_2|\leq 1 $.
Consequently, the result~\eqref{eq:ThetaExp:2} follows once we show
\begin{align*}
	\Zsv(t,x+1)
	\big( \e^{-\frac12}\nabla\Zsv(t,x) \big)
	&=
	\Zsv(t,x+2) \big( \e^{-\frac12}\nabla\Zsv(t,x) \big)
	+
	\e^{\frac12}\big(\Bdd\Zsv^2\big)(t,x),
\\
	\Zsv(t,x)
	\big( \e^{-\frac12}\nabla\Zsv(t,x+1) \big)
	&=
	\Zsv(t,x-1) \big( \e^{-\frac12}\nabla\Zsv(t,x+1) \big)
	+
	\e^{\frac12}\big(\Bdd\Zsv^2\big)(t,x),
\\
	\Zsv(t,x) \big( \e^{-\frac12}\nabla\Zsv(t,x) \big)
	&=
	\Zsv(t,x-2) \big( \e^{-\frac12}\nabla\Zsv(t,x) \big)
	+
	\e^{\frac12}\big(\Bdd\Zsv^2\big)(t,x).
\end{align*}
Going from the l.h.s.\ to the r.h.s.\ amounts to
changing $ \Zsv(t,x+1) \mapsto \Zsv(t,x+2) $ or changing $ \Zsv(t,x) \mapsto \Zsv(t,x-1) $;
note that the $\nabla Z$ factor is never changed.
Thanks to~\eqref{eq:gZtaylor}, these changes
introduce only  error of the form $ \e^{\frac12}(\Bdd \Zsv)(t,x) $.
Also, by~\eqref{eq:gZ:basic}, $ \e^{-\frac12}\nabla\Zsv(t,x) =(\Bdd\Zsv)(t,x) $, $ \e^{-\frac12}\nabla\Zsv(t,x+1) =(\Bdd\Zsv)(t,x) $.
Hence, the overall error caused by the aforementioned changes
is indeed of the form $ \e^{\frac12}(\Bdd\Zsv^2)(t,x) $.
\end{proof}

Lemmas~\ref{lem:ThetaExpand} and \ref{lem:trade} immediately yield
\begin{cor}\label{cor:Theta:Exp}
We have
\begin{align*}
	\e^{-1}\Theta_1(t,x)\Theta_2(t,x)-\frac{2b_1\den(1-\den)}{1+b_1}\Zsv^2(t,x)
	=
	\sqrt\e \Xbd(t,x) + \Xg(t,x) + \Xgg(t,x).
\end{align*}
\end{cor}

\subsection{Time decorrelation via duality}
Given the decomposition from Corollary~\ref{cor:Theta:Exp},
our goal toward proving Proposition~\ref{prop:qv}
is to argue that, each type of expression on the r.h.s.\ is negligible as $ \e\to 0 $.
This is straightforward for $ \sqrt{\e}\Xbd(t,x) $ due to the outstanding $ \e^{\frac12}$ factor.
On the other hand, as mentioned earlier, the terms $ \Xg $ and $ \Xgg $ converge to zero only after time averaging.
This being the case, with $ \Xg $ and $ \Xgg $ being linear combinations of $ \Zg $ and $ \tilZ $,
we direct our focus onto bounding
\begin{align}
	\label{eq:Xgmomt}
	\Xgmomt(\bar t,x^\star_1,x^\star_2)
	&:=
	\Ex\Bigg[\Big(
		\e^{2}\sum_{s=0}^{\bar{t}-1} \Zg(s,x^\star_1(s),x^\star_2(s))
	\Big)^2\Bigg],
\\
	\label{eq:Xggmomt}
	\Xggmomt(\bar t,x^\star_1,x^\star_2)
	&:=
	\Ex\Bigg[\Big(
		\e^{2}\sum_{s=0}^{\bar{t}-1} \tilZ(s,x^\star_1(s),x^\star_2(s))
	\Big)^2\Bigg],
\end{align}
for $ \bar t\in\Z\cap[0,\e^{-2}T] $ and $ x^\star_1\neq x_2^\star\in \Z $, and
$
	x^\star_i(s) := x^\star_i - \mue s + \mues \in \Xi(s).
$
These expressions are expanded into conditional expectations as
\begin{align}
	\label{eq:cndEx1}
	\Xgmomt(\bar t,x^\star_1,x^\star_2)
	&=
	\e^{4} \Big( 2\sum_{s_1<s_2<\bar t} + \sum_{s_1=s_2<\bar t} \Big)
	\Ex\Big[ \Ex\big[\Zg(s_2,x_1,x_2)\big|\filt(s_1) \big] \Zg(s_1,x_1,x_2) \Big],
\\
	\label{eq:cndEx2}
	\Xggmomt(\bar t,x^\star_1,x^\star_2)
	&=
	\e^{4} \Big( 2\sum_{s_1<s_2<\bar t} + \sum_{s_1=s_2<\bar t} \Big)
	\Ex\Big[ \Ex\big[\tilZ(s_2,x_1,x_2)\big|\filt(s_1) \big] \tilZ(s_1,x_1,x_2) \Big],
\end{align}
%where $ x_1:= x_\star+\mue s_1-\mue \lfloor s_1 \rfloor $ and $ x_2:= x_\star+\mue s_2-\mue \lfloor s_2 \rfloor $.
where $ x_i:= x^\star_i - \mue s_i+  \lfloor \mue  s_i \rfloor $ and the notation $(\sum + \sum)(\Cdot) := \sum (\Cdot) + \sum(\Cdot) $.
%Note that $|x_i-x^\star_i|\leq 1$ so we can, in certain bounds (e.g. the proof of Corollary \ref{cor:cndExp}) go between these two variables at the cost of changing constants.
%[this comment has been moved to the relevant proofs]
%
Given~\eqref{eq:cndEx1}--\eqref{eq:cndEx2},
we set out to bounding the following conditional expectations
\begin{align*}
	\Ex\big[(\Zg(t+s,x_1,x_2)|\filt(s)\big] \Zg(s,x_1,x_2),
	\quad
	\Ex\big[\tilZ(t+s,x_1,x_2)|\filt(s)\big] \tilZ(s,x_1,x_2),
\end{align*}
and show that they decay as $ t $ becomes large.
We begin by relating these conditional expectations to the semigroup $ \SGe $ via duality.
Recall that $ \nabla_x $ denotes the discrete gradient acting on a designated variable $ x $.
\begin{lem}
\label{lem:cndEx}
Let $ t,s\in\Z_{\geq 0} $.
For all $ x_1+1<x_2\in\Xi(t) $, we have
\begin{align}
	\Ex\big[ \Zg(t,x_1,x_2) \big| \filt(s) \big]
	=
	\label{eq:gZZ:condEx}
	\sum_{y_1<y_2\in\Xi(s)} \e^{-\frac12}\nabla_{x_1}\SGe\big((y_1,y_2),(x_1,x_2);t\big) \Zsv(s,y_1)\Zsv(s,y_2),
\\
	\Ex\big[ \Zg(t,x_2,x_1) \big| \filt(s) \big]
	\label{eq:ZgZ:condEx}
	=
	\sum_{y_1<y_2\in\Xi(s)} \e^{-\frac12}\nabla_{x_2}\SGe\big((y_1,y_2),(x_1,x_2);t\big) \Zsv(s,y_1)\Zsv(s,y_2).
\end{align}
For all $ x_1<x_2\in\Xi(t) $,
with
\begin{align*}
	\SG^\e_{\nabla+\nabla}((y_1,y_2),(x_1,x_2); t)
	:= \nabla_{y_1}\SGe((y_1-1,y_2),(x_1,x_2);t) +\nabla_{y_2}\SGe((y_1,y_2-1),(x_1,x_2);t),
\end{align*}
we have
\begin{subequations}
\label{eq:tilZ:condEx}
\begin{align}
	\notag
	\Ex&[\tilZ(t+s,x_1,x_2)|\filt(s)]
\\
	\label{eq:tilZ:condEx:gg}
	=\quad &
	- \!\!\!\! \sum_{y_1+1<y_2\in\Xi(s)} \!\!\!\! \e^{-\frac12}\SG^\e_{\nabla+\nabla}((y_1,y_2),(x_1,x_2); t)\,
	\Zsv(s,y_1)\Zsv(s,y_2)
\\
	\label{eq:tilZ:condEx:eps}
	&
	+\sum_{y_1+1<y_2}
		\e^\frac12\SGe\big((y_1,y_2),(x_1,x_2);t\big) \big)\, \Bdd(s,y_1,y_2) \Zsv(s,y_1) \Zsv(s,y_2)
\\
	\label{eq:tilZ:condEx:s}
	&
	+ \sum_{ \begin{subarray}{}|i|,|j|,|i'|,|j'|\leq 3\\ \quad{i<j}\end{subarray}} \!\!\!
	 \Big(
		\sum_{y\in\Xi(s)}
		\SGe\big((y+i,y+j),(x_1,x_2);t\big) \big) \Bdd(s,y) \Zsv(s,y+i') \Zsv(s,y+j')
	\Big).
\end{align}
\end{subequations}
\end{lem}
\begin{rmk}
Recall the discussion regarding $ \nabla $-Weyl chamber from the beginning of Section~\ref{sect:SG}.
With the assumption $ x_1+1<x_2 $,
the expressions in~\eqref{eq:gZZ:condEx}--\eqref{eq:ZgZ:condEx} that involve $ \nabla\SGe $ are
indeed within their $ \nabla $-Weyl chambers,
and similarly for those in~\eqref{eq:tilZ:condEx}.
\end{rmk}
\begin{proof}
Roughly speaking, the proof amounts to translating the duality result from Proposition~\ref{prop:Zduality},
i.e., \eqref{eq:duality:ZZ}--\eqref{eq:duality:eeZZ}, to the relevant context considered.

First, in~\eqref{eq:duality:ZZ}, set $ (x_1,x_2) $ to be $ (x_1+1,x_2) $ and $ (x_1,x_2) $,
and take the difference of the results.
We obtain~\eqref{eq:gZZ:condEx}.
Note that the assumption $ x_1+1<x_2 $ guarantees that $ (x_1+1,x_2) $ lies in the Weyl chamber.
The identity~\eqref{eq:ZgZ:condEx} follows the same way.

We now turn to proving~\eqref{eq:tilZ:condEx}.
To simplify notation, we use ``\eqref{eq:tilZ:condEx:gg}'' to denote
the expression written therein.
Likewise, we use ``\eqref{eq:tilZ:condEx:eps}-type'' and ``\eqref{eq:tilZ:condEx:s}-type''
to denote the types (note the $ \Bdd $'s therein) of expressions written in~\eqref{eq:tilZ:condEx:eps} and \eqref{eq:tilZ:condEx:s}.
First, with $ \tilZ $ defined in~\eqref{eq:tilZ},
taking the difference of \eqref{eq:duality:ZZ} and \eqref{eq:duality:eeZZ} gives
\begin{align*}
%	\label{eq:duality:tilZ}
	&\Ex\Big[ \tilZ(t+s,x_1,x_2) \Big\vert \filt(s) \Big] = \sum_{y_1<y_2\in\Xi(s)} \!\!\!
	\SGe\big((y_1,y_2),(x_1,x_2);t\big) \tilZ(s,y_1,y_2).
\end{align*}
Separate the terms with $ y_1+1=y_2 $. %in~\eqref{eq:duality:tilZ}.
With $ \tilZ(s,y_1,y_2) = \Bdd(s,y_1,y_1)\Zsv(s,y_1)\Zsv(s,y_2) $, we have
\begin{align}
	\label{eq:duality:tilZ:}
	\Ex\Big[ \tilZ(t+s,x_1,x_2) \Big\vert \filt(s) \Big] = \sum_{y_1+1<y_2} \!\!\!
	\SGe\big((y_1,y_2),(x_1,x_2);t\big) \tilZ(s,y_1,y_2)
	+
	\text{\eqref{eq:tilZ:condEx:s}-type}.
\end{align}
Next, with $ \tilZ(s,y_1,y_2) $ defined in~\eqref{eq:tilZ},
adding and subtracting $ \den\Zsv(s,y_1)(\etacnt^{+}\Zsv)(s,y_2)  $, we write
\begin{align*}
	\tilZ(s,y_1,y_2)
	=
	\big( (\etacnt^{+}-\den)\Zsv \big)(s,y_1) \big(\etacnt^{+}\Zsv\big)(s,y_2)
	+
	\big(\den\Zsv\big)(s,y_1) \big( (\etacnt^{+}-\den)\Zsv \big)(s,y_2).
\end{align*}
Use~\eqref{eq:gZtaylor} in reverse:
$ \etacnt^{+}\Zsv = \e^{-\frac12} \nabla \Zsv + \den \Zsv + \e^{\frac12}\Bdd \Zsv $,
we further obtain
\begin{align*}
	\tilZ(s,y_1,y_2)
	&=
	\big( \e^{-\frac12}\nabla\Zsv \big)(s,y_1) \big(\etacnt^{+}\Zsv\big)(s,y_2)
	+
	\den\Zsv(s,y_1) \big( \e^{-\frac12}\nabla\Zsv \big)(s,y_2)
	\\
	&\qquad + \sqrt{\e}\Bdd(s,x_1,y_2)\Zsv(s,y_1)\Zsv(s,y_2).
\end{align*}
Inserting this into~\eqref{eq:duality:tilZ:}, followed by summation by parts:
\begin{align*}
	\sum_{y_1:y_1+1<y_2} f(y_1) \nabla g(y_1)
	&=
	-\sum_{y_1:y_1+1<y_2} \nabla f(y_1-1) g(y_1)
	+ f(y_2-2) g(y_2-1),
\\
	\sum_{y_2:y_1+1<y_2} f(y_2) \nabla g(y_2)
	&=
	-\sum_{y_2:y_1+1<y_2} \nabla f(y_2-1)  g(y_2)
	- f(y_1+1) g(y_1+2),
\end{align*}
we then arrive at the desired result:
\begin{align*}
	\Ex\big[ \tilZ(t+s,x_1,x_2) \big\vert \filt(s) \big] =
	\Big( \eqref{eq:tilZ:condEx:gg} +\text{\eqref{eq:tilZ:condEx:eps}-type}+\text{\eqref{eq:tilZ:condEx:s}-type} \Big)
	+
	\text{\eqref{eq:tilZ:condEx:s}-type}.
\end{align*}
\end{proof}

Given Lemma~\ref{lem:cndEx},
we now incorporate the estimates on $ \SGe $ from Section~\ref{sect:SG} to obtain
bounds on the conditional expectations.

\begin{lem}
\label{lem:cndEx:bd}
Given $ T<\infty $, there exists $ u=u(T)<\infty $ such that,
for all $ s,t\in[0,\e^{-2}T]\cap\Z $ and $ x_1,x_2\in\Xi(t) $,
\begin{align*}
	\ind_{\set{ |x_1-x_2|>1}}
	\Ex\Big[\Big| \Ex\big( \Zg(t+s,x_1,x_2) \big| \filt(s) \big) \Zg(s,x_1,x_2) \Big|\Big]
	&\leq
	C(T) \frac{\e^{-\frac12}}{\sqrt{t+1}} e^{ u\e(|x_1|+|x_2|)};
\\
	\Ex\Big[\Big| \Ex\big(\tilZ(t+s,x_1,x_2)|\filt(s)\big) \tilZ(s,x_1,x_2) \Big|\Big]
	&\leq
	C(T)
	\frac{\e^{-\frac12}}{\sqrt{t+1}} e^{u\e(|x_1|+|x_2|)}.
\end{align*}
\end{lem}
\begin{proof}
First, the moment bound~\eqref{eq:Zmomt} from Proposition~\ref{prop:momt} gives that
$ \Ex[\Zsv(s,y)^4] \leq C(T)e^{u\e|y|} $, for some fixed $ u=u(T)\in(0,\infty) $.
This together with the Cauchy--Schwarz inequality gives
\begin{align}
	\label{eq:cndExp:momt}
	\Ex\big[|\Zsv(s,x_1)\Zsv(s,x_2)\Zsv(s,y_1)\Zsv(s,y_2)|\big] \leq C(T) e^{u\e(|x_1|+|x_2|+|y_1|+|y_2|)}.
\end{align}
To alleviate notation, in the following we often write $ \SGe((y_1,y_2),(x_1,x_2);t)=\SGe $.
Multiply both sides of~\eqref{eq:gZZ:condEx}--\eqref{eq:ZgZ:condEx} by $ \Zg(s,x_1,x_2) $.
Incorporating both the cases $ x_1+1<x_2 $ and $ x_2+1<x_2 $, we write
\begin{align*}
	\Big| &\Ex\big[ \Zg(t+s,x_1,x_2) \big| \filt(s) \big] \Zg(s,x_1,x_2) \Big|\ind_{\set{ |x_1-x_2|>1}}
\\
	&\leq
	C(T) \sum_{y_1<y_2\in\Xi(s)} \e^{-\frac12} \big( |\nabla_{x_1}\SGe|+|\nabla_{x_2}\SGe|\big)
	\Zsv(s,y_1)\Zsv(s,y_2) \big| \Zg(s,x_1,x_2) \big|
\\
	&\leq
	C(T) \sum_{y_1<y_2\in\Xi(s)} \e^{-\frac12} \big( |\nabla_{x_1}\SGe|+|\nabla_{x_2}\SGe|\big)
	\Zsv(s,y_1)\Zsv(s,y_2)\Zsv(s,x_1)\Zsv(s,x_2),
\end{align*}
where, in the last inequality, we used~\eqref{eq:gZ:basic} to write
$ |\Zg(s,x_1,x_2)| \leq C Z(s,x_1) Z(s,x_2) $.
Take expectation on both sides using~\eqref{eq:cndExp:momt}.
For $ f:(y_1<\ldots<y_n)\in\Xi(s)^n \mapsto f(\vec{y})\in\R $, set
\begin{align*}
	[f]_u:= \sum_{y_1<\ldots<y_n} |f(y_1,\ldots,y_n)| e^{u(|y_1|+\ldots+|y_n|)}.
\end{align*}
We then obtain
\begin{align}
	\notag
	\Ex\Big[\Big| \Ex\big( \Zg(t+s,x_1,x_2) &\big| \filt(s) \big) \Zg(s,x_1,x_2) \Big|\Big]\ind_{\set{ |x_1-x_2|>1}}
\\
	\label{eq:lem:cndEx1}
	&\leq
	e^{u\e(|x_1|+|x_2|)} C(T) \e^{-\frac12} \big([\nabla_{x_1}\SGe]_{u\e}+[\nabla_{x_2}\SGe]_{u\e}\big).
\end{align}
Note that $[\nabla_{x_1}\SGe]_{u\e}$, $[\nabla_{x_2}\SGe]_{u\e}$ are only sums over
$y_1<y_2\in\Xi(s)$ and are thus still functions of $x_1,x_2$.

A similar procedure starting with~\eqref{eq:tilZ:condEx} gives
\begin{align}
	\notag
	\Ex\Big[\Big| \Ex\big( &\tilZ(t+s,x_1,x_2) \big| \filt(s) \big) \tilZ(s,x_1,x_2) \Big|\Big]
\\
	\label{eq:lem:cndEx2}
	&\leq
	e^{u\e(|x_1|+|x_2|)} C(T)
	\Big( \e^{-\frac12} \big([\nabla_{x_1}\SGe]_{u\e}+[\nabla_{x_2}\SGe]_{u\e}\big)
	+ \e^\frac12 [\SGe]
	+ \sum_{|i|,|j|\leq 3} [V_{\e,i,j}]_{u\e}
	\Big),
\end{align}
where $ V_{\e,i,j}(y) := \SGe((y+i,y+j),(x_1,x_2);t) $.

With $ t\leq \e^{-2}T $, we set $ \alpha:=3u\sqrt{T} $
so that
$
	\frac{\alpha}{\sqrt{t+1}+C(\alpha)} = \frac{\alpha\e}{\sqrt{T+\e^2}+C(\alpha)\e}
	>
	2u \e,
$
for all $ \e>0 $ small enough.
For such an exponent $ \alpha $, we indeed have
%\begin{align*}
%	\sum_{y\in\Xi(s)} e^{-\frac{\alpha|x-y|}{\sqrt{t+1}}} e^{u\e|y|}
%	=
%	\sum_{y\in -x+\Xi(s)} e^{-\frac{\alpha|y|}{\sqrt{t+1}}} e^{u\e|x-y|}
%	\leq
%	C(\alpha) e^{u\e|x|} (t+1)^\frac12.
%\end{align*}
%This bound holds true even if we replace the exponential $ e^{-\frac{\alpha|x-y|}{\sqrt{t+1}}} $ with
%$ e^{-\frac{\alpha|x-y|}{\sqrt{t+1}+C(\alpha)}} $,
%as long as $ t $ is larger than a threshold $ t_0=t_0(\alpha) $ such that $ \frac{\alpha}{\sqrt{t+1}+C(\alpha)} > u \e $:
\begin{align}
	\notag
	\sum_{y\in\Xi(s)} e^{-\frac{\alpha|x-y|}{\sqrt{t+1}+C(\alpha)}} e^{u\e|y|}
	&\leq
	e^{u|x|}\sum_{y\in\Xi(s)} e^{-\frac{\alpha|x-y|}{\sqrt{t+1}+C(\alpha)}} e^{u\e|x-y|}
\\
	\label{e:sum-y-exp-exp}
	&\leq
	e^{u|x|}\sum_{y\in\Xi(s)} e^{-\frac{\alpha|x-y|}{2(\sqrt{t+1}+C(\alpha))}}
	\leq
	C(\alpha) e^{u\e|x|} (t+1)^\frac12,
%	\quad
%	t \geq t_0(\alpha).
\end{align}
for all $ \e>0 $ small enough.
Now, apply the estimates on $ |\SGe| $ and $ |\nabla\SGe| $ from Proposition~\ref{prop:SG} with this exponent $ \alpha $. We get
\begin{align*}
	&[\nabla_{x_i}\SGe]_{u\e},	\ [\nabla_{y_i}\SGe]_{u\e}
	\leq
	\frac{C(\alpha,T)}{(t+1)^{1/2}}e^{u\e(|x_1|+|x_2|)},
\\
	&[\SGe]_{u\e}
	\leq
	C(\alpha,T)e^{u\e(|x_1|+|x_2|)},
	\quad
	[V_{\e,i,j}]_{u\e}
	\leq
	\frac{C(\alpha,T,i,j)}{(t+1)^{1/2}}e^{u\e(|x_1|+|x_2|)}.
\end{align*}
%for all $ t \geq t_0=t_0(\alpha) $.
Here, upon taking $[\;\Cdot\;]_{u\e}$, the sums over $y_1$ and over $y_2$
of $\SGe((y_1,y_2),(x_1,x_2),t)$ each produces a factor of $(t+1)^\frac12$,
as seen from in \eqref{e:sum-y-exp-exp}.
Insert these bounds into \eqref{eq:lem:cndEx1}--\eqref{eq:lem:cndEx2}.
With $ \frac{\e^{-\frac12}}{\sqrt{t+1}} + \e^{\frac12} + \frac{1}{\sqrt{t+1}} \leq \frac{C(T)\e^{-\frac12}}{\sqrt{t+1}} $,
and with $ \alpha=\alpha(u,T) $,
we conclude the desired result.
%for all large enough $ t \geq t_0(\alpha) $.
%
%As for those values of $ t \leq t_0(\alpha) $,
%since (by~\eqref{eq:gZ:basic}) $ \Zg(t,x_1,x_2)=\Bdd(t,x_1,x_2)\Zsv(t,x_1)\Zsv(t,x_2) $
%and since $ \tilZ(t,x_1,x_2) = \Bdd(t,x_1,x_2)\Zsv(t,x_1)\Zsv(t,x_2) $,
%the desired result readily follows from the moment bound~\eqref{eq:Zmomt} from Proposition~\ref{prop:momt}.
\end{proof}

Recall the definitions of $ \Xgmomt $ and $ \Xggmomt $ from \eqref{eq:Xgmomt}--\eqref{eq:Xggmomt}.
We are now ready to derive the relevant bounds on these quantities.

\begin{cor}
\label{cor:cndExp}
Fix $ T<\infty $,
let $ \bar t\in\Z\cap[0,\e^{-2}T] $ and $ x^\star_1\neq x_2^\star\in \Z $.
%and write $ x^\star_i(s) := x^\star_i + \mue s-\mues \in \Xi(s) $.
We have
\begin{align*}
	\Xgmomt(\bar t,x^\star_1,x^\star_2),
	\quad
	\Xggmomt(\bar t,x^\star_1,x^\star_2)
	\
	\leq
	C(T) \,\e^{\frac12} \, e^{u\e(|x^\star_1|+|x^\star_2|)}.
\end{align*}
\end{cor}
\begin{proof}
This follows by inserting the bounds from Lemma~\ref{lem:cndEx:bd} into~\eqref{eq:cndEx1}--\eqref{eq:cndEx2}:
\begin{align*}
	\Xgmomt(\bar t,x^\star_1,x^\star_2)
	&\leq
	\e^{4} \Big( 2\sum_{s_1<s_2<\bar t} + \sum_{s_1=s_2<\bar t} \Big)
	C(T) \frac{\e^{-\frac12}}{\sqrt{s_2-s_1+1}} e^{u\e(|x_1|+|x_2|)}
	\leq
	C(T) \e^{\frac12}e^{u\e(|x^\star_1|+|x^\star_2|)},
\\
	\Xggmomt(\bar t,x^\star_1,x^\star_2)
	&\leq
	\e^{4} \Big( 2\sum_{s_1<s_2<\bar t} + \sum_{s_1=s_2<\bar t} \Big)
	C(T) \frac{\e^{-\frac12}}{\sqrt{s_2-s_1+1}} e^{u\e(|x_1|+|x_2|)}
	\leq
	C(T) \e^{\frac12}e^{u\e(|x^\star_1|+|x^\star_2|)}.
\end{align*}
Note that, with $ |x_i-x^\star_i|\leq 1 $,
we replaced $ x_i $ with $ x^\star_i $ at the cost of increasing the constant $ C(T) $
by factors of $ e^{u\e} \leq C(T) $ (with $ u=u(T) $).
\end{proof}

We are now ready to prove Proposition~\ref{prop:qv}.
\begin{proof}[Proof of Proposition~\ref{prop:qv}]
Fix $ T<\infty $, $ t \in [0,\e^{-2}T]\cap\Z $, and $ x_\star\in\Z $,
and write $ x_\star(s) := x_\star - \mue s +\mues $.
Given the decomposition in Corollary~\ref{cor:Theta:Exp}, it suffices to prove that
\begin{align}
	\label{eq:propqv:1}
	\Big\Vert \e^2 \sum_{s=0}^{t} A(s,x_\star(s)) \Big\Vert_2
	\leq
	\e^{\frac14} C(T)e^{C\e |x_\star|},
\end{align}
for $ A(s,x) = \e^\frac12\Xbd(t,x) $, $ \Xg(t,x) $, and $ \Xgg(t,x) $,
and for all $ \e>0 $ small enough.

Recall from~\eqref{eq:Xbd} that $ \Xbd(t,x) $ denotes a generic linear combination of $ \Zsv(t,x_1)\Zsv(t,x_2) $,
with random but uniformly exponentially decay coefficients $ \Decay(t,x_1,x_2;x) $.
Consequently,
\begin{align*}
	\Big\Vert \e^2 \sum_{s=0}^{t} \e^{\frac12} \Xbd(s,x_\star(s)) \Big\Vert_2
	\leq
	\e^{\frac12}
	\Bigg(
		\e^2 \sum_{s=0}^{t}
		\sum_{x_1,x_2\in\Xi(s)} e^{-\frac{1}{C}(|x_1-x_\star(s)|+|x_2-x_\star(s)|)} \Vert \Zsv(s,x_1)\Zsv(s,x_2) \Vert_2
	\Bigg).
\end{align*}
Given this,
together with $ |x_\star(s)-x_\star|\leq 1 $,
the statement~\eqref{eq:propqv:1} for $ A(s,x) = \e^\frac12\Xbd(t,x) $
readily follows from the moment bound~\eqref{eq:Zmomt} in Proposition~\ref{prop:momt}.

Next, recall that $ \Xg(t,x) $ and $ \Xgg(t,x) $ denote generic linear combinations of
$ \Zg(t,x_1,x_2) $ and $ \Xgg(t,x_1,x_2) $ with some deterministic coefficients~\eqref{eq:decay}
that decay exponentially off $ x $.
This gives
\begin{align*}
	\Big\Vert \e^2 \sum_{s=0}^{t} \Xg(s,x_\star(s)) \Big\Vert_2
	&\leq
	\sum_{x^\star_1<x^\star_2\in\Z}
	\ind_{\set{|x^\star_1-x^\star_2|>1}}
	e^{-\frac{1}{C}(|x^\star_1(s)-x_\star(s)|+|x^\star_2(s)-x_\star(s)|)}
	\Xgmomt(t,x^\star_1,x^\star_2)^{1/2},
\\
	\Big\Vert \e^2 \sum_{s=0}^{t} \Xgg(s,x_\star(s)) \Big\Vert_2
	&\leq
	\sum_{x^\star_1<x^\star_2\in\Z}
	e^{-\frac{1}{C}(|x^\star_1(s)-x_\star(s)|+|x^\star_2(s)-x_\star(s)|)}
	\Xggmomt(t,x^\star_1,x^\star_2)^{1/2},
\end{align*}
where $ x^\star_i(s):=x^\star_i - \mue s + \mues $.
Given this, the statement~\eqref{eq:propqv:1} for $ A(s,x) =  \Xg(t,x) $, and $ \Xgg(t,x) $,
readily follows from the bounds in Corollary~\ref{cor:cndExp}.
\end{proof}

\appendix
\section{Quadratic variation in ASEP}
\label{sect:ASEP}

In this appendix we expand upon the brief discussion from Sections \ref{sec:ASEPKPZ} and \ref{sect:introDual}
and explain how our Markov duality method can be applied to \ac{ASEP}, which is a simpler limit of the stochastic \ac{6V} model.  We will not carry out the necessary analysis, but rather just point to the main steps.

Recall that \ac{ASEP} is an interacting particle system on $ \Z $,
where particles inhabit sites index by $ \Z $ and jump left and right according to continuous time exponential clocks with rates $ \ell> 0 $ and $ r> 0 $
subject to exclusion (jumps to occupied sites are suppressed). We will assume that $\ell+r=1$ and set $\tau:= r/\ell$.
The \ac{ASEP} height function $ \Nasep(t,x) $ has $1$/$0$ slopes entering occupied/vacant sites as depicted in Figure \ref{fig.asep}.
%\cite{Bertini1997} proved that a variant of the Hopf--Cole transform of $ \Nasep(t,x) $ (recall \eqref{eq.mch}) converges to the solution of the \ac{SHE}. (Note, \cite{Bertini1997} assumed near-stationary initial data with density $ \den=\frac12 $.)
%As discussed in the introduction, the crucial step in this result the convergence of the quadratic martingale for this transformed equation to that of the \ac{SHE}. In \cite{Bertini1997} this is achieved through a nontrivial identity of the (semi)-discrete heat kernel
%that controls part of the underlying martingale. That approach does not seem to generalize to the stochastic \ac{6V} model, hence our development of the Markov duality method.
For \ac{ASEP} with near-stationary initial data of density $ \den=\frac12 $ we define a variant\footnote{This follows immediately from \eqref{eq.mch} by a simple tilting and centering.} of the Hopf--Cole transform of $ \Nasep(t,x) $ by
\begin{align*}
	\Zasep(t,x) := \tau^{\Nasep(t,x)-\frac12 x} e^{t(1-2\sqrt{\ell r})},
	\quad
	t\in[0,\infty),
	\
	x\in\Z.
\end{align*}
This solves the following microscopic \ac{SHE}:
\begin{align}
	\label{eq:asepSHE}
	d \Zasep(t,x)
	=
	\sqrt{\ell r} \Delta \Zasep(t,x) dt
	+
	d\mg(t,x),
\end{align}
where $ \Delta f(x):= f(x+1)+f(x-1)-2f(x) $ denotes the discrete Laplacian, and, for each $ x\in\Z $,
the process $ \mg(t,x) $, $ t\in\R_+ $, is a martingale.

Under weak asymmetry scaling, i.e., $ \tau=\taue:=e^{-\sqrt\e} $ and $ (t,x)\mapsto (\e^{-2}t,\e^{-1}x) $,
an informal scaling argument applied to~\eqref{eq:asepSHE} indicates that the equation should converge to the continuum \ac{SHE}.
Key to establishing this convergence is the identification of the limiting quadratic variation of $ \mg(t,x) $.
Under weak asymmetry scaling, the optional quadratic variation of $ \mg(t,x) $ reads
\begin{align}
	\label{eq:asepQV}
	d\langle \mg(t,x), \mg(t,x') \rangle
	=
	\e \ind_{\set{x=x'}} \Big( \big( \tfrac14 + \e^\frac12 \Bdd(t,x) \big) \Zasep^2(t,x) + \tilde{F}_\e(t,x) \Big) dt,
\end{align}
where, following notations in Section~\ref{sect:qv},
$ \Bdd(t,x) $ is a generic, uniformly bounded process, and
\begin{align}
	\label{eq:asepQV:BG}
	\tilde{F}_\e(t,x)
	:=
	\e^{-\frac12}\nabla \Zasep(t,x) \e^{-\frac12}\nabla \Zasep(t,x-1).
\end{align}

Referring to the r.h.s.\ of~\eqref{eq:asepQV}, we see that $ \e^\frac12 \Bdd(t,x) $
is indeed negligible compared to the constant $ \frac14 $ factor.
Key to identifying the limiting behavior is to argue that $ \tilde{F}(t,x) $ is also negligible.
With $ \nabla\Zasep(t,x) = (e^{-\sqrt\e \eta(t,x+1)}-1)\Zasep(t,x) $,
we indeed have $ \tilde{F}_\e(t,x) = \Bdd(t,x) \Zasep^2(t,x) $,
i.e., pointwise bounded up to a multiplicative factor of $ \Zasep^2(t,x) $.
On the other hand, it is conceivable that this term $ \tilde{F}(t,x) $ does not tend to zero pointwise,
i.e., $ \tilde{F}(t,x) \not\to_\text{P} 0 $.
The crux of the convergence result is to prove that  this term converges to zero after time-averaging:
%\note{perhaps we can say how fast this goes to zero under Bertini-Giacomin's bound versus ours?}
\begin{align}  \label{e:zero after time-averaging}
	\Ex\Big[ \Big( \e^{2}\int_0^{\e^{-2}T} \tilde{F}_\e(t,x) dt \Big)\Big]^2 \longrightarrow 0.
\end{align}
This is first achieved in \cite{Bertini1997} by showing the decay as $ t $ becomes large of the conditional expectation
\begin{align*}
	\Ex\big[ \tilde{F}_\e(t+s,x) \big| \filt(s) \big],
\end{align*}
where $ \filt $ denotes the canonical filtration of \ac{ASEP}.
Roughly speaking, the estimate starts by using~\eqref{eq:asepSHE} to develop a sequence of inequality that bounds the conditional expectation.
`Closing' the series of inequality relies crucially on an identity~\cite[(A.6)]{Bertini1997} for the (semi)-discrete heat kernel. We do not know of a way to generalize this approach from \cite{Bertini1997} to the stochastic \ac{6V} model setting.

Here we provide an alternative approach via duality.
The Markov duality method also begins with bounding conditional expectations.
However, instead of trying to close a sequence of inequalities,
this method provides \emph{direct} access to the conditional expectations.
First, the expression $ \tilde{F}_\e(t,x) $ is not convenient for our purpose.
Use  $ \nabla\Zasep(t,x) = (e^{-\sqrt\e (\eta^+(t,x)-\frac12) }-1)\Zasep(t,x) $ where $ \eta^+(t,x) := \eta(t,x+1) $, and Taylor expand
\begin{align} \label{e:grad-eta-Zasep}
 \nabla\Zasep(t,x) =\sqrt\e (\tfrac12 - \eta^{+}(t,x)) \Zasep + \e \Bdd(t,x)\Zasep(t,x) ,
 \end{align}
 where $ \Bdd(t,x)$ stands for a {\it generic} uniformly bounded process as in Section~\ref{sect:qv}.
We can then write $ \tilde{F}_\e(t,x) = F_\e(t,x) + \e^{1/2} \Bdd(t,x)\Zasep^2(t,x) $, where
\begin{align}
	\label{eq:asepQV:duality}
	F_\e(t,x)
	=
	\tfrac12 \Zg(t,x,x-1) + \tfrac12\Zg(t,x-1,x+1) + \tilZ(t,x-1,x),
\end{align}
where, following the notation in Section~\ref{sect:qv},
\begin{align*}
	\Zg(t,x_1,x_2) &:= (\e^{-\frac12}\nabla \Zasep(t,x_1)) \Zasep(t,x_2),
\\	
	\tilZ(t,x_1,t_2) &:= (\eta^+\Zasep)(t,x_1) (\eta^+\Zasep)(t,x_2) - \tfrac{1}{4} \Zasep(t,x_1) \Zasep(t,x_2).
\end{align*}
To see \eqref{eq:asepQV:duality}, we use \eqref{e:grad-eta-Zasep} just as in the proof of Lemma~\ref{lem:trade} (for the stochastic \ac{6V} model):
\begin{align*}
\tilde{F}_\e (t,x)  & = \big((\tfrac12 - \eta^+)  \Zasep\big)(t,x)  \big((\tfrac12 - \eta^+)  \Zasep\big)(t,x-1)
+ \e^{1/2} \Bdd(t,x)\Zasep^2(t,x)  \\
& =  \tfrac12 \big((\tfrac12 - \eta^+)  \Zasep\big)(t,x)   \Zasep(t,x-1)
+ \tfrac12 \big((\tfrac12 - \eta^+)  \Zasep\big)(t,x-1)   \Zasep(t,x)\\
&\qquad + \tilZ(t,x-1,x) + \e^{1/2} \Bdd(t,x)\Zasep^2(t,x)  \\
%&=  \frac12  \nabla \Zasep (t,x)   \Zasep(t,x-1)
%+ \frac12 \nabla \Zasep(t,x-1)   \Zasep(t,x)
% + \tilZ(t,x-1,x) + \e^{1/2} \Bdd(t,x)\Zasep^2(t,x)
&= \mbox{r.h.s of }\eqref{eq:asepQV:duality} + \e^{1/2} \Bdd(t,x)\Zasep^2(t,x) .
\end{align*}
In the last step, we replace $\Zasep(t,x)$ with $\Zasep(t,x\pm1)$
costing error of order $\e^{1/2} \Bdd(t,x)\Zasep(t,x)$.

As mentioned in Section~\ref{sect:introDual}, \ac{ASEP} enjoys self-duality
via the functions $ Q $ and $ \tilde{Q} $ defined therein.
Specifically, the $ k=2 $ duality translates (after tilting and centering) into the following statement, in which we used the notation
%($\TrPrasep$ represents the space-reversed transition probability)
\begin{align*}
	\SGasep\big((y_1,y_2),(x_1,x_2);t\big) := e^{2t(1-2\sqrt{\ell r})} \tau^{-\frac12(x_1+x_2-y_1-y_2)} \Pr_\text{ASEP}\big((y_1,y_2)\to (x_1,y_2) ;t\big).
\end{align*}
\begin{prop}
\label{prop:asepDual}
For all $ x_1<x_2\in\Z $ and $ t,s\in [0,\infty) $, we have
\begin{align}
	\notag
	&\Ex\Big[ \Zasep(t+s,x_1) \Zasep(t+s,x_2) \Big\vert \filt(s) \Big]
\\
	\label{eq:asepDual}
	&\hphantom{\Ex\Big( (\etacnt^{+}\Zasep)}= \sum_{y_1<y_2\in\Z}
	\SGasep\big((y_1,y_2),(x_1,x_2);t\big) \Zasep(s,y_1) \Zasep(s,y_2),
\\
	\notag
	&\Ex\Big[ (\eta^{+}\Zasep)(t+s,x_1) (\eta^{+}\Zasep)(t+s,x_2) \Big\vert \filt(s) \Big]
\\
	\label{eq:asepDual:}
	&\hphantom{\Ex\Big( (\etacnt^{+}\Zasep)}=
	\sum_{y_1<y_2\in\Z}
	\SGasep\big((y_1,y_2),(x_1,x_2);t\big) \big(\eta^{+}\Zasep\big)(s,x_1)\big(\eta^{+}\Zasep\big)(s,x_2).
\end{align}
\end{prop}

Proposition~\ref{prop:asepDual} provides the necessary ingredients for expressing conditional expectations for the relevant quantities.
Specifically, with $ \tilZ(t,x-1,x) $ being an linear combination the two observables in \eqref{eq:asepDual} and in \eqref{eq:asepDual:}
at $ (x_1,x_2)=(x-1,x) $ we have
\begin{align}
	\label{eq:asep:tilZ}
	\Ex\Big[\tilZ(t+s,x-1,x)\Big|\filt(s)\Big] = \sum_{y_1<y_2\in\Z} \SGasep\big((y_1,y_2),(x-1,x);t\big) \tilZ(t,y_1,y_2).
\end{align}
Likewise, $ \Zg(t,x,x-1) $ is the difference of $ \Zasep(t,x+1)\Zasep(t,x-1) $ and $ \Zasep(t,x)\Zasep(t,x-1) $.
Taking the difference of~\eqref{eq:asepDual} for $ (x_1,x_2) = (x+1,x-1) $ and for $ (x,x-1) $ gives
\begin{align*}
	\Ex\Big[\Zg & (t+s,x,x-1)  \Big|\filt(s)\Big] \\
	& = \sum_{y_1<y_2\in\Z} \e^{-\frac12}\nabla_{x_1}\SGasep\big((y_1,y_2),(x_1,x_2);t\big)\big|_{(x_1,x_2)=(x,x-1)}
		\Zasep(s,y_1)\Zasep(s,y_2),
\end{align*}
where $ \nabla_{x_1} $ denotes the discrete (forward) gradient acting on the variable $ x_1 $.
Similarly,
\begin{align*}
	\Ex\Big[\Zg& (t+s,x-1,x+1)\Big|\filt(s)\Big] \\
	&= \sum_{y_1<y_2\in\Z} \e^{-\frac12}\nabla_{x_1}\SGasep\big((y_1,y_2),(x_1,x_2);t\big)\big|_{(x_1,x_2)=(x-1,x+1)}
		\Zasep(s,y_1)\Zasep(s,y_2).
\end{align*}
%In fact, again using \eqref{e:grad-eta-Zasep}, it is possible to rewrite $\tilZ(t,y_1,y_2)$
%in \eqref{eq:asep:tilZ} into linear combinations of terms of the form $\Zg$ (up to some higher order errors),
%and then by a summation by parts, one can take the $\e^{-\frac12}\nabla$ (hidden in $\Zg$)
%onto $\SGasep$. Lemma~\ref{lem:cndEx} provides the details of analogous arguments
%for S6V model.

From the perspective of duality, roughly speaking,
the mechanism of decay in $ t\to\infty $ arises from the discrete gradient $ \nabla_{x_1} $.
The semigroup $ \SGasep $ behaves similar to (two copies of) the heat kernel,
so that $\sum_{y_1<y_2}\SGasep\big((y_1,y_2),(x_1,x_2);t\big) = \mathcal O(1)$,
and each gradient of $ \SGasep $ effectively produces a factor of $ t^{-1/2} $ for large $ t $.
Under the scaling $ \e^{-2} $ of time, namely $ t^{-1/2}\approx \e^1$, we expect to trade in $ \e^{-1/2}\nabla $ for $ \e^{-1/2}\e^{1}=\e^{1/2} \to 0 $.
In other words, the key heuristic is that the l.h.s of \eqref{e:zero after time-averaging} behaves as
\begin{align} \label{e:heuristicASEP}
\Ex\Big[ \Big( \e^{2}\int_0^{\e^{-2}T} \tilde{F}_\e(t,x) dt \Big)\Big]^2
\approx
\e^4 \int_0^{\e^{-2}T} \!\!\!\! \int_0^{\e^{-2}T} \!\!\! \frac{\e^{-1/2}}{\sqrt{t_1-t_2}}\,dt_1dt_2
 \approx \e^{\frac12} \to 0 .
\end{align}
Note that the identity~\eqref{eq:asep:tilZ} in its current form does not involve {\it gradients} of $ \SGasep $.
This identity can, however, be rewritten via Taylor expansion and summation by parts in a form that exposes the decay in $ t\to\infty $.
We do not perform this procedure here, and direct the readers to Lemma~\ref{lem:cndEx},
where the exact same procedure in carried out for the stochastic \ac{6V} model.
Specifically, the identity \eqref{eq:tilZ:condEx} therein holds with
$ (\SGasep,\Zasep,\Z) $ in place of $ (\SGe,\Zsv,\Xi(s)) $, and with $ s,t\in [0,\infty) $ instead of $ \Z_{\geq 0} $.

Given the preceding discussion,
the task for bounding conditional expectations boils down to estimating the semigroup $ \SGasep $ and its gradients.
Thanks to Bethe ansatz, $ \SGasep $ permits an explicit, analyzable formula in terms double contour integrals.
Under weak asymmetry scaling, we write $ \SGasep=\SGeasep $ and the formula reads
\begin{align*}
	\notag
	\SGeasep\big((y_1,y_2),(x_1,x_2);t\big)
	:=
	\oint_{\Circ_r}\oint_{\Circ_r}
	\Big(
		z_1^{x_1-y_1}z_2^{x_2-y_2} -\SGasepfr(z_1,z_2) z_1^{x_2-y_1}z_2^{x_1-y_2}
	\Big)
	\prod_{i=1}^2
	\frac{e^{t\SGasepE(z_i)}dz_i}{2\pi\img z_i},
\end{align*}
where $ \Circ_r $ is a counter-clockwise oriented, circular contour
 centered at origin, with a large enough radius $ r $
so as to include all poles of the integrand, and
\begin{align*}
	\SGasepfr(z_1,z_2)
	:=
	\frac{1+z_1z_2-(e^{-\frac12\sqrt\e}+e^{\frac12\sqrt\e}) z_2}{1+z_1z_2-(e^{-\frac12\sqrt\e}+e^{\frac12\sqrt\e}) z_1},
	\quad
	\SGasepE(z) := \sqrt{\ell r} \big(z+z^{-1}-2\big).
\end{align*}
This contour integral formula is amenable to steepest decent analysis.
Careful analysis jointly in $ (x_1,x_2,y_1,y_2,t) $ should produce the relevant estimates on $ \SGeasep $ and its gradient (the result and proof should be analogous to Proposition~\ref{prop:SG}). We do not pursue this analysis here.

\addtocontents{toc}{\protect\setcounter{tocdepth}{0}}
\bibliographystyle{alphaabbr}
\bibliography{Reference}

\end{document}